\newif\ifarXiv
\arXivtrue

\ifarXiv
\documentclass[12pt,reqno]{amsart}

\else
\documentclass[english, reqno]{smfart}
\usepackage[english,francais]{babel}
\fi

\usepackage{amssymb}

\ifarXiv
\usepackage[left=1.25in,right=1.25in, bottom=1.25in]{geometry}
\usepackage{times}
\usepackage{newtxmath}
\usepackage[cal=cm]{mathalpha}
\usepackage{xcolor}
\definecolor{MyGreen}{rgb}{0,.6,.2}
\definecolor{MyDarkBlue}{rgb}{.1,.1,.75}
\usepackage[unicode,colorlinks=true, citecolor=MyGreen, linkcolor=MyDarkBlue]{hyperref}
\hypersetup{bookmarksopen=true}

\else
\fi

\usepackage{verbatim}
\usepackage{graphicx}
\usepackage{enumerate}
\usepackage{enumitem}
\usepackage{multicol}
\usepackage{mathrsfs}

\usepackage{cleveref}

\usepackage{setspace}

\renewcommand{\Re}{\operatorname{Re}} 

\renewcommand{\bar}{\overline}
\renewcommand{\epsilon}{\varepsilon}

\newcommand{\KN}{\ast} 

\newcommand{\Defn}[1]{{\boldmath\it\bfseries #1}}
\newcommand{\Reals}{\mathbb{R}}
\newcommand{\Ints}{\mathbb{Z}}
\newcommand{\Nats}{\mathbb{N}}
\newcommand{\Hyp}{\mathbb{H}}
\newcommand{\opP}{\mathcal P}
\newcommand{\HypB}{\mathbb{B}}
\newcommand{\curC}{\mathscr{C}}

\def\ip<#1,#2>{\left\langle#1,#2\right\rangle}
\newcommand{\pipesep}{\,|\, }
\newcommand{\deltrange}{\mathcal{D}}

\DeclareMathOperator{\Riem}{Riem}

\DeclareMathOperator{\R}{R}

\DeclareMathOperator{\Hess}{Hess}

\DeclareMathOperator{\tr}{tr}
\DeclareMathOperator{\supp}{supp}
\DeclareMathOperator{\Id}{id}
\DeclareMathOperator{\im}{im}
\DeclareMathOperator{\Vol}{Vol}
\DeclareMathOperator{\diam}{diam}
\DeclareMathOperator{\coker}{coker}
\DeclareMathOperator{\codim}{codim}
\DeclareMathOperator{\Sym}{Sym}
\DeclareMathOperator{\GL}{GL}
\DeclareMathOperator{\GS}{GS}

\DeclareMathOperator{\cont}{c}

\renewcommand{\tilde}{\widetilde}

\newcommand{\loc}{{\rm loc}}

\newcommand{\ghyp}{{\breve g}}
\newcommand{\Rhn}{\breve{\R}_n}

\def\dual{^*}

\theoremstyle{plain}
\newtheorem{theorem}{Theorem}
\newtheorem{lemma}[theorem]{Lemma}
\newtheorem{proposition}[theorem]{Proposition}
\newtheorem{corollary}[theorem]{Corollary}
\newtheorem{definition}[theorem]{Definition}

\newtheorem{lemma-tp}[theorem]{Lemma (To Be Proved)}

\newtheorem{assumption}{Assumption}

\makeatletter
\newtheorem*{rep@theorem}{\rep@title}
\newcommand{\newreptheorem}[2]{%
\newenvironment{rep#1}[1]{%
 \def\rep@title{#2 \ref{##1}}%
 \begin{rep@theorem}}%
 {\end{rep@theorem}}}
\makeatother
\newreptheorem{theorem}{Theorem}

\numberwithin{theorem}{section}
\numberwithin{equation}{section}

\title[Sobolev-class asymptotically hyperbolic manifolds]{Sobolev-class asymptotically hyperbolic manifolds and the Yamabe problem}

\ifarXiv
\else
\alttitle{Vari\'et\'es asymptotiquement hyperboliques de la classe de Sobolev et le probl\`eme de Yamabe}
\fi

\author{Paul T.~Allen} 
\address{Lewis \& Clark College, Portland, OR, USA}
\email{ptallen@lclark.edu}

\author{John M.~Lee}
\address{University of Washington, Seattle, WA, USA}
\email{johnmlee@uw.edu}

\author{David Maxwell}
\address{University of Alaska, Fairbanks, AK, USA}
\email{damaxwell@alaska.edu}

\ifarXiv
\date{June 26, 2022}
\else
\date{\today}
\fi

\begin{document}

\ifarXiv

\begin{abstract}
We consider asymptotically hyperbolic manifolds whose metrics have Sobolev-class regularity, and introduce several technical tools for studying PDEs on such manifolds. Our results 
employ two novel families of function spaces suitable 
for metrics potentially having a large amount of interior differentiability, but
whose Sobolev regularity implies only a H\"older continuous conformal structure at infinity. 
We establish Fredholm theorems for elliptic operators arising from 
metrics in these families.

To demonstrate the utility of our methods, we solve the Yamabe problem in this category.
As a special case, we show that the asymptotically hyperbolic Yamabe problem is solvable so long as the metric admits a $W^{1,p}$ conformal compactification, with $p$ greater than the dimension of the manifold.  
\end{abstract}
\maketitle
\tableofcontents

\fi

\ifarXiv
\setlength{\parskip}{1ex}
\parindent 0pt
\fi

\ifarXiv
\fi

\section{Introduction}

Given a Riemannian manifold $(M^n,g)$, the Yamabe problem seeks a metric $\tilde g$ that is conformally related to $g$ and that has constant scalar curvature.
Writing $\tilde g = \Theta^{q_n-2} g$, where $q_n = \frac{2n}{n-2}$, the requirement that the scalar curvature of $\tilde g$ be equal to the constant $\R^*$ is equivalent to requiring that the conformal factor $\Theta$ satisfy the semilinear equation
\begin{equation}
\label{intro:yamabe-eqn}
-(q_n+2)\Delta_g\Theta + \R[g]\Theta = \R^*\Theta^{q_n-1},
\end{equation}
where $\R[g]$ is the scalar curvature of $g$, and where our convention for the Laplacian is $\Delta_g = \tr_g\nabla^2$.
Note that $q = \frac{2n}{n-2}$ is the critical exponent associated to the Sobolev embedding of $W^{1,2}$ into $L^q$.
In the case where $M$ is compact without boundary, the problem has a long and famous history, with important work by Trudinger and Aubin and Schoen ultimately leading to a satisfactory resolution of the problem; see \cite{LeeParker}, and the references within.

In this paper, we consider the case where $(M^n,g)$ is asymptotically hyperbolic in the sense that $(M^n,g)$ is a complete manifold,
and that the sectional curvatures of $g$ approach $-1$ at infinity.
In this setting, we seek to conformally deform $g$ to have scalar curvature $\Rhn:= -n(n-1)$, the curvature of $n$-dimensional hyperbolic space.
We further consider the case where the metric $g$ has Sobolev-scale regularity sufficient that $g$ be continuous, but possibly is not continuously differentiable.

We are motivated by regularity results in previous analyses of the Yamabe problem in the asymptotically hyperbolic setting, along with related work in geometric analysis.
Most previous studies of the Yamabe problem in the asymptotically hyperboic setting have assumed that the metric $g$ admits a regular conformal structure at infinity \cite{AnderssonChruscielFriedrich, Mazzeo-Yamabe}.
Related studies of elliptic operators arising from asymptotically hyperbolic metrics also require the geometry to have a regular conformal structure (see e.g.~\cite{Andersson-EllipticSystems, Lee-FredholmOperators}).

Yet even if $g$ is assumed to have a smooth conformal compactification, the asymptotic behavior of generic solutions to \eqref{intro:yamabe-eqn} is incompatible with the resulting metric $\tilde g$ also admitting smooth conformal structure \cite{AnderssonChruscielFriedrich}; similar results hold for the more general Einstein constraint equations \cite{AnderssonChrusciel-Dissertationes, AHEM}.
Geometries that are smooth on the interior, but that do not admit smooth conformal compactifications also appear in the study of boundary regularity for Poincar\'e-Einstein metrics \cite{Chrusciel-Delay-Lee-Skinner}.
We emphasize that in all these examples, the limited boundary regularity of solutions is due to the algebraic structure of the relevant equations, and is thus an inherent feature of the problem, rather than an artifact of the methods.
Thus for geometric analysis problems involving  asymptotically hyperbolic geometry, the appropriate regularity setting is one for which metrics are sufficiently regular on $M$ for theory of elliptic PDE to be established, while simultaneously metrics are permitted to be less regular along the conformal boundary.

The work \cite{WAH} introduced a category of asymptotically hyperbolic metrics that can have a great deal of H\"older regularity on $M$, but which may only have Lipschitz-continuous conformal compactifications.
Fredholm theory for elliptic operators arising from these metrics was established, following the approach of \cite{Lee-FredholmOperators}, and was subsequently used to show that the Yamabe problem can be solved in this category.
In particular, if a metric is initially of this class, then there is a conformally related metric of the same class having constant scalar curvature.

In this work we address the Sobolev setting.
In particular, we introduce 
Banach spaces that allow us to define two different classes of metrics that have $H^{s,p}$ regularity on $M$, but whose conformal compactification is not better than $H^{m,p}$ on $\bar M$.
Under conditions on $s$, $m$, and $p$ that ensure the conformal structure at infinity is continuous, we show that the sectional curvatures approach $-1$ at infinity in an appropriate sense, and thus that the asymptotically hyperbolic condition is well-defined in these categories.
We establish Fredholm theory for elliptic operators arising from such metrics, and subsequently show that the Yamabe problem can be solved in these categories: every asymptotically hyperbolic metric of each regularity class is conformally related to a unique constant-scalar-curvature metric of the same class.

Our motivation for considering Sobolev-class metrics is two-fold.  First, there
is the intrinsic interest of extending theory that exists for 
other settings to
the asymptotically hyperbolic category.  
Our work provides a number of technical tools, notably the Fredholm theory for geometric elliptic operators of \S\ref{sec:fredholm} and the asymptotic decomposition of tensor fields in \S\ref{sec:asymptotics}, that are relevant to a wide range of geometric analysis problems.
Regularity of the type we consider arises
naturally in variational problems and indeed our technique for solving the 
Yamabe problem highlights this character.  We show that the Yamabe problem is
solvable so long as the metric admits a $W^{1,p}$ 
conformal compactification on $\bar M$ with $p>n$, the threshold needed to ensure
H\"older continuity. Greater regularity (in the interior, or at the boundary) is
obtained secondarily after a weak solution has been established.

Second, metrics having Sobolev-class regularity arise routinely in general relativity.
As discussed in \cite{AnderssonChruscielFriedrich, AHEM},
the Yamabe problem is a special case of the conformal method used to generate
initial data for the gravitational Cauchy problem.  It is natural to construct this
data in regularity classes compatible with the $L^2$-based
hyperbolic theory typically used in evolution problems (see \cite{Ringstrom-CauchyProblem} and references therein).  
Moreover, mixed elliptic-hyperbolic schemes for solving the Einstein equations \cite{AnderssonMoncrief-EllipticHyperbolic} involve solving elliptic problems
 arising from Sobolev-class metrics at each time step during the evolution.
An understanding of the asymptotic features of spacetimes generated with these methods
has the potential to 
shed light on open questions regarding conformal structures in general relativity \cite{Friedrich-Peeling}.

\subsection*{Main results}
In order to give a precise statement of our main results, we introduce a number of key definitions and notations, beginning with a discussion of regularity classes.

Let $\bar M$ be a smooth, connected, compact, $n$-dimensional manifold with boundary; let $M$ be the interior of $\bar M$ and denote the boundary by $\partial M$.
Let $\rho$ be a smooth {defining function} on $\bar M$, meaning $\rho\colon \bar M \to [0,\infty)$ with $\rho^{-1}(0) = \partial M$ and $d\rho\neq 0$ along $\partial M$.
A metric $g$ on $M$ is said to be \Defn{conformally compact} if $\bar g = \rho^2 g$ extends to a non-degenerate, and at least continuous, metric on $\bar M$.

Fixing a smoothly conformally compact reference metric, for $k\in \mathbb N_{\geq 0}$, $1\leq p< \infty$, $\delta\in \mathbb R$ we define weighted Sobolev spaces $W^{k,p}_\delta(M) = \rho^\delta W^{k,p}(M)$ to have norm $\|u \|_{W^{k,p}_\delta(M)} = \|\rho^{-\delta} u \|_{W^{k,p}(M)}$.
Using an equivalent norm, defined via a collection of local M\"obius pa\-ram\-et\-riz\-a\-tions, whose scaling is compatible with hyperbolic geometry (see \S\ref{sec:coords}), we generalize these Sobolev spaces to weighted Bessel potential spaces $H^{s,p}_\delta(M)$ having non-integer order of differentiability; see \S\ref{secsec:fspace-defs}.
The collection of local M\"obius parametrizations further allows us to define spaces $X^{s,p}_\delta(M)$ of functions having local differentiability measured in $H^{s,p}$, but with global behavior measured in a uniform sense akin to H\"older spaces; again, see \S\ref{secsec:fspace-defs}.
These spaces, which we refer to as \Defn{Gicquaud-Sakovich spaces}, generalize those introduced in \cite{GicquaudSakovich}, and play a key technical role in our analysis.

As discussed above, regularity classes appropriate for geometric analysis problems in the asymptotically hyperbolic setting must have admit functions with (possibly a great deal) more regularity on the interior $M$ than on the compactified manifold $\bar M$.
However, neither the standard function spaces on $M$, nor those on $\bar M$, sufficiently handle these competing demands.
For example, functions in $C^{k,\alpha}(\bar M)$ have at most $C^{k,\alpha}(M)$ interior regularity, while functions in $C^{k,\alpha}(M)$ may not have well-defined boundary values at all.

In the H\"older setting, this situation was addressed in \cite{WAH} by introducing spaces $\mathscr C^{k,\alpha;m}(M)$, where we require $k\geq m$, of functions satisfying 
\begin{equation*}
C^{k,\alpha}_m(M) 
\subset \mathscr C^{k,\alpha;m}(M)
\subset C^{m-1,1}(\bar M)\cap C^{k,\alpha}_0(M).
\end{equation*}
In \S\ref{sec:asymptotic-spaces} we define analogous Sobolev-scale spaces $\mathscr H^{s,p;m}(M)$ and $\mathscr X^{s,p;m}(M)$, which we refer to as \Defn{fortified spaces}, for $1<p<\infty$, $s\in\mathbb R$, and $m\in \mathbb N_{\geq 0}$ with $s\geq m$.
These are subspaces of weighted $H^{s,p}$ and $X^{s,p}$ spaces respectively but
enforce additional regularity at the boundary.
For functions, we have
\begin{equation*}
\begin{gathered}
H^{s,p}_{m-n/p}(M)
\subset\mathscr H^{s,p;m}(M) 
\subset H^{m,p}(\bar M)\cap H^{s,p}_{-n/p}(M),
\\
X^{s,p}_m(M)
\subset \mathscr X^{s,p;m}(M)
\subset H^{m,p}(\bar M) \cap X^{s,p}_0(M),
\end{gathered}
\end{equation*}
while similar inclusions exist for tensor fields of various ranks; see Lemma \ref{lem:basic-roman-curly} and Proposition \ref{prop:curly-to-M-bar}.  In particular,
every fortified space is a subspace of some $H^{m,p}(\bar M;E)$.

As a consequence of this boundary regularity, 
when $m\ge 1$, elements of fortified spaces possess boundary traces. 
Importantly, these spaces also admit more detailed boundary structure.
The  Asymptotic Structure Theorem \ref{thm:curly-split} shows that a tensor field $u$ belonging
to either $\mathscr H^{s,p;m}$ or $\mathscr X^{s,p;m}$ admits a decomposition $u = \tau + r$, where $\tau$ is smooth on the interior $M$ and describes the leading order asymptotic behavior, and where the remainder $r$ decays as rapidly as the (limited) boundary regularity permits.  Theorem \ref{thm:curly-split} is a consequence of the fine structure provided by a coordinate-dependent variation of Taylor's Theorem,
Theorem \ref{thm:Taylor}.

We now define the two classes of metrics we consider.
To motivate our definition, recall that the curvature operator $\Riem[g]\colon \Lambda^2(TM) \to \Lambda^2(TM)$ relates to that of $\bar g = \rho^2 g$ via the formula
\begin{equation}
\label{intro:relate-curvature-operators}
\Riem[g] = -|d\rho|^2_{\bar g} \Id
+ 2\rho\,\delta \KN \Hess_{\bar g}(\rho)^\sharp
+ \rho^2 \Riem[\bar g];
\end{equation}
here $\Id$ and $\delta$ are the identity operators on $\Lambda^2(TM)$ and $\Lambda^1(TM)$, respectively, and $\KN$ is the Kulkarni-Nomizu product; see \cite{WAH}.
Thus if $\bar g$ is sufficiently regular that the first term on the right side of \eqref{intro:relate-curvature-operators} dominates as $\rho \to 0$, then the curvatures of $g$ are determined by the value of $|d\rho|^2_{\bar g}$ along $\partial M$.
In particular, if $|d\rho|^2_{\bar g}=1$ along $\partial M$, the curvatures of $g$ approach those of hyperbolic space as $\rho\to 0$.
This motivates the following definition:
a continuously conformally compact metric $g$ is said to be \Defn{asymptotically hyperbolic}  either of class
\begin{align}
&\label{intro:H-class}
\mathscr H^{s,p;m} \text{ with }s\geq m > n/p, 1<p<\infty, \text{ or }
\\
&\label{intro:X-class}
\mathscr X^{s,p;m}\text{ with }s \geq m \geq 1\text{ and }s> n/p, 1<p<\infty,
\end{align}
if $\bar g = \rho^2 g$ is an element of $\mathscr H^{s,p;m}(M)$ or $\mathscr X^{s,p;m}(M)$, respectively, and if 
\begin{equation}
\label{intro:ah-condition}
|d\rho|_{\bar g} = 1 \text{ along }\partial M.
\end{equation}
The condition \eqref{intro:ah-condition} is well-defined due to Proposition \ref{prop:curly-to-M-bar}, which implies that under either \eqref{intro:H-class} or \eqref{intro:X-class} we have
\begin{equation}
\label{intro:alpha-exists}
\bar g = \rho^2 g \in C^{0,\alpha}(\bar M) \text{ for some }\alpha\in (0,1].
\end{equation}
The descriptor `asymptotically hyperbolic' is justified by Proposition \ref{prop:scalar-curvature-H}, which shows that the curvature operator of a metric $g$ of class \eqref{intro:H-class} or \eqref{intro:X-class} for which \eqref{intro:ah-condition} holds indeed satisfies
\begin{equation*}
\Riem[g] = -\Id + \text{ remainder},
\end{equation*}
where the remainder is the sum of a smooth term that is $O(\rho)$ at $\partial M$ and a less regular term decaying as rapidly as the boundary regularity of $\bar g$ permits.

Our main result concerning the Yamabe problem is that it can be solved in both the $\mathscr H^{s,p;m}$ and $\mathscr X^{s,p;m}$ asymptotically hyperbolic categories.
\begin{theorem}[Yamabe Theorem]
\label{thm:main-yamabe}
\strut
\begin{enumerate}
\item Suppose $g$ is an asymptotically hyperbolic metric of class \eqref{intro:H-class} with $m\leq n$.
Then there exists a unique conformal factor $\Theta \in \mathscr H^{s,p;m}(M)$ such that $\Theta\big|_{\partial M} =1$ and such that $\R[\Theta^{q_n-2}g]  = -n(n-1)$.

\item Suppose $g$ is an asymptotically hyperbolic metric of class \eqref{intro:X-class} with $m<n$.
Then there exists a unique conformal factor $\Theta \in \mathscr X^{s,p;m}(M)$ such that $\Theta\big|_{\partial M} =1$ and such that $\R[\Theta^{q_n-2}g] = -n(n-1)$.

\end{enumerate}
\end{theorem}
Note the dimensional restrictions on $m$ in Theorem \ref{thm:main-yamabe}, which represent a restriction on the regularity of solutions along the conformal boundary. 
The results of \cite{AnderssonChruscielFriedrich} show that (generically) these restrictions are sharp; see also \cite{AnderssonChrusciel-Dissertationes}, \cite{WAH}.
Note also that working in the $\mathscr H$ category allows for more regularity on $\bar M$, albeit in the $L^p$, rather than pointwise, sense. 

Theorem \ref{thm:main-yamabe} combines the results of Theorems \ref{thm:yamabe-H} and \ref{thm:Yamabe-X}, which address the Yamabe problem by constructing solutions to \eqref{intro:yamabe-eqn}.
This, in turn, requires theory for elliptic differential operators arising from metrics of classes \eqref{intro:H-class} and \eqref{intro:X-class}.
Our approach follows that of \cite{Lee-FredholmOperators} and \cite{WAH}, with significant technical work required to adapt those methods to the  Sobolev regularity setting.
We restrict attention to operators satisfying the following assumption.

\setcounter{assumption}{15}
\begin{assumption}
\label{Assume-P}
We assume that $\mathcal P[g]$ is a $d^\text{th}$-order differential operator, 
mapping sections of bundle $E$ to that same bundle, 
arising from a metric $g$ of class \eqref{intro:H-class} or \eqref{intro:X-class}, 
and satisfying the following.

\begin{enumerate}
\item $\mathcal P[g]$ acts on sections of a \Defn{geometric tensor bundle $E$}, a subbundle of some tensor bundle $T^{k,l}M$ associated to an invariant subspace of the standard representation of $O(n)$ on $T^{k,l}\mathbb R^n$; see \S\ref{secsec:gtb-early}.

\item $\mathcal P[g]$ is \Defn{geometric}, meaning that $\mathcal P[g]u$ is a linear function of the components of $u$ and their derivatives up to order $d$, with the coefficient of any $k^\text{th}$ order derivative of $u$ being a universal polynomial in components of $g$ and $\det(g_{ij})^{-1/2}$ that involves at most $d-k$ derivatives of $g$.

\item $\mathcal P[g]$ is \Defn{scale natural}, meaning that in any coordinate expression of $\mathcal P[g]u$, the coefficients of the $k^\text{th}$ derivative of $u$ involve at most $d-k$ derivatives of components of $g$ distributed over all factors of that coefficient.

\item $\mathcal P[g]$ is \Defn{elliptic} in the usual sense, a condition that is well-defined due to \eqref{intro:alpha-exists}.

\item $\mathcal P[g]$ satisfies the \Defn{weak $L^2$ conditions} 
\begin{equation}
\label{intro:weak-L2}
s\geq \frac{d}{2}
\quad\text{ and }\quad
\frac{1}{p}-\frac{s}{n} \leq \frac{1}{2} - \frac{d/2}{n}.
\end{equation}

\item  $\mathcal P[g]$ is \Defn{formally self-adjoint} in the sense that for all sections $u,v\in H^{s,p}_0(M)$ we have
\begin{equation*}
\int_M\langle u, \mathcal P[g]v\rangle_{g}\,dV_g
=
\int_M\langle \mathcal P[g]u, v\rangle_{g}\,dV_g,
\end{equation*}
which is interpreted in the sense of duality pairings if $s<d$; see \S\ref{secsec:duality-early}.

\end{enumerate}
\end{assumption}
Two remarks on these assumptions are in order.
First, given a metric $g$ of class \eqref{intro:H-class} or \eqref{intro:X-class}, the multiplication properties of Sobolev spaces imply that there exists a set of {compatible Sobolev indices} $\mathcal S^{s,p}_d$ such that if $(\sigma, q)\in \mathcal S^{s,p}_d$ and if $\delta\in \mathbb R$, then a $d^\text{th}$-order  operator $\mathcal P[g]$ satisfying Assumption \ref{Assume-P} gives rise to continuous maps
\begin{equation}
\label{intro:continuous-maps}
\begin{gathered}
H^{\sigma, q}_\delta(M;E)\to H^{\sigma-d, q}_\delta(M;E),
\\
X^{\sigma, q}_\delta(M;E)\to X^{\sigma-d, q}_\delta(M;E);
\end{gathered}
\end{equation}
see Lemma \ref{lem:geometric-op-mapping} and Proposition \ref{prop:L-mapping-S}.
The scale-natural condition is necessary in order to deduce such elementary mapping properties. 
This condition excludes the operator $\Delta_g + \R[g]^2$, as some lowest order terms involve a total of four metric derivatives ($\partial^2 g\partial^2 g$, etc.), while the operator $\Delta_g + \R[g]$ is permitted under our hypotheses.

Second, the weak $L^2$ conditions of Assumption \ref{Assume-P} imply that the $\mathcal S^{s,p}_d$ is nonempty, a situation that occurs if and only if $(d/2,2)\in \mathcal S^{s,p}_d$; see Lemma \ref{lem:S-members}.
Thus all operators under consideration admit a weak $L^2$ theory, a fact that appears in the characterization of the kernel of $\mathcal P[g]$ in the Fredholm results of Theorem \ref{thm:main-fredholm} below.
Note also that for second-order operators, the weak $L^2$ conditions hold as a consequence of the regularity assumed for the metric $g$.

While the maps \eqref{intro:continuous-maps} are continuous for any $\delta \in \mathbb R$, they may not be Fredholm.
In order to determine those values of $\delta$ for which $\mathcal P[g]$ is Fredholm, we consider the corresponding operator $\breve{\mathcal P} = \mathcal P[\breve g]$ arising from the metric $\breve g$ on the Poincar\'e ball hyperbolic space $\mathbb B$, and acting on sections of corresponding bundle $\breve E$.
As detailed in \cite{Lee-FredholmOperators}, for each geometric operator satisfying Assumption \ref{Assume-P}, there exists a set of complex characteristic exponents corresponding to powers of $\rho$ whose coefficients are arbitrary in formal series solutions to $\breve{\mathcal P}u = f$.
The location of these exponents in the complex plane is symmetric about the line $\Re(z) = n/2-w$, where $w$ is the \Defn{weight} of the bundle on which $\breve{\mathcal P}$ acts, defined as the covariant rank minus the contravariant rank.
The \Defn{indicial radius} of $\breve{\mathcal P}$ is defined to be the smallest $R\geq 0$ such that $\breve{\mathcal P}$ has a characteristic exponent with real part equal to $n/2 -w + R$.
The indicial radius of many important operators can be easily calculated; see \cite[Chapter 7]{Lee-FredholmOperators}.

We use the indicial radius of $\breve{\mathcal P}$ to determine those weights $\delta$ for which $\mathcal P[g]$ is Fredholm as follows.
For any $R\geq 0$ and $1<q\leq \infty$, set
\begin{equation}
\label{intro:define-DqR}
\mathcal D_q(R) 
= \left\{ \delta\in \mathbb R\colon \left| \delta +\frac{n-1}{q} - \frac{n-1}{2}\right|<R \right\}.
\end{equation}
Extending the work of \cite[Chapter 5]{Lee-FredholmOperators}, 
Theorems \ref{thm:ball-H-iso} and \ref{thm:ball-X-iso} imply  
that operator $\breve{\mathcal P}$, having indicial radius $R$, 
provides isomorphisms
\begin{equation*}
\begin{gathered}
H^{\sigma, q}_\delta(\mathbb B;\breve E)\to H^{\sigma-d, q}_\delta(\mathbb B;\breve E) \text{ and }
\\
X^{\sigma, q}_\delta(\mathbb B;\breve E)\to X^{\sigma-d, q}_\delta(\mathbb B;\breve E),
\end{gathered}
\end{equation*}
provided $\delta \in \mathcal D_q(R)$ and provided the following isomorphism assumption holds.

\setcounter{assumption}{8}
\begin{assumption}\label{Assume-I}
We assume that 
\begin{equation*}
\breve{\mathcal P}\colon H^{d,2}_0(\mathbb B;\breve E)\to  H^{0,2}_0(\mathbb B;\breve E)
\end{equation*} 
is an isomorphism. This assumption is equivalent, see \cite{Lee-FredholmOperators}, to assuming the existence of compact $K\subset \mathbb B$ and constant C such that
\begin{equation*}
\| u \|_{L^2(\mathbb B;\breve E)} \leq C \|\breve{\mathcal P}u \|_{L^2(\mathbb B;\breve E)}
\qquad\text{ for all }u\in C^\infty_\text{cpct}(\mathbb H\setminus K;\breve E).
\end{equation*}
\end{assumption}
\noindent
Assumption \ref{Assume-I} implies that the indicial radius $R$ of $\breve{\mathcal P}$ is positive; see \cite{Lee-FredholmOperators}.

The hypothesis that $g$ is asymptotically hyperbolic implies that $\mathcal P[g]$ can be closely approximated by $\breve{\mathcal P}$ in a small neighborhood of each point on $\partial M$.
These approximations, together with the isomorphism properties of $\breve{\mathcal P}$, allow us to adapt the parametrix construction in \cite{Lee-FredholmOperators} to operators arising from metrics of class \eqref{intro:H-class} and \eqref{intro:X-class}.
The result is the following Fredholm theorem, which combines the results of Theorems \ref{thm:H-fredholm} and \ref{thm:X-fredholm}, and is accompanied by structure theorems for $\mathcal P[g]$; see \S\ref{secsec:structure-theorems}.

\begin{theorem}[Fredholm Theorem]
\label{thm:main-fredholm}
Suppose $g$ is an asymptotically hyperbolic of class \eqref{intro:H-class} or \eqref{intro:X-class}, and that $\mathcal P = \mathcal P[g]$ is a $d^\text{th}$ order operator satisfing Assumptions \ref{Assume-P} and \ref{Assume-I}.
Let $R$ be the indicial radius of $\mathcal P[\breve g]$.

\begin{enumerate}
\item If $(\sigma, q)\in \mathcal S^{s,p}_d$ and $\delta\in \mathcal D_q(R)$ then $\mathcal P\colon H^{\sigma,q}_\delta(M;E) \to H^{\sigma-d,q}_\delta(M;E)$ is Fredholm of index zero.

\item If $(\sigma, q)\in \mathcal S^{s,p}_d$ and $\delta\in \mathcal D_\infty(R)$ then $\mathcal P\colon X^{\sigma,q}_\delta(M;E) \to X^{\sigma-d,q}_\delta(M;E)$ is Fredholm of index zero.

\end{enumerate}
In both cases, the kernel of $\mathcal P$ is the kernel of $\mathcal P\colon H^{d/2,2}_0(M;E)\to H^{-d/2,2}_0(M;E)$.
\end{theorem}

\subsection*{Organization}
This paper is organized in order to efficiently escort the reader to our main results concerning the Yamabe problem in the Sobolev category, which are contained in \S\ref{sec:yamabe}. Thus a number of more technical theorems are deferred to subsequent sections and the appendices.

In \S\ref{sec:coords} we introduce a preferred collection of coordinate systems, which are used in \S\ref{sec:elementary-spaces} to define H\"older and Sobolev spaces on $M$, along with the ``local Sobolev spaces" of Gicquaud-Sakovich \cite{GicquaudSakovich} that play a key technical role in our analysis.
Section \ref{sec:asymptotic-spaces} contains the definition of
the fortified  spaces.
These spaces are used in \S\ref{sec:ah-metrics} to define, and establish basic properties of, asymptotically hyperbolic metrics in the Sobolev setting.

Section \ref{sec:yamabe} is the heart of the paper, containing our main results concerning the Yamabe problem in the Sobolev category.

The results of \S\ref{sec:yamabe} require technical results that are deferred to subsequent sections: a detailed theory for the asymptotic structure of tensor fields appears in \S\ref{sec:asymptotics},  properties of elliptic operators arising from Sobolev-scale metrics are established in \S\ref{sec:differential-ops}, and in \S\ref{sec:fredholm} we establish Fredholm theory for such operators.
Finally, the  appendices contain additional technical results.

\ifarXiv
\subsection*{Acknowledgements}  
We thank Iva Stavrov for helpful conversations.
This work was supported in part by NSF grant DMS-1263544. 
\fi

\section{Coordinate systems}
\label{sec:coords}

In this section we establish notation and define
a number of coordinate systems required to define
the function spaces that follow.
The constructions are, for the most part,
a specialization of those of \cite[Chapter 2]{Lee-FredholmOperators}, with certain choices imposed to simplify later arguments.

\subsection{Hyperbolic space}
Let $\mathbb H$ be the $n$-dimensional half-space model of hyperbolic space with coordinates
$$
(x,y) = (x^1,\dots, x^{n-1}, y=x^n)\in \mathbb R^{n-1}\times (0,\infty)
$$
and metric
\begin{equation*}
\breve g = y^{-2}\left( (dx^1)^2 + \dots + (dx^{n-1})^2 + dy^2\right).
\end{equation*}

We use two prototypical classes of open sets in $\mathbb H$.
The first are the geodesic balls $B^{\mathbb{H}}_r$ of radius $r>0$
centered at $(x,y)=(0,1)$; it will be convenient at times to use the containment estimate
\begin{equation}\label{eq:ball-containment}
B^{\mathbb{H}}_r\subset \{(x,y)\colon |x|<\sinh(r), e^{-r}<y<e^r\}
\end{equation}
where $|\cdot|$ denotes the Euclidean norm.

The second class of open sets consists of neighborhoods near
infinity.  Given
 $r>0$ we define
\[
Y_r = \{ (x,y) \in \mathbb H \colon |x|<r, 0<y<r\}.
\]

\subsection{Neighborhoods near infinity in \texorpdfstring{$M$}{M}}
For $r>0$ a \Defn{collar neighborhood} is the region
\[
A_r = \{ p\in M \colon \rho(p) < r\},
\]
and we set $\bar A_r = \{ p\in \bar M \colon \rho(p) \le r\}$.
Using compactness of the boundary we can find $r_0>0$ and a smooth
projection
$\Pi_\partial: \bar A_{r_0}\to \partial M$ such that
\[
\Pi_\partial\times\rho: \bar A_{r_0} \to \partial M \times [0,r_0]
\]
is a diffeomorphism.
Fixing this choice of $\Pi_\partial$, we will
sometimes tacitly identify $\bar A_{r_0}$ with
$\partial M\times [0,r_0]$.

Cover the boundary $\partial M$ with finitely many
smooth
\Defn{boundary charts}
$(\Omega\subset \partial M,\theta)$ where each map $\theta$
can be extended smoothly
to a neighborhood of $\bar \Omega$.
Using compactness of the boundary,
and shrinking $r_0$ if needed, we can assume that
for each $\hat p\in \partial M$ we can find
a boundary chart $(\Omega_{\hat p},\theta_{\hat p})$
and a neighborhood $U(\hat p)\subset \Omega_{\hat p}$
such that
\[
\theta_{\hat p}(U(\hat p)) = \{x\in \Reals^{n-1}: |x-\theta_{\hat p}(\hat p)|<r_0\}.
\]
$\theta_{\hat p}(\hat p)=0$.
Let $Z_{r_0}(\hat p) = \Pi_\partial^{-1}(U(\hat p)) \cap A_{r_0}$
which is conveniently identified with $U(\hat p)\times (0,r_0)$.
We define smooth \Defn{coordinates near infinity}
$\Theta_{\hat p} = (\theta_{\hat p}\circ \Pi_\partial - \theta_{\hat p}(\hat p),\rho)$, obtaining diffeomorphisms
\begin{equation*}
\Theta_{\hat p}:  Z_{r_0}(\hat p) \to Y_{r_0}.
\end{equation*}
Extending our earlier notation, for $r\in (0,r_0]$ we have the \Defn{neighborhood near infinity}
\begin{equation*}
Z_r(\hat p) = \Theta^{-1}_{\hat p}(Y_r).
\end{equation*}
\begin{figure}
\centering
\includegraphics[width=.9\textwidth]{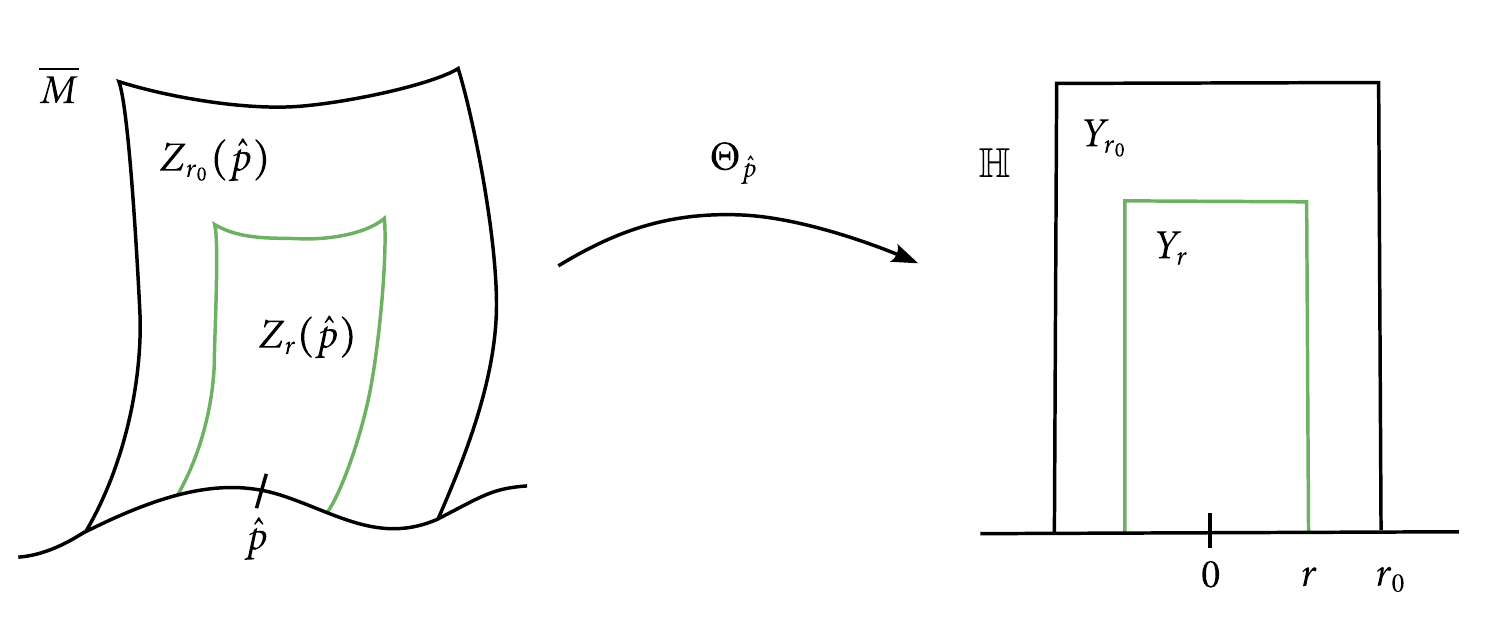}
\caption{Neighborhoods at infinity.}\label{fig:nbhd-at-inf}
\end{figure}

\subsection{M\"obius parametrizations}

For each $p_0\in A_{r_0/8}$, a \Defn{M\"obius parametrization}
centered at $p_0$ is a map $\Phi:B_2^{\mathbb H}\to M$ of the following
form.  First, pick some $\hat p\in \partial M$ such that
$p_0$ lies in $ Z_{r_0/8}(\hat p)$, and let $\Theta=\Theta_{\hat p}$
be the associated coordinates near infinity.
Writing
$\Theta(p_0)=(\theta_0,\rho_0)$ we observe
\begin{equation}\label{eq:ball-containment-gen}
(\theta_0,0) + \rho_0 B_2^\mathbb{H} \subset Y_{r_0};
\end{equation}
demonstrating this containment is an easy exercise using
\eqref{eq:ball-containment} and the
identities $e^2<8$ and $\sinh(2) < 4$.
Relying on the containment \eqref{eq:ball-containment-gen}
we define the M\"obius parametrization
$\Phi :B_2^\Hyp\to M$ by
\begin{equation}
\label{define-Mobius}
(\Theta\circ \Phi)(x,y) = (\theta_0+\rho_0 x, \rho_0 y).
\end{equation}

The restricted images $\Phi(B_{1/2}^\mathbb{H})$ of the M\"obius parametrizations cover $A_{r_0/8}$,
and we choose finitely many additional maps $\Phi:B^\Hyp_2\to M$,
which we will also call M\"obius parametrizations, such
that their restricted images $\Phi(B_{1/2}^\mathbb{H})$ cover $M\setminus A_{r_0/8}$.

For the purposes of defining function spaces it is useful
to choose a preferred countable collection of M\"obius parametrizations.
From the construction of \cite{Lee-FredholmOperators}
Lemma 2.2 we select a countable collection of
points $p_i$ in $A_{r_0/8}$ and corresponding M\"obius parametrizations
$\Phi_i$ such that
\begin{enumerate}
\item the sets $\Phi_i(B_{1/2}^\mathbb{H})$ cover $M$,
\item the sets $\Phi_i(B_{2}^\mathbb{H})$ are uniformly locally finite: there
is an integer $N$ independent of $i$
such that $\Phi_i(B_{2}^\mathbb{H})\cap \Phi_j(B_{2}^\mathbb{H})=\emptyset$ for all but at most
$N$ indices $j$,
\item \label{part:diam-bound-to-I-bound}
more generally, given a metric $\check g$ on $M$
with smooth compactification $\rho^2 \check g$ on $\bar M$,
for every $K>0$ there exists $N_K$
such that if $\diam_{\check g}(U)\le K$
then $\Phi_j(B_2^\mathbb{H})\cap U=\emptyset$ for all but
at most $N_K$ indices $j$.
\end{enumerate}

\section{Elementary function spaces}
\label{sec:elementary-spaces}

A \Defn{tensor bundle} over $\bar M$ is a bundle
$T^{k,l}\bar M$ of tensors with covariant rank $k$
and contravariant rank $l$.
In this section we gather together facts concerning
function spaces of sections of a tensor bundle $E$ over $\bar M$ and its induced
bundle, which we also denote by $E$, over $M$.
In addition to tensor bundles, we also work with `geometric tensor bundles',
which are certain subbundles of tensor bundles
and potentially depend on a metric, and its regularity,
for their definition.  So as to avoid undue initial complication from these
regularity concerns, unless otherwise noted we let $E$
denote a tensor bundle over $\bar M$ throughout this section.
Section \ref{secsec:gtb-early} summarizes the alternations needed to adapt
the results below to geometric tensor bundles.

\subsection{Definitions}\label{secsec:fspace-defs}

Let $\hat g$ be a smooth reference metric on $\bar M$ and
let $\check g= \rho^{-2}\hat g$, which
is a geodesically complete metric on $M$ with sectional curvatures approaching $-|d\rho|^2_{\hat g}$ as $\rho \to 0$; see  \eqref{relate-curvature-operators} below.
The notation $\check g$ is meant to invoke our notation $\breve g$
for the hyperbolic metric.

For $k\in \Nats_{\ge 0}$ and $\alpha\in [0,1]$ we have
the usual Banach space $C^{k,\alpha}(\bar M;E)$
of $k$-times continuously differentiable sections of $E$
with $k^{\rm th}$ derivatives $\hat\nabla^k u$
that are H\"older continuous of order $\alpha$ with
respect to the reference metric $\hat g$.

Given $\delta\in \Reals$
the (interior) weighted H\"older space
$C^{k,\alpha}_\delta(M;E)$ consists of the continuous sections
of $E$ over $M$ satisfying
\begin{equation}\label{eq:holder-def}
\|u\|_{C^{k,\alpha}_\delta(M;E)} =
\sup_i \rho_i^{-\delta} \|\Phi_i^*u \|_{C^{k,\alpha}(B^\Hyp_1)}<\infty,
\end{equation}
where $\rho_i = \rho(p_i)$ and where the $C^{k,\alpha}$ norm on $B^\mathbb{H}_r$ is computed with respect to the Euclidean metric.  When $\alpha=0$
we sometimes omit it from the notation.

For simplicity, for product
bundles $M\times F$ we replace $E$ in the notation 
with the common fiber, e.g. $C^{k,\alpha}_\delta(M;\Reals)$.
It is often clear from context what bundle $E$ is under consideration,
and we drop this part of the notation, especially in the bodies of proofs,
whenever it is convenient to do so to minimize notational clutter.
Indeed, strictly speaking, equation \eqref{eq:holder-def} should
have involved the quantities
$\|\Phi_i^*u \|_{C^{k,\alpha}(B^\Hyp_1;\Phi_i^* E)}$.  

We write $C^{\infty}_{\rm loc}(M;E)$ for the smooth sections of $E$ over $M$
and $C^\infty_{\rm cpct}(M;E)$ consists of the compactly supported
smooth sections. 

For $k\in\Nats_{\ge 0}$ and $1\le p<\infty$
the Sobolev space $W^{k,p}(\bar M)$
consists of distributional sections $u$ of $E$ satisfying
\begin{equation*}
\| u\|_{W^{k,p}(\bar M;E)}:= \left(\sum_{j=0}^k\int_{\bar M} |
\hat\nabla{}^j u |_{\hat g}^p\,dV_{\hat g} \right)^{1/p}<\infty.
\end{equation*}

The interior Sobolev spaces $W^{k,p}(M;E)$ associated with $\check g$
are defined in terms of the analogous norm
\begin{equation}\label{eq:Wkp}
\| u\|_{W^{k,p}(M;E)}:= \left(\sum_{j=0}^k\int_{ M} |
\check\nabla{}^j u |_{\check g}^p\,dV_{\check g} \right)^{1/p}.
\end{equation}
Given $\delta\in\Reals$, we define the weighted Sobolev space
\begin{equation}\label{eq:Wkpdelta}
W^{k,p}_{\delta}(M;E) = \rho^\delta W^{k,p}(M;E)
\end{equation}
with norm $\|u\|_{W^{k,p}_{\delta}(M;E)} = \|\rho^{-\delta} u\|_{W^{k,p}(M;E)}$. We write $L^p_\delta(M;E)$ instead of $W^{0,p}_{\delta}(M;E)$.
If $U$ is an open subset of $M$, the $W^{k,p}_\delta(U)$ norm
is defined by replacing $M$ with $U$ in expression \eqref{eq:Wkp}.

The weighted Sobolev spaces can be generalized to non-integer order of differentiability as
follows.  Recall that for $s\in \mathbb R$ and $1<p<\infty$ the Bessel potential space $H^{s,p}(\mathbb R^n)$ is defined by
\begin{equation*}
H^{s,p}(\mathbb R^n) = \{ u \in S^\prime(\mathbb R^n) \colon \| \mathcal F^{-1}\Lambda^s \mathcal F u\|_{L^p(\mathbb R^n)} <\infty \};
\end{equation*}{}
here $S^\prime(\mathbb R^n)$ is the dual space of the Schwartz functions, $\Lambda(\xi) = (1+|\xi|^2)^{1/2}$, and $\mathcal F$ is the Fourier transform.  We use the same notation for spaces of
tensor-valued distributions over $\Reals^n$ without reference
to the specific bundle, in which case the norm is the sum of
the norms of the tensor components.  If  $U\subset \Reals^n$ is an open set, $H^{s,p}(U)$
consists of the restrictions of distributions in $H^{s,p}(\Reals^n)$
to $U$ and is endowed with the quotient norm.

Let $\chi\ge 0$ be a smooth function on $\Reals^n$  with support contained in
$B_2^\Hyp$ that equals $1$ on $B_{1/2}^\Hyp$.
For $s\in \mathbb R$, $1<p<\infty$, and $\delta\in \mathbb R$ we define the
\Defn{weighted Bessel potential space} $H^{s,p}_\delta(M;E)$
in terms of the preferred M\"obius parametrizations from \S\ref{sec:coords}
to be those distributional fields $u$ such that
\begin{equation}\label{eq:Hspdelta}
\| u \| _{H^{s,p}_\delta(M;E)} := \left( \sum_i \rho_i^{-\delta p} \| \chi \Phi_i^*u\|_{H^{s,p}(\mathbb R^n)}^p\right)^{1/p} <\infty;
\end{equation}
here $\rho_i = \rho(\Phi_i(0,1))$, the value of $\rho$ at the point determining the parametrization $\Phi_i$.
As seen below, these are generalizations of the weighted Sobolev spaces $W^{k,p}_\delta(M)$; furthermore, they are the asymptotically hyperbolic analogues of the
 weighted spaces $h^{s}_{p,\mu}(\Reals^n)$ from \cite{Triebel:1976hk}, which are used in
asymptotically Euclidean contexts.

Gicquaud and Sakovich
\cite{GicquaudSakovich} introduced a class of novel function spaces
with local differentiability measured in Sobolev scales but with
global behavior measured in a uniform sense akin to weighted H\"older norms.
Generalizing their work to real scales of differentiability,
for $s,\delta\in\Reals$ and $1<p<\infty$
we define the \Defn{Gicquaud-Sakovich} space
$X^{s,p}_\delta(M;E)$ to be the set of distributions such that the norm
\begin{equation}
\label{GS-int-norm}
\| u\|_{X^{s,p}_\delta(M;E)} = \sup_i \rho_i^{-\delta} \|\chi \Phi_i^*u \|_{H^{s,p}_{\delta}(\Reals^n)}
\end{equation}
is finite.

We obtain spaces of smooth tensor fields by intersecting the corresponding spaces with finite regularity.
Thus
\begin{gather*}
H^{\infty, p}_\delta(M;E) = \bigcap_{s\in \mathbb R} H^{s,p}_\delta(M;E),
\\
C^\infty_\delta(M;E) = \bigcap_{k\in \mathbb N \atop \alpha\in [0,1]} C^{k,p}_\delta(M;E).
\end{gather*}
Just as the H\"older parameter $\alpha$ plays no role in the latter intersection, the Sobolev embedding of Proposition \ref{prop:SobolevEmbedding} below shows that $p$ is not relevant to the Gicquaud-Sakovich case; thus we define
\begin{equation*}
X^\infty_\delta(M;E) = \bigcap_{s\in \mathbb R} X^{s,p}_\delta(M;E)
\end{equation*}
for any $1<p<\infty$.
Note that the Sobolev embedding, together with Lemma \ref{lem:basic-inclusions}, implies that $X^\infty_\delta(M;E) = C^\infty _\delta(M;E)$.
However, we use the notation $X^\infty_\delta(M;E)$ to emphasize the space's origin.

From the form of equations \eqref{eq:Hspdelta} and \eqref{GS-int-norm} it is clear that
$H^{s,p}_\delta(M;E)$ and $X^{s,p}_\delta(M;E)$
are each part of a spectrum of spaces with the former involving an $\ell^p$ norm and
the latter an $\ell^\infty$ norm.  One could presumably work with the full range of spaces
replacing these with general $\ell^q$ norms and still obtain good elliptic theory.  But special
arguments would presumably be needed in $\ell^\infty$ case regardless, so for simplicity
we work only with these two special cases.

We have the following norm equivalences.  The first two allow us to elide cutoff functions
in the norm definitions, and the second two show that the new weighted spaces reduce to
known spaces when $s=k\in\Nats_{\ge 0}$.
\begin{lemma}\label{lem:many-norms}
Let $s,\delta\in\Reals$ and $1<p<\infty$.
\begin{enumerate}
\item \label{part:noHcutoff} For all $1/2<r<2$,
\[
\displaystyle \|u\|_{H^{s,p}_{\delta}(M;E)} \sim \left[ \sum_i \rho_i^{-\delta p}\|\Phi_i^* u\|_{H^{s,p}(B_r^\Hyp)}\right]^\frac{1}{q}.
\]

\item\label{part:noXcutoff} For all $1/2<r<2$,
\[
\displaystyle \|u\|_{X^{s,p}_{\delta}(M;E)} \sim \sup_i \rho_i^{-\delta}\|\Phi_i^* u\|_{H^{s,p}(B_r^\Hyp)}.
\]

\item\label{part:Hk-is-Wk} For $k\in\Nats_{\ge 0}$
\[
\|u\|_{H^{k,p}_{\delta}(M;E)} \sim \|u\|_{W^{k,p}_{\delta}(M;E)}.
\]

\item\label{part:Xk-is-Xk} Taking a supremum over all M\"obius parametrizations,
\[
\displaystyle \|u\|_{X^{s,p}_{\delta}(M;E)} \sim \sup_{\Phi} \|\Phi^* u\|_{H^{s,p}(B_r^\Hyp)}.
\]
\end{enumerate}
\end{lemma}
These facts are proved in Appendix \ref{app:bpspaces} except for part \eqref{part:Hk-is-Wk}
which follows from part \eqref{part:noHcutoff} and \cite{Lee-FredholmOperators}
Lemma 3.5.  When $s=k\in\Nats_{\ge 0}$
the equivalent norm in part \eqref{part:Xk-is-Xk} is essentially the
norm of the `local weighted spaces' from \cite{GicquaudSakovich} and techniques from
Appendix \ref{app:bpspaces} can be used to show that the $X^{k,p}_{\delta}(M;E)$ norm
here is indeed equivalent to the norm from that earlier work.

\subsection{Embeddings and dense subsets}
The following elementary inclusions are essentially immediate consequences of the definitions involved. The only subtle point is the strict inequalities that
appear relating weight parameters $\delta,\delta'$;
 the same phenomenon
occurs for integer-based Sobolev spaces (cf. \cite{Lee-FredholmOperators} Lemma 3.6).

\begin{lemma}[Basic inclusions on $M$]\label{lem:basic-inclusions}
Suppose $1<p<\infty$, $s,\delta\in\Reals$, $k\in\Nats_{\ge 0}$, and $0\le\alpha\le 1$.
The following embeddings are continuous under the listed restrictions.
\begin{align*}
& H^{s,p}_\delta(M;E) \hookrightarrow  H^{s,p^\prime}_{\delta^\prime}(M;E),
& &p\geq p^\prime, \quad \delta + \frac{n-1}{p} > \delta^\prime + \frac{n-1}{p^\prime};
\\
& H^{s,p}_\delta(M;E) \hookrightarrow X^{s,p^\prime}_\delta(M;E),
&& p \geq p^\prime;
\\
& X^{s,p}_\delta(M;E) \hookrightarrow X^{s,p^\prime}_\delta(M;E),
&& p \geq p^\prime;
\\
& C^{k,\alpha}_\delta(M;E) \hookrightarrow X^{k,p}_\delta(M;E);
\\
& C^{k,\alpha}_\delta(M;E) \hookrightarrow H^{k,p}_{\delta^\prime}(M;E),
&\quad& \delta > \delta^\prime + \frac{n-1}{p};
\\
& X^{s,p}_\delta(M;E) \hookrightarrow H^{s,p}_{\delta^\prime}(M;E),
&\quad& \delta > \delta^\prime + \frac{n-1}{p}.
\end{align*}
\end{lemma}

As discussed in Appendix \ref{app:bpspaces},
Sobolev embedding for the weighted spaces with general $s\in\Reals$
follows from Sobolev embedding for $H^{s,p}(B_r^\Hyp)$
and techniques identical to those used for spaces with $s\in\Nats_{\ge 0}$.
\begin{proposition}[Sobolev embeddings]
\label{prop:SobolevEmbedding}
Suppose $1<p<\infty$, $r,s,\delta\in\Reals$, $k\in\Nats_{\ge 0}$ and $0<\alpha<1$.
The following embeddings are continuous under the listed restrictions.
\begin{align*}
&H^{s,p}_\delta(M;E) \hookrightarrow H^{r,q}_\delta(M;E),
& \quad&  \frac{1}{p} - \frac{s}{n} \leq \frac{1}{q}-\frac{r}{n}
,\quad r\le s,\quad p\le q;
\\
&X^{s,p}_\delta(M;E) \hookrightarrow X^{r,q}_\delta(M;E),
& \quad&  \frac{1}{p} - \frac{s}{n} \leq \frac{1}{q}-\frac{r}{n},\quad r\le s;
\\
& H^{s,p}_\delta(M;E) \hookrightarrow C^{k,\alpha}_\delta(M;E),
& \quad& \frac{1}{p} - \frac{s}{n} \leq -\frac{k+\alpha}{n};
\\
& X^{s,p}_\delta(M;E) \hookrightarrow C^{k,\alpha}_\delta(M;E),
&\quad & \frac{1}{p} - \frac{s}{n} \leq -\frac{k+\alpha}{n}.
\end{align*}
\end{proposition}

Appendix \ref{app:bpspaces} contains proofs of following compactness and density results.
\begin{proposition}[Rellich Lemma]
\label{prop:rellich-lemma}
Suppose $1<p<\infty$ and $s, s^\prime,\delta,\delta'\in \mathbb R$.
The following inclusions are compact under the listed restrictions.
\begin{align*}
&H^{s,p}_\delta(M;E) \to  H^{s^\prime, p}_{\delta^\prime}(M;E),
&\quad & s> s^\prime,\quad \delta > \delta^\prime;
\\
&X^{s,p}_\delta(M;E) \to  X^{s^\prime, p}_{\delta^\prime}(M;E),
& &s> s^\prime,\quad \delta > \delta^\prime.
\end{align*}
\end{proposition}
\begin{proposition}[Density of Smooth Elements]\label{prop:density-early}
Suppose $1<p<\infty$ and $s,\delta\in\Reals$.
\begin{enumerate}
\item\label{part:Hsp-smooth-dense-early}
$C^\infty_{\rm cpct}(M;E)$ is dense in $H^{s,p}_{\delta}(M;E)$.
\item\label{part:Xsp-smooth-dense-early}
$C^\infty_{\rm loc}(M;E)\cap X^{s,p}_{\delta}(M;E)$ is dense in
$X^{s,p}_{\delta}(M;E)$.
\end{enumerate}
\end{proposition}
Note that in general $C^\infty_\delta(M;E)$ is not dense in $X^{s,p}_{\delta}(M;E)$;
see Proposition \ref{prop:Cinf-delta-not-dense-in-X}.

\subsection{Covariant derivatives}
When performing computations in the coordinates of a M\"obius parametrization,
is useful to be able to transfer
between the the pullback connection $\Psi^* \hat \nabla$ and the connection
$\nabla_{\mathbb E}$ of the Euclidean connection with respect to local coordinates.
One can generally substitute one for the other, and the following lemma is
the tool needed to justify this.  See also Lemma \ref{lem:CMbar-into-CMw}
below which concerns transfering tensors such as $\hat g$ on $\bar M$
to tensors in weighted spaces on $M$.

As an application, Lemma \ref{lem:alt-H-X-norms} below constructs equivalent norms for Sobolev and Gicquaud-Sakovich spaces using $\hat\nabla$.

\begin{lemma}\label{lem:christoffel-bound}
Let $\Phi_i$ be a M\"obius parametrization and let
$D = \Phi_i^* \hat\nabla -\nabla_{\mathbb E}$.
For each $k\in\Nats_{\ge0}$ the coordinate expressions
\begin{align}
\label{eq:local-metric-bound}
&(\rho_i^{-1} \nabla_{\mathbb E})^{k} \rho_i^{-2}\Phi_i^*\hat g,
\\
\label{eq:local-metric-inv-bound}
&(\rho_i^{-1} \nabla_{\mathbb E})^{k} \rho_i^{2}\Phi_i^*\hat g^{-1},
\\
\label{eq:christoffel-bound}
&(\rho_i^{-1} \nabla_{\mathbb E})^{k} \rho_i^{-1} D
\end{align}
are uniformly bounded (with respect to $i$) in any $C^\ell(B_2^\Hyp)$.
\end{lemma}
\begin{proof}
Writing $\hat g = \hat g_{ab}d\Theta^a d\Theta^b$ with respect to
the boundary coordinates associated with $\Phi_i$,
\begin{align*}
(\Phi_i^* \hat g) = (\rho_i)^2 \hat g_{ab}
(\theta_0+\rho_i x,\rho_i y) dx^a dx^b\\
(\Phi_i^* \hat g^{-1}) = (\rho_i)^{-2} \hat g^{ab}\partial_{x^a}\partial_{x^b}.
\end{align*}
Since $\hat g$ and $\hat g^{-1}$ are smooth up to the boundary of $\bar M$,
bounds for the expressions \eqref{eq:local-metric-bound}
and \eqref{eq:local-metric-inv-bound} follow from
bounds for derivatives in boundary coordinates along with the
observation that every derivative of the pullbacks introduces
an additional factor of $\rho_i$.  Bounds for
 \eqref{eq:christoffel-bound} follow from writing $D$ as a
 contraction of
 \[
(\Phi_i^* \hat g)^{-1} \nabla_{\mathbb E} \Phi_i^* \hat g
= \rho_i^2(\Phi_i^* \hat g)^{-1} \nabla_{\mathbb E} \rho_i^2\Phi_i^* \hat g.
 \]
and bounds for \eqref{eq:local-metric-bound}--\eqref{eq:local-metric-inv-bound}.
\end{proof}

\begin{lemma}\label{lem:nabla-HX-map} Let $1<p<\infty$, $s\in\Reals$ and $\delta\in\Reals$.
Then $\hat \nabla$
and $\check \nabla$ are continuous maps
\begin{align*}
H^{s,p}_\delta(M;E) \to H^{s-1,p}_\delta(M;T^{0,1}M\otimes E),
\\
X^{s,p}_\delta(M;E) \to X^{s-1,p}_\delta(M;T^{0,1}M\otimes E).
\end{align*}
\end{lemma}
\begin{proof}
The continuity of $\hat \nabla$ follows from working in the
domains of M\"obius parametrizations and Lemma \ref{lem:christoffel-bound}.
Since $\Phi_i^* \check g = y^{-2}\rho_i^{-2} \check g$
a computation shows
$\Phi_i^*\check \nabla  = \Phi_i^* \hat \nabla + \Gamma$
where $\Gamma$ is a universal tensor with coefficients only involving
coordinate derivatives of $y$ and the continuity of $\check\nabla$ follows.
\end{proof}

The following lemma, a consequence of the Gagliardo-Nirenberg inequality, is used below to establish alternate norms for Sobolev and Gicquaud-Sakovich spaces. 

\begin{lemma}\label{lem:missing}
Let $1<p<\infty$, $m\in\Nats_{\ge 0}$, and $s \in\Reals$ with
$s\ge m$. Suppose for some $r>0$ that $u$ is a
real-valued function with
$u\in H^{0,p}(B^\Hyp_r)$
and $\nabla^m_{\mathbb E} u\in H^{s-m,p}(B^\Hyp_r)$.  Then
$u\in H^{s,p}(B^\Hyp_r)$ and
\[
\|u\|_{H^{s,p}(B^\Hyp_r)} \sim
\|u\|_{H^{0,p}(B^\Hyp_r)}
+ \|\nabla^m_{\mathbb E} u\|_{H^{s-m,p}(B^\Hyp_r)}
\]
with implicit constants independent of $u$.
\end{lemma}
\begin{proof}
Let $k = \lfloor s\rfloor$ so $s=k+\alpha$ for some $0\le \alpha < 1$.
Observe that since
$\|\nabla_{\mathbb E}^k u\|_{H^{\alpha,p}(B_r^\Hyp)}\lesssim \|\nabla_{\mathbb E}^m u\|_{H^{s-m,p}(B_r^\Hyp)}$
it suffices to prove the result assuming $m=k$.  That is, we wish to show
\[
\|u\|_{H^{s,p}(B_r^\Hyp)} \lesssim
\|u\|_{H^{0,p}(B_r^\Hyp)}
+ \|\nabla^k_{\mathbb E} u\|_{H^{\alpha,p}(B_r^\Hyp)};
\]
the reverse inequality is obvious.

If $u\in H^{0,p}(B_r^\Hyp)$ and $\nabla_{\mathbb E}^k u\in H^{\alpha,p}(B_r^\Hyp)$
then $\nabla_{\mathbb E}^k u\in H^{0,p}(B_r^\Hyp)$ as well. So the
Gagliardo– {Nirenberg} interpolation
inequality implies $u\in H^{k,p}(B_r^\Hyp)$ and
\begin{equation}\label{eq:extension-0}
\|u\|_{H^{k,p}(B_r^\Hyp)}\lesssim \|u\|_{H^{0,p}(B_r^\Hyp)} + \|\nabla_{\mathbb E}^k u\|_{H^{0,p}(B_r^\Hyp)}.
\end{equation}
This completes the proof when $\alpha=0$.

Now suppose $0<\alpha <1$. Then
\cite[Theorem 3.3.5]{triebel2010theory}  implies that an equivalent norm on $H^{s,p}(B_r^\Hyp)$ is
\[
\|u\|_{H^{s,p}(B_r^\Hyp)} =  \|u\|_{H^{k+\alpha,p}(B_r^\Hyp)} \sim \sum_{0\le j \le k} \|\nabla_{\mathbb E}^j u\|_{H^{\alpha,p}(B_r^\Hyp)}.
\]
Moreover, if $j<k$ then $\|\nabla_{\mathbb E}^j u\|_{H^{\alpha,p}(B_r^\Hyp)}\lesssim \|\nabla_{\mathbb E}^j u\|_{H^{1,p}(B_r^\Hyp)}
\lesssim \|u\|_{H^{k,p}(B_r^\Hyp)}$.  Hence
\begin{equation}\label{eq:extension-1}
\|u\|_{H^{s,p}(B_r^\Hyp)} \lesssim \|\nabla^k_{\mathbb E} u\|_{H^{\alpha,p}(B_r^\Hyp)}
+ \|u\|_{H^{k,p}(B_r^\Hyp)}.
\end{equation}
The result now follows from equations \eqref{eq:extension-0} and \eqref{eq:extension-1}.
\end{proof}

\begin{lemma}\label{lem:alt-H-X-norms}
Let $1<p<\infty$, $m\in\Nats$, and $s\in\Reals$ with
$s\ge m$.  We have the following
norm equivalences for sections of tensor bundle $E$.
\begin{enumerate}
  \item\label{part:alt-H-norm} $\displaystyle \|u\|_{H^{s,p}_\delta(M;E)} \sim \|u\|_{H^{0,p}_{\delta}(M;E)}
  + \|\hat \nabla^m u\|_{H^{s-m,p}_{\delta}(M;T^{0,m}M\otimes E)}$.

  \item $\displaystyle \|u\|_{X^{s,p}_\delta(M;E)} \sim \|u\|_{X^{0,p}_{\delta}(M;E)}
  + \|\hat \nabla^m u\|_{X^{s-m,p}_{\delta}(M;T^{0,m}M\otimes E)}$.
\end{enumerate}
\end{lemma}

\begin{proof}
The two results are proved nearly identically; we treat part \eqref{part:alt-H-norm}.  It suffices to show
\[
\|u\|_{H^{s,p}_\delta(M)} \lesssim \|u\|_{H^{s,p}_{\delta}(M)}
  + \|\hat \nabla^m u\|_{H^{s-m,p}_{\delta}(M)},
\]
as the reverse estimate is obvious.

Let $\Phi_i$ be a M\"obius parametrization. Applying Lemma \ref{lem:christoffel-bound}
to the components of $\Phi_i^*u$ we find
\begin{align*}
\|\nabla_{\mathbb E}^{m} \Phi_i^* u\|_{H^{s-m,p}(B_1^\Hyp)}
&\lesssim
\| \Phi_i^* \hat\nabla^m u\|_{H^{s-m,p}(B_1^\Hyp)}
+ \sum_{j=0}^{m-1} \|\nabla_{\mathbb E}^j \Phi_i^* u\|_{H^{s-m,p}(B_1^\Hyp)}\\
&\lesssim
\| \Phi_i^* \hat\nabla^m u\|_{H^{s-m,p}(B_1^\Hyp)}
+ \|\Phi_i^* u\|_{H^{s-1,p}(B_1^\Hyp)}.
\end{align*}
Hence from Lemma \ref{lem:missing} we conclude
\[
\|\Phi_i^* u\|_{H^{s,p}(B_1^\Hyp)} \lesssim
\|\Phi_i^* u\|_{H^{0,p}(B_1^\Hyp)}
+\| \Phi_i^* \hat\nabla^m u\|_{H^{s-m,p}(B_1^\Hyp)}
+\|\Phi_i^* u\|_{H^{s-1,p}(B_1^\Hyp)}.
\]
A standard interpolation argument to control the final term in this expression then leads to
\[
\|\Phi_i^* u\|_{H^{s,p}(B_1^\Hyp)} \lesssim
\|\Phi_i^* u\|_{H^{0,p}(B_1^\Hyp)}
+\| \Phi_i^* \hat\nabla^m u\|_{H^{s-m ,p}(B_1^\Hyp)}
\]
and the result now follows from multiplying by an appropriate factor of $\rho_i$,
raising to $p^{\rm th}$ powers and summing.
\end{proof}

\subsection{Tensor operations}

It is easy to prove the continuity of contraction operations by working in the domains
of M\"obius parametrizations.
\begin{lemma}[Contractions]
Suppose $E=T^{k_1,k_2}\bar M$ for some $k_1,k_2\in \Nats$ and let
$\cont:E\to E'=T^{k_1-1,k_2-1}\bar M$ be a contraction operation
on a pair of indices.
Then $\cont$ extends to the continuous maps
\begin{align*}
H^{s,p}_\delta(M;E)\to H^{s,p}_{\delta}(M;E'),\\
X^{s,p}_\delta(M;E)\to X^{s,p}_{\delta}(M;E')
\end{align*}
for all $1<p<\infty$ and all $s,\delta\in\Reals$.
\end{lemma}
Multiplication in Sobolev spaces involves a number of compatability
conditions between the spaces parameters that each encode
either a notion of `multiplication cannot generically improve regularity'
or `multiplication must respect duality of Sobolev spaces'.
See, for example, the appendix of \cite{HMT2020} for a demonstration
of the continuity of multiplication between Bessel potential spaces
over bounded open subsets of $\Reals^n$ under the conditions listed
below.  Continuity of tensor products follows from this scalar
result and Lemma \ref{lem:many-norms}.
\begin{proposition}[Product estimates]
\label{prop:multiplication}
Let $E_1$ and $E_2$ be tensor bundles over $\bar M$.
Suppose $s_1, s_2, \sigma\in\Reals$, $1<p_1,p_2,q<\infty)$,
and $\delta_1,\delta_2\in\Reals$.
and define
\[
\frac{1}{r_i} = \frac{1}{p_i}-\frac{s_i}{n},\qquad \frac{1}{r}=\frac{1}{q}-\frac{\sigma}{n}, \qquad i=1,2.
\]
Tensor product is a continuous map
\begin{equation*}
\begin{aligned}
H^{s_1, p_1}_{\delta_1}(M;E_1) \times H^{s_2, p_2}_{\delta_2}(M;E_2) &\to H^{\sigma,q}_{\delta_1+\delta_2}(M;E_1\otimes E_2),\\
X^{s_1, p_1}_{\delta_1}(M;E_1) \times H^{s_2, p_2}_{\delta_2}(M;E_2) &\to H^{\sigma,q}_{\delta_1+\delta_2}(M;E_1\otimes E_2),\\
X^{s_1, p_1}_{\delta_1}(M;E_1) \times X^{s_2, p_2}_{\delta_2}(M;E_2) &\to X^{\sigma,q}_{\delta_1+\delta_2}(M;E_1\otimes E_2)
\end{aligned}
\end{equation*}
provided
\begin{equation*}
\begin{gathered}
\frac{1}{q} \le  \frac{1}{p_1} + \frac{1}{p_2},\\
0\le s_1+s_2,\\
\sigma\leq \min (s_1, s_2),\\
\max\left(\frac 1{r_1},\frac 1{r_2}\right) \le \frac 1r,\\
\frac{1}{r_1}+\frac{1}{r_2} \le \min\left(\frac{1}{r},1\right),\\
\end{gathered}
\end{equation*}
where the last inequality is strict if
\begin{equation*}
\min\left( \frac 1{r_1},\frac1{r_2}\right) = 0.
\end{equation*}
\end{proposition}
We frequently use Proposition \ref{prop:multiplication}
in a context where we are multiplying
\[
\hat\nabla^{k_1}u_1\otimes\cdots\otimes\hat\nabla^{k_\ell}u_\ell
\]
with each $u_i\in H^{s,p}_{\rm loc}(M;E_i)$ with $s>n/p$.
The result lies in $H^{s-k,p}_{\loc}(M;E)$, where $k=k_1+\cdots+k_\ell$ and $E$ is the appropriate bundle, 
under the two restrictions on $k$ in \eqref{eq:multi-mult} below.
We have, for example, the following.
\begin{corollary}
\label{cor:multiplication}
Suppose $1<p<\infty$, $s\in\Reals$ and $s>n/p$.  Given $k_1,\cdots,k_\ell\in\Nats_{\ge 0}$
define $k=k_1+\cdots+k_\ell$ and suppose
\begin{equation}\label{eq:multi-mult}
\begin{aligned}
s &\ge k/2,\\
\frac{1}{p}-\frac{s}{n} &\le \frac{1}{2}-\frac{k}{2n}.
\end{aligned}
\end{equation}
Given tensor bundles $E_k$ over $\bar M$, $1\le k\le \ell$,
tensor product is a continuous map
\[
H^{s-k_1, p}_{\delta_1}(M;E_1) \times \cdots \times H^{s-k_\ell, p}_{\delta_\ell}(M;E_\ell) \to
H^{s-k,}_{\delta_1+\cdots+\delta_\ell}(M;\otimes_{k=1}^\ell E_k)
\]
for any weights $\delta_i$.
\end{corollary}

\subsection{Interpolation}
Suppose $X_1$ and $X_2$ are two Banach spaces that both embed as subspaces of
an ambient Hausdorff topological space; such a pair is called an interpolation couple.  For each
$0<\theta<1$ there is a complex interpolation space $[X_1,X_2]_\theta$
satisfying $X_1\cap X_2\subset [X_1,X_2]_\theta \subset X_1+X_2$.
See \S\ref{secsec:interpolation} for a quick overview complex interpolation
and the proof of the following two interpolation results.  Note that Lemma
\ref{lem:Xinterp-early} concerning interpolation and Gicquaud-Sakovich spaces
is far from optimal but suffices for our purposes.

\begin{theorem}[Complex Interpolation]\label{thm:H-interp-early}
Let $1<p_2<p_2<\infty$, $s_1,s_2\in\Reals$ and $\delta_1,\delta_2\in\Reals$.
Suppose $\theta\in(0,1)$ and let
\begin{equation}\label{eq:spdelta_interp-early}
\begin{aligned}
s &= (1-\theta) s_1 + \theta s_2,\\
\frac 1p &= (1-\theta) \frac{1}{p_1} + \theta \frac{1}{p_2},\\
\delta &= (1-\theta) \delta_1 + \theta \delta_2.
\end{aligned}
\end{equation}
Then
\begin{equation}
[H^{s_1,p_1}_{\delta_1}(M;E),H^{s_2,p_2}_{\delta_2}(M;E)]_\theta = H^{s,p}_{\delta}(M;E)
\end{equation}
as sets, and the norms are equivalent.
\end{theorem}

\begin{lemma}\label{lem:Xinterp-early}
Let $(A_0, A_1)$ be an interpolation pair and suppose $T$
is a linear map
\[
T: A_0 + A_1 \to X^{s_0,p_0}_{\delta_0}(M;E) + X^{s_1,p_1}_{\delta_1}(M;E)
\]
for some $1< p_0,p_1<\infty$, $s_0,s_1\in\Reals$,  and $\delta_0,\delta_1\in\Reals$. If $T$ restricts to continuous maps $T_k:A_k\to X^{s_k,p_k}_{\delta_k}(M;E)$, $k=0,1$, then for each
$0<\theta<1$ the map T also restricts to a continuous map
\[
T_\theta: [A_0, A_1]_\theta \to X^{s,p}_{\delta}(M;E),
\]
where $s$, $p$ and $\delta$ are as in equations \eqref{eq:spdelta_interp-early}.
Moreover,
\[
\|T_\theta\|\lesssim \|T_0\|^{1-\theta} \|T_1\|^\theta
\]
with implicit constant independent of $T_0$, $T_1$.
\end{lemma}

\subsection{Embedding spaces on \texorpdfstring{$\bar M$}{M̅} into weighted spaces on \texorpdfstring{$M$}{M}}
The construction of fortified  spaces in
\S\ref{sec:asymptotic-spaces} relies fundamentally on the observation that
function spaces on $\bar M$ embed continuously into weighted
spaces on $M$. Doing so requires the notion of the \Defn{weight}
of a tensor bundle, the number
of covariant indices minus the number of contravariant indices.
Thus, a metric has weight 2 and a tangent vector has weight -1.
Lemma 3.7(a) of \cite{Lee-FredholmOperators}, reproduced below,
is the template for analogous embeddings for Sobolev spaces.
\begin{lemma}[H\"older $\bar M$ inclusions]
\label{lem:CMbar-into-CMw}
Suppose $k\in\Nats_{\ge 0}$, $0\le \alpha\le 1$, and that
$E$ has weight $w$. Then
\[
C^{k,\alpha}(\bar M;E)\hookrightarrow C^{k,\alpha}_w(M;E).
\]
\end{lemma}

As a first step to generalizing this result we establish the following norm equivalences.
\begin{lemma}\label{lem:normchar}
Suppose $k\in\Nats_{\ge 0}$, $1<p<\infty$, and that $E$ has weight $w$.
We have the following norm equivalences.
\begin{enumerate}
  \item\label{part:W-via-W0p}$\displaystyle \|u\|_{W^{k,p}_\delta(M;E)}  \sim \sum_{j=0}^k \|\hat \nabla^j  u\|_{L^p_\delta(M;T^{0,j}M\otimes E)}$.
  \item\label{part:X-via-X0p}$\displaystyle \|u\|_{X^{k,p}_\delta(M;E)}  \sim \sum_{j=0}^k \|\hat \nabla^j  u\|_{X^{0,p}_\delta(M;T^{0,j}M\otimes E)}$.
  \item\label{part:Wbar-via-W}$\displaystyle\|u\|_{W^{k,p}(\bar M;E)} \sim \sum_{j=0}^k \| \hat\nabla{}^j u\|_{L^p_{w+j - n/p}(M;T^{0,j}M\otimes E)}$.
\end{enumerate}
\end{lemma}

\begin{proof}
Part (\ref{part:X-via-X0p}) follows from Lemmas \ref{lem:nabla-HX-map}
and \ref{lem:alt-H-X-norms}.  The remaining equivalences follow
from an elementary computation using the following observations:
\begin{itemize}
\item  A tensor $u$ of weight $w$ satisfies
$|u|_{\check g} = \rho^w |u|_{\hat g}$.
\item The volume elements of $\hat g$ and $\check g$ are related by
$dV_{\check g} = \rho^{-n} dV_{\hat g}$.
\item Lemma \ref{lem:CMbar-into-CMw} implies $\rho^{-1}d\rho\in C^\infty_0(M;T^{0,1} M)$.
\item Lemma \ref{lem:CMbar-into-CMw} also implies that the
difference tensor between the Levi-Civita connections of $\check g$ and
$\hat g$ lies in $C^\infty_0(M;T^{0,2} M)$ since
\begin{equation*}
\check\nabla - \hat\nabla=\rho^{-1}\bar D,
\end{equation*}
where $\bar D\in C^\infty(\bar M)$ is a tensor of weight $1$.
\item The tensor product is continuous 
\begin{equation*}
C^\infty_0(M;E_1)\otimes L^{p}_\delta(M;E_2)\to L^p_\delta(M;E_1\otimes E_2).
\qedhere
\end{equation*}
\end{itemize}
Note that the first two observations above yield the useful identity
\begin{equation}
\label{Lp-barM-M}
\| u\|_{L^p(\bar M;E)} = \| u\|_{L^p_{w-n/p}(M;E)}.
\end{equation}
\end{proof}

In the following generalization of Lemma \ref{lem:CMbar-into-CMw} the
space $H^{s,p}(\bar M;E)$ is the usual Sobolev space when $s\in\Nats_{\ge 0}$
and is obtained by interpolation for $0<s<\infty$. See, e.g., 
the construction of \cite[Chapter 4]{taylor2010partial} for the case $p=2$, which generalizes
to arbitrary $1<p<\infty$.
\begin{proposition}[Sobolev $\bar M$ inclusions]\label{prop:Mbar-M-HX}
Let $w$ be the weight of $E$ and suppose $1<p<\infty$ and $s\in\Reals$.
The following embeddings are continuous under the listed restrictions.
\begin{align}
H^{s,p}(\bar M;E) &\hookrightarrow H^{s,p}_{w-n/p}(M;E), && s\ge 0\label{eq:HbartoH};\\
H^{s,p}(\bar M;E) &\hookrightarrow X^{s,p}_w(M;E),&& s\ge 1 \text{ and }s > n/p.\label{eq:HbartoX}
\end{align}
\end{proposition}
\begin{proof}
In the case $s\in \Nats_{\ge 0}$, inclusion
\eqref{eq:HbartoH} follows from parts \ref{part:W-via-W0p} and \ref{part:Wbar-via-W} of 
Lemma \ref{lem:normchar}; the general case follows by interpolation.

For embedding \eqref{eq:HbartoX}, first consider the case
$s=k\in \Nats_{\ge 0}$. From
Lemma \ref{lem:normchar}\eqref{part:X-via-X0p}
it suffices to show that 
$\hat\nabla^i u\in X^{0,p}_w(M)$ for $0\le i\le k$.
If $\frac{1}{p} - \frac{k-i}{n}<0$, then Sobolev embedding on $\bar M$,
Lemma \ref{lem:CMbar-into-CMw},
and the basic inclusions from Lemma \ref{lem:basic-inclusions}
imply there exists $\alpha>0$ such that
\begin{align*}
\hat\nabla{}^i u\in H^{k-i,p}(\bar M)
&\hookrightarrow C^{0,\alpha}(\bar M) \\
&\hookrightarrow C^{0,\alpha}_{w+i}(M)\\
&\hookrightarrow C^{0,\alpha}_{w}(M)\\
&\hookrightarrow X^{0,p}_{w}(M).
\end{align*}
On the other hand, if $\frac{1}{p} - \frac{k-i}{n} >0$ then with $\frac{1}{q} = \frac{1}{p} - \frac{k-i}{n}$ we obtain
\begin{align*}
\hat\nabla{}^i u\in H^{k-i,p}(\bar M)
&\hookrightarrow
 L^q(\bar M) \\
&= L^q_{w+i-n/q}(M) \\
&\hookrightarrow L^q_w(M) \\
&\hookrightarrow X^{0,q}_w(M)
\hookrightarrow X^{0,p}_w(M)
\end{align*}
since $i-\frac{n}{q} = k-\frac{n}{p}>0$.
Finally, the boundary case
$\frac{1}{p} - \frac{k-i}{n} = 0$
can be treated by first observing that this can only occur for $i>0$
since $k>n/p$. We can therefore proceed as in the case $\frac{1}{p} - \frac{k-i}{n} = 0$
above but now choosing any $q>p$ with $i-n/q>0$.

The case where $s>1$ is non-integral follows
from interpolation (Lemma \ref{lem:Xinterp-early}) noting that
we can use $H^{k,q_0}(\bar M)$ and $H^{k+1,q_1)}(\bar M)$ as endpoints
of interpolation where $k\in\Nats$ satsfies $k<s<k+1$ and where
$kq_0 = (k+1)q_1 = sp > n$.
\end{proof}
The restriction $s\ge 1$ in embedding \eqref{eq:HbartoX}
is presumably an artifact of the proof technique and $s>0$ would be more natural.
A traveling bump argument shows, however, that the requirement $s>n/p$ is close to optimal:
if $p<n$ we can apply a scaling argument
to find smooth functions $f_i$ each with support on $\Phi_i(B_2^\Hyp)$
such that $\|f_i\|_{H^{1,p}(\bar M;\Reals)}=1$ for each $i$,
and yet $\|f_i\|_{X^{1,p}_0(\bar M;\Reals)}\to\infty$.

\subsection{Geometric tensor bundles}\label{secsec:gtb-early}
Following the terminology of \cite{Lee-FredholmOperators},
a \Defn{geometric tensor bundle} over $\bar M$ is a
subbundle $E$ of some tensor bundle $T^{k,l}\bar M$
that is associated with an invariant subspace of the standard representation
of $O(n)$ (or $SO(n)$ if $\bar M$ is orientable) on $T^{k,l}\mathbb R^n$.
These include bundles associated with $GL(n;\Reals)$
representations such as the symmetric tensors in $T^{0,2}\bar M$,
but also metric-dependent objects such as the trace-free subbundle.

Section \ref{secsec:geometric-E} contains results specific to
function spaces of sections of geometric tensor bundles, and we summarize
the salient points here.
Let $h\in X^{s_0,p_0}_0(M;T^{0,2} M) $ be a metric on $M$, where $1<p_0<n$ and $s_0\ge 1$ satisfy $s_0>n/p_0$, such that
that $\bar h = \rho^2 h$ extends continuously to a metric on $\bar M$.
(The condition $s_0>n/p_0$ ensures $h\in C^{0,\alpha}(M;T^{0,2} M)$
for some $\alpha>0$, whereas $s_0\ge 1$ is a mild technical condition that is always satisfied in our applications here.)
Let $E_h$ be a geometric tensor bundle associated with $h$ and
let $\hat E$ be its ambient tensor bundle.
We say that a section $u$ of $\hat E$ is an element of $X^{s_0,p_0}_0(M;E_h)$ if it is a
section of $E_h$ in the classical $C^0$ sense, and if $u\in X^{s_0,p_0}_0(M;\hat E)$.
Sections with other regularity classes arise as linear combinations of sections in $X^{s_0,p_0}_0(M;E_h)$ with coefficients in either
$H^{s,p}_{\delta}(M;\Reals)$ or $X^{s,p}_{\delta}(M;\Reals)$
for some choice of $s$ or $p$.  In order to make sense of the multiplication
we imposes conditions on these parameters and we define $\mathcal S^{s_0,p_0}_0$
to be the set of pairs $(s,p)\in\Reals\times (1,\infty)$ such that
\[
\begin{gathered}
-s_0 \le s \le s_0\\
\frac{1}{p_0}-\frac{s_0}{n} \le \frac{1}{p}-\frac{s}{n} \le \frac{1}{p\dual_0}+\frac{s_0}{n}.
\end{gathered}
\]
The set $S^{s_0,p_0}_0$ is a special case of related sets $S^{s_0,p_0}_d$ defined in
\S\ref{sec:differential-ops} concerning $d^{\rm th}$ order differential operators.
Proposition \ref{prop:multiplication} ensures that if $u\in H^{s,p}_{\delta}(M;\Reals)$
for some $(s,p)\in \mathcal S^{s_0,p_0}_0$ and if $v\in X^{s_0,p_0}_0(M;\hat E)$ then
$uv\in H^{s,p}_\delta(M;\hat E)$, and a similar fact is true for
coefficients in $X^{s,p}_{\delta}(M;\Reals)$.  We say that $u$ is an element
of $H^{s,p}_\delta(M;E_h)$ for some $(s,p)\in \mathcal S^{s_0,p_0}_0$
if $u$ can be written as a finite sum
\begin{equation}\label{eq:gtb-decomp}
u = \sum_{j=1}^J u_j v_j
\end{equation}
where each $v_j\in X^{s_0,p_0}_0(M;E_h)$ and each $u_j\in H^{\sigma,q}_\delta(M;\Reals)$.
In fact, using the condition that $h$ admits a continuous compactification,
we show in \S\eqref{secsec:geometric-E} that one can find
a fixed collection $\{v_j\}_{j=1}^J$ of elements of $X^{s,p}_0(M;E_h)$
and continuous maps $u\mapsto u_j$ such that each $u\in H^{s,p}_\delta(M;E_h)$
has the form \eqref{eq:gtb-decomp}. As a consequence, $H^{\sigma,q}_{\delta}(M;E_h)$
is a closed subspace of $H^{\sigma,q}_\delta(M;\hat E)$, and
the norm for $H^{\sigma,q}_{\delta}(M;E_h)$ is simply inherited from its
ambient space. A similar analysis holds
in the Gicquaud-Sakovich case and
$X^{\sigma,q}_\delta(M;E_h)$ is a closed subspace of $X^{\sigma,q}_\delta(M;E)$.

With only a few exceptions all of the results discussed up until this
point carry over (either obviously or due to proofs in \S\ref{secsec:geometric-E})
replacing all the tensor bundles with geometric tensor bundles associated with the same
metric $h$ satisfying the regularity conditions listed above.  There is,
of course, the obvious caveat
that all of the Sobolev parameters occurring in the statements must additinally
lie in the set $\mathcal S^{s_0,p_0}_0$ associated with $h$.  The following
additional adjustments are also required.

\begin{enumerate}
  \item Proposition \ref{prop:density-early} concerning density
  needs adjusting as there may be no non-zero smooth sections in a general geometric tensor
  bundle. See Proposition \ref{prop:density-gtb} instead.
  \item Lemma \ref{lem:nabla-HX-map} requires an obvious modification with respect to the target space because  $E_h$ need not be parallel with respect to
  the covariant derivatives $\hat \nabla$ and $\check \nabla$.
  \item Modest additional formalism is required to generalize
  Proposition \ref{prop:Mbar-M-HX}, especially with respect to
  interpreting sections of $H^{s,p}(\bar M;E_{\bar h})$.  We do not use this extension,
  which is left as an exercise for the reader.
\end{enumerate}

\subsection{Duality}\label{secsec:duality-early}
Let $g_{\mathbb E}$ be the Euclidean metric on $\Reals^n$.  The
$L^2$ inner product
\[
\ip<u,v>_{(\Reals^n,g_{\mathbb E})} = \int_{\Reals^n} u v\; dV_{g_{\mathbb E}}
\]
for smooth compactly supported functions extends to a continuous
bilinear form on $H^{-s,p\dual}(\Reals^n)\times H^{s,p}(\Reals^n)$
for any $s\in\Reals$ and $p,p\dual\in(1,\infty)$ with
$\frac 1p+\frac{1}{p\dual} = 1$.  Moreover, the map
\[
u\mapsto \ip<u,\cdot>_{(\Reals^n,g_{\mathbb E})}
\]
yields an isomorphism between $H^{-s,p\dual}(\Reals^n)$
and the dual space
 $(H^{s,p}(\Reals^n))\dual$.  
 An analogous
statement for the weighted Bessel potential spaces
involves a choice of metric, and we can formulate
duality for metrics with the regularity discussed
at the start of \S\ref{secsec:gtb-early}.

\begin{proposition}(Duality)\label{prop:dualH-early}
Suppose metric $h\in X^{s_0,p_0}_0(M;T^{0,2}M)$ for some $1<p_0<\infty$, with
$s_0\ge 1$ and $s_0>n/p_0$, and suppose $\rho^2 h$
admits a continuous extension to $\bar M$.

Let $E_h$ be a geometric tensor bundle over $\bar M$ associated with $h$
and for compactly supported sections $u,v\in X^{s_0,p_0}_0(M;E_h)$ define
\[
\ip<u,v>_{(M,h)} = \int_{M} \ip<u,v>_h\, dV_h,
\]
where $\ip<\cdot,\cdot>_{h}$ is the inner product induced by $h$
on the fibers of $E_h$.

If $(s,p)\in \mathcal S^{s_0,p_0}_0$ then $(-s,p\dual)\in \mathcal S^{s_0,p_0}_0$
as well
and $\ip<\cdot,\cdot>_{(M,h)}$ extends to a continuous bilinear map
\[
H^{-s,p\dual}_{-\delta}(M;E_h)\times H^{s,p}_\delta(M;E_h) \to \Reals.
\]
for each $\delta\in\Reals$. Moreover, the map
\[
v\mapsto \ip<v,\cdot>_{(M,h)}
\]
is an isomorphism $H^{-s,p\dual}_{-\delta}(M;E_h) \to (H^{s,p}_{\delta}(M;E_h))\dual$,
and $H^{-s,p}_{\delta}(M;E_h)$ is
precisely the subset of $H^{-s,p}_{\rm loc}(M;E_h)$ such that
$\left<v,\cdot\right>_{M;E_h}$ extends to an element of
$(H^{s,p}_{\delta}(M;E_h))\dual$.
\end{proposition}

\subsection{Nonlinearities}

Nonlinear functions can be applied to functions in Sobolev spaces once the parameters
are sufficiently regular so as to imply H\"older continuity.  We first consider
the case where the nonlinear function vanishes at zero, in which case decay
rates at infinity are preserved.
\begin{proposition}\label{prop:nonlin}
Suppose $1<p<\infty$, $s\ge 1$, $s>n/p$, and $\delta\ge 0$.
Let $F\in C^{k,\alpha}(\Reals)$ for some $k\in\Nats$ and $\alpha\in [0,1)$ with $k+\alpha>s$
and suppose $F(0)=0$.
\begin{enumerate}
  \item 
 \label{part:nonlin-H} If $u\in H^{s,p}_{\delta}(M;\Reals)$, then $F(u)\in H^{s,p}_{\delta}(M;\Reals)$ and
  \[
  \|F(u)\|_{H^{s,p}_\delta(M;\Reals)} \lesssim
  \|u\|_{H^{s,p}_{\delta}(M;\Reals)} ( 1 + \|u\|_{L^\infty(M;\Reals)}^{k+\alpha}).
  \]
  \item\label{part:nonlin-X} If $u\in X^{s,p}_{\delta}(M;\Reals)$, then $F(u)\in X^{s,p}_{\delta}(M;\Reals)$ and
  \[
  \|F(u)\|_{X^{s,p}_\delta(M;\Reals)} \lesssim
  \|u\|_{X^{s,p}_{\delta}(M;\Reals)} ( 1 + \|u\|_{L^\infty(M;\Reals)}^{k+\alpha}).
  \]
\end{enumerate}
The implicit constants above depend on $F$ and on a choice of constant
$K$ such that $\|u\|_{L^\infty(M;\mathbb R)}\le K$ but are otherwise independent of $u$.
\end{proposition}
\begin{proof}
We prove the result for $H^{s,p}_\delta(M)$;
the Gicquaud-Sakovich case is proved analogously.

Let $\chi$ be a cutoff function that equals 1 on $B_{1/2}^\Hyp$ that is
supported on $B_1^\Hyp$ and let $\eta$ be a cutoff function that equals 1 on
$B_{1}^\Hyp$ that is supported in $B_2^\Hyp$.
Part \eqref{part:nonlin-X} is a consequence
of Theorem 1 of \cite{WS-Composition-96}
(see also \S3.1 of \cite{taylor_pseudodifferential_1991}), which
implies that for a M\"obius parametrization $\Phi_i$
\begin{align*}
\|\chi \Phi_i^* F(u)\|_{H^{s,p}(B_1^\Hyp)} 
&= \| \chi F(\eta \Phi_i^* u)\|_{H^{s,p}(\Reals^n)}\\
&\lesssim \|  F(\eta \Phi_i^* u)\|_{H^{s,p}(B_1^\Hyp)}\\
&\lesssim
\|\eta \Phi_i^* u\|_{H^{s,p}(\Reals^n)} ( 1 + \|\eta \Phi_i^* u\|_{L^\infty(\Reals^n)}^{k+\alpha}),
\end{align*}
where the implicit constants depend on $\chi$ and on a choice of interval $[-K,K]$ containing $u(M)$.
Since $\delta\ge 0$ and $s>n/p$, Sobolev embedding implies $H^{s,p}_\delta(M;E)\hookrightarrow
C^0_0(M;E)$ and hence the $L^\infty$ norms appearing in this expression are uniformly bounded
independent of $i$.  
Part \eqref{part:nonlin-H} now follows from multiplying by $\rho_i$ and summing $p^{\rm th}$
powers.
\end{proof}

If the domain of $F$ in Proposition \ref{prop:nonlin}
is only an interval containing $0$, it is easy to see that
the conclusion still holds for a function $u$
having image contained in a compact set in the domain of $F$. We will use this
variation of Proposition \ref{prop:nonlin} without comment.

Proposition \ref{prop:nonlin} requires the hypothesis $F(0)=0$ to
ensure $F(u)$ decays at infinity when $u$ does.  If we drop this
hypothesis we find instead
$F(u)$ lies in a Gicquaud-Sakovich space.
\begin{corollary}\label{cor:nonlin-into-X}
Suppose $1<p<\infty$, $s\ge 1$, $s>n/p$ and $\delta\ge 0$.
Let $F\in C^{k,\alpha}(\Reals)$ for some $k\in\Nats$ and $\alpha\in [0,1)$ with $k+\alpha>s$.
\begin{enumerate}
  \item If $u\in H^{s,p}_{\delta}(M;\Reals)$ then $F(u)\in X^{s,p}_{0}(M;\Reals)$.
  \item \label{part:X-to-X-zero}If $u\in X^{s,p}_{\delta}(M;\Reals)$ then $F(u)\in X^{s,p}_{0}(M;\Reals)$.
\end{enumerate}
\end{corollary}
\begin{proof}
From the continuous embedding $H^{s,p}_\delta(M)\subset X^{s,p}_\delta(M)$
from Lemma \ref{lem:basic-inclusions}, it suffices
to establish part \eqref{part:X-to-X-zero} and we assume $u\in X^{s,p}_{\delta}(M)$
Write $F = F(0) + \hat F$, so $\hat F(0)=0$.
Proposition \ref{prop:nonlin}\eqref{part:nonlin-X} shows
$\hat F(u)\in X^{s,p}_\delta(M)\subset X^{s,p}_0(M)$.
But then $F(u)=\hat F(u) + F(0)\in X^{s,p}_0(M)$ as well.
\end{proof}

\section{Fortified spaces on \texorpdfstring{$\bar M$}{M̅}}
\label{sec:asymptotic-spaces}

As discussed in the introduction, elliptic problems on asymptotically
hyperbolic manifolds
lead to the consideration of metrics on $M$ that have a
high level of regularity, but whose compactifications have 
a lesser amount of regularity up to the boundary of $\bar M$.  In this section, 
building on prior work in \cite{WAH} in the H\"older category,
we describe two new families of function spaces on $\bar M$ that
are suitable for describing such compactified metrics.  Each of these
spaces resides in some $H^{m,p}(\bar M)$, which can be thought
of as a baseline level of smoothness, and then possesses 
additional regularity measured using weighted spaces on $M$.

We begin by recalling the spaces constructed in \cite{WAH}.
For a tensor bundle $E$ over $\bar M$ of weight $w$,
Lemma \ref{lem:CMbar-into-CMw}
implies that $C^{k,\alpha}(\bar M;E)\subset C^{k,\alpha}_{w}(M;E)$.
Both of these spaces ensure their elements have
$k+\alpha$ interior derivatives
and can be thought of as two
endpoints with respect to a spectrum of boundary regularity.
Whereas functions in $C^{k,\alpha}(\bar M;E)$ have $k+\alpha$
derivatives up to the boundary, a traveling bump argument shows that functions in
$C^{k,\alpha}_{w}(M;E)$ need not admit even a continuous
extension to $\bar M$.  Building on
these observations, we have a family of 
of function spaces that are intermediate
between $C^{k,\alpha}_w(M;E)$ and $C^{k,\alpha}(\bar M;E)$:
 given $\alpha\in[0,1]$ and
$k,m\in \Nats_{\ge 0}$ with $m\le k$, the space
$\curC^{k,\alpha;m}(M;E)$ is defined by the norm
\begin{equation}\label{eq:curC}
\|u\|_{\curC^{k,\alpha;m}(M;E)} = \sum_{j=0}^m \|\hat \nabla^j u\|_{C^{k-j,\alpha}_{w+j}(M;T^{0,j}M\otimes E)}.
\end{equation}
The parameter $k$ controls the amount of interior regularity,
while $m$ allows for a limited amount of regularity
at the boundary.  When $m\ge 1$, functions in $\curC^{k,\alpha;m}(M;E)$ lie in
$C^{m-1,1}(\bar M;E)$ \cite[Lemma 2.3(c)]{WAH}
and hence nearly possess $m$ derivatives up to the boundary.

We introduce two related families of functions spaces motivated by the embeddings
from Proposition \ref{prop:Mbar-M-HX}:
\begin{align*}
H^{s,p}(\bar M;E) &\hookrightarrow H^{s,p}_{w-n/p}(M;E),\\
H^{s,p}(\bar M;E) &\hookrightarrow X^{s,p}_{w}(M;E).
\end{align*}
\begin{definition}
Suppose $1<p<\infty$, $s\in \mathbb R$, and $m\in \mathbb N_{\ge 0}$ with $s\ge m$.
Let $E$ be a tensor bundle over $M$ with weight $w$.
The \Defn{fortified  spaces} $\mathscr H^{s,p;m}(M;E)$ and $\mathscr X^{s,p;m}(M;E)$
consist of the distributional-valued sections of $E$ such that the following norms are finite:
\begin{align}
\label{define-scrH-int}
\| u \|_{\mathscr H^{s,p;m}(M;E)} &= \sum_{j=0}^m \| \hat\nabla{}^j u\|_{H^{s-j,p}_{w+j - n/p}(M;T^{0,j}M\otimes E)},
\\
\label{define-scrX-int}
\| u \|_{\mathscr X^{s,p;m}(M;E)} &=  \sum_{j=0}^m \| \hat\nabla{}^j u\|_{X^{s-j,p}_{w+j}(M;T^{0,j}M\otimes E)}.
\end{align}
\end{definition}
We obtain spaces of smooth functions in the usual way by intersecting
the corresponding spaces with finite regularity.  
Specifically,
\begin{align*} 
\mathscr H^{\infty,p;m}(M;E)&=\bigcap_{s\in\Reals} \mathscr H^{s,p;m}(M;E),\\
\mathscr C^{\infty;m}(M;E) &= \bigcap_{k\in\Nats \atop \alpha\in[0,1]} \mathscr C^{k,\alpha;m}(M;E).
\end{align*}
In the same way that the H\"older parameter $\alpha$ 
plays no essential role in this latter intersection,
an easy application of Sobolev embedding shows that the Lebesgue parameter $p$
is irrelevant in the Gicquaud-Sakovitch case; thus we define
$$\mathscr X^{\infty;m}(M;E)=\bigcap_{s\in\Reals} \mathscr X^{s,p;m}(M;E)$$ regardless
of the value of $p$.  In fact, $\mathscr X^{\infty;m}(M;E)=\mathscr C^{\infty;m}(M;E)$
but we use the notation $\mathscr X^{\infty;m}(M;E)$ in settings where it is helpful
to emphasize the space's origins.

Fortified spaces can be thought of as weighted interior spaces 
with enhanced regularity at the boundary.  The baseline
interior regularity follows from the inclusions 
$\mathscr H^{s,p;m}(M;E)\hookrightarrow H^{s,p}_{w-n/p}(M;E)$
and 
$\mathscr X^{s,p;m}(M;E)\hookrightarrow X^{s,p}_{w}(M;E)$
which are immediate from the definition.  
The following elementary relations, 
which are easy consequences of
Lemmas \ref{lem:basic-inclusions} and \ref{lem:alt-H-X-norms},
demonstrate one facet of the enhanced boundary regularity.
\begin{lemma}\label{lem:curly-basic}
Suppose $1<p<\infty$, $m\in\Nats_{\ge0}$, and $s\in\Reals$ with $s\ge m$.
For any tensor bundle $E$ over $\bar M$,
\begin{equation*}
\mathscr X^{s,p;m}(M; E)\hookrightarrow
\mathscr H^{s,p;m}( M; E)\hookrightarrow
\mathscr H^{m,p;m}( M; E) = H^{m,p}(\bar M;E).
\end{equation*}
Additionally, if $k\in \Nats_{\ge 0}$ with $k\ge m$,
then
\begin{equation*}
\mathscr C^{k,\alpha;m}(M;E)\hookrightarrow \mathscr X^{k,p;m}(M;E)
\end{equation*}
for all $0\le \alpha\le 1$.
\end{lemma}

As a consequence, every fortified  space is a subspace of some $H^{m,p}(\bar M;E)$.
Thus we can interpret fortified  spaces as being a Sobolev space on $\bar M$ supplemented with additional smoothness in the interior.  
For $\mathscr H^{k,p;m}(M;E)$ with $k\in\Nats$ it is straightforward
to interpret this additional regularity.
Lemmas \ref{lemma:nabla-and-rho-on-curly} and \ref{lem:curly-H-improvement} below imply
\begin{equation}\label{eq:def-curlyH-barM}
\|u\|_{\mathscr H^{k,p;m}(M;E)}
\sim \sum_{j=0}^{k-m}\|(\rho \hat\nabla)^j u\|_{H^{m,p}(\bar M;T^{0,j}\bar M\otimes E )},
\end{equation}
and therefore an element of $\mathscr H^{k,p;m}(M;E)$ has $m$ derivatives up to the boundary,
and an additional
$k-m$ interior derivatives that are each allowed to blow up
by a factor of $\rho^{-1}$ beyond what $L^p(\bar M;E)$ allows.
The additional regularity afforded by $\mathscr X^{s,p;m}(M;E)$, beyond
what its inclusion in $\mathscr H^{s,p;m}(M;E)$ implies, is less
transparent but is remarkably strong.
For example, we show in Corollary \ref{cor:X-bdy-awesome}
that the boundary trace of $u\in \mathscr X^{1,p;1}(M;E)$ is
Lipschitz regardless of the value of $1<p<\infty$.

\subsection{Mapping properties of \texorpdfstring{$\hat \nabla$}{∇̂} and \texorpdfstring{$\rho$}{ρ}}
The following lemma justifies the following general principles:
multiplication by $\rho$ improves boundary regularity,
$\hat\nabla$ removes a degree of both interior
and boundary differentiability, and $\rho\hat\nabla$
only removes interior regularity.

\begin{lemma}
\label{lemma:nabla-and-rho-on-curly}
Suppose $1<p<\infty$, $s\in\Reals$, $m\in\Nats$ and $s\ge m$,
and let $E$ be a tensor bundle of weight $w$ over $M$.
The following maps are continuous between the spaces indicated.
\begin{enumerate}
\item $u\mapsto \rho u$:\label{part:curly-rho-map}
\begin{gather*}
\mathscr H^{s,p;m-1}(M;E)\to \mathscr H^{s,p;m}(M;E),\\
\mathscr X^{s,p;m-1}(M;E)\to \mathscr X^{s,p;m}(M;E).
\end{gather*}
\item $u\mapsto \hat\nabla u$:\label{part:curly-nabla-map}
\begin{gather*}
\mathscr H^{s,p;m}(M;E)\to \mathscr H^{s-1,p;m-1}(M; T^{0,1}M\otimes E ),\\
\mathscr X^{s,p;m}(M;E)\to \mathscr X^{s-1,p;m-1}(M;T^{0,1}M\otimes E  ).
\end{gather*}
\item $u\mapsto \rho \hat\nabla u$:\label{part:curly-rho-nabla-map}
\begin{gather*}
\mathscr H^{s,p;m-1}(M;E)\to \mathscr H^{s-1,p;m-1}( M;T^{0,1}M\otimes E  ),\\
\mathscr X^{s,p;m-1}(M;E)\to \mathscr X^{s-1,p;m-1}(M;T^{0,1}M\otimes E  ).
\end{gather*}
\end{enumerate}
\end{lemma}
\begin{proof}
The proofs for the Sobolev and Gicquaud-Sakovich cases are similar and
we treat the Sobolev case.
For part \eqref{part:curly-rho-map} consider some $0\le j\le m$
and observe that
\[
\hat \nabla^j (\rho u) = \sum_{a+b=j}  \binom{j}{a}\, \hat\nabla{}^a \rho\, \hat\nabla{}^b u.
\]
Since $\hat\nabla{}^a\rho$ is a tensor of weight
$a$ and smooth on $\bar M$,
it lies in $C^\infty_{a}(M)$

Consequently, Lemma  \ref{lem:basic-inclusions} and
Proposition \ref{prop:multiplication} imply
\begin{equation*}
\|\hat\nabla^a \rho \hat\nabla{}^b u\|_{H^{s-j,p}_{w+j-n/p}(M)}
\lesssim
\|\hat\nabla{}^b u\|_{H^{s-j,p}_{w+b-n/p}(M)}
\lesssim
\|\hat\nabla{}^b u\|_{H^{s-b,p}_{w+b-n/p}(M)}
\lesssim
\|u\|_{\mathscr H^{s,p;m}(M)},
\end{equation*}
which proves the result.

Part \eqref{part:curly-nabla-map} is immediate from the definition.
Part \eqref{part:curly-rho-nabla-map} follows from parts
\eqref{part:curly-rho-map} and \eqref{part:curly-nabla-map}
except in the case $m=1$, which uses Lemma \ref{lem:nabla-HX-map}
in place of part \eqref{part:curly-nabla-map}.
\end{proof}

Note that $\mathscr H^{s,p;0}(M;E) = H^{s,p}_{w-n/p}(M;E)$ and that
$\mathscr X^{s,p;0}(M;E) = X^{s,p}_{w-n/p}(M;E)$,
 the least regular spaces in the spectrum measured by the parameter $m$.
Starting from
$\mathscr H^{s,p;0}(M;E)=H^{s,p}_{w-n/p}(M;E)$
and multiplying by $\rho^m$ it follows
from repeated applications of Lemma \ref{lemma:nabla-and-rho-on-curly}\eqref{part:curly-rho-map}
that $H^{s,p}_{w+m-n/p}(M;E)\subset \mathscr H^{s,p;m}(M;E)$,
so long as $s\ge m$.  This inclusion is strict, however,
and in \S\ref{secsec:bounday-asymptotics} we show that
$H^{s,p}_{w+1-n/p}(M;E)$
is exactly the subset of $\mathscr H^{s,p;m}(M;E)$ that
vanishes on the boundary.

The following variation of Lemma \ref{lem:alt-H-X-norms}
emphasizes the principle that interior regularity is measured by
applying $\rho\hat\nabla$.
\begin{lemma}\label{lem:curly-H-improvement}
Let $E$ be a tensor bundle over $M$.
Suppose $1<p<\infty$, $s\in\Reals$, $m\in\Nats_{\ge 0}$, and $s\ge m$.
\begin{enumerate}
\item\label{part:H-improvement}
Suppose $u\in \mathscr H^{s,p;m}(M;E)$ and
$\rho\hat\nabla u\in \mathscr H^{s,p;m}(M;T^{0,1}M\otimes E )$.  Then
$u\in \mathscr H^{s+1,p;m}(M;E)$ and
\[
\|u\|_{\mathscr H^{s+1,p;m}(M;E)} \lesssim
\|u\|_{\mathscr H^{s,p;m}(M;E)} +
\|\rho\hat\nabla u\|_{\mathscr H^{s,p;m}(M;T^{0,1}M\otimes E )}.
\]

\item \label{part:X-improvement}
Suppose $u\in \mathscr X^{s,p;m}(M;E)$ and $\rho\hat\nabla u \in \mathscr X^{s,p;m}(M;T^{0,1}M\otimes E )$.
Then $u\in \mathscr X^{s+1,p;m}(M;E)$ and
\begin{equation*}
\| u \|_{\mathscr X^{s+1,p;m}(M;E)}
\lesssim \|u\|_{\mathscr X^{s,p;m}(M;E)}
+ \|\rho\hat\nabla u\|_{\mathscr X^{s,p;m}(M;T^{0,1}M\otimes E )}.
\end{equation*}
\end{enumerate}
\end{lemma}

\begin{proof}
Suppose $0\le k\le m$.  From Lemma \ref{lem:alt-H-X-norms} we have
\begin{align*}
\|\hat\nabla^k u\|_{H^{s-k+1-n/p,p}_{w+k-n/p}(M)} &\lesssim
\|\hat\nabla^k u\|_{H^{s-k,p}_{w+k-n/p}(M)} +
\|\hat\nabla^{k+1} u\|_{H^{s-k,p}_{w+k-n/p}(M)}\\
& \le
\|u\|_{\mathscr H^{s,p;m}(M)} + \|\hat\nabla^{k+1} u\|_{H^{s-k,p}_{w+k-n/p}(M)}.
\end{align*}
So to establish part \eqref{part:H-improvement} it
suffices to show
\begin{equation}\label{eq:k+1-curly}
\|\hat\nabla^{k+1} u\|_{H^{s-k,p}_{w+k-n/p}(M)}\lesssim\|u\|_{\mathscr H^{s,p;m}(M)} +
\|\rho\hat\nabla u\|_{\mathscr H^{s,p;m}(M)}.
\end{equation}

Since $[\hat\nabla^k,\rho]=\sum_{j=0}^{k-1} {k \choose j} \hat\nabla^{k-j}\rho \hat\nabla^{j}$
and since, as argued in Lemma \ref{lemma:nabla-and-rho-on-curly},
 each $\hat\nabla^{k-j} \rho \in C^{\infty}_{k-j}(M;T^{0,k-j} M)$
we find
\begin{equation*}
\begin{aligned}
\|[\hat\nabla^k,\rho]\, \hat\nabla u\|_{H^{s-k,p}_{w+k+1-n/p}(M)}
&\lesssim \sum_{j=0}^{k-1} \|\hat\nabla^{j+1} u\|_{H^{s-k,p}_{w+j+1-n/p}(M)}
\\
&\lesssim \sum_{j=0}^{k-1} \|\hat\nabla^{j+1} u\|_{H^{s-(j+1),p}_{w+j+1-n/p}(M)}
\\
&\lesssim \|u\|_{\mathscr H^{s,p;m}(M)}.
\end{aligned}
\end{equation*}
Hence
\begin{align*}
\|\hat\nabla^{k+1} u\|_{H^{s-k,p}_{w+k-n/p}(M)}
&=
\|\rho\hat\nabla^{k+1} u\|_{H^{s-k,p}_{w+k+1-n/p}(M)}
\\
&\lesssim \|\hat\nabla^k(\rho\hat\nabla u)\|_{H^{s-k,p}_{w+1+k-n/p}(M)} + \|u\|_{\mathscr H^{s,p;m}(M)}\\
&\lesssim
\|\rho \hat\nabla u\|_{\mathscr H^{s,p;m}(M)} + \|u\|_{\mathscr H^{s,p;m}(M)},
\end{align*}
which establishes the desired estimate \eqref{eq:k+1-curly}.
Part \eqref{part:X-improvement}is proved similarly.
\end{proof}

\subsection{Relationships with weighted spaces on \texorpdfstring{$M$}{M}}\label{secsec:rel-M}

The following elementary embeddings follow from
the
definition of the spaces
and the discussion following Lemma \ref{lemma:nabla-and-rho-on-curly}.
\begin{lemma}\label{lem:basic-roman-curly}
Suppose $1<p<\infty$, $s\in\Reals$, $m\in\Nats_{\ge 0}$,
and $s\ge m$; let $E$ be a tensor bundle of weight $w$ over $M$. Then
\begin{align*}
H^{s,p}_{w+m-n/p}(M;E) &\subset \mathscr H^{s,p;m}(M;E) \subset
H^{s,p}_{w-n/p}(M;E),\\
X^{s,p}_{w+m}(M;E) &\subset \mathscr X^{s,p;m}(M;E) \subset
X^{s,p}_{w}(M;E).
\end{align*}
\end{lemma}

In our applications, embeddings into $X^{s,p}_w(M;E)$ are
especially important.
Recall that in \S\ref{secsec:gtb-early} we
developed function spaces of sections of a geometric tangent bundle
associated with a metric $g$ so long as the metric satisfied the following
regularity conditions:
\begin{enumerate}
  \item\label{part:g-in-X} $g\in X^{s,p}_0(M;T^{0,2}M)$ with $s>n/p$ and $s\ge 1$,
  \item $\bar g = \rho^2 g$ admits a continuous extension to $\bar M$.
\end{enumerate}

We now establish that asymptotically hyperbolic metrics of class \eqref{intro:H-class} or \eqref{intro:X-class} satisfy these conditions.
Continuity at the boundary is addressed in \S\ref{secsec:rels-with-barM}, thus
we focus here on condition \eqref{part:g-in-X}.
Lemma \ref{lem:basic-roman-curly} implies that if $\bar g\in\mathscr X^{s,p;m}(M;T^{0,2}M)$
then $g = \rho^{-2}\bar g \in X^{s,p}_0(M;T^{0,2}M)$.  Thus $g$ satisfies \eqref{part:g-in-X}
if $s>n/p$ and $s\ge m\ge 1$.

The situation for a metric with $\bar g\in \mathscr H^{s,p;m}(M;T^{0,2}M)$
is somewhat different.  In this case, Lemma \ref{lem:basic-roman-curly},
in combination with Lemma \ref{lem:basic-inclusions}, only yields
$g \in H^{s,p}_{-n/p}(M;T^{0,2}M)$, which embeds into $ X^{s,p}_{-n/p}(M;T^{0,2}M)$, the key issue being the decay parameter
$-n/p$. One would like to raise this value to $0$, not only for
the considerations above, but also to obtain good mapping properties
of associated differential operators.  For example,
a computation using Proposition \ref{prop:multiplication} shows that
if $s>n/p$ then the Laplacian of a metric in $X^{s,p}_{-n/p}(M;T^{0,2}M)$
takes $H^{s,p}_\delta(M)\to H^{s-2,p}_{\delta-n/p}(M)$
resulting in an undesired loss of decay.

The following result shows that the embedding of
$\mathscr H^{s,p;m}(M;E)$ into Gicquaud-Sakovich spaces
is better than a simple application of Lemma \ref{lem:basic-roman-curly} would
suggest, so long as $m$ is sufficiently large.  In particular,
if $\bar g\in \mathscr H^{s,p;m}(M;T^{0,2}M)$
with $m>n/p$ then $g\in X^{s,p}_0(M;T^{0,2}M)$ as desired.
\begin{proposition}\label{prop:curly-H-into-roman-X}
Suppose $1<p<\infty$, $s\in\Reals$, $m\in\Nats_{\ge 0}$, $s\ge m$.
Let $E$ be a tensor bundle of weight $w$ over $M$.
We have the continuous embedding
\begin{equation}\label{eq:H-curly-to-roman-X}
\mathscr H^{s,p;m}(M;E) \hookrightarrow X^{s,p}_{w+\beta}(M;E),
\end{equation}
where $\beta = \min(m-n/p,0)$ unless $m=n/p$, in which
case \eqref{eq:H-curly-to-roman-X} holds with any $\beta<0$.
\end{proposition}
\begin{proof}

Since $s\ge m$, Lemma \ref{lem:alt-H-X-norms} implies that
\[
\|u\|_{X^{s,p}_{w+\beta}(M)}
\lesssim \|u\|_{X^{0,p}_{w+\beta}(M)} +
\|\hat\nabla^m u\|_{X^{s-m,p}_{w+\beta}(M)}.
\]
But then, using the condition $\beta\le m-n/p$,
we obtain from Lemma \ref{lem:basic-inclusions}
\[
\|\hat\nabla^m u\|_{X^{s-m,p}_{w+\beta}(M)}
\lesssim \|\hat\nabla^m u\|_{X^{s-m,p}_{w+m-n/p}(M)}
\lesssim \|\hat\nabla^m u\|_{H^{s-m,p}_{w+m}(M)}
\lesssim \|u\|_{\mathscr H^{s,p;m}(M)}.
\]
Hence to show embedding \eqref{eq:H-curly-to-roman-X} holds
it suffices to establish the low-order estimate
\begin{equation*}
\|u\|_{X^{0,p}_{w+\beta}(M;E)} \lesssim \|u\|_{\mathscr H^{s,p;m}(M)}.
\end{equation*}
Moreover, $\mathscr H^{s,p;m}(M) \hookrightarrow \mathscr H^{m,p;m}(M) = H^{m,p}(\bar M)$
and we show instead that
\begin{equation}\label{eq:curly-H-to-X-low}
\|u\|_{X^{0,p}_{w+\beta}(M;E)} \lesssim \|u\|_{\mathscr H^{m,p}(\bar M)}.
\end{equation}

Suppose first $m<n/p$, so $\beta = m-n/p < 0$.  Sobolev embedding
on $\bar M$ and 
\eqref{Lp-barM-M} provide continuous inclusions
\begin{equation}\label{eq:curly-H-to-low-roman-H}
H^{m,p}(\bar M) \hookrightarrow L^q(\bar M;E)
= H^{0,q}_{w-n/q}(M)\hookrightarrow X^{0,p}_{w-n/q}(M),
\end{equation}
where
\begin{equation}\label{eq:Sobolev-p-to-q-s4}
\frac{1}{q} = \frac{1}{p} - \frac{m}{n},
\end{equation}
and where we have used $q\ge p$ in the final embedding.
Estimate \eqref{eq:curly-H-to-X-low}
now follows noting that $w-n/q = w+m-n/p = w+\beta$.

The proof of estimate
\eqref{eq:curly-H-to-X-low}
in the boundary case $m=n/p$ is
argued the same way as above,
except the choice of $q\ge p$ in \eqref{eq:Sobolev-p-to-q-s4}
is arbitrary.
So suppose $m>n/p$ and therefore $\beta=0$.
Sobolev embedding on $\bar M$ along with
Lemmas \ref{lem:CMbar-into-CMw} and \ref{lem:basic-inclusions}
implies there exists $\alpha\in (0,1)$ such that
\begin{equation}\label{eq:H-into-C-into_X}
H^{m,p}(\bar M) \hookrightarrow  C^{0,\alpha}(\bar M)
\hookrightarrow C^{0,\alpha}_{w}( M)
\hookrightarrow X^{0,p}_{w}( M)=X^{0,p}_{w+\beta}(M)
\end{equation}
as desired.
\end{proof}

\subsection{Relationships with spaces on \texorpdfstring{$\bar M$}{M̅}}\label{secsec:rels-with-barM}
We work with metrics admitting conformal compactifications in fortified 
spaces and the following embeddings into spaces on $\bar M$ play a
central role in determining which spaces are suitable for this purpose.

\begin{proposition}[Inclusion into spaces on $\bar M$]\label{prop:curly-to-M-bar}
Suppose $1<p,q<\infty$, $s\in\Reals$, $m,k\in\Nats_{\ge 0}$ with $s\ge m$, anbd $0<\alpha<1$.
Let $E$ be a tensor bundle over $\bar M$. 
We have the following embeddings:
\begin{align}
&\mathscr H^{s,p;m}(M;E)\hookrightarrow H^{m,p}(\bar M;E);
\label{eq:curlyH-to-H-Mbar}
\\
&\mathscr X^{s,p;m}(M;E)\hookrightarrow H^{m,q}(\bar M;E),
&& \frac{1}{p}-\frac{s}{n}\le \frac{1}{q}-\frac{m}{n};
\label{eq:curlyX-to-H-Mbar}
\\
&\mathscr H^{s,p;m}(M;E)\hookrightarrow C^{k,\alpha}(\bar M;E),
&& m\ge  k+\alpha + \frac{n}{p};
\label{eq:curlyH-to-C-Mbar}
\\
&\mathscr X^{s,p;m}(M;E)\hookrightarrow C^{k,\alpha}(\bar M;E),
&& s\ge k+\alpha + \frac{n}{p}\text{ and }m\geq k+1.
\label{eq:curlyX-to-C-Mbar}
\end{align}

\end{proposition}
\begin{proof}
Inclusion \eqref{eq:curlyH-to-H-Mbar} follows immediately from  equation \eqref{define-scrH-int}
defining the norm and Lemma \ref{lem:normchar}\eqref{part:Wbar-via-W}.
Sobolev embedding on $\bar M$ then implies under the hypotheses
of \eqref{eq:curlyH-to-C-Mbar}
\[
\mathscr H^{s,p;m}(M)\hookrightarrow
H^{m,p}(\bar M)
\hookrightarrow C^{k,\alpha}(\bar M)
\]
as desired.

Turning to \eqref{eq:curlyX-to-H-Mbar},
the identity \eqref{Lp-barM-M}
implies that for a tensor $u$ of weight $w$ and for $0\le j \le m$,
\begin{equation}
\|\hat\nabla^j u\|_{L^q(\bar M)}^q \lesssim
\|\hat \nabla^j u \|_{L^q_{w+j-n/q}(M)}^q
\lesssim \sum_i \rho_i^{-(w+j)q+n} \| \Phi_i^*(\hat \nabla^j u)\|_{L^q(B_1^\Hyp)}^q.
\end{equation}
Sobolev embedding on $B_1^\Hyp$ implies
\[
\|\Phi_i^*(\hat \nabla^ju)\|_{L^q(B_1^\Hyp)}
\lesssim
\|\Phi_i^*(\hat \nabla^ju)\|_{H^{s,p}(B_1^\Hyp)}
\lesssim \rho_i^{w+j} \|\hat\nabla^j u\|_{X^{s,p}_{w+j}(M)}
\le \rho_i^{w+j} \|u\|_{\mathscr X^{s,p;m}(M)}
\]
and we conclude
\[
\|\hat\nabla^j u\|_{L^q(\bar M;E)}^q \lesssim
\sum_i \rho_i^{n} \| u\|_{\mathscr X^{s,p;m}(M)}^q
\lesssim \|1\|_{L^q(\bar M);E}^q \|u\|_{\mathscr X^{s,p;m}(M)}^q;
\]
note that the estimate $\sum \rho_i^q\lesssim \|1\|_{L^q(\bar M)}$
is a consequence of Lemma \ref{lem:many-norms} and \eqref{Lp-barM-M}.
This establishes \eqref{eq:curlyX-to-H-Mbar}.

Inclusion \eqref{eq:curlyX-to-C-Mbar} is a consequence of
\eqref{eq:curlyX-to-H-Mbar} using Sobolev embedding on $\bar M$
as follows.  First consider the case $k=0$, in which case
we can assume $m=1$.  Lowering $s\ge 1$ (and additionally
lowering $p$ if $s=1$ is not sufficient)
we can assume $s=\alpha+n/p$ and we pick $q$ satisfying
\[
\frac{1}{q} - \frac{1}{n} = \frac{1}{p}-\frac{s}{n} = -\frac{\alpha}{n}.
\]
Since $s\ge 1$, and since $0<\alpha<1$, a computation shows
$1< p\le q<\infty$.  Thus Sobolev embedding on $\bar M$ yields
\[
\mathscr X^{s,p;1}(M;E)\hookrightarrow
H^{1,q}(\bar M)\hookrightarrow C^{0,\alpha}(\bar M).
\]
The embedding $\mathscr X^{s,p;m}(M)\subset C^{k,\alpha}(\bar M)$
when $s \ge k+\alpha-n/p$ and $m\ge k+1$
follows from the $k=0$ case using Lemma \ref{lemma:nabla-and-rho-on-curly}\eqref{part:curly-nabla-map}.
\end{proof}

Proposition \ref{prop:curly-to-M-bar} demonstrates that some of the additional
regularity possessed by elements of $\mathscr X^{s,p;m}(M;E)$, as compared to their
$\mathscr H^{s,p;m}(M;E)$ counterparts, continues up to the boundary.
For example, although both spaces embed in $H^{m;p}(M;E)$,
the inclusion \eqref{eq:curlyX-to-H-Mbar} shows that if $s>m$ then $X^{s,p;m}(M;E)$
embeds in a space $H^{m,q}(\bar M;E)$ with $q>p$.
Indeed the embedding is as good as if the initial space were
$H^{s,p}(\bar M;E)$ rather than $\mathscr X^{s,p;m}(M;E)$.

As discussed in \S\ref{secsec:rel-M}, we wish to work with
metrics $g$ admitting continuous compactifications $\bar g$.
If $\bar g\in \mathscr H^{s,p;m}(M;T^{0,2}M)$ then the inclusion \eqref{eq:curlyH-to-C-Mbar} of 
Proposition \ref{prop:curly-to-M-bar}
leads to the condition $m>n/p$.  But if $\bar g\in \mathscr X^{s,p;m}(M;T^{0,2}M)$, then 
this can be substantially weakened to $s>n/p$ and $m\ge 1$. In both cases,
the compactified metric is in fact H\"older continuous up to the boundary,
and this additional regularity is essential
in the parametrix construction of \S\ref{sec:fredholm}.

Sufficiently regular spaces on $\bar M$ are contained in
fortified  spaces, and the following
inclusions are an easy application of Proposition \ref{prop:Mbar-M-HX}.
\begin{proposition}[Containment of spaces on $\bar M$]\label{prop:M-bar-to-curly}
Let $E$ be a tensor bundle over $M$.
Suppose $s\in\Reals$, $m\in\Nats_{\ge 0}$, $1<p<\infty$ and $s\ge m$.
Then
\begin{align}
\label{eq:HbarM_into_curlyH}
H^{s,p}(\bar M;E) &\hookrightarrow \mathscr H^{s,p;m}(M;E); \\
\label{eq:HbarM_into_curlyX}
H^{s,p}(\bar M;E) &\hookrightarrow \mathscr X^{s,p;m}(M;E), &&
 s>\frac{n}{p}+m\text{ and } s\ge m+1.
\end{align}
\end{proposition}
Note that the hypothesis $s\ge m+1$
in embedding \eqref{eq:HbarM_into_curlyX} is
superfluous if $s\in\Nats$, and an improvement to
Proposition \ref{prop:Mbar-M-HX} would be needed to
remove it in the general case.  The requirement $s>n/p+m$
is fundamental, however, and can be thought of as arising
from additional boundary regularity possessed by
elements of $\mathscr X^{s,p;m}(M;E)$.  For example,
Corollary \ref{cor:X-bdy-awesome} below
shows that functions in $\mathscr X^{1,p;1}(M;E)$ have
Lipschitz boundary traces regardless of the value of $p$.
But $H^{s,p}(\bar M;E)$ guarantees Lipschitz boundary
traces only when $s>n/p+1$.

\subsection{Boundary asymptotics}\label{secsec:bounday-asymptotics}
Elements of fortified  spaces admit asymptotic expansions at the boundary, a
central feature that permits their use in geometric analysis.
See \S\ref{sec:ah-metrics} where these expansions are used
to precisely describe the sense in which asymptotically
hyperbolic metrics have curvatures that decay
to those of $\breve g$
when their compactifications lie in sufficiently regular fortified  spaces.
These curvature expansions
play a central role in solving the Yamabe problem, where an initial
asymptotic scalar curvature correction is necessary to first bring the
scalar curvature anomaly into the Fredholm range of the requisite PDE.
The somewhat involved
proof establishing these expansions is the subject of \S\ref{sec:asymptotics},
which culminates in the following.

\begin{reptheorem}{thm:curly-split}[Asymptotic Structure Theorem]
Suppose $1<p<\infty$, $s\in\Reals$, $m\in\Nats_{\ge 0}$, $s\ge m$
and $E$ is a tensor bundle over $M$ of weight $w$.
\begin{enumerate}
  \item If $u\in \mathscr H^{s,p;m}(M;E)$ then
 $
    u = \tau + r,
 $
  where $\tau\in \mathscr H^{\infty,p;m}(M;E)$ and $r\in H^{s,p}_{w+m-n/p}(M;E)$.
  
  Moreover, if $u\in H^{s,p}_{w+k+\alpha-n/p}(M;E)$ for some integer
  $0\le k\le m-1$ and some $\alpha\in[0,1)$, then
  $\tau=\rho^k \eta$ for some $\eta \in \mathscr H^{\infty,p;m-k}(M;E)\cap H^{\infty,p}_{w+\alpha-n/p}(M;E)$.  
  
  Finally, if $m\ge 1$ then $u|_{\partial M}=0$ if and only if $u\in H^{s,p}_{w+1-n/p}(M;E)$.

  \item If $u\in \mathscr X^{s,p;m}(M;E)$ then
  $
    u = \tau + r,
  $
  where $\tau\in \mathscr X^{\infty;m}(M;E)$ and $r\in X^{s,p}_{w+m}(M;E)$.
  
  Moreover, if $u\in X^{s,p}_{w+k}(M;E)$ for some
  integer $0\le k\le m$ and some $\alpha\in[0,1)$, then
  $\tau=\rho^k \eta$ for some $\eta \in \mathscr X^{\infty;m-k}(M;E)\cap X^{\infty}_{w+\alpha}(M;E)$. 
  
  Finally, if $m\ge 1$ then $u|_{\partial M}=0$ if and only if $u\in X^{s,p}_{w+1}(M;E)$.
\end{enumerate}
\end{reptheorem}
Theorem \ref{thm:curly-split} is a consequence of Theorem \ref{thm:Taylor},
which establishes a finer-grained Taylor expansion in a coordinate-dependent setting.
The decomposition $u=\tau + r$ should be thought of as a split into a component
$\tau$ that carries boundary information and which is
smooth in the interior,
along with a remainder $r$ that is less regular in the interior and that vanishes on the boundary
at a rate controlled by $u$'s limited boundary regularity.
Indeed, to see that $r|_{\partial M}=0$
observe
\[
H^{s,p}_{w+m-n/p}(M;E) \hookrightarrow  H^{1,p}_{w+1-n/p}(M;E)
\hookrightarrow \mathscr H^{1,p;1}(M;E) = H^{1,p}(M;E)
\]
and hence $r$ admits a trace.  Moreover, since the smooth compactly supported sections
are dense, the trace vanishes for all
elements of $H^{s,p}_{w+m-n/p}(M;E)$.  
The same conclusion holds if $r\in X^{s,p}_{w+m}(M;E)$, since
$X^{s,p}_{w+m}(M;E)\hookrightarrow H^{s,p}_{w+m-n/p}(M;E)$.

Our earlier claim that elements of $\mathscr X^{1,p;1}(M;E)$ have
Lipschitz traces follows from Theorem \ref{thm:curly-split} and
the regularity of $\mathscr X^{\infty;1}(M;E)$.
\begin{corollary}\label{cor:X-bdy-awesome}
Suppose $1<p<\infty$ and $E$ is a tensor bundle over $M$.
If $u\in \mathscr X^{1,p;1}(M;E)$ then $u|_{\partial M}$ is Lipschitz
continuous.
\end{corollary}
\begin{proof}
Suppose $u\in \mathscr X^{1,p;1}(M)$.  Theorem \ref{thm:curly-split} implies
$u|_{\partial M} = \tau|_{\partial M}$ for some
$\tau\in \mathscr X^{\infty;m}(M)$.
But we observed at the start of the section that
$\mathscr X^{\infty;m}(M)=\mathscr C^{\infty;m}(M)$, and
\cite{WAH} Lemma 2.3(c) implies $\mathscr C^{\infty;m}(M)
\subset C^{0,1}(\bar M;E)$.
\end{proof}

\subsection{Relationships between fortified  spaces}

The following result enumerates relationships
between fortified  spaces. Some of these have been mentioned earlier
and are reiterated now for convenience.  The important new principle
is that embedding $\mathscr H^{s,p;m}(M;E)$ into $\mathscr X^{s,q;\ell}(M;E)$
entails a loss of boundary regularity and is possible only if $m>\ell$.
\begin{lemma}
Suppose $1<p,q<\infty$, $k,\ell,m\in\Nats_{\ge 0}$, $k\ge m$, $0\le \alpha\le 1$, $s\in\Reals$, and $s\ge m$. The following relationships hold, and all stated inclusions are continuous.
\begin{align}
  \label{eq:cX-to-cH-aa}
  \mathscr X^{s,p;m}(M;E)&\hookrightarrow \mathscr H^{s,p;m}(M;E);
  \\
  \label{eq:cH-to-cX-bb}
  \mathscr H^{s,p;m}(M;E)&\hookrightarrow \mathscr X^{s,p;\ell}(M;E),&\quad&
  m>\ell+\frac{n}{p};
  \\
  \label{eq:cC-to-cX-cc}
  \mathscr C^{k,\alpha;m}(M;E)&\hookrightarrow \mathscr X^{k,p;m}(M;E);
  \\
  \label{eq:cX-to-cC-dd}
   \mathscr X^{s,p;m}(M;E)&\hookrightarrow \mathscr C^{k,\alpha;m}(M;E),
   &\quad& \frac{1}{p}-\frac{s}{n} \le -\frac{k+\alpha}{n}, \quad \alpha>0;
   \\
  \label{eq:cXinf-to-cCinf-ee}
  \mathscr C^{\infty;m}(M;E)&=\mathscr X^{\infty;m}(M;E);
  \\
  \label{eq:cHinf-to-cCinf-ff}
  \mathscr H^{\infty;m}(M;E)&\hookrightarrow \mathscr C^{\infty;\ell}(M;E),
  &\quad& m>\ell+\frac{n}{p};
  \\
  \label{eq:cH-not-to-cX-gg}
  \mathscr H^{s,p;1}(M;E) &\not\subset \mathscr X^{1,q;1}(M;E),
  &\quad& s\ge 1;
  \\
  \label{eq:cHinf-not-to-cCinf-hh}
  \mathscr H^{\infty,p;1}(M;E)&\not\subset \mathscr C^{\infty;1}(M;E).
\end{align}
\end{lemma}
\begin{proof}
Embeddings \eqref{eq:cX-to-cH-aa}, \eqref{eq:cC-to-cX-cc}, and \eqref{eq:cX-to-cC-dd}
follow from Lemma \ref{lem:basic-inclusions} and  Proposition \ref{prop:SobolevEmbedding}.
To establish \eqref{eq:cH-to-cX-bb}, suppose that $u\in \mathscr H^{s,p;m}(M;E)$ and 
$m>\ell+n/p$.  For each $0\le j \le m$ we apply Proposition \ref{prop:curly-H-into-roman-X}
and the inequality $m>\ell-n/p\ge j-n/p$ to find
\[
\|\hat\nabla^j u\|_{X^{s-j,p}_{w+j}(M)} \lesssim \|\hat\nabla^j u\|_{\mathscr H^{s-j,p;m-j}(M)}
\lesssim \|u\|_{\mathscr H^{s,p;m}(M)}.
\]

Embedding \eqref{eq:cXinf-to-cCinf-ee} follows from embeddings \eqref{eq:cC-to-cX-cc} and \eqref{eq:cX-to-cC-dd},  whereas equality \eqref{eq:cHinf-to-cCinf-ff} stems from
embeddings \eqref{eq:cH-to-cX-bb} and \eqref{eq:cX-to-cC-dd}.

The non-containments \eqref{eq:cH-not-to-cX-gg} and \eqref{eq:cHinf-not-to-cCinf-hh}
come from a single example.  Consider the function $u=\rho^{1-1/p+\epsilon}$ for
some $0<\epsilon<1/p$. For any $s$ and $p$ we have
\[
1\in X^{s,p}_0(M) \hookrightarrow H^{s,p}_{1/p-n/p-\epsilon}(M),
\]
and a similar inclusion yields $\rho^{-1}d\rho \in H^{s,p}_{1/p-n/p-\epsilon}(M)$.
Hence $u\in H^{s,p}_{1-1/p}(M)$ and  $\hat\nabla u\in H^{s,p}_{1-n/p}(M)$,
implying $u\in \mathscr H^{s,p;m}(M)$.
Since $1-1/p<1$, and since each
\[
\Phi_i^* \hat \nabla u = \rho_i^{1-1/p+\epsilon} dy,
\]
we find $\hat\nabla u\not \in X^{1,q}_{1}(M)$ for any choice of $q$
and therefore $u\not\in \mathscr X^{1,q;1}(M)$.  This
establishes \eqref{eq:cH-not-to-cX-gg}, and the non-containment \eqref{eq:cHinf-not-to-cCinf-hh} follows from the observation in the proof
of Corollary \ref{cor:X-bdy-awesome} that elements of $\mathscr C^{\infty;1}(M)$
are Lipschitz on $\bar M$ whereas $u$ is not.
\end{proof}

\subsection{Multiplication and nonlinearities}

The following variation of Corollary \ref{cor:multiplication},
suffices for the limited circumstances where we multiply elements
of fortified  spaces.  Much of the
work in the proof stems from managing terms of the form $-n/p$ appearing in weights
in the Sobolev case; the inclusion provided by Proposition \ref{prop:curly-H-into-roman-X} and asymptotic structure of Theorem
\ref{thm:curly-split} play central roles.

\begin{proposition}\label{prop:curly-H-tensor}
Suppose $1<p<\infty$, $s\in\Reals$, $m\in\Nats$, $s\ge m$, and $s>n/p$. Moreover, suppose that for each $j=1,\ldots,\ell$ we have a tensor bundle $E_j$ of weight $w_j$
over $M$ and a nonnegative integer $k_j$ such that $k_1+\cdots + k_\ell = k\leq m$.

\begin{enumerate}
\item Assume in addition that $m>n/p$. If $u_j\in \mathscr H^{s,p;m}(M;E_j)$, then
$
\hat \nabla^{k_1} u_1\otimes\cdots\otimes\hat\nabla^{k_\ell} u_\ell \in \mathscr H^{s-k,p;m-k}(M;\otimes_{j} E_j).
$

\item If $u_j\in \mathscr X^{s,p;m}(M;E_j)$, then
$
\hat \nabla^{k_1} u_1\otimes\cdots\otimes\hat\nabla^{k_\ell} u_\ell \in \mathscr X^{s-k,p;m-k}(M;\otimes_{j} E_j).
$

\end{enumerate}
In both cases
$
(u_1,\ldots,u_\ell)\mapsto \hat \nabla^{k_1} u_1\otimes\cdots\otimes\hat\nabla^{k_\ell} u_\ell
$
is continuous between the stated spaces.
\end{proposition}
\begin{proof}
We prove the first claim, which is harder, and outline the changes needed for the
second claim at the end of the proof.
Let $w_1,\dots, w_\ell$ be the weights of $u_1,\dots, u_\ell$; set $w = w_1+\dots + w_\ell$.
From Theorem \ref{thm:curly-split} we have
\begin{equation}\label{eq:uj_dcomp}
u_j = \tau_j + r_j,
\end{equation}
where $\tau_j\in \mathscr H^{\infty,p;m}(M)$ and
$r_j\in H^{s,p}_{w_j + m-n/p}(M)$.
Write the product $\nabla^{k_1} u_1\otimes\cdots\otimes\hat\nabla^{k_\ell} u_\ell$ as the sum
of
\begin{equation}
\label{eq:mult-smooth-part}
\tau = \hat\nabla^{k_1}\tau_1 \otimes \cdots  \otimes \hat\nabla^{k_\ell}\tau_\ell
\end{equation}
and $r$, which is the sum of products containing at least one factor $\hat\nabla^{k_j} r_j$; both $\tau$ and $r$ are tensors of weight $w+k$.

First we analyze $\tau$.
For each $j$, Proposition \ref{prop:curly-to-M-bar} implies
\[
\hat\nabla^{k_j}\tau_j \in \mathscr H^{\infty, p;m-k_j}(M) \subset H^{m-k_j, p}(\bar M).
\]
Because $s>n/p$, standard multiplication properties of Sobolev spaces on $\bar M$ analogous to
those of Proposition \ref{prop:multiplication}
together with Proposition \ref{prop:M-bar-to-curly}
imply
\[
\tau \in H^{m-k,p}(\bar M) \subset \mathscr H^{m-k, p; m-k}(M).
\]
Applying $\rho\hat\nabla$ to $\tau$, and using that $\rho\hat\nabla \hat\nabla^{k_j}\tau_j \in \mathscr H^{\infty, p;m-k_j}(M)$ for each $j$, we find $\rho\hat\nabla \tau \in \mathscr H^{m-k, p; m-k}(M)$ as well.
Thus Lemma \ref{lem:curly-H-improvement} implies $\tau \in \mathscr H^{m-k+1, p; m-k}(M)$.
Iterating this argument, we find $\tau \in \mathscr H^{\infty, p; m-k}(M)$.

We show that $r\in \mathscr H^{s-k, p;m-k}(M)$ by establishing $\hat\nabla^i r\in H^{s-k-i, p}_{w+k+i-n/p}(M)$ for each $i\leq m-k$.
Consider first those terms of $\hat\nabla^i r$ of the form
\begin{equation}
\label{curlyHtensor-first-mixed-term}
\hat\nabla^{a_1}r_1\otimes \hat\nabla^{a_2}\tau_2\otimes\dots\otimes \hat\nabla^{a_\ell}\tau_\ell,
\end{equation}
where $a_1+\dots+a_\ell = i+k\leq m$.
Recall $r_1\in H^{s,p}_{w_1 + m - n/p}(M)$ and
therefore $\hat\nabla^{a_1}r_1\in H^{s-a_1, p}_{w_1 + m - n/p}(M)$.
For $j=2,\dots, \ell$, we use Proposition \ref{prop:curly-H-into-roman-X} and the hypothesis $m>n/p$
to obtain $\tau_j \in X^{\infty}_{w_j}(M)$ and thus $\hat\nabla^{a_j}\tau_j \in X^{\infty}_{w_j}(M)$.
The expression \eqref{curlyHtensor-first-mixed-term} is thus an element of
\begin{equation}
\label{curlyHtensor-mixed-space}
H^{s-a_1, p}_{w_1 + m - n/p}(M)
\times X^{\infty}_{w_2}(M)
\times\dots\times X^{\infty}_{w_\ell}(M),
\end{equation}
which by
Proposition \ref{prop:multiplication} is contained in $H^{s-k-i, p}_{w + m-n/p}(M) \subset H^{s-k-i,p}_{w+k_i-n/p}(M)$ as desired.
Terms of $\hat\nabla^i r$ containing exactly one $r_j$ factor are similarly in $H^{s-k-i,p}_{w+k_i-n/p}(M)$.

The remaining terms of $\hat\nabla^i r$ arise from taking terms analogous to \eqref{curlyHtensor-first-mixed-term} and replacing some of the $\hat\nabla^{a_j}\tau_j$ with $\hat\nabla^{a_j}r_j$.
This results in the replacement of $X^{\infty}_{w_j}(M)$ by $H^{s-a_j,p}_{w_j + m-n/p}(M)$
in \eqref{curlyHtensor-mixed-space}.
As $m-n/p>0$, one may use Proposition \ref{prop:multiplication} to show these products are also included in $H^{s-k-i, p}_{w + m-n/p}(M) \subset H^{s-k-i,p}_{w+k_i-n/p}(M)$.
Thus $r$, and hence $\tau + r$, is in $\mathscr H^{s-k,p;m-k}(M)$.

It remains to establish continuity, and we employ
the Closed Graph Theorem.  Suppose $v_a$ is a sequence in $\mathscr H^{s,p;m}(M)$ converging
to $u_1$, and that
\[
\hat\nabla^{k_1} v_a\otimes \hat\nabla^{k_2}u_2\otimes\cdots\otimes
\hat\nabla^{k_\ell}u_\ell \to w \in \mathscr H^{s-k,p;m-k}(M).
\]
Reducing to subsequences, since $m-k\ge 0$, we can
assume that the above convergence, as well as
$\hat\nabla^{k_1} v_a \to \hat\nabla^{k_1} u$
occurs pointwise almost everywhere.
But then $w = \hat\nabla^{k_1} u_1\otimes\cdots\otimes
\hat\nabla^{k_\ell} u_\ell$, which proves the continuity
of the map
\[
v\mapsto \hat\nabla^{k_1}v\otimes\cdots\otimes \hat\nabla^{k_\ell}u_\ell.
\]
The same argument works for each factor, and continuity of the multilinear map then
follows from the Uniform Boundedness Principle.

In the Gicquaud-Sakovich case the argument above can be simplified.  One can
show that $\tau$ in equation \eqref{eq:mult-smooth-part}
belongs to $\mathscr X^{\infty;m-k}(M)$ as above or
directly from the definition of $\mathscr X^{\infty;m-k}(M)$
and Proposition \ref{prop:multiplication} using only the hypothesis $s>n/p$.
In mixed expressions
such as \eqref{curlyHtensor-first-mixed-term}, terms of the form
$\hat\nabla^{a_i}\tau_i$ lie in $X^{s-a_i,p}_{w_i+a_i}(M)$
and each $\hat\nabla^{a_i}r_i$ lies in
$X^{s-a_i,p}_{w_i+m}(M;*)\subset X^{s-a_i,p}_{w_i+a_i}(M)$, again by virtue
of the definition in equation \eqref{define-scrX-int}.  The product
then belongs to $X^{s-k,p}_{w+k}(M)$ by application of Proposition
\ref{prop:multiplication}.  There is no use of Proposition \ref{prop:curly-H-into-roman-X}
and hence the additional hypothesis $m>n/p$ is not needed.
\end{proof}

There are two important
consequences of Proposition \ref{prop:curly-H-tensor}.
First, specializing to two factors and taking no derivatives
we have the following Banach algebra property.
\begin{proposition} \label{prop:curly-h-banach-algebra}
Suppose $1<p<\infty$, $s\in\Reals$, $m\in\Nats$, $s\ge m$, and
$s>n/p$.
Let $E_1$ and $E_2$ be tensor bundles over $M$.
Then the tensor product yields continuous
bilinear maps
\begin{gather*}
\mathscr H^{s,p;m}(M;E_1)\times \mathscr H^{s,p;m}(M;E_2)
\to \mathscr H^{s,p;m}(M;E_1\otimes E_2),
\\
\mathscr X^{s,p;m}(M;E_1)\times \mathscr X^{s,p;m}(M;E_2)
\to \mathscr X^{s,p;m}(M;E_1\otimes E_2),
\end{gather*}
where the first claim requires additionally $m>n/p$.
\end{proposition}

The second important consequence of Proposition
\eqref{prop:curly-H-tensor} is that smooth nonlinear functions
determine well-defined maps on $\mathscr H^{s,p;m}(M;\Reals)$
whenever the Banach algebra condition
$s\ge m>n/p$ holds and $s\ge 1$.

\begin{proposition}\label{prop:nonlin-curly}
Suppose $1<p<\infty$,
$m\in\Nats$ and $s> n/p$, and $s\ge m$.
Let $F\in C^\infty(\Reals)$.
\begin{enumerate}
\item
Assume additionally $m>n/p$.  If $u\in \mathscr H^{s,p;m}(M;\Reals)$
then
$
F(u) \in \mathscr H^{s,p;m}(M;\Reals).
$
\item
If $u\in \mathscr X^{s,p;m}(M;\Reals)$
then
$
F(u) \in \mathscr X^{s,p;m}(M;\Reals).
$
\end{enumerate}
\end{proposition}
\begin{proof}
The two results are proved similarly and we assume that $u\in \mathscr H^{s,p;m}(M)$.
We need to show that
\[
\hat\nabla^k f(u) \in H^{s,p}_{k-n/p}(M)
\]
for each $0\le k\le m$.
Now $\hat\nabla^k f(u)$ is a sum of terms of the form
\[
f^{(a)}(u) \hat\nabla^{k_1} u\otimes\cdots \otimes \hat\nabla^{k_\ell} u,
\]
where $k_1+\cdots+k_\ell =k$ and where $a\le k$.  Corollary \ref{cor:nonlin-into-X}
implies $f^{(a)}(u)\in X^{s,p}_0(M)$
whereas Proposition \ref{prop:curly-H-tensor} implies
$\hat\nabla^{k_1} u\otimes\cdots \otimes \hat\nabla^{k_\ell} u
\in \mathscr H^{s-k,p;m-k}(M;*)\subset H^{s-k,p}_{k-n/p}(M)$;
we have used the fact that $\hat\nabla^{k_1} u\otimes\cdots \otimes \hat\nabla^{k_\ell} u$ has weight $k$.  An application of Proposition \ref{prop:multiplication}
finishes the proof.
\end{proof}

\section{Asymptotically hyperbolic metrics}
\label{sec:ah-metrics}
In the introduction, we motivated the definition of asymptotically hyperbolic metrics with the identity
\begin{equation}
\label{relate-curvature-operators}
\Riem[g] = -|d\rho|^2_{\bar g} \Id
+ 2\rho\,\delta \KN \Hess_{\bar g}(\rho)^\sharp
+ \rho^2 \Riem[\bar g]
\end{equation}
relating the curvature operator $\Riem[g]$ of metric $g$ with that of $\bar g = \rho^2 g$.
(Recall that we view $\Riem[g]$ as a $(2,2)$ tensor, that $\Id$ and $\delta$ are the identity operators on $\Lambda^2(TM)$ and $\Lambda^1(TM)$, and that $\KN$ is the Kulkarni-Nomizu product; see \cite{WAH}.)
In particular, the identity \eqref{relate-curvature-operators} indicates that if $\bar g$ is sufficiently regular, then curvatures of $g$ are asymptotically determined by the value of $|d\rho|_{\bar g}$ along $\partial M$.
The main result of this section, 
Proposition \ref{prop:scalar-curvature-H} below, shows that the regularity assumptions of the introduction are sufficient to conclude that if $|d\rho|_{\bar g}=1$ along $\partial M$, then the curvatures of $g$ approach those of hyperbolic space as $\rho\to 0$.

To this end, recall from the introduction our assumption that $g$ satisfies either
\begin{itemize}
\item condition \eqref{intro:H-class}, in which case 
$\bar g \in \mathscr H^{s,p;m}(M;T^{0,2}M)$ with $s\in \mathbb R$ and $m\in \mathbb N$ satisfying $s\geq m > n/p$ and $1<p<\infty$;
 or
\item condition \eqref{intro:X-class}, in which case $\bar g \in \mathscr X^{s,p;m}(M;T^{0,2}M)$ with $s\in \mathbb R$ and $m\in \mathbb N$ satisfying $s\geq m \geq 1$, $s<n/p$, and $1<p<\infty$.
\end{itemize}
In either case, Proposition \ref{prop:curly-to-M-bar} implies that $\bar g$ extends continuously to $\bar M$.
We say that $g$ is \Defn{asymptotically hyperbolic of class $\mathscr H^{s,p;m}$ or $\mathscr X^{s,p;m}$}, respectively, if the relevant regularity assumption above is satisfied, if $g$ is conformally compact (so that $\bar g$ extends to a Riemannian metric on $\bar M$), and if $|d\rho|_{\bar g} =1$ along $\partial M$.

\begin{proposition}\label{prop:scalar-curvature-H}
\strut
\begin{enumerate}
\item Suppose $g$ is an asymptotically hyperbolic metric of class $\mathscr H^{s,p;m}$ satisfying \eqref{intro:H-class}.
Then the curvature operator of $g$ satisfies
$$\Riem[g] =-\Id +  \rho \tau_\textup{Riem} + r_\textup{Riem},$$
where $\tau_\textup{Riem}\in \mathscr H^{\infty, p;m-1}(M; T^{2,2}M)$ and $r_\textup{Riem} \in H^{s-2,p}_{m-n/p}(M; T^{2,2}M)$.

Consequently, the scalar curvature of $g$ satisfies
$$\R[g]  = -n(n-1) + \rho \tau_\textup{R} + r_\textup{R},$$ where $\tau_\textup{R}\in \mathscr H^{\infty, p;m-1}(M;\mathbb R)$ and where $r_\textup{R} \in H^{s-2,p}_{m-n/p}(M;\mathbb R)$. 

Moreover
\begin{equation}\label{eq:scalar-curly-H-to-X}
\R[g]+n(n-1) \in X^{s-2,p}_{\beta}(M;\mathbb R),
\end{equation}
where $\beta=\min(m-n/p,1)$ unless $m=n/p+1$, in which case
inclusion \eqref{eq:scalar-curly-H-to-X} holds for any $\beta<1$.

\item Suppose $g$ is an asymptotically hyperbolic metric of class $\mathscr X^{s,p;m}$ satisfying \eqref{intro:X-class}.
Then the curvature operator of $g$ satisfies
$$\Riem[g] =-\Id +  \rho \tau_\textup{Riem} + r_\textup{Riem},$$
where $\tau_\textup{Riem}\in \mathscr X^{\infty;m-1}(M;T^{2,2}M)$ and $r_\textup{Riem} \in X^{s-2,p}_{m}(M;T^{2,2}M)$.

Consequently, the scalar curvature of $g$ satisfies
$$\R[g]  = -n(n-1) + \rho \tau_\textup{R} + r_\textup{R},$$ where $\tau_\textup{R}\in \mathscr X^{\infty;m-1}(M;\mathbb R)$ and where $r_\textup{R} \in X^{s-2,p}_{m}(M;\mathbb R)$. 

Moreover
\[
\R[g] + n(n-1) \in X^{s-2,p}_{1}(M;\mathbb R).
\]
\end{enumerate}
\end{proposition}
\begin{proof}
The result follows immediately from Lemmas \ref{lem:drho-H} - \ref{lem:R-bar-H} below,
the identity \eqref{relate-curvature-operators}, and its consequence
\[
\R[g] = -n(n-1)|d\rho|_{\bar g}^2 + 2(n-1)\rho \Delta_{\bar g}\rho +\rho^2\R[\bar g].\qedhere
\]
\end{proof}

The following lemmas repeatedly use that $\hat\nabla{}^k \rho$ is an element of both $\mathscr H^{s,p;m}(M;T^{0,k}M)$ and $\mathscr X^{s,p;m}(M;T^{0,k}M)$ for all $k$; this follows from Proposition \ref{prop:M-bar-to-curly}.
We furthermore use that if $\bar g\in \mathscr H^{s,p;m}(M;T^{0,2}M)$, then $\bar g^{-1}\in  \mathscr H^{s,p;m}(M;T^{2,0}M)$ and, similarly, that if  $\bar g\in \mathscr X^{s,p;m}(M;T^{0,2}M)$, then $\bar g{}^{-1}\in \mathscr X^{s,p;m}(M;T^{2,0}M)$.
These latter facts follow from $\bar g$ being continuous on $\bar M$, the determinant being a polynomial, and Proposition \ref{prop:curly-H-tensor}.

\begin{lemma}\label{lem:drho-H}
\strut
\begin{enumerate}
\item 
Suppose $g = \rho^{-2}$ is an asymptotically hyperbolic metric of class $\mathscr H^{s,p;m}$ satisfying \eqref{intro:H-class}.
Then
\[
|d\rho|_{\bar g}^2 = 1 + \rho \tau + r,
\]
where $\tau\in \mathscr H^{\infty,p;m-1}(M;\mathbb R)$ and $r\in H^{s,p}_{m-n/p}(M;\mathbb R)$. 

Moreover
\begin{equation}\label{eq:drho-in-X}
|d\rho|_{\bar g}^2 - 1 \in X^{s,p}_\beta(M;\mathbb R),
\end{equation}
where $\beta=\min(m-n/p,1)$ unless $m=n/p+1$, in which case \eqref{eq:drho-in-X}
holds for any $\beta<1$.

\item Suppose $g = \rho^{-2}\bar g$ is an asymptotically hyperbolic metric of class $\mathscr X^{s,p;m}$ satisfying \eqref{intro:X-class}. 
Then
\[
|d\rho|_{\bar g}^2 = 1 + \rho \tau + r,
\]
where $\tau\in \mathscr X^{\infty;m-1}(M;\mathbb R)$ and $r\in X^{s,p}_{m}(M;\mathbb R)$.  

Moreover
\[
|d\rho|_{\bar g}^2 - 1 \in X^{s,p}_{1}(M;\mathbb R).
\]
\end{enumerate}
\end{lemma}

\begin{proof}
Suppose first that $\bar g\in \mathscr H^{s,p;m}(M)$ and consider $f = |d\rho|^2_{\bar g} - 1$.
Since $|d\rho|^2_{\bar g}$ is a contraction of $\bar g{}^{-1} \otimes \hat\nabla\rho\otimes\hat\nabla\rho$, Proposition \ref{prop:curly-H-tensor} implies $f\in \mathscr H^{s,p;m}(M)$.
By Proposition \ref{prop:curly-to-M-bar} $f\in C^{0,\alpha}(\bar M)$, $\alpha = m-n/p$, and by assumption $f =0$ along $\partial M$.
Thus Theorem \ref{thm:curly-split} implies that $f\in H^{s,p}_{1-n/p}(M)$ and implies further that $f = \rho\tau + r$ with $r = H^{s,p}_{m-n/p}(M)$, and with $\tau \in \mathscr H^{\infty,p;m-1}(M)$.
Applying Proposition \ref{prop:curly-H-into-roman-X} to $\tau$, and subsequently multiplying by $\rho$, yields $\rho\tau \in X^{s,p}_\beta(M)$, where $\beta = \min(m-n/p, 1)$, unless $m = n/p+1$, in which case any $\beta<1$ suffices.

If instead $\bar g\in \mathscr X^{s,p;m}(M)$, then Proposition \ref{prop:curly-H-tensor}
implies $f\in \mathscr X^{s,p;m}(M)$ and Proposition \ref{prop:curly-to-M-bar} again shows $f\in C^{0,\alpha}(\bar M)$ with zero trace along $\partial M$.
Hence Theorem \ref{thm:curly-split} implies $f\in X^{s,p}_1(M)$
and moreover that $f = \rho\tau + r$ with $r\in X^{s,p}_m(M)$ and $\tau \in \mathscr X^{\infty;m-1}(M)$.
This immediately yields the remaining claims of this case.
\end{proof}

\begin{lemma}\label{lem:Delta-rho-H}
\strut
\begin{enumerate}
\item Suppose $g = \rho^{-2}\bar g$ is an asymptotically hyperbolic metric of class $\mathscr H^{s,p;m}$ satisfying \eqref{intro:H-class}. 
Then
\begin{enumerate}
\item $
\Hess_{\bar g}(\rho)^\sharp = \tau_H + r_H
$
with 
\begin{align*}
\tau_H&\in \mathscr H^{\infty,p;m-1}(M;T^{1,1}M),\\
r_H&\in H^{s-1,p}_{m-1-n/p}(M;T^{1,1}M).
\end{align*}

\item
$
\rho \Delta_{\bar g} \rho = \rho \tau_L + r_L
$
with 
$
\tau_L \in \mathscr H^{\infty,p;m-1}(M;\mathbb R)
$ and 
$
r_L\in H^{s-1,p}_{m-n/p}(M;\mathbb R).
$
In particular,
\begin{equation}\label{eq:rhoLap-in-X}
\rho\Delta_{\bar g}\rho \in X^{s-1,p}_\beta(M;\mathbb R),
\end{equation}
where $\beta=\min(m-n/p,1)$ unless $m=n/p+1$, in which case
\eqref{eq:rhoLap-in-X} holds with any $\beta<1$.

\end{enumerate}

\item Suppose $g =\rho^{-2}\bar g$ is an asymptotically hyperbolic metric of class $\mathscr X^{s,p;m}$ satisfying \eqref{intro:X-class}.
Then
\begin{enumerate}
\item $
\Hess_{\bar g}(\rho)^\sharp = \tau_H + r_H
$
with 
\begin{align*}
\tau_H&\in \mathscr X^{\infty;m-1}(M;T^{1,1}M),\\
r_H& \in X^{s-1,p}_{m-1}(M;T^{1,1}M).
\end{align*}

\item
$
\rho \Delta_{\bar g} \rho = \rho_L \tau + r_L
$
where $\tau_L \in \mathscr X^{\infty;m-1}(M;\mathbb R)$ and $r_L\in X^{s-1,p}_{m}(M;\mathbb R)$.
In particular
\begin{equation}\label{eq:rhoLap-X-X}
\rho\Delta_{\bar g}\rho \in X^{s-1,p}_1(M;\mathbb R).
\end{equation}
\end{enumerate}
\end{enumerate}
\end{lemma}

\begin{proof}
First consider the case when $\bar g \in \mathscr H^{s,p;m}(M)$ and write $\Hess_{\bar g}(\rho)^\sharp $ as a linear combination of contractions of $\bar g{}^{-1}\otimes \hat\nabla^2 \rho$ and
$\bar g{}^{-1}\otimes \bar g^{-1} \otimes \hat\nabla \bar g \otimes \hat\nabla \rho.$
At most one derivative of $\bar g$ appears in these expressions; thus Proposition \ref{prop:curly-H-tensor}
and the smoothness of $\rho$ implies
$\Hess_{\bar g}(\rho)^\sharp \in \mathscr H^{s-1,p;m-1}(M)$.
Theorem \ref{thm:curly-split} then implies the first claim in this case.
Contracting and multiplying by $\rho$ yields the second claim from the first.
The embedding \eqref{eq:rhoLap-in-X}
follows from applying Proposition \ref{prop:curly-H-into-roman-X} to $\Delta_{\bar g}\rho\in \mathscr H^{s-1,p;m-1}(M)$ and then multiplying by $\rho$.

The argument when $\bar g\in \mathscr X^{s,p;m}(M)$ follows from analogous logic,
except that Proposition \ref{prop:curly-H-into-roman-X} is not needed
to establish \eqref{eq:rhoLap-X-X}, which follows immediately from
$\Delta_{\bar g}\rho\in \mathscr X^{s-1,p;m-1}(M)$.
\end{proof}

\begin{lemma}\label{lem:R-bar-H}
\strut
\begin{enumerate}
\item Suppose that $g =\rho^{-2}\bar g$ is an asymptotically hyperbolic metric of class
$\mathscr H^{s,p;m}$ satisfying \eqref{intro:H-class}.
If $m\ge 2$ then
\begin{enumerate}
\item \label{part:Riem-barg-H-decomp}
$\Riem[\bar g] = \tau_\textup{Riem} + r_\textup{Riem}$, where
\begin{align*}
\tau_\textup{Riem} &\in \mathscr H^{\infty,p;m-2}(M;T^{2,2}M),\\
r_\textup{Riem}&\in H^{s-2,p}_{m-2-n/p}(M;T^{2,2}M).
\end{align*}

\item \label{part:R-barg-H-decomp}
$\rho^2\R[\bar g] = \rho^2\tau_\textup{R} + r_\textup{R}$, where $\tau_\textup{R}\in \mathscr H^{\infty,p,m-2}(M;\mathbb R)$ and $r_\textup{R}\in H^{s-2,p}_{m-n/p}(M;\mathbb R)$.
\end{enumerate}
If $m=1$, in which case $\mathscr H^{\infty,p;m-2}(M)$ is not defined,
then (i) and (ii) remain true setting $\tau_\textup{Riem}=0$ and $\tau_\textup{R}=0$.
In all cases
\begin{equation}\label{eq:R-barg-H-X}
\rho^2 \R[\bar g]\in X^{s-2,p}_\beta(M;\mathbb R),
\end{equation}
where $\beta=\min(m-n/p,2)$ unless $m=n/p+2$ in which case \eqref{eq:R-barg-H-X} holds
with any $\beta<2$.

\item Suppose $g =\rho^{-2}\bar g$ is an asymptotically hyperbolic metric of class
$\mathscr X^{s,p;m}$ satisfying \eqref{intro:X-class}. 
If $m\ge 2$ then
\begin{enumerate}
\item $\Riem[\bar g] = \tau_\textup{Riem} + r_\textup{Riem}$, where
\begin{align*}
\tau_\textup{Riem} &\in \mathscr X^{\infty;m-2}(M;T^{2,2}M),\\
r_\textup{Riem}&\in X^{s-2,p}_{m-2}(M;T^{2,2}M).
\end{align*}

\item $\rho^2\R[\bar g] = \rho^2\tau_\textup{R} + r_\textup{R}$ where $\tau_\textup{R}\in \mathscr X^{\infty;m-2}(M;\mathbb R)$ and $r_\textup{R}\in X^{s-2,p}_{m}(M;\mathbb R)$.
\end{enumerate}
If instead $m=1$, then (i) and (ii) remain true setting 
$\tau_\textup{Riem}=0$ and $\tau_\textup{R}=0$.
In all cases
\begin{equation}\label{eq:R-barg-X-X}
\rho^2 \R[\bar g]\in X^{s-2,p}_{\min(m,2)}(M;\mathbb R).
\end{equation}
\end{enumerate}
\end{lemma}

\begin{proof}
The proof relies on writing the curvature operator $\Riem[\bar g]$ as a linear combination
of contractions of
\begin{equation}
\label{eq:R-bar-g-pieces}
\bar g^{-1} \otimes \bar g^{-1} \otimes \hat \nabla^2 \bar g
\quad\text{ and }\quad
\bar g^{-1} \otimes \bar g^{-1} \otimes \hat \nabla \bar g
\otimes \bar g^{-1} \otimes \hat \nabla \bar g.
\end{equation}

If $\bar g\in \mathscr H^{s,p;m}(M)$ with $m\ge 2$ then we can apply Proposition \ref{prop:curly-H-tensor}
to conclude each of these terms lies in $\mathscr H^{s-2,p;m-2}(M)$,
and hence so does $\Riem[\bar g]$.
Theorem \ref{thm:curly-split}
then implies the first claim of this case;
the second claim follows from the first via contraction and multiplication by $\rho^2$.
The final claim of the lemma follows from applying Proposition \ref{prop:curly-H-into-roman-X} to $\R[\bar g]\in \mathscr H^{s-2,p;m-2}(M)$ and then multiplying by $\rho^2$.

If $\bar g\in \mathscr H^{s,p;1}(M)$ so that $m=1> n/p$, then Proposition \ref{prop:curly-H-into-roman-X} implies $\bar g{}^{-1}\in X^{s,p}_{-2}(M)$, while by definition we have $\hat\nabla \bar g\in H^{s-1,p}_{3-n/p}(M)$ and thus $\hat\nabla{}^2\bar g\in H^{s-2,p}_{3-n/p}(M)$.
We then obtain $\Riem[\bar g]\in H^{2-s,p}_{-1-n/p}(M)$, which is the first claim in this case; the second claim follows from contraction and multiplying by $\rho^2$.
The final claim follows from Lemma \ref{lem:basic-inclusions}.

The cases where $\bar g\in \mathscr X^{s,p;m}(M)$ follow from analogous, though simpler, logic.
\end{proof}

\section{The Yamabe problem}
\label{sec:yamabe}

The Yamabe problem in the asymptotically hyperbolic setting
amounts to the following: starting from an asymptotically hyperbolic
metric $g$ we wish to find
a conformally related metric
that is also asymptotically hyperbolic and that has constant scalar curvature equal to that of hyperbolic
space:
$
\Rhn:= -n(n-1).
$
Concretely, we seek a conformal
factor $\Theta>0$ such that $\Theta^{q_n-2} g$ is asymptotically hyperbolic
and such that
\begin{equation}\label{eq:Yamabe-problem}
\R[\Theta^{q_n-2} g] = \Rhn;
\end{equation}
here $q_n=2n/(n-2)$ is the critical exponent
associated with the Sobolev embedding of $W^{1,2}$ into $L^q$ and its use
simplifies some later expressions.
The standard formula for the conformal change of scalar curvature
\begin{equation}
\label{eq:change-scalar-curvature}
\R[\Theta^{q_n-2}] = \Big(-(q_n+2)\Delta_g\Theta + \R[g]\Theta\Big)\,\Theta^{1-q_n}
\end{equation}
allows us to rewrite equation \eqref{eq:Yamabe-problem} as a
semilinear PDE for $\Theta$.

One readily verifies that the requirement that $\Theta^{q_n-2} g$ be asymptotically hyperbolic
imposes the boundary condition $\Theta|_{\partial M}=1$.
Writing $\Theta=1+u$ with
$u$ vanishing  on $\partial M$,
we rewrite equation \eqref{eq:Yamabe-problem} as a
semilinear PDE
\begin{equation}\label{eq:Yamabe-eq}
-a_n \Delta_{g} u + \R[g] u =
\Rhn\left[ (1+u)^{q_n-1} - 1\right] - (\R[g]-\Rhn),
\end{equation}
where $a_n = q_n+2$, that we wish to solve for $u$.

An important consideration is that the regularity class of the metric should not
be changed under the conformal transformation.  In particular, the solution metric
should be as acceptable a
starting point for the construction as the metric we began with.  In this section we obtain a unique solution, within the regularity class, of the Yamabe problem for asymptotically hyperbolic metrics of the following classes
\begin{align}
\mathscr H^{s,p;m} &\text{ with $s\ge m> n/p$,}\label{eq:curly-H-class}\\
\mathscr X^{s,p;m} &\text{ with $s\ge m\ge 1$ and $s>n/p$}.\label{eq:curly-X-class}
\end{align}
In both cases the restrictions on the parameters ensure the important condition
that the compactified metric $\bar g = \rho^2 g$ be H\"older continuous up to the boundary
on $\bar M$.

To begin, we observe that these classes of metric behave well under conformal
changes.
\begin{lemma}\label{lem:conf-change-curly}
Let $1<p<\infty$, $s\in\Reals$ and $m\in\Nats$.
\begin{enumerate}
  \item Suppose $g$ is an asymptotically hyperbolic metric of class $\mathscr H^{s,p;m}$
  with $s\ge m> n/p$.  If $\Theta\in \mathscr H^{s,p;m}(M;\mathbb R)$ is a conformal factor such that
  $\Theta|_{\partial M}=1$, then $\Theta^{q_n-2} g$ is an asymptotically hyperbolic metric
  of class $\mathscr H^{s,p;m}$.  In particular, this holds for $\Theta=1+u$ with
  $u\in H^{s,p}_{m-n/p}(M;\mathbb R)$.
  
  \item Suppose $g$ is an asymptotically hyperbolic metric of class $\mathscr X^{s,p;m}$
  with $s>n/p$ and $s\geq m\ge 1$.
  If $\Theta\in \mathscr X^{s,p;m}(M;\mathbb R)$ is a conformal factor such that
  $\Theta|_{\partial M}=1$, then $\Theta^{q_n-2} g$ is an asymptotically hyperbolic metric
  of class $\mathscr X^{s,p;m}$.  In particular, this holds for $\Theta=1+u$ with
  $u\in X^{s,p}_{m}(M;\mathbb R)$.
\end{enumerate}
\end{lemma}
\begin{proof}
The invariance of the regularity classes under the given conformal changes
is an easy consequence of Propositions \ref{prop:curly-h-banach-algebra} and
\ref{prop:nonlin-curly}; the resulting metric is asympotically hyperbolic
because $\Theta|_{\partial M}=1$ and hence the norm of $d\rho$ with respect to
to the new compactified metric is unchanged on $\partial M$.
If $u\in H^{s,p}_{m-n/p}(M)$,
Lemma \ref{lem:basic-roman-curly} implies $u\in \mathscr H^{s,p;m}(M)$
and Theorem \ref{thm:curly-split} implies $u|_{\partial M} = 0$.  Hence $\Theta=1+u\in \mathscr H^{s,p;m}(M)$ and $\Theta|_{\partial M}$=1.  An analogous argument works in
the $\mathscr X^{s,p;m}$ category if $u\in X^{s,p}_{m}(M)$.
\end{proof}

When $g$ is an asymptotically hyperbolic metric of one of the regularity
classes \eqref{eq:curly-H-class}--\eqref{eq:curly-X-class}, the Laplacian $\Delta_g$ has excellent mapping properties between
weighted Sobolev spaces.
The limited regularity of the metric restricts the spaces
on which the Laplacian can naturally operate, however, and we have the following
consequence of Proposition
\ref{prop:L-mapping-S} and Lemma \ref{lem:geometric-op-mapping} below.
\begin{lemma}\label{lem:Lap-mapping}
Suppose $1<p<\infty$, $m\in\Nats$, $s\in\Reals$ and that $g$ is an asymptotically
hyperbolic metric in one of the two regularity classes \eqref{eq:curly-H-class} and \eqref{eq:curly-X-class}.  Then the Laplacian $\Delta_g$ determines continuous maps
\begin{equation}
\label{eq:lap-mapping}
\begin{aligned}
H^{\sigma,q}_\delta(M;\mathbb R) &\to H^{\sigma-2,q}_\delta(M;\mathbb R),
\\
X^{\sigma,q}_\delta(M;\mathbb R) &\to X^{\sigma-2,q}_\delta(M;\mathbb R)
\end{aligned}
\end{equation}
for any $\delta\in\Reals$ so long as $(\sigma,q)\in\Reals\times (1,\infty)$ satisfy
\begin{equation}\label{eq:lap-conds}
\begin{gathered}
2-s\leq \sigma \leq s,t,
\\
\frac{1}{p} - \frac{s}{n}
\leq \frac{1}{q} - \frac{\sigma}{n}
\leq \frac{1}{p\dual}-\frac{2-s}{n},
\end{gathered}
\end{equation}
where $p\dual$ is the Lebesgue exponent conjugate to $p$, i.e. $1/p + 1/p\dual=1$.
\end{lemma}

The restrictions \eqref{eq:lap-conds} arise as
a consequence of Proposition \ref{prop:multiplication} concerning
multiplication in Sobolev spaces
\footnote{In fact, because the Laplacian is missing a low-order term compared
to a more general geometric operator such as the conformal Laplacian,
conditions \eqref{eq:lap-conds} appearing in Lemma \ref{eq:lap-mapping}
could be weakened; see Appendix \ref{app:Lap-extra}.}.
Following the notation of \S\ref{sec:differential-ops}, we define
the \Defn{compatible Sobolev indices}
$\mathcal S^{s,p}_{2}$ of $\Delta_g$ to be the
pairs $(\sigma,q)\in\Reals\times(1,\infty)$
satisfying conditions \eqref{eq:lap-conds}; the subscript $2$ in
the notation is associated with the fact that the Laplacian is a second-order
operator.

The regularity
classes \eqref{eq:curly-H-class} and \eqref{eq:curly-X-class}
both require $s\ge 1$ and $s>n/p$.  A simple computation
shows that under these conditions $\mathcal S^{s,p}_2$ is non-empty and
indeed always contains $(\sigma,q)=(1,2)$
corresponding to $H^{1,2}_\delta$.  When $s=1$
and $p>n$, $(\sigma,q)\in\mathcal S^{s,p}_2$
if and only if $\sigma=1$ and $p\dual\le q\le p$.
This minimal level of regularity forms
the baseline of our construction.  Using the conditions $s\ge 1$ and $s>n/p$ it is easy to see that
$2-s\le 1$ and $1/p\dual-(2-s)/n>0$.
Hence we find that
if $(\sigma,q)$ satisfies
\begin{equation}
\label{eq:reduced-lap-conds}
\begin{gathered}
1\le\sigma\le s,\\
\frac{1}{p}-\frac{s}{n}\le \frac{1}{q}-\frac{\sigma}{n}<0,
\end{gathered}
\end{equation}
then $(\sigma,q)\in\mathcal S^{s,p}_2$.  This restricted sufficient condition suffices for all
of our purposes in this section
and has an easy interpretation: if $\sigma\ge 1$, then $H^{\sigma,q}_{\rm loc}(M;\mathbb R)$ consists of
H\"older continuous functions, and $H^{\sigma,q}_{\rm loc}(M;\mathbb R)$ contains $H^{s,p}_{\rm loc}(M;\mathbb R)$.

Our construction employs operators of the form $-\Delta_g + \Lambda$, where $\Lambda\in\Reals$;
these evidently also define continuous actions \eqref{eq:lap-mapping} for
$(\sigma,q)\in\mathcal S^{s,p}_{2}$.  On hyperbolic space these are Fredholm maps
only for a limited range of weights $\delta$ depending on $\Lambda$.
Identical restrictions on $\delta$ apply
for asymptotically hyperbolic metrics
of one of the regularity classes \eqref{eq:curly-H-class} and \eqref{eq:curly-X-class}.
Thus we  have following consequence of our main results of \S\ref{sec:fredholm}.
\begin{proposition}\label{prop:Lap-Fredholm-range}
Suppose $1<p<\infty$, $m\in\Nats$, $s\in\Reals$ and that $g$ is an asymptotically
hyperbolic metric in one of the regularity classes \eqref{eq:curly-H-class}, \eqref{eq:curly-X-class}.
Let $\Lambda\ge 0$ and define
\begin{equation}\label{eq:ind-radius-lap}
R = \sqrt{\left(\frac{n-1}{2}\right)^2 + \Lambda}.
\end{equation}
Then
$-\Delta_{g}+\Lambda$ is an isomorphism
\begin{equation}\label{eq:lap-H-map}
H^{\sigma,q}_\delta(M;\mathbb R)\to H^{\sigma-2,q}_\delta(M;\mathbb R)\\
\end{equation}
so long as $(\sigma,q)\in S^{s,p}_{2}$ and
\begin{equation}\label{eq:Fred-delta-H}
\left|\delta+\frac{n-1}{q} -\frac{n-1}2\right| < R.
\end{equation}
Similarly,  $-\Delta_{g}+\Lambda$ is an isomorphism
\begin{equation}\label{eq:lap-X-map}
X^{\sigma,q}_\delta(M)\to X^{\sigma-2,q}_\delta(M)\\
\end{equation}
if $(\sigma,q)\in S^{s,p}_{2}$ and
\begin{equation}\label{eq:Fred-delta-X}
\left|\delta -\frac{n-1}2\right| < R.
\end{equation}
\end{proposition}
\begin{proof}
Lemma 7.1 of \cite{Lee-FredholmOperators} implies that $R$ defined by
equation \eqref{eq:ind-radius-lap} is the
indicial radius of $-\Delta_{\breve g} + \Lambda$ acting on hyperbolic space.
So we can apply Theorems \ref{thm:H-fredholm} and \ref{thm:X-fredholm}
to find that  $-\Delta_g + \Lambda$ is a Fredholm map with index zero
when acting as in \eqref{eq:lap-H-map} or \eqref{eq:lap-X-map},
so long as
$(\sigma,q)$ and $\delta$ satisfy the stated restrictions.

To show these maps are isomorphisms it is enough to show
they have trivial kernel.
Suppose $u$ is in the kernel of $-\Delta_g + \Lambda$.
Then $u\in H^{1,2}_0(M)$ by Proposition \ref{prop:P-regularity}.
A density argument shows integration by parts
applies on $H^{1,2}_0(M)$ and therefore
\[
\int_{M} \left[|du|_g^2 +\Lambda u^2\right]\;dV_g = 0.
\]
Since $\Lambda\ge0$ we conclude that  $u$ is a constant element of $H^{1,2}_0(M)$ and therefore 0.
\end{proof}

As the constant $R$ defined in \eqref{eq:ind-radius-lap} is the indicial radius for the operator $\mathcal P[g] = -\Delta_g + \Lambda$, the inequalities \eqref{eq:Fred-delta-H}and \eqref{eq:Fred-delta-X} define the
\Defn{Fredholm ranges} of the associated operators, and correspond to \eqref{intro:define-DqR} from the introduction.  
Note that $R$ is increasing in $\Lambda$, and
hence the Fredholm range of the operators grows as $\Lambda$ increases.  Two
choices of $\Lambda$ play central roles in our construction. First, when $\Lambda=0$
inequalities \eqref{eq:Fred-delta-H}and \eqref{eq:Fred-delta-X} become
\begin{gather}
0 < \delta +\frac{n-1}{q} < n-1\label{eq:Fred-Lap-H},\\
0 < \delta  < n-1,\label{eq:Fred-Lap-X}
\end{gather}
respectively.  Second, when $\Lambda=n$ a computation shows
inequalities \eqref{eq:Fred-delta-H}and \eqref{eq:Fred-delta-X} become
\begin{gather}
-1 < \delta +\frac{n-1}{q} < n\label{eq:Fred-Lap-n-H},\\
-1 < \delta  < n.\label{eq:Fred-Lap-n-X}
\end{gather}
The Fredholm ranges for $-\Delta+n$ play an especially important role in the construction
because equation \eqref{eq:Yamabe-problem} can be rewritten
\begin{equation}\label{eq:Yamabe-rewrite}
(-\Delta_g + n)u = \frac{\Rhn}{a_n}\left[(1+u)^{q_n-1}-1-(q_n-1)u\right] + \frac{\Rhn-\R[g]}{a_n}(1+u).
\end{equation}
If $u$ has some decay at infinity, the first expression on the right-hand side of equation
\eqref{eq:Yamabe-rewrite} vanishes $O(u^2)$ at infinity and hence the known decay of $u$ can
be improved up to that of $\R[g]-\Rhn$, but only within the limits imposed by the
Fredholm range of $-\Delta_g +n$.

\subsection{Solution in the \texorpdfstring{$\mathscr H^{s,p;m}$}{Hsp;m} category}

In this section we present the solution of the Yamabe problem for
asymptotically hyperbolic metrics of class $\mathscr H^{s,p;m}$.  This
case is mildly more technical than the corresponding $\mathscr X^{s,p;m}$
construction, and we summarize the changes needed in the  $\mathscr X^{s,p;m}$
in the following section.

The construction the consists of three steps:
\begin{itemize}
  \item \textbf{Conditioning at infinity.}  We perform an initial conformal transformation
  so that $\R[g]-\Rhn$ has `rapid' decay: $\R[g]-\Rhn\in H^{s-2,p}_{m-n/p}(M;\mathbb R)$.
  The notion of rapidity is limited by the regularity of $g$ at infinity.
  Indeed the decomposition in Proposition \ref{prop:scalar-curvature-H}, based
  on Theorem \ref{thm:curly-split}, implies that
  \begin{equation}\label{eq:R-expand-H}
\R[g] = \Rhn + \rho \tau + r,
  \end{equation}
  where $\tau\in \mathscr H^{\infty,p;m-1}(M;\mathbb R)$, $r\in H^{s-2,p}_{m-n/p}(M;\mathbb R)$,
  and $\rho \tau \in H^{\infty,p}_{1-n/p}(M;\mathbb R)$.
  In effect, this
  stage is an asymptotic
  correction to the scalar curvature to eliminate $\rho \tau$
  in this expansion, and no PDEs are solved.  Nevertheless the Fredholm
  range from inequality \eqref{eq:Fred-Lap-n-H} manifests itself and we are able
  to accomplish the desired correction only if $m\le n$.  Note that this first step is needed
  only if $m>1$;
  when $m=1$ then $\rho\tau+r$ from equation \eqref{eq:R-expand-H}
  already lies in $H^{s-2,p}_{1-n/p}(M;\mathbb R)$.
\item \textbf{Candidate solution at low regularity.}
Having made the conformal change to rapid scalar curvature decay,
the comments following equation \eqref{eq:Yamabe-rewrite} suggest that
it is reasonable to seek a solution $u$ of equation \eqref{eq:Yamabe-rewrite}
in $H^{s,p}_{m-n/p}(M;\mathbb R)$. If we can find such a solution, the rapid decay of $u$ implies
$u\in \mathscr H^{s,p;m}(M;\mathbb R)$ and therefore $1+u$ is a suitable conformal
factor; see Lemma \ref{lem:conf-change-curly}.  At this stage
we seek a somewhat less regular solution.  Sobolev embedding
implies that asympototically hyperbolic metrics of class $\mathscr H^{s,p;m}$ are
also asymptotically hyperbolic of class $\mathscr H^{1,q;1}$ for some $q>n$ (see details below).
Moreover, if there exists a
solution of equation \eqref{eq:Yamabe-rewrite} in $H^{s,p}_{m-n/p}(M;\mathbb R)$,
then it also lies in $X^{1,q}_{1-n/q}(M;\mathbb R)$. Using a barrier method,
we construct a solution at this limited level of regularity.
As an initial step, we perform a conformal change to continuous scalar
curvature satisfying $\R[g] \le \Rhn$, which facilitates finding suitable barriers.
\item \textbf{Bootstrap.}  Having found a solution in $X^{1,q}_{1-n/q}(M;\mathbb R)$,
a bootstrap is used to improve the regularity up to $X^{s,p}_{1-n/q}(M;\mathbb R)$,
and then improve the decay to $X^{s,p}_{m-n/p}(M;\mathbb R)$ and finally to
upgrade to $H^{s,p}_{m-n/p}(M;\mathbb R)$.  This step uses the structural form
of the right-hand side of equation \eqref{eq:Yamabe-rewrite} to improve the
weights and again encounters the restriction $m\le n$ arising from the Fredholm range.
\end{itemize}
We remark that the obstruction $m\le n$ is real.  In three dimensions,
for smooth metrics $g$ a generic solution of \eqref{eq:Yamabe-eq}
will be polyhomogeous and contains a term $\rho^3\log\rho$,
which lies in $H^{3,p}(\bar M;\mathbb R)$ for all $1<p<\infty$ but
does not lie in $H^{4,p}(\bar M;\mathbb R)$ for any $p>1$; see \cite{ACF1992}.

\subsubsection{Conditioning at infinity}
The following proposition implements the first step of the construction outlined
and increases the decay of $\R[g]-\Rhn$ up to a threshold determined jointly by
the regularity of the metric at infinity and the Fredholm range of the
shifted Laplacian $-\Delta_g+n$.

\begin{proposition}\label{prop:solve-at-infinity}
Suppose $1<p<\infty$, $m\in\Nats$,
$s\in\Reals$, with $s\ge m>n/p$ and $m\le n$.
Given an asymptotically hyperbolic metric $g$ on $M$
of class $\mathscr H^{s,p;m}$
there exists a positive
conformal factor $\Theta \in \mathscr H^{s,p;m}(M)$
such that $\Theta|_{\partial M}=1$ and such that
\[
\R[\Theta^{-2}g] -\Rhn \in H^{s-2,p}_{m-n/p}(M;\mathbb R).
\]
\end{proposition}
\begin{proof}
First, if $m=1$ then Proposition \ref{prop:scalar-curvature-H}
implies $\R[g]\in H^{s-2,p}_{m-n/p}(M)$
without any additional work.  So we can assume $m\ge 2$.

Let $\alpha = \min(1/2,m-n/p)$; by hypothesis $\alpha >0$.
We claim that if
$\R[g]-\Rhn\in H^{s-2,p}_{k+\sigma-n/p}(M)$
for some integer $0\le k\le m-1$ and some $0\le \sigma<1$,
then there exists
a positive conformal factor $\Theta\in \mathscr H^{\infty,p;m}$
such that $\Theta =1$ along $\partial M$ and such that
\begin{equation*}
\R[\Theta^{-2}g]-\Rhn\in H^{s-2,p}_{k-n/p+\min(1,\sigma+\alpha)}(M).
\end{equation*}
That is, the decay of $\R[\Theta^{-2}g]$ can be improved by the amount
$\alpha$, up to the next integer above $k$.
Assuming the claim, the result now follows by an inductive construction.
Proposition \ref{prop:scalar-curvature-H}
implies $\R[g]-\Rhn\in H^{s-2,p}_{1-n/p}(M)$, so we
can initially take $k=1$ and $\sigma=0$; note that $k\le m-1$.
Applying the the claim iteratively we can find a sequence
of finitely many conformal factors, each
in $\mathscr H^{s,p;m}(M)$ and equal to 1 on the boundary,
that improve the scalar curvature decay rate to the next
integer, and then repeat
the process $m-2$ more times until
the final scalar curvature lies in $H^{s-2,p}_{m-n/p}(M)$.
At each step Lemma \ref{lem:conf-change-curly} and Proposition
\ref{prop:nonlin-curly} ensure that the
new metric is asymptotically hyperbolic of class $\mathscr H^{s,p;m}$.

To prove the claim,
suppose that for some $1\le k \le m-1$ and $\sigma\in [0,1)$
we have
$\R[g]-\Rhn\in H^{s-2,p}_{k+\sigma-n/p}(M)$.
From Proposition \ref{prop:scalar-curvature-H} we have $\R[g] = \Rhn + \rho\tau_0 + r_0$, where $r_0\in H^{s-2,p}_{m-n/p}(M)$ and $\tau_0 \in \mathscr H^{\infty, p;m-1}(M)$.
Hence
\[
\rho\tau_0 = \R[g] - \Rhn - r_0 \in H^{s-2,p}_{k+\sigma - n/p}(M).
\]
Lemma \ref{lemma:nabla-and-rho-on-curly} implies $\rho\tilde\tau_0 \in \mathscr H^{\infty, p;m}(M)$
as well and therefore
Theorem \ref{thm:curly-split} implies $\rho\tau_0 = \rho^k\tau + r_1$, where $\tau \in \mathscr H^{\infty, p, m-k}(M)\cap H^{\infty, p}_{\sigma - n/p}(M)$ and $r_1\in H^{s-2,p}_{m-n/p}(M)$.
We conclude finally that
\begin{equation}\label{eq:R-decomp}
\R[g] = \Rhn + \rho^{k}\tau + r
\end{equation}
with $\tau\in \mathscr H^{\infty,p;m-k}(M)\cap H^{\infty,p}_{\sigma-n/p}(M)$ and $r\in H^{s-2,p}_{m-n/p}(M)$, and we seek a conformal change that increases
the weight $\sigma-n/p$ appearing in this expression.

For a sufficiently smooth conformal factor $\omega$, possibly vanishing on $\partial M$,
a computation shows
\begin{equation}\label{eq:scalar-curvature-pieces}
\R[\omega^{-2}\bar g] - \Rhn
= \Rhn(|d\omega|_{\bar g}^2-1) + 2(n-1)\omega\Delta_{\bar g}\omega +\omega^2 \R[\bar g].
\end{equation}
As a consequence, if $\Theta\in \mathscr H^{s,p;m}(M)$ then
\begin{equation}
\label{Theta-R-conformal-change}
\begin{aligned}
\R[\Theta^{-2}g] - \Rhn
&=
R[(\Theta\rho)^{-2}\bar g] - \Rhn
\\
&=
\Rhn(|d(\Theta\rho)|_{\bar g}^2 - 1 )
+2(n-1)(\Theta\rho)\Delta_{\bar g}\Theta\rho +(\Theta\rho)^2\R[\bar g].
\end{aligned}
\end{equation}
We suppose that $\Theta$ is of the form $1+\rho^k u$, with
$u\in \mathscr H^{\infty,p;m-k}(M)\cap H^{\infty,p}_{\sigma-n/p}(M)$
to be chosen later, and gather a number of results to use when analyzing the terms of \eqref{Theta-R-conformal-change}.
From Lemma \ref{lemma:nabla-and-rho-on-curly} and Proposition \ref{prop:curly-H-into-roman-X}
we find
\begin{align*}
\rho^k u &\in \mathscr H^{\infty,p;m}(M) \cap H^{\infty,p}_{k+\sigma-n/p}(M) \cap X^{\infty}_{\alpha}(M),\\
\rho^{k+1} \hat\nabla u &\in
\mathscr H^{\infty,p;m}(M) \cap H^{\infty,p}_{k+2-n/p}(M) \cap X^{\infty}_{1}(M),
\end{align*}
where in the second line we used also that $k< m$;
note that the condition $\alpha\leq 1/2$ handles the boundary case in Proposition \ref{prop:curly-H-into-roman-X}.
In particular, $\hat\nabla u\in H^{\infty,p}_{1-n/p}(M)$ and therefore
$\hat\nabla^2 u \in H^{\infty,p}_{1-n/p}(M)$ as well.
Proposition \ref{prop:curly-H-into-roman-X} also implies that
$\bar g {}^{-1}\in X^{s,p}_{-2}(M)$ and $\hat\nabla\bar g \in X^{s-1,p}_{-1}(M).$
We use these latter facts to analyze $\Delta_{\bar g}(\Theta\rho)$, which is comprised of terms that are contractions of $\bar g{}^{-1}\otimes\hat\nabla{}^2(\Theta\rho)$ and $\bar g{}^{-1}\otimes\bar g{}^{-1}\otimes \hat\nabla\bar g \otimes\hat\nabla(\Theta\rho)$.
Finally, from Lemmas \ref{lem:drho-H} - \ref{lem:R-bar-H} we have
\begin{equation*}
|d\rho|_{\bar g}^2 -1 \in  X^{s,p}_{\alpha}(M),\qquad
\rho\Delta_{\bar g}\rho \in X^{s-1,p}_{\alpha}(M),\qquad
\rho^2 \R[\bar g] \in X^{s-2,p}_{\alpha}(M).
\end{equation*}

A lengthy computation using these facts and
Proposition \ref{prop:multiplication} results in
\begin{align*}
|d(\Theta\rho)|^2_{\bar g} &= |d\rho|^2 +2(k+1)\rho^k u + H^{s,p}_{k-n/p+\min(1,\sigma+\alpha)}(M),
\\
(\Theta\rho)\Delta_{\bar g}(\Theta\rho) &= \rho\Delta_{\bar g}\rho +
k(k+1)\rho^k u + H^{s-1,p}_{k-n/p+\min(1,\sigma+\alpha)}(M),
\\
(\Theta\rho)^2 \R[\bar g] &= \rho^2 \R[\bar g] + H^{s-2,p}_{k-n/p+\sigma+\alpha}(M).
\end{align*}
Hence, from  \eqref{eq:scalar-curvature-pieces} and \eqref{eq:R-decomp} we find
\begin{equation*}
\begin{aligned}
\R[\Theta^{-2}g]-\Rhn
&= \R[g]-\Rhn
+ 2(k+1)(n-1)(k-n)\rho^k u  + H^{s-2,p}_{k-n/p+\min(1,\sigma+\alpha)}(M),
\\
&=\rho^k\left(\tau +  2(k+1)(n-1)(k-n) u\right) + H^{s-2,p}_{k-n/p+\min(1,\sigma+\alpha)}(M).
\end{aligned}
\end{equation*}
Since $k<m\le n$, the coefficient on $u$ above
does not vanish, and we can take
\[
u = \frac{\tau}{2(k+1)(n-1)(n-k)}
\]
in a neighborhood of $\partial M$, small enough to ensure $\Theta >0$, and obtain $$\R[\Theta^{-2}g]-\Rhn\in H^{s-2,p}_{k-n/p+\min(1,\sigma+\alpha)}(M)$$ as claimed.
\end{proof}

\subsubsection{Candidate solution at low regularity}

Suppose $g$ is a metric of class $\mathscr H^{s,p;m}$ with $s\ge m>n/p$.
Hence $\rho^2 g \in \mathscr H^{s,p;m}(M;T^{0,2}M) \subset \mathscr H^{m,p;m}(M;T^{0,2}M) = H^{m,p}(\bar M;T^{0,2}M)$.
Choose $q \ge p$ such that
\[
\frac{1}{p} -\frac{m}{n} \le \frac{1}{q} - \frac{1}{n} < 0.
\]
Sobolev embedding on $\bar M$ implies that $H^{m,p}(\bar M;T^{0,2}M)\subset H^{1,q}(\bar M;T^{0,2}M)$
and therefore $g$ is asymptotically hyperbolic of class $\mathscr H^{1,q;1}$.
Throughout this stage we work with metrics at this level of regularity with
the aim of finding a low-regularity solution of equation \eqref{eq:Yamabe-rewrite}.
In the process we make extensive use of the following maximum principle,
which can be proved with straightforward
modifications of the techniques of Lemmas 5.2 and 5.3 of \cite{Maxwell-RoughAE}.
\begin{lemma}\label{lem:max} Let $g$ be asymptotically hyperbolic of class
$\mathscr H^{1,q;1}$ or of class $\mathscr X^{1,q;1}$ with $q>n$.
\begin{enumerate}
\item\label{part:weak-max}
Suppose $V\in X^{-1,q}_0(M;\mathbb R)$ with $V\ge 0$.  If $u\in X^{1,q}_\alpha(M;\mathbb R)$ for some $\alpha>0$
and
\[
-\Delta_g u + Vu \ge 0
\]
then $u\ge 0$.
\item\label{part:harnack}
Suppose $V\in H^{-1,q}_{\rm loc}(M;\mathbb R)$ and that $u\in H^{1,q}_{\rm loc}(M;\mathbb R)$, is nonnegative,
and satisfies
\[
-\Delta_g u + Vu \ge 0.
\]
If $u=0$ at some point then $u\equiv 0$.
\end{enumerate}
\end{lemma}

The scalar curvature of a metric of class $\mathscr H^{1,q;1}$ only lies
in $H^{-1,q}_{\rm{loc}}(M;\mathbb R)$.  The following lemma shows that we can
conformally transform to a metric with \textit{continuous} scalar curvature
that also satisfies a bound from above needed for our subsequent barrier construction.

\begin{lemma}\label{lem:smooth-scalar}
Suppose $g$ is asymptotically hyperbolic of class $\mathscr H^{1,q;1}$
for some $q>n$.  There exists $u\in H^{1,q}_{1-n/q}(M;\mathbb R)$ such that $\Theta=1+u>0$
and such that $\Theta^{q_n-2} g$ is asymptotically hyperbolic
of class $\mathscr H^{1,q;1}$ with
\begin{align*}
\R[\Theta^{q_n-2} g] - \Rhn &\in H^{1,q}_{1-n/q}(M;\mathbb R),\\
\R[\Theta^{q_n-2} g]  \le \Rhn.
\end{align*}
\end{lemma}
\begin{proof}
The proof proceeds by two conformal transformations. The first
attains the regularity criterion, and the second lowers the scalar curvature.

For ease of notation, set $\alpha = 1-n/q$.
Since
$\alpha$ lies in the Fredholm range defined by inequality \eqref{eq:Fred-Lap-H},
 Proposition \ref{prop:Lap-Fredholm-range} implies
$\Delta_g:H^{1,q}_{\alpha}(M)\to H^{-1,q}_{\alpha}(M)$ is
an isomorphism.
Proposition \ref{prop:scalar-curvature-H} and Lemma \ref{lemma:nabla-and-rho-on-curly}(\ref{part:curly-rho-map}) imply $\R[g]-\Rhn\in H^{-1,q}_{\alpha}(M)$.

Since $\alpha>0$, Proposition \ref{prop:multiplication} implies that multiplication
$H^{-1,q}_{\alpha}(M)\times H^{1,q}_{\alpha}(M)\to H^{-1,q}_{\alpha}(M)$
is bilinear and continuous.
Thus, using Proposition \ref{prop:density},
 we may choose $W\in H^{1,q}_\alpha(M)$
such that $V_1=\R[g]-\Rhn-W$ is sufficiently small
in $H^{-1,q}_\alpha(M)$ so that
\[
-a_n \Delta_g + V_1: H^{1,q}_{\alpha}(M) \to H^{-1,q}_{\alpha}(M)
\]
remains an isomorphism.

For each $\eta\in[0,1]$ let $u_\eta\in H^{1,q}_{\alpha}(M)$
be the unique solution of
\[
-a_n\Delta_g u_\eta + \eta V_1 u_\eta = -\eta V_1.
\]
We claim that each $u_\eta$ satisfies $1+u_\eta>0$.  Indeed, the map
$\eta\mapsto u_\eta$ is continuous into $H^{1,q}_{\alpha}(M) \subset C^{0,\alpha}_{\alpha}(M)$.  Since $u_0\equiv 0$, if the claim were false there would be some first
value of $\eta$ such that $1+u_\eta\ge 0$ and $1+u_\eta=0$ at some point.  Since
$-a_n\Delta(1+u_\eta) + V_1(1+u_\eta) = 0$, Lemma \ref{lem:max}(\ref{part:harnack}) implies
$1+u_\eta\equiv 0$.  But this is impossible since $\alpha>0$ and therefore
since $1+u_\eta\to 1$ at $\partial M$.

Let $\Theta_1=1+u_1$.  Lemma \ref{lem:conf-change-curly}
implies $g_1 = \Theta^{q_n-2}_1 g$ is asympotically hyperbolic of class
$\mathscr H^{1,q;1}$.  From equation \eqref{eq:change-scalar-curvature} it follows that
\[
\R[g_1] = \Theta^{2-q_n}_1(\R[g]-V_1) =  \Theta^{2-q_n}_1(\Rhn+W).
\]
Thus Proposition \ref{prop:nonlin}, and its following comments,
applied to $F(x)=(1+x)^{2-q_n}-1$ imply
\begin{equation}
\label{eq:first-conformal-R-improvement}
\R[g_1]-\Rhn \in H^{1,q}_{\alpha}(M).
\end{equation}
This completes the first conformal transformation.

We now construct $V_2\in H^{1,q}_{\alpha}(M)$ such that $V_2\ge 0$ and $V_2\ge \R[g_1]-\Rhn$.
To accomplish this, let $\chi$  be a cutoff function that equals $1$ on $B_1^\Hyp$ and supported
in $B_2^\Hyp$.
For each $i$ let $k_i = \max_{B_1^\Hyp} |\Phi_i^*(\R[g_1]-\Rhn)|$.
Since $q>n$, by Sobolev embedding we have $k_i\lesssim \|\Phi_i^*(\R[g_1]-\Rhn)\|_{H^{1,q}(B_2^\Hyp)}$.
We define $V_2=\sum_i k_i\left((\Phi_i)_* \chi\right)$, noting that a local uniform finiteness
argument implies
$$
\|V_2\|_{H^{1,p}_{\alpha}(M)} \lesssim \|\R[g_1]-\Rhn\|_{H^{1,p}_{\alpha}(M)}.
$$

Multiplication by $V_2$ is continuous from $H^{1,q}_{\alpha}(M)$ to $H^{1,q}_{2\alpha}(M)$
and therefore a compact operator $H^{1,q}_{\alpha}(M)\to H^{0,q}_{\alpha}(M)$; see Proposition \ref{prop:rellich-lemma}.
So $-a_n \Delta_{g_1} + V_2$ is a compact perturbation of $-a_n\Delta_{g_1}$ and hence
a Fredholm map with index 0 from $H^{1,q}_{\alpha}(M)$ to $H^{-1,q}_{\alpha}(M)$.  Moreover,
since $V_2\ge 0$, an integration by parts argument using the density 
(see Proposition \ref{prop:density}) of smooth compactly supported functions in $H^{1,q}_{\alpha}(M)$
shows that the map has trivial kernel and hence $-a_n\Delta_{g_1}+V_2$ is an isomorphism.
Let $u_2\in H^{1,q}_{\alpha}(M)$ be the solution of
\[
-a_n \Delta_{g_1} u_2 + V_2u_2 = - V_2.
\]
A homotopy argument as in the first conformal transformation shows
$\Theta_2 = 1+u_2> 0$.
Since $V_2\ge 0$, and since $u_2$ and $ V_2$ are in $H^{1,q}_\alpha(M)\subset X^{1,q}_\alpha(M)$, the maximum principle Lemma \ref{lem:max}\eqref{part:weak-max} implies $u_2\le 0$
and hence $\Theta_2< 1$.

Corollary \ref{cor:nonlin-into-X} implies
$\Theta_2^{2-q_n}\in X^{1,q}_0(M)$.
Thus \eqref{eq:first-conformal-R-improvement} and Proposition \ref{prop:nonlin}, applied to $F(x)=(1+x)^{2-q_n}-1$ as above, in combination with \eqref{eq:change-scalar-curvature}, imply that
\begin{equation*}\Rhn
\R[g_2]- \Rhn
= \Theta_2^{2-q_n}\left( \R[g_1]-\Rhn - V_2\right)
+ \left(\Theta_2^{2-q_n} -1\right)\Rhn
\in H^{1,q}_\alpha(M).
\end{equation*}
Furthermore, since $\R[g_1]-\Rhn - V_2<0$ and $\Theta_2< 1$, we have $\R[g_2]\leq \Rhn$.

The result now follows from noting
that Proposition \ref{prop:multiplication} implies $\Theta_1\Theta_2-1\in H^{1,q}_\alpha(M)$.
\end{proof}

We now construct the low regularity candidate solution.

\begin{proposition}\label{prop:Yamabe-low-solve}  Suppose $g$ is an asymptotically hyperbolic metric
of class $\mathscr H^{1,q;1}$ for some $q>n$.  There exists
$u\in X^{1,q}_{1-n/q}(M;\mathbb R)$ solving equation \eqref{eq:Yamabe-eq}.
\end{proposition}
\begin{proof}
For simplicity of notation, let $\alpha =1-n/q$; note $0<\alpha<1$.
Due to Lemma \ref{lem:smooth-scalar} we can assume, without loss of generality,
that $\R[g]-\Rhn\in H^{1,q}_{\alpha}(M)$ and that
$\R[g]\le \Rhn$.
Indeed, if this is not the case, we first use the conformal factor $\Theta$
from Lemma \ref{lem:smooth-scalar} to transform to a metric satisfying these conditions.
Then if $u\in X^{1,q}_{\alpha}(M)$ is a solution of equation
\eqref{eq:Yamabe-rewrite} for the new metric, we find $\Theta(1+u)-1\in X^{1,q}_{\alpha}(M)$
is a solution for the original metric.

We write \eqref{eq:Yamabe-eq} as
\begin{equation}\label{eq:Yamabe-basic}
-a_n \Delta u = F(u),
\end{equation}
where
$$
F(u) = \Rhn(1+u)^{q_n-1}-\R[g](1+u).
$$
Since $\R[g]\le \Rhn$ it follows that $u_- = 0$
satisfies $-a_n \Delta u_- = 0 \le \Rhn - \R[g]=F(u_-)$ and provides a lower barrier.

To find an upper barrier, we first let $v \in X^{1,q}_{\alpha}(M)$ be the solution of
\[
-\Delta_g v = \alpha(n-1-\alpha)\rho^\alpha,
\]
which exists since $\alpha\in(0,1)$ lies
in the Fredholm range defined by inequality \eqref{eq:Fred-Lap-X} and
since $\rho^\alpha\in X^{1,q}_{\alpha}(M)$.
The maximum principle in Lemma \ref{lem:max}\eqref{part:weak-max} implies that $v\geq 0$.

A direct computation using $g = \rho^{-2}\bar g$ shows
\[
-\Delta_g \rho^\alpha = -\alpha\rho^{\alpha+1}\Delta_{\bar g} \rho + \alpha(n-1-\alpha)\rho^\alpha |d\rho|^2_{\bar g}
\]
and therefore Lemmas \ref{lem:drho-H} and \ref{lem:Delta-rho-H} imply
\[
-\Delta_g (v-\rho^\alpha)
= \alpha \rho^\alpha\left[
(n-1-\alpha)\left(1-|d\rho|^2_{\bar g}\right)
+
\rho\Delta_{\bar g}\rho
\right]
\in X^{0,q}_{2\alpha}(M)\subset X^{-1,q}_{2\alpha}(M).
\]
Since $0<2\alpha<2\le n-1$, we see that the weight $2\alpha$ satisfies inequality \eqref{eq:Fred-Lap-X} and thus
Proposition \ref{prop:Lap-Fredholm-range}
implies $-\Delta_g$ is an isomorphism $X^{1,q}_{2\alpha}(M)\to X^{-1,q}_{2\alpha}(M)$.
Hence $v = \rho^\alpha + \xi$ for some $\xi\in X^{1,q}_{2\alpha}(M) \subset  C^{0,\alpha}_{2\alpha}(M)$.
Thus, since $v\geq 0$, we conclude that there exists $c>0$
such that $v\ge c \rho^\alpha$.

We claim that $u_+=K v$ forms an upper barrier for
some $K$ sufficiently large, i.e. that $-a_n\Delta u_+\ge F(u_+)$.
Since $-a_n \Delta K v > 0$ it suffices to show that we may choose $K$ so that $F(Kv)\le 0$.
To see this, first note that since $x\mapsto (1+x)^{q_n-2}$ is monotone increasing for $x\geq 0$,
\begin{equation*}
F(kv) \leq (1+Kv)\left(\Rhn(1+Kcv)^{q_n-2} - \R[g]\right).
\end{equation*}
Thus it suffices to show that
\begin{equation}\label{eq:supersol}
(1+Kc\rho^\alpha)^{q_n-2} \ge \frac{\R[g]}{\Rhn}.
\end{equation}
Since
\[
\R[g] - \Rhn \in H^{1,q}_\alpha(M)\subset \mathscr H^{1,q;1}(M)\subset C^{0,\alpha}(\bar M)
\]
we have $\R[g] = \Rhn + O(\rho^\alpha)$.
The existence of $K$ satisfying \eqref{eq:supersol} now follows from a direct computation, using the compactness of $\bar M$.

With barriers $u_+>u_-$ in hand we now
pick $\Lambda >0$ such that $F_\Lambda(z) = F(z)+\Lambda z$ is monotone increasing
in $z$ for $0\le z \le \max u_+$.  Starting with $u_0=u_+\in X^{1,q}_\alpha(M)$ we iteratively construct a sequence
$\{u_k\}$ in $X^{1,q}_{\alpha}(M)$ as follows.
Assuming $u_k\in X^{1,q}_\alpha(M)$ has been found, observe that
\begin{equation}\label{eq:F-Lambda-rewrite}
F_\Lambda(u_k)
= \Rhn\left[(1+u_k)^{q_n-1}-1 \right]
+ (\Rhn-\R[g])(1+u_k)
+ (\Lambda -\Rhn) u_k.
\end{equation}
Since $\Rhn-\R[g]\in X^{1,q}_{\alpha}(M)$ and since $u_k\in X^{1,q}_{\alpha}(M)$, Proposition \ref{prop:multiplication} and Proposition \ref{prop:nonlin} imply
$F_\Lambda(u_k)\in X^{1,q}_\alpha(M)\subset X^{-1,q}_\alpha(M)$.
Note that $\alpha = 1-n/q$ lies the Fredholm range
of
\begin{equation*}
-a_n\Delta_g + \Lambda\colon X^{1,q}_\alpha(M) \to X^{-1,q}_\alpha(M);
\end{equation*}
see \eqref{eq:ind-radius-lap} and \eqref{eq:Fred-delta-X}.
Hence Proposition \ref{prop:Lap-Fredholm-range} implies there
exists $u_{k+1}\in X^{1,q}_\alpha(M)$ solving
\begin{equation*}
-a_n \Delta_g u_{k+1} +\Lambda u_{k+1} = F_\Lambda(u_k).
\end{equation*}

Note that by construction we have
\begin{equation*}
(-a_n\Delta_g + \Lambda)(u_1 - u_0) = F(u_+) + a_n\Delta_g u_+ \leq 0.
\end{equation*}
Thus the maximum principle Lemma \ref{lem:max}(\ref{part:weak-max}) implies $u_1\le u_0=u_+$.
Furthermore, the monotonicity of $F_\Lambda$ implies
\begin{equation*}
-a_n\Delta_g u_1 + \Lambda u_1
= F_\Lambda(u_0) \geq F_\Lambda(0)
= \Rhn - \R[g] \geq 0,
\end{equation*}
and thus $u_1\geq 0 = u_-$.
Proceeding inductively, we find that $u_- \leq u_{k+1} \leq u_k \leq u_+$ for all $k$.

From the monotonicity of $F_\Lambda$ and the fact that $u_-\le u_k\le u_+$ for each $k$ it follows
that $\{F_\Lambda(u_k)\}$
is bounded in $X^{0,q}_\alpha(M)$ and thus in $X^{-1,q}_\alpha(M)$.
Proposition \ref{prop:Lap-Fredholm-range} then implies
$\{u_k\}$ is bounded in $X^{1,q}_\alpha(M)$. We conclude from
Proposition \ref{lem:basic-inclusions} that the sequence is
bounded in $H^{1,q}_{\delta}(M)$ for any $\delta < \alpha - (n-1)/q$.
Since $H^{1,q}_{\delta}(M)$ is reflexive,
a subsequence $\{u_{k_j}\}$
converges weakly in $H^{1,q}_{\delta}(M)$
to a limit $u$ which, by monotonicity, is the pointwise limit of the full sequence
and hence also lies in $X^{0,q}_{\alpha}(M)$. Now
\[
-a_n \Delta_g u_{k_j} + \Lambda u_{k_j} = F_\Lambda(u_{ (k_j)-1})
\]
Since $F_\Lambda(u_{(k_j)-1})\to F_\Lambda(u)$ pointwise,
the Dominated Convergence Theorem and the weak convergence of $\{u_{k_j}\}$
imply that $u$ is a distributional solution of $-a_n\Delta_gu + \Lambda u=F_\Lambda(u)$ and therefore
also solves
\[
-a_n \Delta_g u = F(u).
\]
Since $u\in X^{0,q}_\alpha(M)$, and since we can rewrite $F(u)$  in
a form analogous to \eqref{eq:F-Lambda-rewrite},
it follows that $F(u)\in X^{0,q}_\alpha(M)$. Proposition
\ref{prop:Lap-Fredholm-range} and the inclusion $X^{1,q}_{\alpha}(M)\subset H^{1,q}_{\delta}(M)$ then imply  $u\in X^{1,q}_\alpha(M)$ as well.
\end{proof}

\subsubsection{Bootstrap}
The following theorem implements the
boostraps needed to assert full regularity from the candidate
solution found in Proposition \ref{prop:Yamabe-low-solve}.

\begin{theorem}\label{thm:yamabe-H}
Suppose $1<p<\infty$, $m\in\Nats$, $s\in \Reals$ with
$s\ge m > n/p$ and $m\le n$.  Given an asymptotically hyperbolic metric $g$ of
class $\mathscr H^{s,p;m}$ there exists a unique conformal factor
$\Theta\in \mathscr H^{s,p;m}(M;\mathbb R)$ such that $\Theta|_{\partial M}=1$ and such that
\[
\R[\Theta^{q_n-2}g] = \Rhn.
\]
\end{theorem}
\begin{proof}
In light of Lemma \ref{lem:conf-change-curly} and Propositions \ref{prop:curly-H-tensor}
and \ref{prop:nonlin-curly},
using the the construction of Proposition \ref{prop:solve-at-infinity} we
may assume without loss of generality that $\R[g]-\Rhn\in H^{s-2,p}_{m-n/p}(M)$.
As discussed prior to Lemma \ref{lem:smooth-scalar} we can choose
$q \ge p$ such that
\[
\frac{1}{p} -\frac{m}{n} \le \frac{1}{q} - \frac{1}{n} < 0.
\]
and such that $g$ is asymptotically hyperbolic of class $\mathscr H^{1,q;1}$.
We set $\alpha = 1-n/q$ and observe that $0<\alpha<1$
and that $m-n/p\ge \alpha$.

Proposition \ref{prop:Yamabe-low-solve} implies there exists $u\in X^{1,q}_{\alpha}(M)$
that solves
\begin{equation}\label{eq:Yamabe-basic-2}
-a_n \Delta_g u +\R[g](1+u) = \Rhn(1+u)^{q_n-1}.
\end{equation}
We claim that in fact $u\in H^{s,p}_{m-n/p}(M)$.
If this were true, then $u\in \mathscr H^{s,p;m}(M)$ by Lemma \ref{lem:basic-roman-curly} and $u=0$ at $\rho=0$ by Theorem \ref{thm:curly-split}.
Consequently, $\Theta=1+u\in \mathscr H^{s,p;m}(M)$ is the conformal factor we seek.

The proof of the claim follows from four bootstraps.
In all cases we use the
fact that equation \eqref{eq:Yamabe-basic-2} can be rewritten
\begin{equation}\label{eq:Yamabe-remap-2}
-\Delta_g u + n u = \frac{\Rhn}{a_n} \left[(1+u)^{q_n-1}-1-(q_n-1)u\right]
+\frac{\R[g]-\Rhn}{a_n} (1+u).
\end{equation}
Let $h$ be the smooth function on $(-1,\infty)$, vanishing at $0$, such that
\begin{equation*}
(1-x)^{q_n-1} - 1 - (q_n-1)x = h(x)x.
\end{equation*}
Then \eqref{eq:Yamabe-remap-2} can be expressed
\begin{equation}
\label{eq:Yamabe-h-version}
-\Delta_g u + nu = \frac{\Rhn}{a_n} h(u) u + \frac{\R[g]-\Rhn}{a_n} (1+u).
\end{equation}

It is helpful to have sufficient conditions for
\begin{equation}
\label{eq:bootstrap-isomorphism}
-\Delta_g +n :X^{\tau,r}_\delta(M)\to X^{\tau-2,r}_\delta(M)
\end{equation}
to be an isomorphism.
Recall the discussion of compatible Sobolev indicates $\mathcal S^{s,p}_2$ that follows Lemma \ref{lem:Lap-mapping}: if $(\tau,r)$ satisfy \eqref{eq:reduced-lap-conds}, namely
\begin{equation*}
1\leq \tau\leq s
\qquad\text{ and }\qquad
\frac{1}{p}-\frac{s}{n} \le \frac{1}{r}-\frac{\tau}{n} < 0,
\end{equation*}
then $(\tau,r)\in \mathcal S^{s,p}_2$ and thus the map \eqref{eq:bootstrap-isomorphism} is continuous.
If we furthermore have
\begin{equation}\label{eq:lap-shifted-n-delta}
-1 < \delta  < n
\end{equation}
then Proposition \ref{prop:Lap-Fredholm-range} implies
\eqref{eq:bootstrap-isomorphism} is an
isomorphism.
Note that $\alpha$ satisfies \eqref{eq:lap-shifted-n-delta}.

\textbf{First Bootstrap.}
Initially we know that $u\in X^{1,q}_{\alpha}(M)$
with $q\ge p$ and $1/p-s/n\le 1/q-1/n<0$.
The goal at this step is to increase the number of derivatives to the maximum extent possible
while leaving $q$ fixed.
 Suppose we have determined that $u\in X^{\tau,q}_{\alpha}(M)$ for some $1\le \tau \le s$,
and that $1/p-s/n < 1/q-\tau/n<0$.  Define $\hat\tau=\tau+2$ unless $1/q-(\tau+2)/n<1/p-s/n$,
in which case we pick $\hat \tau\in (\tau, \tau+2]$ such that $1/q-\hat \tau/n = 1/p-s/n$.  Then $1/p-s/n\le 1/q-\hat\tau/n<0$.  Moreover, since $q>p$, we find $1\le\hat\tau\le s$ and we conclude
$(\hat\tau,q)\in \mathcal S^{s,p}_2$.
Thus, since $\alpha$ satisfies \eqref{eq:lap-shifted-n-delta},
$$-\Delta + n : X^{\hat\tau,q}_{\alpha}(M)
\to X^{\hat\tau-2,q}_{\alpha}(M)$$ is an isomorphism.

We now analyze the right side of \eqref{eq:Yamabe-h-version}.
Since $u\in X^{\tau,q}_\alpha(M)$, Proposition \ref{prop:nonlin} implies $h(u)\in X^{\tau,q}_\alpha(M)$
and, since $1/q-\tau/n<0$,
Proposition \ref{prop:multiplication} implies $h(u)u\in X^{\tau,q}_{2\alpha}(M)$.
Because $\R[g]-\Rhn\in H^{s-2,p}_{m-n/p}(M)\subset X^{s-2,p}_{m-n/p}(M)$
we conclude that the right-hand side of equation \eqref{eq:Yamabe-h-version}
has the schematic form
\[
X^{\tau,q}_{2\alpha}(M) + X^{s-2,p}_{m-n/p}(M) +X^{s-2,p}_{m-n/p}(M) \cdot X^{\tau,q}_{\alpha}(M).
\]
Each of these terms lies in $X^{\hat\tau-2,q}_{\alpha}(M)$: the first one is obvious;
the second follows from Sobolev embedding along with the observation
$\alpha\le m-n/p$; and the third via Proposition \ref{prop:multiplication}
using the conditions $s>n/p$, $s\ge 1$, and the definition of $\hat\tau$.
From Proposition \ref{prop:Lap-Fredholm-range}, and the fact that $X^{\hat\tau,q}_\alpha(M)$ embeds in
$X^{\tau,q}_\alpha(M)$, we conclude
that $u\in X^{\hat\tau,q}_{\alpha}(M)$.
Repeating this argument we can improve the number of derivatives to conclude that $u\in X^{\tau,q}_{\alpha}(M)$ for some $\tau$
with $1/q-\tau/n=1/p-s/n$.

\textbf{Second Bootstrap.}
Suppose we have determined that $u\in X^{\tau,r}_{\alpha}(M)$ for some $r> p$ and
$1\le \tau \le s$, and that $1/p-s/n = 1/r-\tau/n$. Let $\hat\tau = \min(\tau+2,s)$
and define $\hat r$ by
\[
\frac{1}{\hat r} -\frac{\hat \tau}{n} = \frac{1}{p}-\frac{s}{n}.
\]
Arguing as in the first bootstrap, the right-hand side of equation \eqref{eq:Yamabe-h-version}
has the schematic form
\[
X^{\tau,r}_{2\alpha}(M) + X^{s-2,p}_{m-n/p}(M) +X^{s-2,p}_{m-n/p}(M) \cdot X^{\tau,r}_{\alpha}(M).
\]
Each of these terms lies in $X^{\hat\tau-2,\hat r}_{\alpha}(M)$.  This follows for the
first two terms by Sobolev embedding whereas justifying this claim for third term
requires a computation
with Proposition \ref{prop:multiplication} using the facts that $s>n/p$, $s\ge 1$,
and $\hat \tau>n/\hat r$ along with the definition of $\hat\tau$ and $\hat r$.
Proposition \ref{prop:Lap-Fredholm-range} and the fact that $X^{\hat\tau,\hat r}_{\alpha}(M)$
embeds in $X^{\tau,r}_{\alpha}(M)$ implies $u\in X^{\hat\tau,\hat r}_{\alpha}(M)$.
This bootstrap terminates in finitely many steps when we conclude $u\in X^{s,p}_\alpha(M)$.

\textbf{Third Bootstrap.}
We have found $u\in X^{s,p}_\alpha(M)$ for some $\alpha>0$.  Arguing as in the first
two bootstraps we find that the right-hand side of equation \eqref{eq:Yamabe-remap-2}
has the form
\[
X^{s,p}_{2\alpha}(M) + X^{s-2,p}_{m-n/p}(M) +X^{s-2,p}_{m-n/p}(M) \cdot X^{s,p}_{\alpha}(M)
\]
which lies in $X^{s-2,p}_{\min(2\alpha,m-n/p)}(M)$.  Since $\delta=\min(2\alpha,m-n/p)$
satisfies inequality \eqref{eq:lap-shifted-n-delta}  we find that $u\in X^{s,p}_{\delta}(M)$.
Repeating this procedure finitely many times we conclude $u\in X^{s,p}_{m-n/p}(M)$.

\textbf{Final Bootstrap.}  At this stage we know that $u\in X^{s,p}_{m-n/p}(M)$
and we wish to show $u\in H^{s,p}_{m-n/p}(M)$. Choose an initial $\delta$ such
that $\delta < m-n/p -(n+1)/p$, in which case
Lemma \ref{lem:basic-inclusions} implies $u\in H^{s,p}_\delta(M)$.
Since $0<m-n/p<n$ we can nevertheless take $\delta$ sufficiently large so
that
\[
-1 < \delta + \frac{n-1}{p} < n
\]
and therefore $\delta$ lies in the Fredholm range from inequality \eqref{eq:Fred-Lap-n-H}
for $-\Delta_g+n$ mapping $H^{s,p}_\delta(M)\to H^{s-2,p}_\delta(M)$.
As in the first bootstrap
\[
(1+u)^{q_n-1}-1-(q_n-1)u = h(u)u
\]
where Proposition \ref{prop:nonlin} now implies $h(u)\in X^{s,p}_{m-n/p}(M)$
and we conclude from Proposition \ref{prop:multiplication} that
$h(u)u\in X^{s,p}_{m-n/p}(M) \cdot H^{s,p}_{\delta}(M)\subset
H^{s,p}_{\delta+m-n/p}(M)$. Hence
the right-hand side of equation \eqref{eq:Yamabe-remap-2} has the schematic form
\[
H^{s,p}_{\delta+m-n/p}(M) + H^{s-2,p}_{m-n/p}(M) \subset H^{s-2,p}_{\delta'}(M)
\]
where $\delta'=\min(m-n/p+\delta,m-n/p)$.
We can then conclude $u\in H^{s,p}_{\delta'}(M)$ so long as $\delta'$
lies in the Fredholm range from inequality \eqref{eq:Fred-Lap-n-H}.  But
\[
-1<\delta+\frac{n-1}{p}\le  \delta'+\frac{n-1}{p}
\]
and
\[
\delta'+\frac{n-1}{p}\le m-\frac{n}{p}
+ \frac{n-1}{p} = m-\frac{1}{p} < n
\]
as required.  Since $m-n/p>0$, repeating this argument finitely many times the
bootstrap terminates when we conclude $u\in H^{s,p}_{m-n/p}(M)$.

\textbf{Uniqueness.} To investigate uniqueness, we are welcome to start
the construction at one of the constant scalar curvature metrics in
the conformal class, in which case equation \eqref{eq:Yamabe-remap-2}
becomes
\begin{equation}\label{eq:Y-unique}
-\Delta_g u + n u = F(u)
\end{equation}
with $F(u)=\frac{\Rhn}{a_n}\left[(1+u)^{q_n-1}-1-(q_n-1)u\right].$
Since $q_n>2$, $F''(x)<0$ for $x>-1$ and we find $F$ has a maximum on $(-1,\infty)$
at $x=0$ where $F(0)=0$. A solution of equation \eqref{eq:Y-unique} in
$H^{s,p}_{m-n/p}(M)$ then satisfies $-\Delta_g u+n u \le 0$ and we can lower
reguality so that the
maximum principle Lemma \ref{lem:max} implies $u\le 0$. It follows that any conformal
factor $\Theta$ taking some constant scalar curvature metric to another satisfies $\Theta\le1$.
But this is only possible if there is a single such metric.
\end{proof}

\subsection{Solution in the \texorpdfstring{$\mathscr X^{s,p;m}$}{Xsp;m} category}

The strategy is parallel to that for $\mathscr H^{s,p;m}$
but is somewhat less technical at every stage.  We state
the results obtained and briefly sketch the changes needed.

\begin{proposition}\label{prop:solve-at-infinity-X}
Suppose $1<p<\infty$, $m\in\Nats$,
$s\in\Reals$, with $s>n/p$, $s\ge m$ and $m\le n$.
Given an asymptotically hyperbolic metric $g$ on $M$
of class $\mathscr X^{s,p;m}$
there exists a positive
conformal factor $\Theta \in \mathscr X^{s,p;m}(M;\mathbb R)$
such that $\Theta|_{\partial M}=1$ and such that
\[
\R[\Theta^{-2}g] -\Rhn \in X^{s-2,p}_{m}(M;\mathbb R).
\]
\end{proposition}
\begin{proof}
When $m=1$ Proposition \ref{prop:scalar-curvature-H} implies
immediately that $\R[g]-\Rhn\in X^{s-2,p}_1(M)$.  Otherwise
we proceed as in Proposition \ref{prop:solve-at-infinity}
through a sequence of conformal transformations with conformal
factors $\Theta\in \mathscr X^{s,p;m}(M)$ so that
if $\R[g]-\Rhn\in X^{s-2,p}_{k}(M)$ for some integer $1\le k \le n-1$,
then $\R[\Theta^2 g]-\Rhn\in X^{s-2,p}_{k+1}(M)$.  The argument
is somewhat simplified because we are able to improve
the decay by a full power of $\rho$ at each step.
When $\R[g]-\Rhn\in X^{s-2,p}_{k}(M)$, Proposition \ref{prop:scalar-curvature-H}
and Theorem \ref{thm:curly-split} imply
\[
\R[g] = \Rhn + \rho^k \tau + r
\]
with $\tau\in \mathscr X^{\infty;m-k}(M)$ and $r\in X^{s-2,p}_m(M)$.
We take the conformal factor of the form
$\Theta = 1+\rho^k u$ with $u\in \mathscr X^{\infty;m-k}(M)$,
in which case
\begin{equation*}
\rho^k u \in X^{\infty}_k(M),\qquad \rho^{k+1} \hat\nabla u \in X^{\infty}_{k+2}(M),\qquad
\hat\nabla^2 u \in X^{\infty}_{k+1}(M).
\end{equation*}
Lemmas \ref{lem:drho-H}-\ref{lem:R-bar-H} imply
\begin{equation*}
|d\rho|_{\bar g}^2 -1 \in  X^{s,p}_{1}(M),\qquad
\rho\Delta_{\bar g}\rho \in X^{s-1,p}_{1}(M),\qquad
\rho^2 R[\bar g] \in X^{s-2,p}_{1}(M).
\end{equation*}
A computation using these facts leads to
\[
\R[\Theta^2 g] -\Rhn = \rho^k \left( \tau + 2(k+1)(n-1)(n-k)u\right) + X^{s-2,p}_{k+1}(M)
\]
and since $k<n$ we can take $u=-\tau/(2(k+1)(n-1)(n-k))$ near $\partial M$.
\end{proof}

Consider an asymptotically hyperbolic metric $g$ of class
$\mathscr X^{s,p;m}$ with $s>n/p$, $s\ge m$ and $m\ge 1$.
Choose $q>1$ so that
\[
\frac{1}{p}-\frac{s}{n} \le \frac{1}{q}-\frac{1}{n} < 0
\]
which is possible since $s\ge 1$ and $s>n/p$.
Lemma \ref{lem:curly-basic} and Proposition \ref{prop:curly-to-M-bar} imply
$\bar g = \rho^2 g \in H^{1,q}(\bar M) = \mathscr H^{1,q;1}(M)$
and hence $g$ is also asymptotically hyperbolic of class $\mathscr H^{1,q;1}$.
Thus, the second step of the $\mathscr H^{s,p;m}$ construction applies
without change and we need only establish the desired regularity of
the candidate solution.

\begin{theorem}\label{thm:Yamabe-X}
  Suppose $1<p<\infty$, $m\in\Nats$, $s\in \Reals$ with
$s>n/p$, $s\ge m$ and $m<n$.  Given an asymptotically hyperbolic metric $g$ of
class $\mathscr X^{s,p;m}$ there exists a unique conformal factor
$\Theta\in \mathscr X^{s,p;m}(M;\mathbb R)$ such that $\Theta|_{\partial M}=1$  and such that
\[
\R[\Theta^{q_n-2}g] = \Rhn.
\]
\end{theorem}
\begin{proof}
The proof is a truncated version of the proof of Theorem \ref{thm:yamabe-H}.
Using Proposition \ref{prop:solve-at-infinity-X} we can assume
$\R[g]-\Rhn\in X^{s-2,p}_{m}(M)$, and from Proposition
\ref{prop:Yamabe-low-solve} and the comments above
we know that there exists a solution
$u\in X^{1,q}_{1-n/q}(M)$ of equation \eqref{eq:Yamabe-rewrite}
for some $q$ with $1/p-s/n\le 1/q-1/n<0$.  The first two bootstraps
to improve the interior regularity of $u$ are essentially identical
to those of Theorem \ref{thm:yamabe-H},
replacing some instances of $X^{s-2,p}_{m-n/p}(M)$ with the even
more regular space
$X^{s-2,p}_{m}(M)$ arising from scalar curvature terms, and we
arrive at the conclusion $u\in X^{s,p}_{1-n/q}(M)$.  In the course of
the third bootstrap we find that if $u\in X^{s,p}_{\alpha}(M)$
for some $\alpha>0$ then
\[
-\Delta_g u + n u \in X^{s,p}_{2\alpha}(M) + X^{s-2,p}_{m}(M).
\]
Since $\alpha>0$ and $m<n$, $\delta = \min(2\alpha,m)$ lies in the Fredholm range
from inequality \eqref{eq:Fred-Lap-n-X} and we conclude
$u\in X^{s,p}_{\delta}(M)$.  Since $\alpha>0$, this bootstrap
terminates after finitely many steps
when we conclude $u\in X^{s,p}_{m}(M)\subset \mathscr X^{s,p;m}(M)$ and, by Theorem \ref{thm:curly-split}, vanishes at $\rho=0$.
The fourth bootstrap of Theorem \ref{thm:yamabe-H}
is not needed and uniqueness follows from uniqueness in the $\mathscr H^{1,q;1}$ category.
\end{proof}
Note that if $m=n$, the proof of Theorem \ref{thm:Yamabe-X} almost
succeeds.  The only failure occurs at the very end of the third bootstrap
because $\delta=n$ just fails to lie in the Fredholm range
from inequality \eqref{eq:Fred-Lap-n-X}, and
we can only conclude that $u\in X^{s,p}_\delta(M)$ for any $\delta<n$.
This is not enough to conclude that $u\in \mathscr X^{s,p;n}(M;\mathbb R)$, however,
as is needed to preserve the original metric's regularity class.
As mentioned previously, for a smooth background metric when $n=3$
a generic solution will contain a term of the form $\rho^3\log(\rho)$.
Such a term does not lie in $\mathscr X^{3,p;3}(M;\mathbb R)$, for otherwise
after taking three $\partial_\rho$ derivatives we would conclude
$\log(\rho)\in X^{0,p}_0(M;\mathbb R)$, which is false.

\section{Boundary asymptotics in fortified  spaces}\label{sec:asymptotics}

The aim of this section is to prove the following result from \S\ref{sec:asymptotic-spaces}, which shows that a
function in a fortified  space can be decomposed into
two pieces, one of which is smooth in the interior of $M$ and
captures the asymptotic behavior, whereas the second piece is a non-smooth
remainder that decays as rapidly as the limited boundary regularity allows.

\begin{theorem}[Asymptotic Structure Theorem]\label{thm:curly-split}
Suppose $1<p<\infty$, $s\in\Reals$, $m\in\Nats$, $s\ge m$
and $E$ is a tensor bundle over $M$ of weight $w$.
\begin{enumerate}
  \item\label{part:curly-split-H} If $u\in \mathscr H^{s,p;m}(M;E)$ then
  \[
    u = \tau + r
  \]
  where $\tau\in \mathscr H^{\infty,p;m}(M;E)$ and $r\in H^{s,p}_{w+m-n/p}(M;E)$.
  Moreover, if $u\in H^{s,p}_{w+k+\alpha-n/p}(M;E)$ for some integer
  $0\le k\le m-1$ and some $\alpha\in[0,1)$, then
  $\tau=\rho^k \eta$ for some $\eta \in \mathscr H^{\infty,p;m-k}(M;E)\cap H^{\infty,p}_{w+\alpha-n/p}(M;E)$.

  Finally, if $m\ge 1$ then $u|_{\partial M}=0$ if and only if $u\in H^{s,p}_{w+1-n/p}(M;E)$.
  \item\label{part:curly-split-X} If $u\in \mathscr X^{s,p;m}(M;E)$ then
  \[
    u = \tau + r
  \]
  where $\tau\in \mathscr X^{\infty;m}(M;E)$ and $r\in X^{s,p}_{w+m}$.
  Moreover, if $u\in X^{s,p}_{w+k}$ for some
  integer $0\le k\le m$ and some $\alpha\in[0,1)$, then
  $\tau=\rho^k \eta$ for some $\eta \in \mathscr X^{\infty;m-k}(M;E)\cap X^{\infty}_{w+\alpha}(M;E)$. 

  Finally, if $m\ge 1$ then $u|_{\partial M}=0$ if and only if $u\in X^{s,p}_{w+1}(M;E)$.
\end{enumerate}
\end{theorem}

The proof of Theorem \ref{thm:curly-split} relies on two technical tools.
\begin{enumerate}[label=(\arabic*)]
  \item For functions in fortified  spaces, a vanishing
  boundary trace manifests itself as an additional rate of decay at the boundary. Moreover,
  higher-order $\partial_\rho$ derivatives vanishing at the boundary are quantitatively
  reflected in terms of higher-order decay.
  Section \ref{secsec:boundary-trace} contains precise statements of these facts, which generalize corresponding results in \cite{WAH} for H\"older continuous functions.

  \item $\mathbb H$-mollification, introduced previously in \cite{WAH} as group mollification, extends
  to the current setting and can be used to improve a function's interior regularity while leaving
  its boundary trace unchanged.  Section \ref{secsec:group-mollification} contains the details.
\end{enumerate}
Section \ref{secsec:asymptotic-structure} uses these tools to establish Theorem \ref{thm:Taylor}, a variation of Taylor's Theorem for expansions at the conformal boundary,
and Theorem \ref{thm:curly-split} follows as a straightforward consequence.

$\mathbb H$-mollification is based on the half-space model of hyperbolic space and
is better adapted to a subtly different class of function spaces than those used elsewhere in
this paper; the key difference is that the boundary defining function $y$ on $\Hyp$
is singular at a point
on the boundary of the ball model.
For functions on $\Hyp$ supported in some fixed domain
avoiding a neighborhood of the singularity the norms in the two classes of function
spaces are equivalent, and this observation allows us to transfer facts from the half-space
model when proving Theorem \ref{thm:curly-split}.  Nevertheless, for the remainder of
this section we work within this new framework and we begin by defining
its associated function spaces and outlining their basic properties.

Let $\mathbb R^n_+$ be the Euclidean half space with coordinates $(x, y)$, where $x \in \mathbb R^{n-1}$ and $y\geq 0$, and metric
\begin{equation*}
g_{\mathbb E} = (dx^1)^2 + \dots +(dx^{n-1})^2 + dy^2.
\end{equation*}
The Euclidean connection is $\nabla_{\mathbb E}$.
As introduced in \S\ref{sec:coords},
$\mathbb H$ denotes the $n$-dimensional space half space model of hyperbolic space, corresponding to the interior of $\mathbb R^n_+$, and equipped with the metric
$$\breve g = y^{-2}g_\mathbb{E} = y^{-2}\left((dx^1)^2 + \dots +(dx^{n-1})^2 + dy^2\right).$$

Let $\Phi_i\colon B^\mathbb{H}_2\to \mathbb H$, given by $\Phi_i\colon (\xi,\eta)\mapsto (x_i + y_i \xi, y_i \eta)$, be a countable collection of M\"obius parametrizations with uniformly locally finite
images such that the sets $\Phi_i(B^\Hyp_{1/2})$ cover $\mathbb H$.
We define the half-space versions of the Sobolev and Gicquaud-Sakovich spaces of
sections of a tensor bundle $E$ over $\Hyp$ of
scalar-valued functions via
\begin{equation}
\label{roman-half-space-norms}
\begin{gathered}
\|u \|_{H^{s,p}_\delta(\mathbb H;E\pipesep y) }
= \left( \sum_i y_i^{-\delta p} \|\Phi_i^*u\|_{H^{s,p}(B^\mathbb{H}_1)}^p\right)^{1/p},
\\
\|u \|_{X^{s,p}_\delta(\mathbb H; E\pipesep y)}
= \sup_i y_i^{-\delta p} \|\Phi_i^*u\|_{H^{s,p}(B^\mathbb{H}_1)};
\end{gathered}
\end{equation}
here $H^{s,p}(B^\mathbb{H}_1)$ is computed as in \S\ref{sec:elementary-spaces}.
Uniformly equivalent norms are obtained if $B^\mathbb{H}_1$ is replaced by $B^\mathbb{H}_r$ for $1/2\leq r \leq 2$ in \eqref{roman-half-space-norms}.
Throughout this section the bundle $E$ is easily inferred from context and
we almost universally only consider scalar valued functions.  Hence 
we usually drop the bundle from 
the notation and write $H^{s,p}_{\delta}(\Hyp\pipesep y)$ instead,
emphasizing the role of the singular boundary defining function $y$.

For fortified  spaces in this setting we only work with scalar-valued functions and define
\begin{equation}
\label{curly-half-space-norms}
\begin{gathered}
\|u\|_{\mathscr H^{s,p;m}(\mathbb H\pipesep y)} 
= \sum_{j=0}^m \|\nabla_\mathbb{E}^ju\|_{H^{s-j, p}_{j-n/p}(\mathbb H \pipesep y)},
\\
\|u\|_{\mathscr X^{s,p;m}(\mathbb H\pipesep y)} = \sum_{j=0}^m \|\nabla_\mathbb{E}^ju\|_{X^{s-j, p}_{j}(\mathbb  H\pipesep y)},
\end{gathered}
\end{equation}
which should be compared with the definitions \eqref{define-scrH-int} and \eqref{define-scrX-int}.

We make use of equivalent norms obtained from coordinate derivatives as follows.
For a multiindex $\alpha = (\alpha_1,\dots, \alpha_n)\in \mathbb Z_{\geq 0}^n$ we define
\begin{equation*}
\partial^\alpha = \partial_{x^1}^{\alpha_1}\cdots \partial_{x^{n-1}}^{\alpha_{n-1}}\partial_y^{\alpha_n},
\quad\text{ and }\quad
(y\partial)^\alpha = (y\partial_{\theta^1})^{\alpha_1}\cdots (y\partial_{\theta^{n-1}})^{\alpha_{n-1}}(y\partial_y)^{\alpha_n}
\end{equation*}
and set $|\alpha| = \alpha_1+\dots+\alpha_n$.
For $1<p<\infty$ we have the norm equivalences for functions on $\mathbb H$:
\begin{equation}
\label{half-space-norm-equivalences}
\begin{gathered}
\|u\|_{\mathscr H^{s,p;m}(\mathbb H\pipesep y)}
\sim
\sum_{ |\beta|\leq m}
\|\partial^\beta u\|_{H^{s-|\beta|,p}_{-n/p}(\mathbb H\pipesep y)},
\\
\|u\|_{\mathscr X^{s,p;m}(\mathbb H\pipesep y)}
\sim
\sum_{|\beta|\leq m}
\|\partial^\beta u\|_{X^{s-|\beta|,p}_0(\mathbb H\pipesep y)}.
\end{gathered}
\end{equation}
Furthermore, since $y$ is uniformly bounded above and below on $B^\mathbb{H}_2$, we have the local norm equivalences
\begin{equation}
\label{local-V0-norm-equivalence}
\| u\|_{H^{k,p}(B^\mathbb{H}_r)}
\sim
\sum_{|\alpha|\leq k} \| (y\partial)^\alpha u\|_{L^p(B^\mathbb{H}_r)}
\end{equation}
for all $k\in \Nats_{\ge 0}$ and $1/2\leq r \leq 2$.

While many results for the weighted spaces and fortified  spaces
proved in Sections \ref{sec:elementary-spaces} and \ref{sec:asymptotic-spaces} transfer
to the spaces here, not all facts do, and some care is needed because of the
non-compactness of $\Reals^{n}_+$.
For example, it is not true that $X^{s,p}_{0}(\Hyp\pipesep y)$
embeds in $H^{s,p}_{-n/p}(\Hyp\pipesep y)$.  The following observations can be proved
using straightforward modifications of earlier results.
\begin{itemize}
  \item The spaces are complete.
  \item For scalar-valued functions, the $H^{0,p}_{-n/p}(\Hyp\pipesep y)$ norm is equivalent to
  the usual Lebesgue norm $L^p(\Reals^n_+)$ and
  $\mathscr H^{1,p;1}(\Hyp\pipesep y)$ is the usual Sobolev space $H^{1,p}(\Reals^n_+)$.
  \item $y^\beta H^{s,p}_\delta(\Hyp\pipesep y) = H^{s,p}_{\delta+\beta}(\Hyp\pipesep y)$ for any $\beta\in\Reals$, and similarly for Gicquaud-Sakovich spaces.
  \item $\partial^\alpha : H^{s,p}_{\delta}(\Hyp\pipesep y)\to H^{s-|\alpha|,p}_{\delta-|\alpha|}(\Hyp\pipesep y)$ if $|\alpha|\le s$, and similarly for the spaces $X^{s,p}_{\delta}(\Hyp\pipesep y)$
  \item $\partial^\alpha : \mathscr H^{s,p;m}(\Hyp\pipesep y)\to \mathscr H^{s-|\alpha|,p;m-|\alpha|}(\Hyp\pipesep y)$ if $|\alpha|\le m$, and similarly for the spaces $\mathscr X^{s,p;m}(\Hyp\pipesep y)$
  \item If $s\ge m+1$ then $y \mathscr H^{s,p;m}(\Hyp\pipesep y) \subset H^{s,p;m+1}(\Hyp\pipesep y)$ and similarly for the spaces $\mathscr X^{s,p;m}(\Hyp\pipesep y)$.
  \item The direct analog of Lemma \ref{lem:alt-H-X-norms}, i.e.
  $\|u\|_{H^{s,p}_\delta(\Hyp\pipesep y)} \sim \|u\|_{H^{0,p}_{\delta}(\Hyp\pipesep y)} +
  \|\nabla_{\mathbb E}^k u\|_{H^{s-k,p}_{\delta}}$ and similarly for the spaces $X^{s,p}_{\delta}(\Hyp\pipesep y)$.
\end{itemize}

\subsection{Boundary traces and asymptotic decay}\label{secsec:boundary-trace}
As a consequence of the elementary embeddings
\[
\mathscr X^{s,p;m}(\Hyp\pipesep y)
\hookrightarrow \mathscr H^{s,p;m}(\Hyp\pipesep y)
\hookrightarrow \mathscr H^{m,p;m}(\Hyp\pipesep y) = H^{m,p}(\Reals^n_+)
\]
it follows that elements of fortified  spaces admit boundary traces when $m\ge 1$.
In this section we establish the connection between the rate of decay
of a function $u$ in a fortified  function space
and the vanishing trace of its derivatives $\partial_y^k u$ at the boundary.  As
a first step, the next two lemmas show that if $u$ has a vanishing
trace, the decay of $u$ at the boundary is one order
better than the decay of $\partial_y u$.  The first lemma
treats Gicquaud-Sakovich spaces, and the second
weighted Sobolev spaces.

\begin{lemma}\label{lem:first-deriv-zero-X}

Suppose for some $1<p<\infty$
that $u$ is a real-valued distribution such
that $u\in \mathscr X^{1,p;1}(\Hyp\pipesep y)$,
$u|_{y=0}=0$, and
$\partial_y u \in X^{0,p}_\delta(\mathbb H\pipesep y)$
for some real $\delta\ge 0$.
Then $u\in X^{0,p}_{\delta+1}(\Hyp\pipesep y)$ and
\begin{equation*}
\|u\|_{X^{0,p}_{\delta+1}(\Hyp\pipesep y)} \lesssim
\|\partial_y u\|_{X^{0,p}_{\delta}(\Hyp\pipesep y)}.
\end{equation*}
\end{lemma}
\begin{proof}
Let $\Phi_i$ be a M\"obius parametrization.  Let $Q_i=\Phi_i(Y_5)$
and observe that since $B_1^\Hyp\subset Y_5$, $\Phi_i(B_1)^\Hyp\subset Q_i$.
Since $\Phi_i^*dV_{g_{\mathbb E}}=y_i^n dV_{g_{\mathbb E}}$, and recalling from 
\S\ref{secsec:fspace-defs} 
that we compute norms on $B^\mathbb{H}_1$ using the Euclidean metric $g_\mathbb{E}$,
we find
\begin{equation}\label{eq:1st-deriv-A}
\|\Phi_i^* u\|_{L^p(B_1^\Hyp)}^p
=  y_i^{-n} \|u\|_{L^p(\Phi_i(B_1^\Hyp))}^p
\le y_i^{-n} \|u\|_{L^p(Q_i)}^p.
\end{equation}

Assume for the moment that $u$ is smooth and vanishes on $\partial \Hyp$.
For $(x,y)\in Q_i$ we have
\[
|u(x,y)| = \left|\int_0^{y}\partial_y u(x,\eta)\;d\eta\right|
\le y^{\frac1{p\dual}}
\left[\int_0^{y}|(\partial_y u)(x,\eta)|^p\;d\eta\right]^{\frac1p}.
\]
Integrating $p^{\rm th}$ powers and using the relation $p/p\dual = p-1$
\begin{multline*}
\int_{0}^{5y_i} |u(x,y)|^p\;dy
\le \left(\int_0^{5y_i} y^{\frac p{p\dual}}\;dy\right)
\int_0^{5y_i}|(\partial_y u)(x,\eta)|^p\;d\eta 
\\
= 
\frac{1}{p}(5y_i)^p
\int_0^{5y_i}|(\partial_y u)(x,\eta)|^p\;d\eta. 
\end{multline*}
Now integrating in $x$ we conclude
\begin{equation}\label{eq:1st-deriv-A2}
\int_{Q_i} |u|^p\; dV_{g_\mathbb{E}}
\lesssim y_i^p \int_{Q_i} |\partial_y u|^p\;dV_{g_\mathbb{E}}
\end{equation}
and a 
density argument shows that this same inequality holds for all $u$ in the usual Sobolev space
$W^{1,p}(Q_i)$
that vanish at $y=0$. Moreover, since the volume of $Q_i$ is finite,
a computation from the definition analogous to that of
Lemma \ref{lem:normchar}\eqref{part:Wbar-via-W}
shows that $\mathscr X^{1,p;1}(\Hyp\pipesep y)$ restricts continuously into $W^{1,p}(Q^*_z)$ and hence inequality
\eqref{eq:1st-deriv-A2} applies to these functions as well.

A computation shows that if $B_1^\Hyp(p)\cap Y_5\neq\emptyset$ then
$B_1^\Hyp(p)\subset Y_{100}$.  Let $Q^*_i=\Phi_i(Y_{100})$ and
let $J=\{j:\Phi_j(B^\Hyp_1)\cap Q_i\neq \emptyset\}$.
Observe that if $j\in J$ then $y_j\le 100 y_i$.
Setting $A=\cup_{j\in J} \Phi_j(B^\Hyp_1)\subset Q_i^*$ we compute
\begin{equation}\label{eq:1st-deriv-B}
\begin{aligned}
\int_{Q_i} |\partial_y u|^p\;dV_{g_\mathbb{E}}
&\le \int_{A} |\partial_y u|^p\;dV_{g_\mathbb{E}}
\\
&\le \sum_{j\in J} \int_{\Phi_j(B^\Hyp_1)} |\partial_y u|^p\; dV_{g_\mathbb{E}}\\
&= \sum_{j\in J} y_j^n \|\Phi_j^*\partial_y u\|^p_{L^p(B^\Hyp_1)} \\
&\le \sum_{j\in J} y_j^n y_j^{\delta p}
\|\partial_y u\|^p_{X^{0,p}_\delta(\Hyp\pipesep y)}\\
&\le 100^{\delta p} y_i^{\delta p} \|\partial_y u\|^p_{X^{0,p}_\delta(\Hyp\pipesep y)}
\sum_{j\in J} y_j^n\\
&\lesssim y_i^{\delta p} \|\partial_y u\|^p_{X^{0,p}_\delta(\Hyp\pipesep y)} \int_A\;dV_{g_\mathbb{E}}
\end{aligned}
\end{equation}
where uniform local finiteness was used in the final step.
Moreover, since $A\subset Q_i^*$,
\begin{equation}\label{eq:1st-deriv-C}
\int_A\;dV_{g_{\mathbb E}} \le
\Vol_{g_{\mathbb E}}(Q_i^*)  = y_i^n\Vol_{g_{\mathbb E}}(Y_{100}).
\end{equation}
Combining inequalities \eqref{eq:1st-deriv-A},
\eqref{eq:1st-deriv-A2}, \eqref{eq:1st-deriv-B}, \eqref{eq:1st-deriv-C} we find
\[
y_i^{-\delta-1}\|\Phi_i^* u\|_{L^p(B^\Hyp_1)}
 \lesssim \|\partial_y u\|_{X^{0,p}_{\delta}(\mathbb H\pipesep y)};
\]
taking a supremum over $i$ completes the proof.
\end{proof}

\begin{lemma}\label{lem:first-deriv-zero-W}
Suppose  $u\in\mathscr H^{1,p;1}(\Hyp\pipesep y)$ for some
$1<p<\infty$
such that $u|_{y=0}=0$
and such that
$\partial_\rho u\in H^{0,p}_{\delta-n/p}(\mathbb H\pipesep y)$
for some  $\delta\ge 0$.
Then $u\in H^{0,p}_{\delta+1-n/p}(\Hyp\pipesep y)$ and
\begin{equation*}
\|u\|_{H^{0,p}_{\delta+1-n/p}(\Hyp\pipesep y)} \lesssim
\|\partial_y u\|_{H^{0,p}_{\delta-n/p}(\Hyp\pipesep y)}.
\end{equation*}
\end{lemma}
\begin{proof}
The hypotheses on $u$ can be reformulated
as $u\in W^{1,p}(\Reals^n_+)$, $u|_{y=0}=0$, and
$y^{-\delta}\partial_y u \in L^p(\mathbb R^n_+)$.
We wish to show that $y^{-(\delta+1)} u\in L^p(\mathbb R^n_+\pipesep y)$
and indeed we will show
\begin{equation}\label{eq:u-decays-fast-1}
\|y^{-(\delta+1)}u\|_{L^p(\Reals^n_+)} \le
\frac{p}{(\delta+1)p-1} \|y^{-\delta}\partial_y u\|_{L^p(\Reals^n_+)}.
\end{equation}
Note that since $p>1$ and $\delta\ge 0$ the coefficient
on the right-hand side of equation \eqref{eq:u-decays-fast-1}
is finite.

First, suppose that $u$ is smooth, compactly supported on the interior of $\Reals^n_+$
and is non-negative.  Then
\begin{align*}
\int_{\Reals^n_+} y^{-(\delta+1)p} u^p\;dV_{\mathbb E}
&= \frac{1}{1-(\delta+1)p}
\int_{\Reals^n_+} \partial_y (y^{1-(\delta+1)p})u^p\; dV_{g_{\mathbb E}}\\
&= \frac{p}{(\delta+1)p-1}
\int_{\Reals^n_+} ( y^{-(\delta+1)} u)^{p-1} y^{-\delta}\partial_y u\; dV_{g_{\mathbb E}}\\
&\le \frac{p}{(\delta+1)p-1}
\| y^{-(\delta+1)} u\|_{L^p(\Reals^n_+)}^{{p}/{p\dual}}
\|y^{-\delta}\partial_y u\|_{L^p(\mathbb R^n_+)}.
\end{align*}
Noting that that $p/p\dual=p-1$, we arrive at
\eqref{eq:u-decays-fast-1}.
In fact, inequality \eqref{eq:u-decays-fast-1}
holds for all functions in $W^{1,p}(\Reals^n_+)$
that vanish for $y\le\epsilon$ for some $\epsilon >0$.
The smoothness hypothesis can be dropped by considering a sequence
of smooth non-negative compactly supported functions that vanish for $y\le\epsilon/2$
and converge to $u$ in $W^{1,p}(\Reals^n_+)$.
This argument requires the observation that the integrals appearing in
\eqref{eq:u-decays-fast-1} are continuous on the subspace of
of functions in $W^{1,p}(\Reals^n_+)$ that vanish for
$y\le\epsilon/2$.
The non-negativity can be relaxed since $u\mapsto |u|$
is continuous on $W^{1,p}(\Reals^n_+)$ and since
$|\partial_y u|= |\partial_y |u| |$.

Now consider an arbitrary function $u\in W^{1,p}(\Reals^n_+)$
such that $u|_{y=0}=0$ and such that
$y^{-\delta}\partial_y u \in L^p(\Reals^n_+)$.
Let $\chi:[0,\infty)\to [0,1]$ be a monotone smooth function that equals 1 on $[0,1]$
and vanishes on $[2,\infty)$.
For $m\in\Nats$
define $\kappa_m(y)=1-\chi(my)$ and $\chi_m(x)=\chi(|x|^2/m^2)$
so $\kappa_m$ vanishes for $y<1/m$ and $\chi_m$ vanishes for $|x|^2>2m^2$.
Inequality
\eqref{eq:u-decays-fast-1}
applied to $\kappa_m\chi_m u$ yields
\begin{multline}\label{u-decays-fast-2}
\|\chi_m \kappa_m y^{-(\delta+1)}u\|_{L^p(\Reals^n_+)} \le\\\qquad
\frac{p}{(\delta+1)p-1}\left[
\|\chi_m\kappa_m y^{-\delta}\partial_y u\|_{L^p(\Reals^n_+)}
+ K m\left(\int_{1/m}^{2/m}  \int_{\Reals^{n-1}} \chi_m |y^{-\delta}u|^p\; dx\;dy\right)^{\frac{1}{p}}\right],
\end{multline}
where $K$ is a constant depending on the choice of $\chi$ but is independent of $m$, and where we write $dV_{g_{\mathbb E}} = dx\,dy$.
The result then follows from an application of the Monotone Convergence Theorem so long as we can show that the second term
in the brackets
on the right-hand side of \eqref{u-decays-fast-2} vanishes
in the limit $m\to\infty$.

For smooth functions that vanish at $y=0$
\begin{equation}\label{eq:u-decays-fast-3}
\int_{\Reals^{n-1}} \chi_m |u(x,y)|^p\;dx
= \int_{\Reals^{n-1}} \chi_m \left| \int_0^{y}(\partial_y u)(x,\eta)\;d\eta\right|^p dx
\end{equation}
for any $y>0$. A density argument shows
this same equality holds for all functions in $W^{1,p}(\Reals^n_+)$
with zero boundary trace
because both sides of the equality are continuous on
$W^{1,p}(\Reals^n_+)$; this last claim uses the compact support of $\chi_m$.

Consider the map $y\mapsto (\partial_y u)(x,y)$, which
is an $L^p$ function of $y$ on $[0,y]$ for almost every $x\in\Reals^{n-1}$.
H\"older's inequality implies
\begin{equation}\label{eq:u-decays-fast-4}
\begin{aligned}
\left| \int_0^{y}(\partial_y u)(x,\eta)\;d\eta\right|^p
&\le \left[\int_0^{y} \eta^{\delta p\dual}\;d\eta\right]^{\frac{p}{p\dual}} \int_0^{y} \eta^{-\delta p}|(\partial_y u)(x,\eta)|^p\;d\eta
\\
& = \left(\frac{1}{\delta p\dual+1}\right)^{\frac{p}{p\dual}}
y^{\delta p+p-1} \int_0^{y} \eta^{-\delta p}|(\partial_y u)(x,\eta)|^p\;d\eta.
\end{aligned}
\end{equation}
Multiplying \eqref{eq:u-decays-fast-3} by $m^py^{-\delta p}$,
integrating in $y$, and applying estimate \eqref{eq:u-decays-fast-4}
we conclude
\begin{align*}
m^p\int_0^{2/m} &\int_{\Reals^{n-1}} \chi_m y^{-\delta p}|u|^p\;dx\;dy
\\
&\lesssim m^p
\int_0^{2/m} y^{p-1}
\int_{\Reals^{n-1}}\chi_m \int_0^{y} \eta^{-\delta p}|(\partial_y u)(x,\eta)|^p\;d\eta\;dx\;dy\\
&\le m^p \left(\int_0^{2/m} y^{p-1}\;dy\right)
\int_{\Reals^{n-1}} \int_0^{2/m} y^{-\delta p}|\partial_y u|^p\;dy\;dx\\
&= \frac{2^p}{p} \int_{\Reals^{n-1}} \int_0^{2/m} y^{-\delta p}|\partial_y u|^p\;dy\;dx.
\end{align*}
This final integral converges to $0$ as $m\to\infty$,
which establishes the desired limit.
\end{proof}

We are now ready to quantitatively connect rates of decay at the boundary to
vanishing boundary traces.

\begin{proposition}\label{prop:fast-decay-vs-vanishing-derivatives}
Suppose $1<p<\infty$, $s\in\Reals$, $m\in\Nats$, and $s\ge m$.
\begin{enumerate}
\item \label{part:fast-decay-implies-zero-derivatives-H}
Suppose $u\in \mathscr H^{s,p;m}(\mathbb H\pipesep y)$. Then $u\in H^{s,p}_{m-n/p}(\mathbb H\pipesep y)$
if and only if then $\partial_y^k u|_{y=0} = 0$
for $0\le k \le m-1$.
\item \label{part:fast-decay-implies-zero-derivatives-X}
Suppose $u\in \mathscr X^{s,p;m}(\mathbb H\pipesep y)$. Then  $u\in X^{s,p}_{m}(\mathbb H\pipesep y)$
if and only if $\partial_y^k u|_{y=0} = 0$
for $0\le k \le m-1$.
\end{enumerate}
\end{proposition}
\begin{proof}
We begin with part \eqref{part:fast-decay-implies-zero-derivatives-H}.
For the forward direction, suppose $u\in H^{s,p}_{m-n/p}(\mathbb H\pipesep y)$;
we wish to show $\partial_y^k u|_{y=0}=0$ for $k \le m-1$.  Since
$$\partial_y: \mathscr H^{s,p;m}(\mathbb H\pipesep y)\cap H^{s,p}_{m-n/p}(\mathbb H\pipesep y)\to \mathscr H^{s-1,p;m-1}(\mathbb H\pipesep y)\cap H^{s-1,p}_{m-1-n/p}(\mathbb H\pipesep y)$$
an iterative argument shows
it suffices to prove the result assuming $m=1$, and there is
no loss of generality assuming $s=1$ as well.

Let $\chi:[0,\infty)\to\Reals$
be a smooth function that equals $1$ on $[0,1]$ and equals $0$ on
$[2,\infty)$.  For a smooth function in $u\in\mathscr H^{1,p;1}(\Hyp\pipesep y)$ an easy computation using Lemma \ref{lem:first-deriv-zero-W} implies
\[
\|u-\chi(y) u_0\|_{H^{0,p}_{1-n/p}(\mathbb H\pipesep y)} \lesssim \|\partial_y u\|_{H^{0,p}_{-n/p}(\mathbb H\pipesep y)} + \|u_0\|_{L^p(\Reals^{n-1})},
\]
where $u_0=u|_{y=0}$.
Since the smooth functions in $\mathscr H^{1,p;1}(\Hyp\pipesep y)=W^{1,p}(\mathbb R^n_+)$
are dense, and since the right-hand side of the previous inequality
is continuous on $\mathscr H^{1,p;1}(\Hyp\pipesep y)$, the inequality
holds for arbitrary elements of $\mathscr H^{1,p;1}(\Hyp\pipesep y)$.

Now suppose $u\in \mathscr H^{1,p;1}(\Hyp;\Reals)\cap H^{0,p}_{1-n/p}(\Hyp\pipesep y)$ and let
$u_0=u|_{y=0}$.  We have just seen that
$u-\chi u_0\in H^{0,p}_{1-n/p}(\Hyp\pipesep y)$ and therefore
$\chi u_0\in H^{0,p}_{1-n/p}(\Hyp\pipesep y)$ as well.
That is, $y^{-1}\chi u_0 \in L^p(\mathbb R^n_+)$.
But $\chi=1$ in a neighborhood of $y=0$, so this
is only possible if $\|u_0\|_{L^p(\Reals^{n-1})}=0$.
This completes the proof of the forward direction of part \eqref{part:fast-decay-implies-zero-derivatives-H}.

For the converse we begin by showing
that
$u\in H^{0,p}_{m-n/p}(\Hyp\pipesep y)$.  The proof is by induction
on $m$ and the case $m=1$ follows from Lemma \ref{lem:first-deriv-zero-W}
and the observation that if $u\in\mathscr H^{s,p;1}(\Hyp\pipesep y)$ then
$\partial_y u\in H^{0,p}_{-n/p}(\Hyp\pipesep y)$.  Suppose the result
holds for some $m\leq s-1$ and consider some $u\in \mathscr H^{s,p;m+1}(\Hyp\pipesep y)$
such that $\partial_y^k u|_{y=0}=0$ for $0\le k\le m$.
Applying the inductive hypothesis to $\partial_y u\in \mathscr H^{s-1,p;m}(\mathbb H\pipesep y)$
satisfying $\partial_y^j (\partial_y u)|_{y=0}=0$ for $0\le j\le m-1$
we conclude $\partial_y u \in H^{0,p}_{m-n/p}(\Hyp\pipesep y)$. Lemma \ref{lem:first-deriv-zero-W}
then implies $u\in  H^{0,p}_{m+1-n/p}(\mathbb H\pipesep y)$ as required.

To address the higher derivatives we note that since $u\in \mathscr H^{s,p;m}(\mathbb H\pipesep y)$,
\[
\nabla^m_{\mathbb E} u \in H^{s-m,p}_{m-n/p}(\Hyp\pipesep y).
\]
Since $u\in H^{0,p}_{m-n/p}(\mathbb H\pipesep y)$ as well,
the analog of Lemma \ref{lem:alt-H-X-norms} discussed at the beginning of this
section
implies that $u\in H^{s,p}_{m-n/p}(\Hyp\pipesep y)$.  This completes the proof of
part \eqref{part:fast-decay-implies-zero-derivatives-H}.

To establish part \eqref{part:fast-decay-implies-zero-derivatives-X}
we first observe that although
$X^{s,p}_{m}(\Hyp\pipesep y)$ does not embed in $H^{s,p}_{m-n/p}(\Hyp\pipesep y)$
as is the case for standard Gicquaud-Sakovich spaces,
it is easy to see that a function in $X^{s,p}_{m}(\Hyp\pipesep y)$
with bounded support in $\Reals^n_+$ does indeed lie in
$H^{s,p}_{m-n/p}(\Hyp\pipesep y)$.  Consequently a function in
$\mathscr X^{s,p;m}(\Hyp\pipesep y) \cap X^{s,p}_m(\Hyp\pipesep y)$
with bounded support in $\Reals^n_+$ belongs to
$\mathscr H^{s,p;m}(\Hyp\pipesep y) \cap H^{s,p}_{m-n/p}(\Hyp\pipesep y)$.
The forward direction of part
\eqref{part:fast-decay-implies-zero-derivatives-X} therefore reduces
to part \eqref{part:fast-decay-implies-zero-derivatives-H} via a localization
argument, noting that vanishing trace is a local property.
The converse direction is proved identically
to part \eqref{part:fast-decay-implies-zero-derivatives-H}
using Lemma \ref{lem:first-deriv-zero-X} in place of
Lemma \ref{lem:first-deriv-zero-W}.
\end{proof}

\subsection{\texorpdfstring{$\Hyp$}{H}-mollification}\label{secsec:group-mollification}
Our asymptotic analysis uses a mollification procedure based on the group structure on $\mathbb H$, described in \cite{WAH}, given by the multiplication
$(\xi,\eta)\cdot(x,y) = (\xi + \eta x, \eta y)$.
The identity element is $(0,1)$ and $(x,y)^{-1} = (-y^{-1}x, y^{-1})$.
Note that the inverse map $w\mapsto w^{-1}$ is a diffeomorphism $B^\mathbb{H}_r \to B^\mathbb{H}_r$ for each $r>0$.

For $z\in \mathbb H$, we denote the left multiplication map by $\Phi_z:w\mapsto z\cdot w$.
Thus the M\"obius parametrization $\Phi_i$ centered at $(x_i,y_i)$ is the restriction of $\Phi_{(x_i,y_i)}$ to $B^\mathbb{H}_2$.
Note that $\breve g$ and the vector fields $y\partial_x,y\partial_y$ are left-invariant with respect to this group structure: for all $z\in \mathbb H$ we have $\Phi_z^*\breve g = \breve g$, $\Phi_z^*(y\partial_x) = y\partial_x$, and $\Phi_z^*(y\partial_y) = y\partial_y$.

Following \cite{WAH}, we define \textbf{$\mathbb H$\nobreakdash-convolution} of $u\in C^\infty(\mathbb H)$ with $\psi$ by
\begin{align}
(u\ast \psi)(z)
&= \int_\mathbb{H} u(w)\,\psi(w^{-1}\cdot z)\,dV_{\breve g}(w)
\label{def-convolution}
\\
&= \int_\mathbb{H} u(z\cdot w)\,\psi(w^{-1})\,dV_{\breve g}(w).
\label{alt-def-convolution}
\end{align}
For notational convenience we define $\psi^\dagger(w)=\psi(w^{-1})$
in which case these formulas can be written
\begin{equation}\label{eq:conv-via-apply}
(u\ast \psi)(z) = \ip<\Phi_z^* u,\psi^{\dagger} >_{(M,g_{\mathbb H})} = \ip<u,\Phi_{z^{-1}}^*\psi^\dagger>_{(M,g_{\mathbb H})}.
\end{equation}

\begin{lemma}
\label{lem:convolution-properties}
Suppose $\psi\in C^\infty_{\rm cpct}(M)$ and $u\in C^\infty(M)$.
\begin{enumerate}
\item \label{part:conv-support}
$\supp u\ast \psi \subset \supp u \cdot \supp \psi = \{a\cdot b:a\in \supp u, b\in \supp v\}$.

\item 
\label{pt:pullback-convolution} $\Phi_z^*(u\ast \psi) = (\Phi_z^*u)\ast\psi$ for all $z\in \mathbb H$.

\item
\label{pt:convolve-with-V0}
$(y\partial)^\alpha (u\ast \phi) = u\ast ((y\partial)^\alpha \phi)$ for all multiindices $\alpha$.

\item\label{pt:convolve-apply}
$\ip< u\ast\psi, \phi>_{(\mathbb H, \breve g)}
=
\ip <u, \phi\ast\psi^\dagger>_{(\mathbb H, \breve g)}$ for all $\phi\in C^\infty_{\rm cpct}(M)$.

\end{enumerate}
\end{lemma}

Now suppose that $\psi\in C^\infty_{\rm cpct}(M)$ and that $u$ is
a distribution on $M$. Motivated by equation \eqref{eq:conv-via-apply} we define
\[
(u\ast \psi)(z) = \ip<u,\Phi_{z^{-1}}^* \psi^{\dagger}>.
\]
Note that this is well defined since $\Phi_{z^{-1}}^* \psi^{\dagger}$ is smooth and compactly
supported.  Moreover, the continuity of $z\mapsto \Phi_{z^{-1}}^* \psi^{\dagger}$
in any $C^k(M)$ implies that $u\ast\psi$ is a continuous function.  

In order to investigate the local effects of $\Hyp$-convolution, we 
suppose that $\psi\in C^{\infty}_{\rm cpct}(B_{r}^\Hyp)$ for some $r>0$.
Given a distribution $u\in H^{s,p}(B_{R+r}^\Hyp)$ we define
$u\ast \psi$ on $B_R^\Hyp$ by $(u\ast\psi)(z) = (\tilde u\ast \psi)(z)$ 
where $\tilde u$ is any compactly supported extension of $u$ in $H^{s,p}(\Hyp)$.
One readily verifies that Lemma \eqref{lem:convolution-properties}\eqref{part:conv-support}
also holds for distributions and hence $u\ast\psi$ is well defined on $B_R^\Hyp$.

The following lemma shows that $\Hyp$-convolution locally improves regularity.

\begin{lemma}
\label{lem:local-convolution}
Suppose $\psi\in C^\infty_{\rm cpct}(B_r^\Hyp)$
and that $u\in H^{s_0, p_0}(B^\mathbb{H}_{R+r})$ for some $r,R>0$, 
$s_0\in \mathbb R$ and $1<p_0<\infty$.
Then $u\ast\psi\in H^{s,p}(B^\mathbb{H}_R)$ for all $s\in \mathbb R$ and $1<p<\infty$ with
\begin{equation*}
\| u\ast\psi\|_{H^{s,p}(B^\mathbb{H}_R)}
\lesssim
\|u\|_{H^{s_0, p_0}(B^\mathbb{H}_{R+r})}.
\end{equation*}
The implicit constant is independent of $u$ but depends on all other data.
\end{lemma}

\begin{proof}
Suppose $u$ is smooth on $B^\mathbb{H}_{R+r}$.
Direct computation shows that for $k\in \mathbb N$ we have
\begin{equation*}
\|\Phi_w^*\psi\|_{C^k(B^\mathbb{H}_{R+r})}
\lesssim \| \psi\|_{C^k(B^\mathbb{H}_r)}
\lesssim 1.
\end{equation*}
where the implicit constants are independent of $w\in B^\mathbb{H}_r$.
Since $z\in B^\mathbb{H}_r$ implies $z^{-1}\in B^\mathbb{H}_r$ we have, for $k \geq |s_0|$, that
\begin{equation*}
|(u\ast\psi)(z)|
= |\ip< u, \Phi_{z^{-1}}^*\psi^\dagger>|
\lesssim \|u\|_{H^{s_0, p_0}(B^\mathbb{H}_{R+r})}
\| \psi\|_{C^k(B^\mathbb{H}_r)}
\lesssim \|u\|_{H^{s_0, p_0}(B^\mathbb{H}_{R+r})}.
\end{equation*}

For each multiindex $\alpha$ we have
\begin{align*}
\|(\rho\partial_\Theta)^\alpha (u\ast \psi)\|_{L^\infty(B^\mathbb{H}_R)}
&= \|u\ast((\rho\partial_\Theta)^\alpha \psi)\|_{L^\infty(B^\mathbb{H}_R)}
\\
&\lesssim \| u\|_{H^{s_0, p_0}(B^\mathbb{H}_{R+r})}
\| \psi \|_{C^{k+|\alpha|}(B^\mathbb{H}_r)}
\\
&\lesssim \| u\|_{H^{s_0, p_0}(B^\mathbb{H}_{R+r})}.
\end{align*}
The lemma now follows for smooth functions from 
the norm equivalence \eqref{local-V0-norm-equivalence} and holds generally
by a density argument.
\end{proof}

Extending the previous result to the global setting, $\Hyp$-convolution 
with smooth compactly supported functions
improves regularity while leaving decay rates at the boundary unchanged.  
For convenience, we restrict $\psi$ to be supported in a small ball
so that its local effects can be observed within the domains of individual
M\"obius parametrization.
\begin{proposition}
\label{prop:prop-conv}
Let $\psi\in C^\infty_{\rm cpct}(B_{1/2}^\Hyp)$
and suppose $s_0\in \Reals$, $1< p_0< \infty$
and $\delta\in\Reals.$
\begin{enumerate}

\item\label{pt:conv-addl-regularity-H}
If $u\in H^{s_0,p_0}_\delta(\mathbb H\pipesep y)$
then $u*\psi\in  H^{s,p}_\delta(\mathbb H\pipesep y)$ for all $s\in\Reals$ and $p\ge p_0$,
with
$$
\|u*\psi\|_{H^{s,p}_\delta(\mathbb H\pipesep y)}
\lesssim \|u\|_{H^{s_0,p_0}_\delta(\mathbb H\pipesep y)}.
$$

\item\label{pt:conv-addl-regularity-X}
If $u\in X^{s_0,p_0}_\delta(\mathbb H\pipesep y)$
then $u*\psi\in  X^{s,p}_\delta(\mathbb H\pipesep y)$ for all $s\in\Reals$ and $1<p<\infty$,
with
$$
\|u*\psi\|_{X^{s,p}_\delta(\mathbb H\pipesep y)}
\lesssim \|u\|_{X^{s_0,p_0}_\delta(\mathbb H\pipesep y)}.
$$

\item\label{pt:conv-addl-script-regularity-H}
If $u\in \mathscr H^{m,p_0;m}(\mathbb H\pipesep y)$ for some $m\in\Nats$,
then $u*\psi\in  \mathscr H^{s,p_0;m}(\mathbb H\pipesep y)$ for all $s\geq m$,
with
$$
\|u*\psi\|_{\mathscr H^{s,p_0;m}(\mathbb H\pipesep y)}
\lesssim \|u\|_{\mathscr H^{m,p_0;m}(\mathbb H\pipesep y)}.
$$

\item\label{pt:conv-addl-script-regularity-X}
If $u\in \mathscr X^{m,p_0;m}(\mathbb H\pipesep y)$ for some $m\in\Nats$,
then $u*\psi\in  \mathscr X^{s,p;m}(\mathbb H\pipesep y)$ for all $s\geq m$ and $1<p<\infty$,
with
$$
\|u*\psi\|_{\mathscr X^{s,p;m}(\mathbb H\pipesep y)}
\lesssim \|u\|_{\mathscr X^{m,p_0;m}(\mathbb H\pipesep y)}.
$$
\end{enumerate}
The implicit constants above are independent of $u$ and $\delta$, but depend on $\psi$, $s_0$, $p_0$, $s$, $p$, and $m$.
\end{proposition}

\begin{proof}
For each M\"obius parametrization $\Phi_i$, Lemmas \ref{lem:convolution-properties}\eqref{pt:pullback-convolution}
and \ref{lem:local-convolution} imply
\begin{equation*}
\|\Phi_i^*(u\ast\psi)\|_{H^{s,p}(B^\mathbb{H}_{1})}
=
\|(\Phi_i^*u)\ast\psi\|_{H^{s,p}(B^\mathbb{H}_{1})}
\lesssim
\|\Phi_i^*u\|_{H^{s_0,p_0}(B^\mathbb{H}_{3/2})}.
\end{equation*}
Recalling \eqref{roman-half-space-norms},
part \eqref{pt:conv-addl-regularity-H} follows from the inclusion
 $\ell^{p_0} \hookrightarrow \ell^{p}$, while part \eqref{pt:conv-addl-regularity-X} is immediate.

To establish part \eqref{pt:conv-addl-script-regularity-H}, we first apply coordinate derivatives $\partial_x$ and $\partial_y$ to \eqref{alt-def-convolution} in order to see that for each multiindex $\alpha$ there exist functions $\psi^{\alpha,\beta}\in C^\infty_{\rm cpct}(B^\mathbb{H}_{1/2})$ such that
\begin{equation*}
\partial^\alpha(u\ast \psi) = \sum_{|\beta|=|\alpha|} (\partial^\beta u)\ast \psi^{\alpha,\beta};
\end{equation*}
see also \cite[Lemma 2.7]{WAH}.
Thus the norm equivalences \eqref{half-space-norm-equivalences} imply 
\begin{align*}
\|u\ast\psi\|_{\mathscr H^{s,p;m}(\mathbb H\pipesep y)}
& \lesssim
\sum_{|\alpha|\leq m} \|\partial^\alpha (u\ast \psi)\|_{H^{s-|\alpha|,p}_{-n/p}(\mathbb H\pipesep y)}
\\
&\leq  \sum_{|\alpha|\leq m}\sum_{|\beta|=|\alpha|} \|(\partial^\beta u)\ast \psi^{\alpha,\beta}\|_{H^{s-|\alpha|,p}_{-n/p}(\mathbb H\pipesep y)}.
\end{align*}
Applying part \eqref{pt:conv-addl-regularity-H} yields the desired result; part \eqref{pt:conv-addl-script-regularity-X} follows similarly.
\end{proof}

We highlight the subtle changes in the hypotheses concerning the Lebesgue parameter
$p$ in the various cases of Proposition \ref{prop:prop-conv}.  In 
part \eqref{pt:conv-addl-regularity-H} the assumption $p\ge p_0$ arises simply
because $\ell^{p_0}\not\subset \ell^{p}$ for $p<p_0$, and it is easy to verify
that it is necessary. There is no analogous restriction in part \eqref{pt:conv-addl-regularity-X}
because $\ell^\infty$ replaces the role of $\ell^p$ regardless of the value of $p$.
In part \eqref{pt:conv-addl-script-regularity-H} the
Lebesgue parameter is fixed at $p=p_0$ and this is also necessary.  
It cannot be lowered because of the obstacle in part \eqref{pt:conv-addl-regularity-H},
and considerations at the boundary imply that it cannot be raised. 
Indeed, first recall that 
$\mathscr H^{m,p_0;m}(\mathbb H\pipesep y) = W^{m,p_0}(\mathbb R^n_+)$.
If $\Hyp$\nobreakdash-mollification were able to improve $p_0$ to
some $p>p_0$, this would imply an increase of boundary regularity to that of
$W^{m,p}(\mathbb R^n_+)$ functions.  But it follows from Proposition \ref{prop:fast-decay-vs-vanishing-derivatives}  
along with Lemma \ref{lem:first-order-approx} proved below that $\Hyp$-convolution
can modify boundary traces at most by multiplication by a constant. Hence
no improvement in $p$ is possible. 
By contrast, the 
fact that $\Hyp$\nobreakdash-mollification takes $\mathscr X^{1,p_0,1}(\mathbb H\pipesep y)$ to $\mathscr X^{\infty;1}(\mathbb H\pipesep y)$ regardless of the value of $p_0$ indicates that  
of elements of $\mathscr X^{1, p_0;1}(\mathbb H\pipesep y)$ 
intrinsically 
have a large amount of boundary regularity.  See, e.g., Corollary \ref{cor:X-bdy-awesome}.

To investigate the effect of $\Hyp$-convolution at the boundary it is helpful
to restrict our attention to functions
$\psi\in C^\infty_{\rm cpct}(B_{1/2}^\Hyp)$ that additionally satisfy
\[
\int_{B_{1/2}^\Hyp}\psi^\dagger\,dV_{\breve g}=1,
\]
in which case we refer to $u\ast\psi$ as the \Defn{$\Hyp$\nobreakdash-mollification} of $u$.
Proposition \ref{prop:prop-conv} implies that $\Hyp$\nobreakdash-mollification
is regularity improving inside $\Hyp$, and we now wish to show that 
it simultaneously leaves boundary traces unaffected. As an initial
step, we have the following analogue of a standard fact about ordinary mollification.
\begin{lemma}\label{lem:group-translate-diff}
Suppose $\psi\in C^\infty_{\rm cpct}(B_{1/2}^\Hyp)$ and that
$\int_{B_{1/2}^\Hyp} \psi^\dagger dV_{\breve g}=1$.
For all $1<p<\infty$, if $u \in H^{1,p}(B_{3/2}^\Hyp)$ then
\[
\|u*\psi-u\|_{L^p(B_1^\Hyp)}
\lesssim \|du\|_{L^p(B_{3/2}^\Hyp)}.
\]
\end{lemma}
\begin{proof}
Let $R_w$ denote right multiplication by $w$.
We claim that
\begin{equation}\label{eq:translate-est}
\|R_w^* u-u\|_{L^p(B_1^\Hyp)} \lesssim \|du\|_{L^p(B_{3/2}^\Hyp)}
\end{equation}
with the implicit constant independent of $w\in B_1^\Hyp$.
Assuming \eqref{eq:translate-est} for the moment, and using that $\int \psi^\dagger\,dV_{\breve g} =1$, we find
\begin{align*}
\int_{B^\Hyp_1} |u*\psi - u|^p dV_{\breve g}
&=
\int_{B^\Hyp_1} \left|
	\int_{B^\Hyp_{1/2}}
	[u(y\cdot w)- u(y)] \psi(w^{-1})
	dV_{\ghyp}(w)\right|^p dV_{\ghyp}(y)\\
&\le
\int_{B^\Hyp_{1/2}}
	\int_{B^\Hyp_1}
	\left| [u(y\cdot w)- u(y)] \psi(w^{-1})\right|^p
	dV_{\ghyp}(y) dV_{\ghyp}(w)\\
&\lesssim
\sup_{w\in B_{1/2}^\Hyp} \|R_w^*u-u\|_{L^p(B^{\mathbb H}_1)}^p\\
&\lesssim
\|du\|_{L^p(B^{\mathbb H}_{3/2})}^p.
\end{align*}
So it only remains to establish inequality \eqref{eq:translate-est}.

Suppose that $u$ is smooth.
Fix $w=(x,y)\in B_{1/2}^\Hyp$ and let $\gamma:[0,1]\to B^\Hyp_1$
be the curve $\gamma(t)=(tx,y^t)$.
Direct computation shows that
\begin{equation*}
|R_w^* u(z) - u(z)|^p
=\left| \int_0^1 \frac{d}{dt} u(z\cdot \gamma(t))\;dt\right|^p
\lesssim
\int_0^1 \left| du_{z\cdot\gamma(t)}\right|_{\ghyp}^p\; dt
\end{equation*}
with implicit constant independent of $u\in W^{1,p}(B_{3/2}^\Hyp)$, $w\in B^{\Hyp}_{3/2}$ and
$z\in B^\Hyp_1$.
Although $\ghyp$ is not invariant with respect to
right multiplication, 
a compactness argument shows 
that $dV_{\breve g} \lesssim R_w^* dV_{\breve g}$
and $R_w^* \breve g \lesssim \breve g \lesssim R_w^* \breve g$
with implicit constants independent of $w\in B_{1/2}^\Hyp$.
Therefore
\begin{align*}
\int_{B^\Hyp_1} |R_w^*u-u|^p\;dV_{\ghyp}
&\lesssim
\int_{B^\Hyp_1}
		\int_0^1 |du_{z\cdot\gamma(t)}|_{\ghyp}^p\;dt
	\;dV_{\ghyp}(z)\\
&\lesssim \int_0^1 \int_{R_{\gamma(t)}(B_1^\Hyp)}
|du_z|_{\ghyp}^p\;  dV_{\ghyp}(z) \;dt\\
&\lesssim
\int_{B_{3/2}^\Hyp}  |du|_{\ghyp}^p \; dV_{\ghyp}.
\end{align*}
Hence inequality \eqref{eq:translate-est} holds for smooth functions and holds
generally by a density argument.
\end{proof}

The following result shows that $\Hyp$\nobreakdash-mollification preserves
boundary values in the sense of an improved decay
rate; in combination with Proposition \ref{prop:fast-decay-vs-vanishing-derivatives}
we conclude that $\Hyp$\nobreakdash-mollification preserves boundary traces.

\begin{lemma}\label{lem:first-order-approx}
Suppose $\psi\in C^\infty_{\rm cpct}(M)$ and that $\int_{B_{1/2}^\Hyp} \psi^\dagger\;dV_{\breve g}=1$
.
\begin{enumerate}
\item\label{part:first-order-H}
If $u\in \mathscr H^{s,p;1}(\Hyp\pipesep y)$
 for some $s\ge 1$ and $1<p<\infty$, then
\[
u*\psi - u \in H^{s,p}_{1-n/p}(\Hyp\pipesep y)
\]
and
\[
\|u*\psi-u\|_{H^{s,p}_{1-n/p}(\Hyp\pipesep y)} \lesssim \|u\|_{\mathscr H^{s,p;1}(\Hyp\pipesep y)}.
\]
\item\label{part:first-order-X}
 If $u\in \mathscr X^{s,p;1}(\Hyp\pipesep y)$
for some $s\ge 1$ and $1<p<\infty$ then
\[
u*\psi - u \in X^{s,p}_{w+1}(\Hyp\pipesep y)
\]
and
\[
\|u*\psi-u\|_{X^{s,p}_{1}(\Hyp\pipesep y)} \lesssim \|u\|_{\mathscr X^{s,p;1}(\Hyp\pipesep y)}.
\]
\end{enumerate}
\end{lemma}
\begin{proof}
We prove \eqref{part:first-order-H}; the proof of
\eqref{part:first-order-X} is similar and a bit simpler.

Suppose $u\in\mathscr H^{s,p;1}(\mathbb H\pipesep y)$, so
$u\in H^{s,p}_{-n/p}(\mathbb H\pipesep y)$ and $du\in H^{s-1,p}_{1-n/p}(\mathbb H\pipesep y)$.
Then the properties of convolution and \ref{lem:group-translate-diff} imply
\begin{align*}
\|u*\psi-u\|_{L^p_{1-n/p}(\mathbb H\pipesep y)}^p
&\lesssim \sum_i y^{-p+n}_i\|\Phi_i^*(u*\psi) - \Phi_i^*u\|_{L^p(B^{\mathbb H}_1)}^p\\
&\lesssim \sum_i y^{-p+n}_i\|(\Phi_i^*u)*\psi - \Phi_i^*u\|_{L^p(B^{\mathbb H}_1)}^p\\
&\lesssim \sum_i  y^{-p+n} \|d(\Phi_i^* u)\|_{L^p(B^{\mathbb H}_{3/2})}^p\\
&= \sum_i  y^{-p+n} \|\Phi_i^* du\|_{L^p(B^{\mathbb H}_{3/2})}^p\\
&\lesssim \|du\|_{L^p_{1-n/p}(\mathbb H\pipesep y)}^p
\\
&\lesssim \|u\|_{\mathscr H^{s,p;1}(\mathbb H\pipesep y)}^p.
\end{align*}
On the other hand, since $u\in\mathscr H^{s,p;1}(\mathbb H\pipesep y)$,
Proposition
\ref{prop:prop-conv}\eqref{pt:conv-addl-regularity-H}
implies
$
d(u*\psi) \in H^{s-1,p}_{1-n/p}(\mathbb H\pipesep y)
$
with norm controlled by the
$H^{s-1,p}_{1-n/p}$ norm of $u$.  That is,
\begin{align*}
u*\psi-u &\in H^{0,p}_{1-n/p}(\mathbb H\pipesep y)\\
d(u*\psi-u) & \in H^{s-1,p}_{1-n/p}(\mathbb H\pipesep y)
\end{align*}
with both norms controlled by the $\mathscr H^{s,p;1}(\Hyp\pipesep y)$ norm
of $u$.  The result now follows from the analog
of Lemma \ref{lem:alt-H-X-norms} discussed at the beginning
of this section.
\end{proof}

\subsection{Asymptotic structure}\label{secsec:asymptotic-structure}
We now turn to the construction of asymptotic expansions
and begin with a technical lemma used
to ensure $\partial_y$
derivatives of coefficient functions appearing in the expansions
can be arranged to vanish to high order.

\begin{lemma} \label{lem:almost-const-coeff}
Suppose $1<p<\infty$ and $k\in\Nats$.
For $u\in C^\infty(\mathbb H)$ define
\[
S_k(u) = \sum_{j=0}^{k-1} \frac{(-1)^jy^j}{j!}\partial_y^j u.
\]
\begin{enumerate}
\item \label{pt:H-coef-regularity}
If $u\in \mathscr H^{\infty,p;k}(\Hyp\pipesep y)$ then $S_k(u)\in \mathscr H^{\infty,p;k}(\Hyp\pipesep y)$.

\item \label{pt:X-coef-regularity}
If $u\in \mathscr X^{\infty;k}(\Hyp\pipesep y)$ then $S_k(u)\in \mathscr X^{\infty;k}(\Hyp\pipesep y)$.

\item \label{pt:coef-machine} In either case \eqref{pt:H-coef-regularity} or case \eqref{pt:X-coef-regularity} we have
\[
\left.\partial^j_y \left[\frac{y^\ell}{\ell!} S_k(u)\right]\right|_{y=0} =
\begin{cases} 0 & 0\le j \le k+\ell-1,\quad j\neq \ell\\
u|_{y=0}& j=\ell
\end{cases}
\]
for each $\ell\in \mathbb Z_{\geq0}$.
\end{enumerate}
\end{lemma}
\begin{proof}
Parts \eqref{pt:H-coef-regularity} and \eqref{pt:X-coef-regularity} follow from the properties of $\mathscr H^{s,p;k}(\mathbb H\pipesep y)$ and $\mathscr X^{s,p;k}(\mathbb H\pipesep y)$ discussed at the beginning of \S\ref{sec:asymptotics}.
We prove part \eqref{pt:coef-machine} in the case that $u\in \mathscr H^{\infty,p;k}(\Hyp\pipesep y)$; the proof of the other case follows from analogous logic.

First consider the case $\ell=0$.
If $1\le j \le k-1$ then
\[
y^j\partial_y^j u\in \mathscr H^{\infty,p;k}(\Hyp\pipesep y)\cap H^{\infty,p}_{j-n/p}(\Hyp\pipesep y)\]
and Proposition \ref{prop:fast-decay-vs-vanishing-derivatives} implies
$y^j\partial_y^j u|_{y=0}=0$.
Hence $S_k(u)|_{y=0}=u|_{y=0}$.
Moreover, direct computation shows that, since
$\partial_y^{k} u \in \mathscr H^{\infty,p;0}(\Hyp\pipesep y)=H^{\infty,p}_{-n/p}(\Hyp\pipesep y)$, we have
\[
\partial_y S_k(u) = \frac{(-1)^{k-1}}{(k-1)!}y^{k-1}  \partial_y^{k} u\in H^{\infty,p}_{k-1-n/p}(\Hyp\pipesep y).
\]
So Proposition \ref{prop:fast-decay-vs-vanishing-derivatives}
implies $\partial_y^{k} S_k(u)|_{y=0} = 0$ for $1\le j\le k-1$.
This proves the result for $\ell=0$, and the result for $\ell>0$ is
a straightforward inductive argument with the $\ell=0$ case as the
base case.
\end{proof}

\begin{theorem}\label{thm:Taylor}
Suppose $1<p<\infty$, $m\in\Nats$, $s\in\Reals$ with
$s\ge m$.
\begin{enumerate}
\item\label{part:Taylor-H}
 Given $u\in \mathscr H^{s,p;m}(\Hyp\pipesep y)$
there exist $u_0,\ldots,u_{m-1}$ such that
for each $0\le k\le m-1$
\begin{enumerate}
	\item $u_k\in \mathscr H^{\infty,p;m-k}(\Hyp\pipesep y)$,
	\item $u_k|_{y=0}=\partial_y^k u|_{y=0}$,
	\item $u_k=0$ if $\partial_y^k u|_{y=0}=0$,
	\item The function $P_k[u]=\sum_{j=0}^k (y^j/j!)u_j\in \mathscr H^{\infty,p;m}(\mathbb H\pipesep y)$ and
	\begin{enumerate}
		\item $u-P_k[u]\in H^{s,p}_{k+1-n/p}(\Hyp\pipesep y)$,
		\item $\partial_y^j P_k[u]|_{y=0}=0$, $k+1\le j\le m-1$.
	\item If $u\in H^{s,p}_{\sigma-n/p}(\Hyp\pipesep y)$ for some $\sigma\in \Reals$
	 then then $P_k[u]\in H^{\infty,p}_{\sigma-n/p}(\Hyp\pipesep y)$.
	\end{enumerate}
\end{enumerate}
In particular, $P_{m-1}[u] \in \mathscr H^{\infty,p;m}(\Hyp\pipesep y)$
and $u = P_{m-1}[u] + H^{s,p}_{m-n/p}(\Hyp\pipesep y)$.

\item\label{part:Taylor-X}
 Given $u\in \mathscr X^{s,p;m}(\Hyp\pipesep y)$
there exist $u_0,\ldots,u_{m-1}$ such that
for each $0\le k\le m-1$
\begin{enumerate}
	\item $u_k\in \mathscr X^{\infty;m-k}(\Hyp\pipesep y)$,
	\item $u_k|_{y=0}=\partial_y^k u|_{y=0}$,
	\item $u_k=0$ if $\partial_y^k u|_{y=0}=0$,
	\item The function $P_k[u]=\sum_{j=0}^k (y^j/j!)u_j\in \mathscr X^{\infty;m}(\mathbb H\pipesep y)$ and
	\begin{enumerate}
		\item $u-P_k[u]\in X^{s,p}_{k+1}(\Hyp\pipesep y)$,
		\item $\partial_y^j P_k[u]|_{y=0}=0$, $k+1\le j\le m-1$.
	\item If $u\in X^{s,p}_{\sigma}(\Hyp\pipesep y)$ for some $\sigma\in \Reals$
	 then then $P_k[u]\in X^{\infty}_{\sigma}(\Hyp\pipesep y)$.
	\end{enumerate}
\end{enumerate}
In particular, $P_{m-1}[u] \in \mathscr X^{\infty;m}(\Hyp\pipesep y)$
and $u = P_{m-1}[u] + X^{s,p}_{m}(\Hyp\pipesep y)$.
\end{enumerate}
\end{theorem}
\begin{proof}
The proofs of parts \eqref{part:Taylor-H} and
\eqref{part:Taylor-X} are essentially the same,
and we focus on part \eqref{part:Taylor-H}.
For $0\le k \le m-1$ define
\[
u_k=\begin{cases} 0 & \partial_y^k u|_{\partial H} \equiv 0\\
S_{m-k}(\partial_y^k u * \psi) &\text{otherwise}.
\end{cases}
\]
Here $S_{m-k}$ is the map defined in Lemma \ref{lem:almost-const-coeff}.
Note that $\partial_y^k u\in\mathscr H^{s-k,p;m-k}(\Hyp\pipesep y)$,
so
Proposition \ref{prop:prop-conv}\eqref{pt:conv-addl-script-regularity-H} and Lemma \ref{lem:almost-const-coeff} ensure
$u_k\in \mathscr H^{\infty,p;m-k}(\Hyp\pipesep y)$.
Since $m-k\ge 1$,
Lemma \ref{lem:first-order-approx}, Proposition \ref{prop:fast-decay-vs-vanishing-derivatives}
and Lemma \ref{lem:almost-const-coeff} imply
$u_k|_{y=0}=\partial_y^k u|_{y=0}$.

The properties of $\mathscr H^{s,p;m}(\mathbb H\pipesep y)$ at the beginning of this section imply that $P_k[u] = \sum_{j=0}^k (y^j/j!)u_j\in
\mathscr H^{\infty,p;m}(\mathbb H\pipesep y)$ for $0\le k \le m-1$.
Thus using
Lemma \ref{lem:almost-const-coeff}
we compute
\[
\left.\partial_y^j P_k[u]\right|_{y=0} = \begin{cases}
\left.\partial_y^j u\right|_{y=0} & 0\le j\le k\\
0 & k+1\le j \le m-1.
\end{cases}
\]
Note that the above computation remains valid even
for those values of $j$ where $\partial^j_y u|_{y=0}\equiv 0$.
As a consequence we find
$\partial_y^j (u-P_k[u])|_{y=0}=0$ for $0\le j \le k$,
and Proposition \ref{prop:fast-decay-vs-vanishing-derivatives} implies $u-P_k[u]\in H^{s,p}_{k+1-n/p}(\Hyp\pipesep y)$.

Concerning the decay rate of $P_k[u]$, suppose $u\in H^{s,p}_{\sigma-n/p}(\mathbb H\pipesep y)$.  If $\sigma\le 0$ then $P_k[u]\in H^{\infty,p}_{\sigma-n/p}(\mathbb H\pipesep y)$
trivially and if $\sigma\ge m$ then
Proposition \ref{prop:fast-decay-vs-vanishing-derivatives}
implies $\partial_j u|_{\partial \Hyp}=0$, $0\le j\le m-1$
and therefore $P_k[u]= 0\in H^{\infty,p}_{\sigma-n/p}(\mathbb H\pipesep y)$.

If $0<\sigma<m$, let $j=\lfloor \sigma\rfloor$ and observe that $0\le j \le m-1$.
Proposition \ref{prop:fast-decay-vs-vanishing-derivatives} implies
$u_i\equiv 0$ if $i<j$, whereas if $j<i\le m-1$  then
$y^i u_i\in H^{\infty,p}_{i-n/p}(\mathbb H\pipesep y)\subset H^{\infty,p}_{\sigma-n/p}(\mathbb H\pipesep y)$.
So it suffices to show that $y^j u_j\in H^{\infty,p}_{\sigma-n/p}(\mathbb H\pipesep y)$.

Recall $u_j = S_{m-j}((\partial_j u)*\psi)$.
Since $\partial_j u\in H^{s-j,p}_{\sigma-j-n/p}(\mathbb H\pipesep y)$
Proposition \ref{prop:prop-conv}\eqref{pt:conv-addl-regularity-H}
implies $(\partial_j u)*\psi\in H^{\infty,p}_{\sigma-j-n/p}(\mathbb H\pipesep y)$.
The proof of Lemma \ref{lem:almost-const-coeff} shows
\[
S_{m-j}\left((\partial_j u)*\psi\right) =
(\partial_j u)*\psi + H^{\infty,p}_{1-n/p}(\mathbb H\pipesep y)
\]
and since $\sigma-j\le 1$ we conclude
$u_j\in H^{\infty,p}_{\sigma-j-n/p}(\mathbb H\pipesep y)$ and $y^j u_j \in H^{\infty,p}_{\sigma-n/p}(\mathbb H\pipesep y)$ as required.
\end{proof}

The fine details of the asymptotic expansions of
Theorem \ref{thm:Taylor} do not translate to to general
manifolds without assuming additional structure near $\partial M$
to make sense of higher order derivatives.  Nevertheless,
we obtain Theorem \ref{thm:curly-split}.

\begin{proof}[Proof of Theorem \ref{thm:curly-split}]
We prove part \eqref{part:curly-split-H} concerning $\mathscr H^{s,p;m}(M;E)$;
the Gicquaud-Sakovich case is proved similarly.
Using a partition of unity
it suffices to show that the result holds
for tensors with bounded support in a single coordinate chart
at infinity.  When $m=0$, we can take $\tau=0$ and $r=u$ and
the result is simply a consequence
of the definition of the space $\mathscr H^{s,p;0}(M;E)$.  Otherwise,
the technique of Lemma \ref{lem:X-other-charts} implies
each of the components $u_A^B$ of $u$
relative to the coordinates at infinity lie in
$\mathscr H^{s,p;m}(\Hyp\pipesep y)$ with $m\ge 1$
and Theorem \ref{thm:Taylor} implies we can write
\[
u_A^B = \tau_A^B + r_A^B
\]
with $\tau_A^B = P_m[u_A^B]\in \mathscr H^{\infty,p;m}(\Hyp\pipesep y)$ and $r_A^B\in H^{s,p}_{m-n/p}(\Hyp\pipesep y)$.

Suppose now that $u$ has additional decay, specifically that
each $u_A^B\in H^{s,p}_{k+\alpha-n/p}(\Hyp\pipesep y)$ for some integer
$0\le k \le m-1$ and some $\alpha\in [0,1)$. We wish to show
that
$\tau_A^B\in y^k (\mathscr H^{s,p;m-k}(\Hyp\pipesep y)\cap H^{\infty,p}_{\alpha}(\Hyp\pipesep y)$
Theorem \ref{thm:Taylor} implies
$\tau_A^B =P_m[u_A^B] \in H^{\infty,p}_{k+\alpha}(\Hyp\pipesep y)=y^k H^{\infty,p}_{\alpha}(\Hyp\pipesep y)$ which finishes
the proof if $k=0$.  If $k>0$ then
Proposition \ref{prop:fast-decay-vs-vanishing-derivatives}
implies that the traces $\partial_y^j u_A^B|_{y=0}$ vanish for $0\le j \le k-1$.
But then $\tau_A^B$ has the form $\sum_{j=k}^m (y^j/j!) a_j$ where
each $a_j\in \mathscr H^{\infty,p;m-j}(\Hyp\pipesep y)$. Observe
that for each $j\ge k$
\[
y^{-k} y^{j} a_j \in \mathscr H^{\infty,p;m-j+(j-k)}(\Hyp\pipesep y)=
\mathscr H^{\infty,p;m-k}(\Hyp\pipesep y)
\]
and hence $\tau_A^B \in y^k \mathscr H^{\infty,p;m-k}(\Hyp\pipesep y)$ as required.

Finally, the fact that for $m\ge 1$ that 
$u|_{\partial M}=0$ if and only if $u\in H^{s,p}_{1-n/p}(\Hyp\pipesep y)$
follows from Proposition \ref{prop:fast-decay-vs-vanishing-derivatives}.
\end{proof}

\section{Elliptic differential operators}\label{sec:differential-ops}

In this section, we establish a priori estimates for operators $\mathcal P[g]$ satisfying Assumptions \ref{Assume-P} and \ref{Assume-I} of the introduction, where $g$ is an asymptotically hyperbolic metric of either class \eqref{intro:H-class} or class \eqref{intro:X-class}.
The main result is Proposition \ref{prop:semi-fred-estimate}, which implies that
these operators are semi-Fredholm and which is
used in \S\ref{sec:fredholm} to obtain a Fredholm theory for such operators.
In the course of obtaining Proposition \ref{prop:semi-fred-estimate}
we also develop some of the machinery needed to build the parametrices
used to obtain the full Fredholm theory; see especially \S\ref{secsec:interp-P}.

Many of the results leading to Proposition \ref{prop:semi-fred-estimate} do not require every point in Assumptions \ref{Assume-P}.
For example, the self-adjoint property is not required by any of the intermediate results.
However, for simplicity of exposition, we frequently invoke Assumption \ref{Assume-P} in its entirety.

\subsection{Linear differential operators}\label{secsec:op-mapping}
We begin with a discussion of the mapping properties of a rather general class of differential operators having coefficients in Gicquaud-Sakovich spaces.
Recall from \S\ref{sec:elementary-spaces} the reference connection $\hat\nabla$ on $\bar M$.
We say that a linear differential operator $L$, taking sections
of a tensor bundle $E$ to sections of tensor bundle $F$ over $M$,
is of \Defn{class $\mathcal L^{s,p}_d$}
for some $s\in\Reals$ and $1<p<\infty$ if it has the form
\begin{equation}\label{eq:curly-L}
L = \sum_{0\le k \le d} a_k\cdot \hat\nabla^k,
\end{equation}
where the dot indicates tensor contraction, and where each
$a_k \in X^{s-d+k,p}_0(M;T^{*,*}M)$ as a section of its $k$-dependent bundle.
These operators determine maps from $C^\infty(M;E)$ to
distributions $\mathcal D'(M;F)$ and we would like to
expand the domain to various Sobolev spaces.
When $s>n/p$, and hence $X^{s,p}_0(M;\Reals)$ is an algebra,
repeated applications of Proposition \ref{prop:multiplication}
yield the following.
\begin{proposition}\label{prop:L-mapping-S}
Suppose $1< p,q <\infty$, $\delta\in\Reals$, $s>n/p$, $\sigma\in\Reals$ and
$d\in\Nats_{\ge 0}$.  An operator
$L=\sum_{k=0}^d a_k \cdot \hat\nabla^k$ of class $\mathcal L^{s,p}_{d}$
extends to continuous linear maps
\[
\begin{aligned}
&H^{\sigma,q}_{\delta}(M;E)\to H^{\sigma-d,q}_{\delta}(M;F),\\
&X^{\sigma,q}_{\delta}(M;E)\to X^{\sigma-d,q}_{\delta}(M;F)\\
\end{aligned}
\]
so long as
\begin{equation}\label{eq:S-conds}
\begin{gathered}
d-s\leq \sigma\leq s,
\\
\frac{1}{p} - \frac{s}{n}\leq
\frac{1}{q} - \frac{\sigma}{n}
\leq \frac{1}{p\dual}-\frac{d-s}{n}.
\end{gathered}
\end{equation}
Moreover, viewed as a map between these spaces, $L$ depends continuously
on its coefficients $a_k \in X^{s-d+k}_0(M;T^{*,*}M)$.
\end{proposition}

\begin{figure}
\centering
\includegraphics[width=0.7\textwidth]{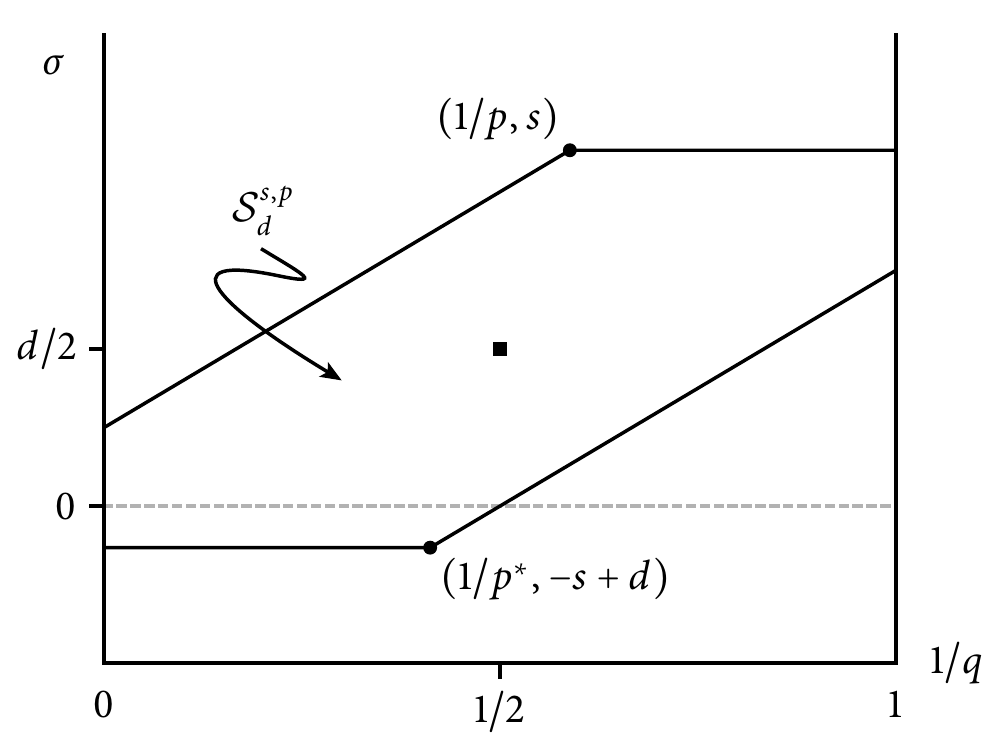}
\caption{The region $\mathcal S^{s,p}_d$ for $n=3$, $d_0=0$, $d=2$, $p=1.7$
and $s=\frac n p + \frac 12$.}\label{fig:Sspd}
\end{figure}

It is convenient to have notation for the Sobolev indices
satisfying conditions \eqref{eq:S-conds}.
Given $1<p<\infty$, $s\in\Reals$ and $d\in\Nats_{\ge 0}$,
the \textbf{compatible Sobolev indices $\mathcal S^{s,p}_d$}
for an operator of class $\mathcal L^{s,p}_d$ is the set
of tuples $(\sigma,q)\in \Reals\times (1,\infty)$  satisfying \eqref{eq:S-conds}.

For a general choice of $s$, $p$, $d$, the set
$\mathcal S^{s,p}_{d}$ may be empty. The following
lemma provides conditions that guarantee
that $\mathcal S^{s,p}_d$ is nontrivial and is
proved identically to its counterpart in \cite{HMT2020},
using Lemma \ref{lem:basic-inclusions} and Proposition \ref{prop:SobolevEmbedding}
to establish the embeddings.
\begin{lemma}\label{lem:S-members}
Suppose $1<p<\infty$, $s\in\Reals$ and $d\in\Nats_{\ge 0}$.
Then $\mathcal S^{s,p}_d$
is nonempty if and only if
\begin{equation}\label{eq:S-p-2}
s\geq \frac{d}{2}
\quad\text{ and }\quad
\frac{1}{p} - \frac{s}{n} \le \frac{1}{2}-\frac{d/2}{n}.
\end{equation}
Consequently:
\begin{enumerate}
\item If $n\ge 2$, $s\ge 1$, and $s>n/p$, then $\mathcal S^{s,p}_2$ is non-empty.

\item If $\mathcal S^{s,p}_d$ is non-empty, then it contains $(s,p)$, $(d-s,p\dual)$, and $(d/2,2)$.

\item If $(\sigma,q)\in \mathcal S^{s,p}_{d}$,
then for any $\delta\in\Reals$ and $\epsilon>0$
we have the continuous inclusions for any tensor bundle $E$ over $M$
\begin{equation}\label{eq:S-include}
\begin{gathered}
H^{s,p}_{\delta+\frac{n-1}{p}}(M;E) \hookrightarrow
H^{\sigma,q}_{\delta+\frac{n-1}{q}+\epsilon}(M;E) \hookrightarrow
H^{d-s,p\dual}_{\delta+\frac{n-1}{p\dual}+2\epsilon}(M;E),\\
X^{s,p}_{\delta}(M;E) \hookrightarrow
X^{\sigma,q}_{\delta}(M;E) \hookrightarrow
X^{d-s,p\dual}_{\delta}(M;E).
\end{gathered}
\end{equation}
\end{enumerate}
\end{lemma}
The main conclusion of Lemma \ref{lem:S-members} is striking: $\mathcal S^{s,p}_d$
is nonempty if and only if it contains $(d/2,2)$, which corresponds to a weak $L^2$
theory (e.g., $W^{1,2}$ for a second-order operator).
We therefore call \eqref{eq:S-p-2} the \Defn{weak $L^2$ condition} and reiterate that it always
holds for a second-order operator when $n\geq 2$, $s\ge 1$, and $s>n/p$.

If $s>n/p$ then the highest-order coefficients of an operator
$L\in \mathcal L^{s,p}_d$ are continuous.
We say
that $L$ is \Defn{elliptic} on open set $U$ if at each point the principle operator $a_d\cdot\hat\nabla^d$ is elliptic
in the usual sense.
If $U$ is an open set in $M$, we let $H^{s,p}(U;E)$ be the quotient space of
$H^{s,p}_0(M;E)$, where two sections are identified if they agree as distributions on 
$U$; this construction is analogous to the definition of Bessel potential spaces on 
open subsets of $\Reals^n$. The local elliptic regularity theorem of \cite{HMT2020}
then readily implies the following interior estimate.
\begin{proposition}\label{prop:interior-reg}
Suppose $1<p<\infty$, $s\in\Reals$ with $s>n/p$, and $d\in \Nats_{\ge 0}$.
Suppose $L$ is a differential operator of class $\mathcal L^{s,p}_d$, acting sections of tensor bundle $E$ over $M$,
that is elliptic on some precompact
open set $U$.  If $u\in H^{d-s,p\dual}(U;E)$
and $L u\in H^{\sigma-d,q}(U;E)$
for some $(\sigma,q)\in \mathcal S^{s,p}_d$,
then for any open set $V$ with $\overline V\subset U$, we have $u\in H^{\sigma,q}(V;E)$ and
\begin{equation}\label{eq:int-reg}
\|u\|_{H^{\sigma,q}(V;E)} \lesssim \|Lu\|_{H^{\sigma-d,q}(U;E)} +
\|u\|_{H^{d-s-1,p\dual}(U;E)},
\end{equation}
where the implicit constant depends on $U$ and $V$.
\end{proposition}

We have the following commutator result; note
the factor of $\rho^{-1}$, which is essential
in our applications.

\begin{lemma}\label{lem:commutator}
Suppose $L$ is a differential operator
of class $\mathcal L^{s,p}_d$ on sections of tensor bundle
$E$ over $M$.
If $\phi$ is a smooth function on $\overline M$ then
\[
\rho^{-1} [L,\phi] \in \mathcal L^{s,p}_{d-1}.
\]
between the same bundles.
In particular, if $(\sigma,q)\in\mathcal S^{s,p}_d$,
it determines continuous linear maps
$H^{\sigma,q}_{\delta}(M;E) \to H^{\sigma-d+1,q}_\delta(M;E)$
and
$X^{\sigma,q}_{\delta}(M;E) \to X^{\sigma-d+1,q}_\delta(M;E)$.
\end{lemma}
\begin{proof}
Since $\phi\in C^\infty(\bar M)$,
and since $\hat\nabla^k \phi$ is a tensor of weight $k$,
Lemma \ref{lem:CMbar-into-CMw} implies
$\rho^{-1}\hat\nabla^k \phi \in C^\infty_{k-1} (M)
\subset C^\infty_0(M)$ for any $k\ge 1$.  Noting that
all the coefficients of $[L,\phi]$ involve at least
one derivative of $\phi$, the sum decomposition \eqref{eq:curly-L}
of $[L,\phi]$ is then a straightforward computation.
Verification of the mapping property
follows from the fact that $\mathcal S^{s,p}_{d-1}\supset
\mathcal S^{s,p}_d$.
\end{proof}

\subsection{Geometric operators and Assumption \ref{Assume-P}}\label{secsec:geometric-P}

We now discuss the contents of Assumption \ref{Assume-P} in more detail, and show that results of \S\ref{secsec:op-mapping} apply to such operators.
We assume that $g$ is an asymptotically hyperbolic metric satisfying either
\begin{itemize}
\item assumption \eqref{intro:H-class}, in which $\bar g\in \mathscr H^{s,p;m}(M;T^{0,2}M)$ with $s\geq m > n/p$  and $1<p<\infty$; or
\item assumption \eqref{intro:X-class}, in which $\bar g\in \mathscr X^{s,p;m}(M;T^{0,2}M)$ with $s\geq m \geq 1$, $s > n/p$, and $1<p<\infty$.
\end{itemize}
For such a metric, we consider $d^\text{th}$-order differential operator $\mathcal P[g]$ satisfying a list of assumptions that we now explain.

The first part of Assumption \ref{Assume-P} is that $\mathcal P[g]$  acts on sections of a geometric tensor bundle $E$ over $M$, yielding sections of that same bundle.
It is important that, as noted in \eqref{intro:alpha-exists}, under either of the regularity conditions above imply we have $\bar g\in C^{0,\alpha}(\bar M;T^{0,2}M)$ for some $0<\alpha\leq 1$, and thus the discussion of geometric tensor bundles in \S\ref{secsec:gtb-early} applies.
That $\bar g$ is H\"older continuous on $\bar M$ for positive $\alpha$, and not merely continuous, is essential for estimates of \S\ref{secsec:boundary-decay}, and subsequent Fredholm theory of \S\ref{sec:fredholm}.

The second part of Assumption \ref{Assume-P} is that $\mathcal P[g]$ is {geometric}, so that in any coordinate system $\mathcal P[g]u$ is a linear function of the components of
 $u$ and their derivatives up to order $d$, and such that the coefficients are universal
 polynomials in the components of $g$, their derivatives, and $\det(g_{ij})^{-1/2}$.
It is further required that the coefficient
 of any $k^\text{th}$ derivative of $u$ involves at most the the first $d-k$  derivatives of $g$.
As is discussed \S\ref{secsec:boundary-decay} below, this assumption allows us to compare $\mathcal P[g]$ with the corresponding operator defined on hyperbolic space.

Due to restrictions imposed by multiplication in Sobolev
 spaces (see Proposition \ref{prop:multiplication}) we are obliged to consider a smaller class of operators, in order to ensure that $\mathcal P[g]$ be of class $\mathcal L^{s,p}_d$.
This motivates the third part of Assumption \ref{Assume-P}, that $\mathcal P[g]$ be {scale natural}.
Thus each term in the polynomial comprising the coefficient of any $k^{\rm th}$ derivative of $u$ has a total of at most
 $d-k$ metric derivatives distributed over all the factors comprising the
 coefficient. 
 As noted in the introduction, although $\Delta_g+\R[g]^2$ is a geometric second-order operator,
 it is not scale natural because the low-order coefficient contains terms
 with a total of four metric derivatives ($\partial^2 g \partial^2 g$, etc.).
 On the other hand, $\Delta_g + \R[g]$ is scale natural.

Using these assumptions, along with the fifth point of Assumption \ref{Assume-P}, which is that the weak $L^2$ condition is satisfied, we see $\mathcal P[g]$ is indeed an operator of class $\mathcal L^{s,p}_d$.

\begin{lemma}\label{lem:geometric-op-mapping}
Let $g$ be an asymptotically hyperbolic metric of either class \eqref{intro:H-class} or class \eqref{intro:X-class}, and suppose that $\mathcal P[g]$ satisfies Assumption \ref{Assume-P}.
Then $\mathcal P[g]$ is of class $\mathcal L^{s,p}_d$.

Furthermore, writing
\begin{equation}
\label{geometric-op-form}
\mathcal P[g] = \sum_{k=0}^d a_k[g]\cdot \hat\nabla^k,
\end{equation}
each map $g\mapsto a_k[g]$ is Lipschitz continuous at $g$
as a map taking a neighborhood of $g$ in $X^{s,p}_0(M;T^{0,2}M)$ to $X^{s-d+k,p}_0(M;T^{*,*}M)$.
\end{lemma}

\begin{proof}
Consider a M\"obius parametrization $\Phi_i$.
From the definition
of a geometric operator
\begin{equation*}
\Phi_i^*\mathcal P[g]
= \sum_{k=0}^d \Phi_i^*(a_k[\Phi_i^*g]) \cdot (\Phi_i^* \hat \nabla)^k
= \sum_{k=0}^d b_k[\Phi_i^*g]\cdot \nabla_\mathbb{E}^k
= \sum_{k=0}^d b_k[\Phi_i^*g]\cdot (\Phi_i^* \hat \nabla - D)^k,
\end{equation*}
where $D$ is the difference tensor $\Phi_i^* \hat\nabla-\nabla_{\mathbb E}$
and where each component of $b_k[g_i]$ is a polynomial in
components of $g_i=\Phi_i^*g$, their derivatives, and $\det(g_i)^{-1/2}$.
The uniform estimates for $D$ from Lemma \ref{lem:christoffel-bound}
allows us to ignore these terms and the result follows if we can
establish uniform bounds for the coefficients
\begin{equation}\label{eq:P-coeff-bound}
\|b_k[g_i]\|_{H^{s-d+k,p}(B_1^\Hyp)} \lesssim 1
\end{equation}
along with a Lipschitz estimate
\begin{equation}\label{eq:P-coeff-lip}
\|b_k[g_i+\delta g_i]-b_k[g_i\|_{H^{s-d+k,p}(B_1^\Hyp)} \lesssim \|\delta g_i\|_{H^{s,p}(B_1^\Hyp)},
\end{equation}
where the implicit constants are independent of $i$, and where
inequality \eqref{eq:P-coeff-lip} holds when
$\|\delta g_i\|_{H^{s,p}(B_1^\Hyp)} \le c$ for some $c$ independent of $i$.

Let $c[g_i] (\det g_i)^{-\ell/2}$ be a term of the polynomial defining
a component of $b_k[g_i]$, where $\ell\in\Nats_{\ge 0}$ and where
$c[g_i]$ consists of a product of components of $g_i$ and their derivatives.
Since $s>n/p$, since no more than $d-k$ derivatives appear among
the factors of $c[g_i]$, and since the weak $L^2$ condition holds,
the local coordinate equivalent of Corollary \ref{cor:multiplication}
shows that $c[g_i]\in H^{s-d+k,p}(B_1^\Hyp)$ with a norm that is
uniformly estimated in terms of $\|g\|_{X^{s,p}_0(M;T^{0,2}M)}$, and that the map
$g_i\to c[g_i]$ is Lipschitz at $g_i$
with a constant estimated in terms of an upper bound for $\|g_i\|_{H^{s,p}(B_1^\Hyp)}$
so long as
the perturbation $\delta g_i$ satisfies $\|\delta g_i\|_{H^{s,p}(B_1^\Hyp)}\le K$
for any fixed choice of $K$.

Since $s>n/p$, (and since the weak $L^2$ condition is satisfied and since $k\le d$),
a multiplication argument shows that the proof is complete if we can show that
$\delta g \mapsto \det(g_i+\delta g)^{-1/2}$
is Lipschitz continuous into $H^{s,p}(B_1^\Hyp)$ on a ball containing $0$ of
size independent of $i$.  To do so,
first observe that $\det(g_i)\ge \epsilon>0$
for some $\epsilon$ independent of $i$.
Indeed, writing
$\bar g = \bar g_{ab} d\Theta^a d\Theta^b$ in the coordinates near infinity corresponding
to $\Phi_i$, $\det(g_i) = y^{-n}\det(\bar g_{ab}\circ \Phi_i)$ which along with the
continuity of $\bar g$
establishes the
desired lower bound on $B_1^\Hyp$.  Since $H^{s,p}(B_1^\Hyp)$ embeds into the continuous
functions on $B_1^\Hyp$,
if $\|\delta g\|_{H^{s,p}(B_1^\Hyp)}\le c$
for some $c$ sufficiently small (depending on $\epsilon$ but independent of $i$),
we can ensure that $\det (g_i+\delta g)> \epsilon/2$. This lower bound, along with the obvious
control of $\det(g_i+\delta g)$ in $H^{s,p}(B_1^\Hyp)$, allows us to apply
Lemma \ref{lem:power}
to conclude $\delta g\mapsto \det(g_i+\delta g)^{-1/2}$ is  Lipschitz continuous in $\delta g$
if $\|\delta g\|_{H^{s,p}(B_1^\Hyp)}\le c$ as required.
\end{proof}

Assumption \ref{Assume-P} concludes with one final hypothesis, that $\mathcal P[g]$ is formally self-adjoint in the sense that 
for all sections $u,v\in H^{s,p}_0(M;E)$ we have
\[
\int_M \langle u,\mathcal P[g]v\rangle_g\,dV_g = \int_M \langle \mathcal P[g] u, v\rangle_g\,dV_g.
\]
Note that if $s<d$ the integrals above should be interpreted in the sense of the duality
pairing from \S\ref{secsec:duality-early}, and a
computation using the weak $L^2$ condition
shows that the pairing is well defined for $u,v\in H^{s,p}_0(M;E)$.
We emphasize that the self-adjoint assumption is not, in fact, used until \S\ref{secsec:semi-fred}.
However, for simplicity of exposition we invoke Assumption \ref{Assume-P} in its entirety in the intervening results.

\subsection{Boundary decay estimates}\label{secsec:boundary-decay}

Our technique for obtaining semi-Fredholm estimates for $\mathcal P[g]$
involves localizing  the operator at points at infinity where
elliptic estimates for $\mathcal P[\breve g]$ are used instead.
In order to accomplish this, we require that $g$ can be approximated
by a metric isometric to $\breve g$ at points of $M$ sufficiently close
to $\partial M$.  The following lemma is the key technical tool
needed to ensure the requisite decay, and in the $\mathscr H^{s,p;m}$
case can be thought of as a refinement of Proposition \ref{prop:curly-H-into-roman-X}.

\begin{lemma}[cf.~Lemma 6.1 in \cite{Lee-FredholmOperators} ]\label{lem:r_rescale}
Let $1<p<\infty$, $s\in\Reals$, $m\in\Nats$, $s\ge m>n/p$.
Suppose
$u \in \mathscr H^{s,p;m}(M;E)$ is a tensor of weight $w$ and thus by Proposition \ref{prop:curly-to-M-bar} $u\in C^{0,\alpha}(\bar M;E)$ for some $\alpha\in (0,1]$.
If $\Phi_i$ is a M\"obius parametrization
with image $\Phi_i(B_1^\Hyp)$ intersecting $Z_r(\hat p)$ for some $0<r<r_0$ then
\[
\rho_i^{-w}\| \Phi_i^* u \|_{H^{s,p}(B_1^\Hyp)} \lesssim r^\alpha \|u\|_{\mathscr H^{s,p;m}(M;E)},
\]
with implicit constant independent of $i$, $\hat p$, $r$ and $u$.

An analogous statement holds if instead $s>n/p$ and $s\ge m\ge 1$ and $u\in \mathscr X^{s,p;m}(M;E)$.
\end{lemma}
\begin{proof}
Lemma \ref{lem:missing}, which follows from the Gagliardi-Nirenberg inequality,
implies
\begin{equation}\label{eq:Phi_i_u_split_est}
\|\Phi^*_i u\|_{H^{s,p}(B_1^\Hyp)} \lesssim \|\Phi^*_i u\|_{H^{0,p}(B_1^\Hyp)} +
\|\nabla_\mathbb{E}^m\Phi^*_i u\|_{H^{s-m,p}(B_1^\Hyp)}.
\end{equation}
To estimate the first term on the right-hand side of equation \eqref{eq:Phi_i_u_split_est}
we rely on the embedding $\mathscr H^{s,p;m}(M)\hookrightarrow C^{0,\alpha}(\bar M)$.

Since the geodesic distance from $\hat p$ to a point
in $Z_r(\hat p)$ is controlled uniformly by $r$, and since
$\diam_{\hat g} \Phi_i(B_1^\Hyp) \lesssim r$,
\[
|u|_{\hat g} \lesssim r^\alpha \| u \|_{\mathscr H^{s,p;m}(M)}
\]
on $\Phi_i(B_1^\Hyp)$. Using the fact that $\Phi_i^*\check g$ is uniformly close to the Euclidean metric on $B_1^\Hyp$ and the identity
$|u|_{\check g} = \rho^{w}|u|_{\hat g}\sim \rho_i^{w}|u|_{\hat g}$
on $\Phi_i(B_1^\Hyp)$ we find
\[
\|\Phi_i^* u\|_{H^{0,p}(B_1^\Hyp)}
\lesssim
\|\Phi_i^*\ |u|_{\check g} \|_{H^{0,p}(B_1^\Hyp)}\lesssim
\rho_i^w
\|\Phi_i^* |u|_{\hat g} \|_{H^{0,p}(B_1^\Hyp)}
\lesssim \rho_i^w r^\alpha \| u \|_{\mathscr H^{s,p;m}(M)}
\]
as required.

To estimate the second term on the right-hand side of equation, we claim that
for each $0\le k\le m$ that
$\|\nabla_{\mathbb E}^k\Phi_i^* u\|_{H^{s-k,p}(B_1^\Hyp)} \lesssim \rho_i^{w+k-n/p}
\|u\|_{\mathscr{H}^{s,p;m}(M)}$. Assuming the claim is true, then
\[
\rho_i^{-w}\|\nabla_{\mathbb E}^m\Phi_i^* u\|_{H^{s-k,p}(B_1^\Hyp)} \lesssim
r^{m-n/p} \|u\|_{\mathscr{H}^{s,p;m}(M)} \lesssim r^{\alpha} \|u\|_{\mathscr{H}^{s,p;m}(M)}
\]
since $m-n/p\ge \alpha>0$ (Proposition \ref{prop:curly-to-M-bar}) and since $\rho_i\lesssim r$, which completes the desired estimate.

The claim is proved inductively with the base case being obvious.
Suppose for some $k<m$ that
the claim holds for each  $0\le j \le k$.
Then
\[
\nabla_{\mathbb E}^{k+1}\Phi_i^* u =  \Phi_i^*\hat\nabla^{k+1} u + Au,
\]
where $Au$ is a sum of contractions of tensors of the form $A_j\nabla_{\mathbb E}^ju$
with $j\le k$, and where each $A_j$ is a sum of contractions of tensors of the
form $\nabla_{\mathbb E}^{a_1}D\otimes\cdots \otimes \nabla_{\mathbb E}^{a_\ell}D$
with $a_1+\cdots+a_\ell+\ell = k+1-j$ where $D=\Phi_i^*\hat\nabla-\nabla_{\mathbb E}$.
The claim then follows from the induction hypothesis, the uniform estimates on
derivatives of $D$ from  Lemma \ref{lem:christoffel-bound}, and the fact that
$\|\Phi_i^* \hat\nabla^{k+1} u\|_{H^{s-k-1,p}(B_1^{\Hyp})}\lesssim \rho_i^{w+k+1 - n/p} \|u\|_{\mathscr H^{s,p;m}(M)}$ .

The proof in the $\mathscr X^{s,p;m}$ category is similar. With the new hypotheses
on the parameters and with the new definition of $\alpha$ it follows that
$\mathscr X^{s,p;m}(M)$ embeds in $C^{0,\alpha}(\bar M)$
and the estimate of the first term on the right-hand side of equation \eqref{eq:Phi_i_u_split_est}
proceeds identically.  An analogous argument for the second term
shows $\|\Phi_i^*\nabla^m_{\mathbb E} u\|_{H^{s,p}(B_1^\Hyp)}\lesssim \rho_i^{w+m}\|u\|_{\mathscr X^{s,p;m}(M)}$.
But then
\[
\rho_i^{-w}\|\Phi_i^*\nabla^m_{\mathbb E} u\|_{H^{s,p}(B_1^\Hyp)}\lesssim r \|u\|_{\mathscr X^{s,p;m}(M)}
\]
since $m\ge 1$ and since $\rho_i\lesssim r$.
\end{proof}

\subsection{Localized approximate operators at infinity}\label{secsec:localized-ops}

Let $g = \rho^{-2}\bar g$ be an asymptotically hyperbolic metric of class \eqref{intro:H-class} or class \eqref{intro:X-class}, and
let $\mathcal P[g]$ satisfy Assumption \ref{Assume-P} acting on sections of geometric tensor bundle $E$.
Our goal is to approximate, close to points $\hat p\in \partial M$, the operator $\mathcal P[g]$ by an operator constructed from $\mathcal P[\breve g]$.

To accomplish this, we first construct coordinates at $\hat p\in \partial M$ adapted to $\bar g$, following the construction in Chapter 6 of \cite{Lee-FredholmOperators}.
Recall from \S\ref{sec:coords} the coordinates near infinity $\Theta_{\hat p}\colon Z_{r_0}(\hat p) \to Y_{r_0}$ defined where $\rho\leq r_0$.
Let $K_{\hat p}\in T^{1,1}\bar M$ be the tensor at $\hat p$ obtained by the Gram-Schmidt
algorithm such that the forms
$\beta^i = (K_{\hat p})^i_j d\Theta^j$ form a $\bar g$-orthonormal co-frame at $\hat p$ and
such that $\beta^n = |d\rho|_{\bar g}^{-1}d \rho$.
We define the \Defn{adapted coordinates near infinity $\tilde\Theta_{\hat p}$} by $\tilde\Theta_{\hat p}^i =  (K_{\hat p})^i_j\Theta^j_{\hat p}$.
Observe that $\tilde\Theta_{\hat p}^n = \Theta_{\hat p}^n = \rho$,
and also that in adapted coordinates we have
$\bar g(\hat p) = \delta_{ij}d\tilde\Theta_{\hat p}^i\tilde\Theta_{\hat p}^j$.
The continuity of $\bar g$
on the compact boundary ensures that the values $(K_{\hat p})^i_j$ and $(K_{\hat p}^{-1})^j_i$
are bounded independent of $\hat p$ and these estimates ensure that
the Jacobians  $D(\tilde \Theta\circ \Theta^{-1})$ and their inverses admit
uniform $C^k$ estimates for any boundary coordinates $\Theta$.

We now show that the metrics $\tilde\Theta_{\hat p}^*\breve g$ and $g$
are uniformly close in a neighborhood of $\hat p$ and, consequently, that
geometric differential operators built from each are comparable.
\begin{lemma}
\label{lem:local-compare-operators}
Suppose $g$ is an asymptotically hyperbolic metric of class \eqref{intro:H-class} or \eqref{intro:X-class}, and let $\alpha$ be as in \eqref{intro:alpha-exists}.  
Let $\mathcal P[g]$ satisfy Assumption \ref{Assume-P}.
There exists $r_0'\le r_0$
depending on $g$ and $\mathcal P$ such that
for all $\hat p\in \partial M$,
if $\Phi_i$ is a M\"obius parametrization
with $\Phi_i(B_1^\Hyp)$ intersects $Z_{r}(\hat p)$ for some $0<r<r_0'$, then
\begin{enumerate}
\item\label{part:B-in-Z}
$\Phi_i(B_1^\Hyp)\subset Z_{r_0}(\hat p)$ and hence $\Phi_i^*\tilde\Theta_{\hat p}^* \breve g$ is well defined,
\item\label{part:g-is-close}
 $\|\Phi_i^*((\tilde\Theta_{\hat p}^* \breve g)-g)\|_{H^{s,p}(B_1^\mathbb H)} \lesssim r^\alpha$,
\item\label{part:P-is-close}
if $(\sigma, q)\in \mathcal S^{s,p}_d$ and $u\in H^{\sigma,q}(B_1^\mathbb H)$ we have
\begin{gather}
\label{eq:P-is-bounded}
\|\Phi^*_i(\mathcal P[\tilde\Theta_{\hat p}^*\breve g)] u\|_{H^{\sigma-d,q}(B_1^\mathbb H)} \lesssim \|u\|_{H^{\sigma,q}(B_1^\mathbb H)},
\\
\label{eq:P-is-close}
\|\Phi^*_i(\mathcal P[\tilde\Theta_{\hat p}^*\breve g] - \mathcal P[g])u)\|_{H^{\sigma-d,q}(B_2^\mathbb H)}
\lesssim r^\alpha \|u\|_{H^{\sigma,q}(B_1^\mathbb H)}.\
\end{gather}
\end{enumerate}
The implicit constants are independent of $i$, $\hat p$, $r$ and $u$.
\end{lemma}
\begin{proof}
For part \eqref{part:B-in-Z},
a compactness argument shows that we can take $r_0'$ sufficiently small independent of $\hat p$
so that if $\Phi_i(B_1^\Hyp)$ intersects $Z_{r_0'}(\hat p)$ then $\Phi_i(B_1^\Hyp)\subset Z_{r_0}(\hat p)$; see also Lemma \ref{lem:kappa} where a stronger fact is proved.

Let $\tilde g$ be a smooth tensor on $\bar M$
that agrees with $\tilde \Theta^* g_{\mathbb E}$ on $Z_{r_0}$
and such that $\tilde g$ is bounded in any $C^k(\bar M)$ independent of $\hat p$;
uniform bounds for transition Jacobians permit this.
Part \eqref{part:g-is-close} follows by applying Lemma \ref{lem:r_rescale}
 to $\bar g - \tilde g$ and noting that since $\tilde\Theta_{\hat p}^n = \rho$,
 \[
\Phi_i^*(g-\tilde\Theta^*_{\hat p}\breve g) =
(\rho_i y)^{-2}
\Phi_i^*(\bar g-\tilde\Theta^*_{\hat p}g_{\mathbb E}).
\]
Note that the implicit constant of part \eqref{part:g-is-close}
arises from Lemma \ref{lem:r_rescale} and is independent
of $i$ and $\hat p$ in addition to $r$.

In order to obtain estimate \eqref{eq:P-is-close}
 of part \eqref{part:P-is-close}, write
$\Phi_i^*\mathcal P[g] = \sum_k b_k[\Phi_i^* g]\cdot \nabla_{\mathbb E}^k$
as in the proof of Lemma \ref{lem:geometric-op-mapping}.  The Lipschitz estimate
\eqref{eq:P-coeff-lip} from that proof implies that
if  $\Phi_i^*\tilde\Theta_{\hat p}^*\breve g$
is close enough to $\Phi_i^*g$ in $H^{s,p}(B_1^\Hyp)$,
with a closeness bound independent of $i$, then
\begin{equation}\label{eq:P-local-est}
\| b_k[\Phi_i^*\tilde\Theta_{\hat p}^*\breve g] - b_k[\Phi_i^*g] \|_{H^{s-d+k,p}(B_1^\mathbb{H})}
\lesssim
\|\Phi_i^*\tilde\Theta_{\hat p}^*\breve g-\Phi_i^*g\|_{H^{s,p}(B_1^\mathbb{H})}
\end{equation}
with an implicit constant that does depend on $g$ but
can be arranged as in Lemma \ref{lem:geometric-op-mapping} to be independent
of $i$; note that $\hat p$ plays no role.
The closeness bound can be satisfied part \eqref{part:g-is-close} and
shrinking $r_0'$ if needed by a quantity independent of $i$ and $\hat p$.
Inequality \eqref{eq:P-is-close}
then follows from estimate \eqref{eq:P-local-est}, part \eqref{part:g-is-close} and the local version of the product estimates in Proposition \ref{prop:multiplication}, and
inequality \eqref{eq:P-is-close} and the uniform estimates \eqref{eq:P-coeff-bound}
from Lemma \ref{lem:geometric-op-mapping}.
\end{proof}

We seek to compare the operator $\mathcal P[g]$, which acts on sections of bundle $E$, to the operator $\mathcal P[\breve g]$.
However, the geometric tensor bundle $\breve E\to \mathbb H$ on which $\mathcal P[\breve g]$ acts might not pull back under $\tilde\Theta_{\hat p}$ to the bundle $E\to M$ on which $\mathcal P[g]$ acts; this occurs, for example, if the definition of $E$ is contingent upon the metric $g$, in which case $\breve E$ would be defined using $\breve g$ instead; see \S\ref{secsec:gtb-early} for a more detailed discussion.
For such bundles, the comparison in Lemma \ref{lem:local-compare-operators} is not sufficient for the Fredholm theory of \S\ref{sec:fredholm}, where we use the fact that $\mathcal P[\breve g]$ is invertible on $\breve E$.

In order to address this situation,
we construct a locally-defined tensor $A[g]\in T^{1,1}\bar M\big|_{Z_{r_0}(\hat p)}$,
again applying the Gram-Schmidt algorithm but now starting with the
$\tilde\Theta^*\breve g$-orthonormal co-frame
$\rho^{-1}d\tilde\Theta^j$ and resulting in a $g$-orthonormal co-frame
$\omega^i=(A[g])^i_j \rho^{-1}d\tilde\Theta^j$.
Let $T^{k_1,k_2} M$ be the ambient tensor bundle of $E$.
Then $A_{\hat p}[g]$ gives rise to a bundle automorphism tensor $\Pi_{\hat p}[g]\in T^{k_1+k_2,k_1+k_2}M\big|_{Z_{r_0}(\hat p)}$,
formed from tensor products of $A_{\hat p}[g]$ and $A_{\hat p}^{-1}[g]$,
such that contraction $u\mapsto \Pi_{\hat p}[g]\cdot u$ takes $g$ orthonormal frames to $\tilde\Theta_{\hat p}^*\breve g$ orthonormal frames. As a consequence,
if $u$ is a section of $E$, then the contraction $\Pi_{\hat p}[\bar g]\cdot u$ is a section of $\tilde\Theta_{\hat p}^*\breve E$,
,where $\breve E$ is the bundle over $\mathbb H$ corresponding to $E$. Moreover,
from the invertibility of $A_{\hat p}[g]$, $\Pi_{\hat p}[g]$
admits an inverse $\Pi_{\hat p}^{-1}[g]$ that reverses this process.
The bundle automorphisms $\Pi_{\hat p}[g]$ satisfy estimates analogous
to those of Lemma \ref{lem:local-compare-operators}.
\begin{lemma}
\label{lem:Pi-properties}
Suppose $g$ is an asymptotically hyperbolic metric of class \eqref{intro:H-class} or \eqref{intro:X-class}, and let $\alpha$ be as in \eqref{intro:alpha-exists}.
Suppose $E$ is a geometric tensor bundle over $M$.
For the same constant $r_0'$ as in Lemma \ref{lem:local-compare-operators},
for all $\hat p\in\partial M$, and for all $0<r\le r_0'$,
if $\Phi_i$ is a M\"obius parametrization with $\Phi_i(B_1^\Hyp)$ intersecting
$Z_r(\hat p)$, then the pullbacks $\Phi_i^*\Pi_{\hat p}[g]$ and $\Phi_i^*\Pi_{\hat p}[g]^{-1}$
are uniformly bounded in $H^{s,p}(B_1^\Hyp)$ independent of $i$.
Furthermore we have
\begin{align*}
\|\Phi_i^*\Pi_{\hat p}[g]-\Id\|_{H^{s,p}(B_2^\Hyp)}&\lesssim r^\alpha,
\\
\|\Phi_i^*\Pi_{\hat p}^{-1}[g]-\Id\|_{H^{s,p}(B_2^\Hyp)}&\lesssim r^\alpha.
\end{align*}
The implicit constants are independent of $i$, $\hat p$, and $r$.
\end{lemma}
\begin{proof}
From the construction of $\Pi_{\hat p}[g]$ it suffices to show that
the pullbacks $\Phi_i^* A_{\hat p}[g]$ and $\Phi_i^* A_{\hat p}^{-1}[g]$
are uniformly bounded (independent of $i$, $\hat p$ and $r$) in $H^{s,p}(B_2^\mathbb H)$ and that
\begin{equation}
\begin{aligned}\label{eq:A-is-close}
\| \Phi_i^* A_{\hat p}[g] - \Id\|_{H^{s,p}(B_2)} \lesssim r^\alpha\\
\| \Phi_i^* A^{-1}_{\hat p}[g] - \Id\|_{H^{s,p}(B_2)} \lesssim r^\alpha.
\end{aligned}
\end{equation}
To obtain the estimates \eqref{eq:A-is-close} let $\Theta$ be the boundary coordinates
associated with $\Phi_i$ and for $1\le k\le n$
let $\omega^k=\omega^k_{j} dx^j = \Phi_i^*(\rho^{-1}d\tilde\Theta^k) = y^{-1}d(\tilde\Theta^k\circ \Theta^{-1})$. Given a metric $h$ on $B_2^\Hyp$ let $\{(C_{(i)}[h])_j^k dx^j\}$ be the
orthonormal coframe obtained from the Gram-Schmidt algorithm
applied to $h$ and the coframe $\{\omega^k\}$.

Using uniform $C^k$ bounds for the transition Jacobians $D(\tilde \Theta\circ\Theta^{-1})$ and
and their inverses to control the frame $\{\omega^k\}$ independent of $i$,
Lemma \ref{lem:GS} implies
$h\mapsto C_{(i)}[h]$ is Lipshitz $X^{s,p}(B_2^\Hyp)\to X^{s,p}(B_2^\Hyp)$ with
a Lipschitz constant independent of $i$ so long as $h$ is constrained to lie in a
bounded region of $X^{s,p}(B_2^\Hyp)$ with determinant uniformly bounded below.
Arguing as in Lemma \ref{lem:geometric-op-mapping} we can find such bounds
for $\Phi_i^* g$ and the analogous bounds for $\Phi_i^*\tilde \Theta^*\breve g$
follow easily from estimates for transition Jacobians.
Inequalities \eqref{eq:A-is-close} now follow from Lemma \ref{lem:local-operator-compare}\eqref{part:g-is-close} and the observation that
\begin{gather*}
C_{(k)}[\Phi_i^* g] = \Phi_k^*A[g], \\
C_{(k)}[\Phi_k^*\tilde\Theta^*\breve g]= \Id.
\qedhere
\end{gather*}
\end{proof}

We are now ready to compare operators in a neighborhood of a point along the boundary,
and we use cutoff functions from the following lemma to perform the localization.
\begin{lemma}
\label{lem:bdy-cutoff}  Let $\hat p\in \partial M$, let $0<r\le r_0/2$.
There exists a nonnegative function $\eta_{\hat p,r}\in C^\infty(\bar M;\mathbb R)$ that equals 1 on $Z_{r}(\hat p)$
and has support in $M$ contained in $Z_{2r}(\hat p)$.
Moreover, $\eta_{\hat p,r}$ can be chosen such that for each $\ell\in\Nats_{\ge 0}$ there is $K>0$ independent of $\hat p$ and $r$, but depending on the parameter $a$,
such that  $\|\eta_{\hat p,r}\|_{C^{\ell}_0(M;\mathbb R)}\le K$.
\end{lemma}
\begin{proof}
Let $\psi:\bar \Hyp\to [0,1]$ be a smooth function that vanishes outside $\bar {Y_{1}}$
and equals 1 on $ Y_{1/2}$ and define
cutoff function $\eta_{\hat p,r}$ on $M$,
supported in $Z_{r}(\hat p)$, by $\eta_{\hat p, r}(\theta,\rho) = \psi(\theta/(2r), \rho/(2r))$.
Since the transition functions between background coordinate charts are uniformly bounded, it is straightforward to show that for any M\"obius parametrization $\Phi_i$ we have $\|\Phi_i^*\eta_{\hat p, r}\|_{C^{l}(B_2^\mathbb H)}$ is bounded independent of $r$ and $\hat p$, which yields the desired bound.
\end{proof}

Suppose $r\le r_0'$ from Lemma \ref{lem:local-operator-compare} and additionally that
$r\le r_0/2$ so that the cutoff functions $\eta_{\hat p, r}$ from Lemma \ref{lem:bdy-cutoff}
are well-defined. We define localized operators $\mathcal P_{\hat p,r}$ and $\breve P_{\hat p, r}$
of scale $r$ at the point $\hat p$ by
\begin{equation*}
\mathcal P_{\hat p, r}u = \mathcal P[g](\eta_{\hat p, r} u)
\end{equation*}
and
\begin{equation}
\label{define-breve-P-loc}
\begin{aligned}
\breve{\mathcal P}_{\hat p, r}u
&= \Pi_{\hat p}^{-1}[g]\cdot\mathcal P[\tilde\Theta_{\hat p}^*\breve g] \left(\Pi_{\hat p}[ g]\cdot(\eta_{\hat p, r}u)\right).
\end{aligned}
\end{equation}
The bundle automorphism tensor $\Pi_{\hat p}[g]$ gives rise to the map
\begin{equation}\label{eq:Xi-def}
u\mapsto \Xi_{\hat p} u := (\tilde\Theta_{\hat p})_* \Pi_{\hat p}[g] u,
\end{equation}
which is a local bundle isomorphism $E\big|_{Z_{\hat r}(\hat p)} \to \breve E\big|_{Y_{\hat r}}$, and we can alternatively write
\begin{equation}
\label{alt-breve-P-loc}
\breve{\mathcal P}_{\hat p, r}u
= \Xi_{\hat p}^{-1}\, \mathcal P[\breve g]\, \Xi_{\hat p} (\eta_{\hat p, r}u).
\end{equation}

The two localized operators are comparable in the following sense.
\begin{lemma}
\label{lem:local-operator-compare}
Suppose $g$ is an asymptotically hyperbolic metric of class \eqref{intro:H-class} or \eqref{intro:X-class}, and let $\alpha$ be as in \eqref{intro:alpha-exists}.
Suppose $\mathcal P[g]$ acts on sections of bundle $E$ and satisfies Assumption \ref{Assume-P}.
For any $\hat p\in \partial M$, $(\sigma, q)\in \mathcal S^{s,p}_d$, and $\delta\in \mathbb R$ the operator norms of
\begin{align}
\label{local-H-mapping}
\breve{\mathcal P}_{\hat p,r} - \mathcal P_{\hat p, r}
& \colon H^{\sigma, q}_\delta(M;E) \to H^{\sigma-d,q}_\delta(M;E),
\\
\label{local-X-mapping}
\breve{\mathcal P}_{\hat p,r} - \mathcal P_{\hat p, r}
& \colon X^{\sigma, q}_\delta(M;E) \to X^{\sigma-d,q}_\delta(M;E)
\end{align}
satisfy
\begin{equation*}
\| \breve{\mathcal P}_{\hat p,r} - \mathcal P_{\hat p, r} \|
\lesssim r^\alpha.
\end{equation*}
The implicit constants are independent of $\hat p$ and $r\le\min(r_0',r_0/2)$,
where $r_0'$ is the constant from Lemma \ref{lem:local-compare-operators}.
\end{lemma}

\begin{proof}
Consider the case \eqref{local-H-mapping} and write
\begin{multline}
\label{expand-local-difference}
\breve{\mathcal P}_{\hat p,r}u - \mathcal P_{\hat p, r}u
=
\left( \Pi_{\hat p}[ g]^{-1}-\Id\right)
\cdot\mathcal P[\tilde\Theta_{\hat p}^*\breve g] \left(\Pi_{\hat p}[ g]\cdot(\eta_{\hat p, r}u)\right)
\\
+
\mathcal P[\tilde\Theta_{\hat p}^*\breve g]\left((\Pi_{\hat p}[ g]-\Id)\cdot(\eta_{\hat p, r}u)\right)
+
\left(\mathcal P[\tilde\Theta_{\hat p}^*\breve g]-\mathcal P[g]\right)(\eta_{\hat p, r}u),
\end{multline}
for $u\in H^{\sigma, q}_\delta(M)$.

Suppose $\Phi_i$ is a M\"obius parametrization such that $\Phi_i(B_1^\Hyp)$
intersects the support of $\eta_{\hat p,r}$.  The estimates of
Lemmas \ref{lem:local-compare-operators} and \ref{lem:Pi-properties}
together with the
local versions of the product estimates in Proposition \ref{prop:multiplication}
then imply
\begin{align*}
\|\Phi_i^*(\breve{\mathcal P}_{\hat p,r}u - \mathcal P_{\hat p, r}u)\|_{H^{s,p}(B_1^\Hyp)}
&\lesssim r^\alpha \|\Phi_i^* (\eta_{\hat p,r} u)\|_{H^{\sigma,q}(B_1^\Hyp)}
\\
&\lesssim r^\alpha \|\eta_{\hat p,r}\|_{X^{s,p}(M;\Reals)}\|\Phi_i^*  u\|_{H^{\sigma,q}(B_1^\Hyp)}
\end{align*}
with implicit constant independent of $i$, $\hat p$, and $r\le \min(r_0',r_0/2)$.
The proof is completed by summing $q^{\rm th}$ powers, noting that
$\|\eta_{\hat p,r}\|_{X^{s,p}(M)}$ is bounded independent of $r$ and $\hat p$.
The case \eqref{local-X-mapping} is proved similarly replacing the sum with a supremum.
\end{proof}

\subsection{Interpolations of \texorpdfstring{$\mathcal P$}{𝒫} at infinity}\label{secsec:interp-P}
The semi-Fredholm estimates of \S\ref{secsec:semi-fred} and the
parametrix construction of \S\ref{sec:fredholm} rely on
approximating $\mathcal P$ on a collar neighborhood $A_r$ of $\partial M$
with a collection of
`hyperbolic space' operators $\breve{\mathcal P}_{\hat p_j,r}$
localized at points $\hat p_j\in\partial M$, which are then
stitched together with a suitable partition of unity.
We call this approximating
collection an interpolation of $\mathcal P$ at infinity; see
Definition \ref{def:interp-of-P} below.

As $r\to 0$, the number of points $p_j$ associated with the interpolation
increases without bound. For the applications of \S\ref{sec:fredholm}, we
need to mitigate against this and ensure that every M\"obius
parametrization $\Phi_i$ has an image that intersects a bounded number of
regions on which some $\breve{\mathcal P}_{\hat p_j,r}$ acts non-trivialy,
with a fixed bound as $r\to 0$.  The following two
technical results are the tools needed to obtain the bound and 
address a gap in the parametrix construction of \cite{Lee-FredholmOperators}.

The first step is to show that if the image of a M\"obius
parametrization intersects some $Z_{r}(\hat p)$, then it is
contained in a larger neighborhood of infinity
with a size that scales naturally in $r$.
\begin{lemma}\label{lem:kappa} There exists $\kappa\in (0,1)$ such that for all
$r\le r_0$, all $\hat p\in\partial M$, and all
M\"obius parametrizations $\Phi_i$,
if $\Phi_i(B_1^\Hyp)\cap Z_{\kappa r}(\hat p)\neq \emptyset$ then
$\Phi_i(B_1^\Hyp)\subset Z_{r}(\hat p)$.
\end{lemma}
\begin{proof}
A compactness argument shows that there
exists a constant $C\ge 1$ such that if $\theta$ and $\tilde \theta$ two boundary
charts with $\hat a$ and $\hat b$ in the domains of both, then
$|\theta(\hat a)-\theta(\hat b)|_{\Reals^{n-1}}\le C|\tilde\theta(\hat a)-\tilde \theta(\hat b)|_{\Reals^{n-1}}$.  Let $\kappa=(1+100C)^{-1}/2$.

Suppose for some $r\le r_0$ and M\"obius parametrization
$\Phi_i$ that $\Phi_i(B_1^\mathbb{H})\cap Z_{\kappa r}(\hat p)\neq \emptyset$.
We will show that $\Phi_i(B_1^\mathbb{H})\cap Z_{r}(\hat p)$ is both open
and closed in $\Phi_i(B_1^\mathbb{H})$, which proves the desired containment.  The
intersection is clearly open, and to show that it is closed it
suffices to show $\Phi_i(B_1^\mathbb{H})\cap Z_{r}(\hat p)\subset Z_{r/2}(\hat p)$.
Let $a\in \Phi_i(B_1^\mathbb{H})\cap Z_{\kappa r}(\hat p)$ and let $b\in \Phi_i(B_1^\mathbb{H})\cap Z_{r}(\hat p)$.

Since $a,b\in \Phi_i(B_1^\mathbb{H})$ and since $a\in Z_{\kappa r}(\hat p)$, we conclude
$\rho(b)< \exp(2) \rho(a) < 9 \kappa r < r/2$.  So
to show $b\in Z_{r/2}(\hat p)$ we need only demonstrate
$|\theta(\hat p)-\theta\circ \Pi_\partial(b)|_{\Reals^{n-1}}<r/2$
where $\theta$ is the boundary chart associated with $\hat p$.  Let $\tilde
\theta$ be the boundary chart associated with $\Phi_i$.  Then
\begin{align*}
|\theta(\hat p)-\theta\circ \Pi_\partial(b)|_{\Reals^{n-1}}
&\le |\theta(\hat p)-\theta\circ \Pi_\partial(a)|_{\Reals^{n-1}}+
|\theta\circ \Pi_\partial(a)-\theta\circ \Pi_\partial(b)|_{\Reals^{n-1}}\\
&< \kappa r + C|\tilde \theta\circ \Pi_\partial(a)-\tilde \theta\circ \Pi_\partial(b)|_{\Reals^{n-1}}\\
&<\kappa r +  2C\sinh(1) \rho(p_i)\\
& < \kappa r + 10 C \kappa r < r/2.
\end{align*}
where we have used the fact that $\rho(a)<\kappa r$ implies $\rho(p_i)< 3\kappa r$.
\end{proof}

The following variation of Lemma 2.2 of \cite{Lee-FredholmOperators} ensures that
for any choice of $r\le r_0$ we can find a uniform (in $r$) locally finite family
of neighborhoods at infinity at scale $r$ that cover a collar neighborhood
at the same scale. We will subsequently see that a suitable
choice of the parameter $a$ appearing in it will, in combination with Lemma \ref{lem:kappa},
ensure that uniform local finiteness
for the neighborhoods at infinity extends to the images of M\"obius parametrizations
that intersect the collar neighborhood.

\begin{lemma}\label{lem:boundary-loc-finite}
Suppose $0<a<1$.
For any $0<r\le r_0$ there exist finitely many
points $\hat p_j\in \partial M$
such that:
\begin{enumerate}
\item the sets $\{Z_{ar}(\hat p_j)\}$ cover $A_{ar}$, and
\item the neighborhoods $\{Z_{r}(\hat p_j)\}$
are uniformly locally finite independent of $r$.
\end{enumerate}
\end{lemma}
\begin{proof}
Recall the smooth reference metric $\hat g$ on $\bar M$ and let $\gamma =\hat g|_{\partial M}$.
Given $\hat p\in \partial M$,
let $\theta$ be the boundary coordinates associated with
$Z_{r_0}(\hat p)$; see \S\ref{sec:coords}.
Let $C_r(\hat p)$
be the $\theta$ coordinate ball of radius $r\le r_0$ centered at $\hat p$.
Similarly, let $B_r^\gamma(\hat p)$ be the $\gamma$-geodesic ball in
centered at $\hat p$ of radius $r$.

Since the geometry of $\gamma$ is uniformly close to Euclidean in any
of the finitely many boundary coordinate systems
there exist constants $m$ and $M$ such that for any $r\leq r_0$ we have
\begin{equation}
\label{vol-Z-compare}
B_{mr}^\gamma(\hat p) \subset C_r(\hat p) \subset B_{Mr}^\gamma(\hat p).
\end{equation}
Since $\partial M$ is compact, standard volume
comparison theorems imply
there exists constants $C_1, C_2$, independent of
$\hat p\in \partial M$ and  $r\le r_0$, such that
\begin{equation}
\begin{gathered}
\label{vol-B-compare}
C_1 r^{n-1} \leq \Vol B_{mr}(\hat p),\\
\Vol B_{2Mr} \le C_2 r^{n-1}.
\end{gathered}
\end{equation}

Fixing $r\leq r_0$, choose a maximal disjoint collection $\{B_{m ar/2}(\hat p_j)\}_{j=1}^J$; a volume comparison argument shows the collection is finite.
The maximality, by means of the triangle inequality, implies that $\{B_{mar}(\hat p_j)\}_{j=1}^J$ covers $\partial M$.
The inclusion \eqref{vol-Z-compare}, together with the
assumed product structure of coordinates near infinity
shows that $\{Z_{ar}(\hat p_j)\}_{j=1}^J$
covers $A_{ar}$.  The uniform (in $r$)
estimate of the number of sets $Z_{r}(\hat p_j)$ that can intersect
follows from the relations \eqref{vol-B-compare}, their common scaling
in $r$, and the argument of \cite{Lee-FredholmOperators} Lemma 2.2.
\end{proof}

With the two previous technical results in hand, we can now construct the
requisite partitions of unity used to stitch together localized operators
at infinity. The key point is item \eqref{part:why-we-did-the-labor}.
\begin{lemma}\label{lem:boundary-pou}
There exists $r_0''\le r_0$ such that
for each $r\le r_0''$ there exist boundary points $\{\hat p_j\}_{j=1}^J$
and smooth functions
$\varphi_0,\varphi_1,\ldots, \varphi_J$ on $\bar M$ such that:
\begin{enumerate}
\item The support of $\varphi_0$ in $M$ is contained in $M\setminus A_{r/4}$.
\item For $1\le j\le J$, the support in $M$ of $\varphi_j$ is contained in $Z_{r}(\hat p_j)$.
\item The functions $\psi_j=\varphi_j^2$ form a partition of unity on $\bar M$.
\item For any choice of $\ell\in\Nats_{\ge 0}$,  the functions
$\varphi_j$
are uniformly bounded in $C^{\ell,0}_0(M;\mathbb R)$ independent of $r$.
\item There is a bound $K$, independent of $r$, such that for each $j\ge 1$
there are at most $K$ indices $k\ge 1$ such that the support of $\varphi_k$
overlaps that of $\varphi_j$.
\item\label{part:why-we-did-the-labor} Given a M\"obius parametrization $\Phi_i$,
there are at most $K$ indices $j\ge 1$ such that $\Phi_i(B_2^\mathbb{H})$
intersects $Z_{r}(\hat p_j)$. Consequently, there are at most $K+1$
functions $\varphi_j$, $j\ge 0$, with support intersecting the image of $\Phi_i$.
\end{enumerate}
\end{lemma}
\begin{proof}
Let $r_0''=\kappa r_0$ where $\kappa$ is the parameter from Lemma \ref{lem:kappa}
and suppose $r\le r_0''$.
Applying Lemma \ref{lem:boundary-loc-finite} with its parameter $a=\kappa/2$
implies we can find boundary points $\hat p_j$
such that the sets $Z_{r/2}(\hat p_j)$ cover $A_{r/2}$ and such that
the sets $Z_{\kappa^{-1}r}(\hat p_j)$ are uniformly locally finite:
there exists $K>0$ such that
each of these sets intersects no more than $K$
sets in the collection, with $K$
independent of $r$.

From Lemma \ref{lem:bdy-cutoff}, for each $j$ we can choose a
smooth cutoff function $\chi_j$ on $\bar M$ that equals 1 on $Z_{r/2}(\hat p_j)$
and which has a restriction to $M$ with support contained in
$Z_{r}(\hat p_j)$. Moreover,
we can ensure that for any choice of $\ell$ the $C^{\ell}_0$
norms of the functions $\chi_j$ are bounded independent of
$r$ and the choice of points $\hat p_j$.
Choose a smooth
function $\chi:[0,\infty)\to [0,\infty)$
such that $\chi=0$ on $[0,1/4]$ and $\chi=1$ on $[1/2,\infty)$
and define $\chi_{0}=\chi\circ(\rho/r)$, which is smooth on $\bar M$.
A scaling argument shows that $\chi_{0}$ is uniformly
bounded in $C^{\ell}_0(M)$ for each $\ell\in \Nats_{\ge 0}$
independent of $r\le r_0$.

By construction, for $j\ge 1$ the function $\chi_j$ is supported in $M$
on $Z_{r}(\hat p_j)\subset Z_{\kappa^{-1}r}(\hat p_j)$ and hence its support
can overlap with that of at most $K$ functions $\chi_k$ with $k\ge 1$.
Taking $\chi_0$ into account, its support can therefore overlap
at most $K+1$ functions $\chi_k$.
Moreover, if $\Phi_i$ is a M\"obius parametrization,
and if for some $j\ge 1$ the function $\chi_j$ has support intersecting $\Phi_i(B_1^\mathbb{H})$,
Lemma \ref{lem:kappa} implies $\Phi_i(B_1^\mathbb{H})\subset Z_{\kappa^{-1}r}(\hat p_j)$ and
it follows from the uniform local finiteness of the sets
$Z_{\kappa^{-1}r}(\hat p_i)$ that $\Phi_i(B_1^\mathbb{H})$ can intersect at most $K$ total of these.

The function $\sum_{j=0}^J \chi_{j}^2\ge 1$ and
uniform local finiteness
implies it is bounded in $C^{\ell}_0(M)$ for each $\ell\in \Nats_{\ge 0}$
independent of $r$ and the choice of points $\hat p_j$.
Hence, defining
\[
\varphi_{i} = \frac{\chi_{i}} {\sqrt{\sum_{j=0}^J \chi_{j}^2}},
\]
the functions $\psi_i=\varphi_i^2$
form a partition of unity of $\bar M$ that satisfy the
desired $C^{\ell}_0$ bounds.
\end{proof}

\begin{definition}\label{def:interp-of-P}
Let $g$ be an asymptotically hyperbolic metric of class \eqref{intro:H-class} or \eqref{intro:X-class} and let $\mathcal P[g]$ be an operator satisfying Assumption \ref{Assume-P}.
Recall the constants $r_0'$ and $r_0''$ from Lemmas \eqref{lem:local-compare-operators}
and \ref{lem:boundary-pou}, respectively.
For positive $r\le \min(r_0/2,r_0',r_0'')$
apply Lemma \ref{lem:boundary-pou} to construct a partition of unity
$\psi_j = \varphi_j^2$, with $0\le j\le J$, for this choice of $r$
and associated boundary points $\hat p_j$.

An \Defn{interpolation of $\mathcal P[g]$ at infinity with scale $r$} is comprised of the corresponding collection of bundle automorphism tensors $\Xi_j = \Xi_{\hat p_j}$, defined for $1\le j\le J$ by \eqref{eq:Xi-def}; 
localized approximation operators 
$\breve{ \mathcal P}_j:=\breve{\mathcal P}_{\hat p_j, r}$, defined for $1\le j\le J$ by \eqref{alt-breve-P-loc};
and the partition of unity $\{\varphi_j\}_{j=0}^J$.
\end{definition}

\subsection{A priori estimates}\label{secsec:semi-fred}
We are now able to prove a-priori estimates for elliptic geometric operators arising from asymptotically hyperbolic metrics of class \eqref{intro:H-class} and \eqref{intro:X-class}.
Our estimates rely on the isomorphism property of $\mathcal P[g]$, recorded in Assumption \ref{Assume-I} of the introduction.
This assumption allows us to invoke Theorems \ref{thm:ball-H-iso} and \ref{thm:ball-X-iso}, which establish isomorphism properties of $\mathcal P[g_\mathbb{B}]$ on the ball model $(\mathbb B, g_\mathbb{B})$ of hyperbolic space.
In our application here, and subsequently in \S\ref{sec:fredholm}, we use the fact that the isomorphism results on the ball model give rise to corresponding isomorphism theorems for $\mathcal P[\breve g]$ on the half space model $(\mathbb H, \breve g)$.
To do so, we choose an isometry $\mathcal I:\Hyp\to\mathbb B$ and simply define
norms on $\Hyp$ in terms of the norms on $\mathbb B$ defined in 
Appendix \ref{app:hyp-elliptic}:
\[
\|u\|_{H^{\sigma,q}_{\delta }(\Hyp;\breve E)} := \|\mathcal I_* u\|
_{H^{\sigma,q}_{\delta}(\mathbb B;\,\mathcal{I}_*\breve E)
}
\]
for any $1<q<\infty$ and $s,\delta\in\Reals$.  It is easy to see that the norms defined this way
are equivalent for any choice of $I$.  The following lemma allows us to transfer between sections
of geometric bundles on $M$ and bundles on $\Hyp$.

\begin{lemma}
\label{lem:xi-norm-equiv}
Suppose $g$ is an asymptotically hyperbolic metric of class \eqref{intro:H-class} or \eqref{intro:X-class}, and suppose $\hat p\in \partial M$.
Let $E$ is a geometric tensor bundle over $M$ and let $\breve E$ be the corresponding bundle over $\mathbb H$.
Let $(\sigma, q)\in \mathcal S^{s,p}_0$ and $\delta\in \mathbb R$.
Recalling the cutoff functions $\eta_{\hat p,r}$ from Lemma \ref{lem:bdy-cutoff},
for each $0< r\leq \bar r$ from Lemma \ref{lem:local-operator-compare}
we have the norm equivalence
\begin{equation*}
\| \Xi_{\hat p} (\eta_{\hat p, r}u) \|_{H^{\sigma, q}_\delta(\mathbb H;\breve E)}
\sim
\| \eta_{\hat p, r} u \|_{H^{\sigma, q}_\delta(M; E)}
\end{equation*}
for all $u\in H^{\sigma, q}_\delta(M;E)$, with the implicit constants
independent of $\hat p$ and $r$.
\end{lemma}

\begin{proof}
In view of the estimates of Lemma \ref{lem:Pi-properties}, it suffices to show that
\begin{equation*}
\| \tilde\Theta_{\hat p}^* (\eta_{\hat p, r}u) \|_{H^{\sigma, q}_\delta(\mathbb H;\breve F)}
\sim
\| \eta_{\hat p, r} u \|_{H^{\sigma, q}_\delta(M; F)}
\end{equation*}
for any tensor bundle $F$. The proof of this
is a straightforward variation of the proof of Lemma \ref{lem:X-other-charts},
using the following facts:
\begin{itemize}
  \item We can pick $\bar r$ sufficiently large so that 
  $\tilde\Theta_{\hat p}(Z_{\hat p}(r_0)\subset Y_{\bar r}$
  for any $\hat p\in \partial M$.
  \item  Writing $\rho_{\mathbb B}$ for the defining function on $\mathbb B$,
  $(\mathcal I^*\rho_{\mathbb B})/y$ is bounded above and below
  on $Y_{\bar r}$
  \item On $Y_{\bar r}$, the Jacobians of $\Theta^{\mathbb B}\circ \mathcal I$
  and its inverse admit uniform $C^k$ bounds for any boundary chart $\Theta^{\mathbb B}$
  on $\mathbb B$.
  \item For any boundary chart $\Theta$ on $M$,
  the Jacobians of $\tilde \Theta_{\hat p}\circ \Theta^{-1}$ 
  and its inverse admit uniform $C^k$ bounds independent of $\hat p$.
\end{itemize}
See also the proof of Lemma \ref{lem:alt-B-norms}.
\end{proof}

We now establish the a priori estimates that imply $\mathcal P[g]$ is semi-Fredholm.
Recall that Assumption \ref{Assume-I} implies the positive of the indicial radius $R$ of $\mathcal P[\breve g]$; see \cite{Lee-FredholmOperators}.
Recall also, from \eqref{intro:define-DqR}, the notation 
\begin{equation*}
D_q(R) = \left\{ \delta\in \mathbb R\colon \left| \delta - \frac{n-1}{q} + \frac{n-1}{2}\right|< R \right\}.
\end{equation*}

\begin{proposition}\label{prop:semi-fred-estimate}
Suppose $g$ be an asymptotically hyperbolic metric of class \eqref{intro:H-class} or \eqref{intro:X-class}.
Let $\mathcal P[g]$ be an operator on sections of geometric tensor bundle $E$ that satisfies Assumptions \ref{Assume-P} and \ref{Assume-I}.
Let $R>0$ be
the indicial radius of the corresponding operator $\breve{\mathcal P}$
on hyperbolic space.
\begin{enumerate}
\item If $(\sigma,q)\in\mathcal S^{s,p}_d$ and if
$\delta \in D_q(R)$
then for all $u\in H^{\sigma,q}_{\delta}(M;E)$
\begin{equation}\label{eq:semifredholm-H}
\|u\|_{H^{\sigma,q}_{\delta}(M;E)} \lesssim
\|\mathcal P u\|_{H^{\sigma-d,q}_{\delta}(M;E)} + \|u\|_{H^{\sigma-1,q}_{\delta-1}(M;E)}.
\end{equation}
\item If $(\sigma,q)\in\mathcal S^{s,p}_d$ and if
$\delta\in D_\infty(R)$
then for all $u\in X^{\sigma,q}_{\delta}(M;E)$
\begin{equation}\label{eq:semifredholm-X}
\|u\|_{X^{\sigma,q}_{\delta}(M;E)} \lesssim
\|\mathcal P u\|_{X^{\sigma-d,q}_{\delta}(M;E)} + \|u\|_{X^{\sigma-1,q}_{\delta-1}(M;E)}.
\end{equation}
\end{enumerate}
Consequently, $\mathcal P$ thus has finite-dimensional kernel and closed range (i.e.~is semi-Fredholm)
as a map $H^{\sigma,p}_\delta(M;E)\to H^{\sigma-d,p}_\delta(M;E)$ or
$X^{\sigma,p}_\delta(M;E)\to X^{\sigma-d,p}_\delta(M;E)$ under the stated restrictions on the parameters $\sigma$, $p$ and $\delta$.
\end{proposition}

\begin{proof}
We first address the Sobolev case.
Let $\{\breve{\mathcal P}_j\}$ be an interpolation of $\mathcal P$
at infinity with scale $r$ to be chosen sufficiently small below
and let $\{\psi_j\}$ be the associated partition of unity.

Consider some $j\ge 1$; let $\tilde\Theta_j= \tilde\Theta_{\hat p_j}$ and $\Xi_j = \Xi_{\hat p_j}$ be the corresponding adapted coordinates and local bundle isomorphism.
Since $\delta\in \mathcal D_q(R)$
and since the cutoff function $\eta_{\hat p_j,r}$ appearing
in equation \eqref{define-breve-P-loc} satisfies
$\eta_{\hat p,j}\psi_j=\psi_j$
we can apply Lemma \ref{lem:xi-norm-equiv} and Theorem \ref{thm:ball-H-iso} to obtain
\begin{equation}
\label{ap-hyp-P-iso}
\begin{aligned}
\| \psi_j u\|_{H^{\sigma, q}_\delta(M)}
&\lesssim \| \Xi_j(\psi_j u) \|_{H^{\sigma, q}_\delta(\mathbb H)}
\\
&\lesssim \| \mathcal P[\breve g]\Xi_j(\psi_j u) \|_{H^{\sigma-d, q}_\delta(\mathbb H)}
\\
&\lesssim \| \Xi_j^{-1}\mathcal P[\breve g]\Xi_j(\psi_j u) \|_{H^{\sigma-d, q}_\delta(M)}
\\
&= \| \breve{\mathcal P}_j (\psi_ju) \|_{H^{\sigma-d, q}_\delta(M)}.
\end{aligned}
\end{equation}
Again using the identity $\eta_{\hat p_j,r}\psi_j = \psi_j$, we have
$\mathcal P_j u = \mathcal P[g](\psi_ju)$ and hence
the triangle inequality implies that the right side of \eqref{ap-hyp-P-iso} is bounded by
\begin{equation*}
\| \psi_j \mathcal P[g]u\|_{H^{\sigma-d, q}_\delta(M)}
+
\| [\mathcal P[g], \psi_j]u\|_{H^{\sigma-d, q}_\delta(M)}
+
\| (\breve{\mathcal P}_j - \mathcal P_j)(\psi_j u)\|_{H^{\sigma-d, q}_\delta(M)}.
\end{equation*}
Moreover, Lemma \ref{lem:local-operator-compare} implies
\begin{equation*}
\| (\breve{\mathcal P}_j - \mathcal P_j)(\psi_j u)\|_{H^{\sigma-d, q}_\delta(M)}
\lesssim r^\alpha \| \psi_j u\|_{H^{\sigma, q}_\delta(M)}.
\end{equation*}
All of the implicit constants thus far are independent of $\hat p_j$ and $r$,
so at this point we can select $r$ sufficiently small so that for all $j$
this term can be absorbed into the left side of \eqref{ap-hyp-P-iso}.
Lemma \ref{lem:commutator} implies that
\begin{align*}
\| [\mathcal P[g], \psi_j]u\|_{H^{\sigma-d, q}_\delta(M)}
&\lesssim
\|\rho^{-1} [\mathcal P[g], \psi_j]u\|_{H^{\sigma-d, q}_{\delta-1}(M)}
\\
&\lesssim
\|\rho^{-1} [\mathcal P[g], \psi_j]u\|_{H^{\sigma-d, q}_{\delta-1}(M)}
\\
&\lesssim 
\| u\|_{H^{\sigma-1, q}_{\delta-1}(M)}
\end{align*}
and combining the inequalities thus far we find that for each $j\geq 1$ that
\begin{equation*}
\| \psi_j u\|_{H^{\sigma, q}_\delta(M)}
\lesssim
\|\mathcal P[g]u\|_{H^{\sigma-d, q}_\delta(M)}
+
\| u\|_{H^{\sigma-1, q}_{\delta-1}(M)}.
\end{equation*}

In order to estimate $\psi_0 u$ we pick precompact 
open sets $V$ and $U$ with smooth boundaries such that $\supp \psi_0\subset V\subset\bar V\subset U$.  
Since $(\sigma,q)\in \mathcal S^{s,p}_{d}$, 
Sobolev embedding on $U$ implies $H^{\sigma-1-d, q}(U;E)\hookrightarrow H^{d-s-1,p\dual}(U;E)$.
Hence Proposition \ref{prop:interior-reg} and straightforward arguments
relating weighted spaces on $M$ and their restrictions to domains imply
\begin{equation*}
\begin{aligned}
\| \psi_0u \|_{H^{\sigma, q}_\delta(M)}
\lesssim \|u\|_{H^{\sigma,q}(V)}
&\lesssim \| \mathcal P[g] u\|_{H^{\sigma-d, q}(U)}
+ \| u\|_{H^{d-s-1, p\dual}(U)}
\\
&\lesssim \| \mathcal P[g] u\|_{H^{\sigma-d, q}_\delta(M)}
+ \| u\|_{H^{\sigma -1, q}_{\delta-1}(M)}.
\end{aligned}
\end{equation*}
Estimate \eqref{eq:semifredholm-H} now follows by summing over the partition of unity.
The proof of \eqref{eq:semifredholm-X} follows from analogous logic, using Theorem \ref{thm:ball-X-iso} in place of Theorem
\ref{thm:ball-H-iso}.

The fact that $\mathcal P$ has finite dimensional kernel
follows from inequalities \eqref{eq:semifredholm-H}
and \eqref{eq:semifredholm-X} and the fact that the maps
$H^{\sigma,q}_{\delta}(M)\hookrightarrow H^{\sigma-1,q}_{\delta-1}(M)$
and
$X^{\sigma,q}_{\delta}(M)\hookrightarrow X^{\sigma-1,q}_{\delta-1}(M)$
are compact by Proposition \ref{prop:rellich-lemma}.
The closedness of the image follows from this; see, for example, \cite[Lemma 4.10]{Lee-FredholmOperators}.
\end{proof}

\section{Fredholm theory}\label{sec:fredholm}
In this section we establish the Fredholm results of Theorem \ref{thm:main-fredholm} for elliptic operators associated to asymptotically hyperbolic metrics.
To this end, let $g$ be a asymptotically hyperbolic metric of class \eqref{intro:H-class} or \eqref{intro:X-class}
Let $\mathcal P = \mathcal P[g]$ be a differential operator on sections of geometric tensor bundle $E$ that satisfies Assumption \ref{Assume-P}.
Recall that Lemma \ref{lem:S-members}, together with the assumptions on $g$, imply that the set of compatible Sobolev indices $\mathcal S^{s,p}_d$ is nonempty.

Let $R$ be the indicial radius of $\breve{\mathcal P} = \mathcal P[\breve g]$, the associated operator on hyperbolic space, which acts on sections of corresponding bundle $\breve E$.
Theorem \ref{thm:ball-H-iso} implies that $\breve{\mathcal P}\colon H^{\sigma,q}_{\delta}(\Hyp;\breve E)\to H^{\sigma-d,q}_{\delta}(\Hyp;\breve E)$ is an isomorphism for $\delta \in \mathcal D_q(R)$, while Theorem \ref{thm:ball-X-iso} indicates that $\breve{\mathcal P}\colon X^{\sigma,q}_{\delta}(\Hyp;\breve E)\to X^{\sigma-d,q}_{\delta}(\Hyp;\breve E)$
is an isomorphism if $\delta\in\deltrange_\infty(R)$.
Thus our primary goal is to show that $\mathcal P$ is Fredholm of index zero for those exact weight parameters $\delta$ for which $\breve{\mathcal P}$ is an isomorphism.

\subsection{Parametrix construction}
Adapting the strategy of \cite{Lee-FredholmOperators}, we establish the Fredholm properties of $\mathcal P$
using a parametrix, the operator $\widetilde Q_r$ defined below, based on
the inverse $\breve {\mathcal P}^{-1}$  applied to sections supported near $\partial M$.
To that end, let $\{\breve{\mathcal P}_j\}_{j=1}^J$ be an interpolation of $\mathcal P$ satisfying Definition \ref{def:interp-of-P} with associated bundle automorphism tensors $\{\Xi_j\}_{j=1}^J$ and cutoff functions $\{\varphi_j\}_{j=1}^J$.
For $1\leq j \leq J$, define operators $\mathcal P_j$ by $\mathcal P_ju = \mathcal P(\varphi_j u)$.
Finally, define the maps
\begin{equation*}
\begin{aligned}
Q_r u &= \sum_{j\geq 1} \varphi_j \Xi_j^{-1}\breve{\mathcal  P} ^{-1} \Xi_j(\varphi_j u),
\\
S_r u &= \sum_{j\geq 1} \varphi_j \Xi_j^{-1}\breve{\mathcal  P} ^{-1} \Xi_j((\mathcal P_j - \breve{\mathcal  P}_j) \varphi_j u),
\\
T_r u &= \sum_{j\geq 1} \varphi_j \Xi_j^{-1}\breve{\mathcal  P} ^{-1} \Xi_j([\varphi_j, \mathcal P] u)
\end{aligned}
\end{equation*}
for sufficiently regular, compactly supported sections $u$ of $E$.

\begin{proposition}\strut
\label{prop:basic-Q-S-R}
\begin{enumerate}
\item If $(\sigma,q)\in \mathcal S^{s,p}_{d}$ and
$\delta\in\deltrange_q(R)$ then $Q_r, S_r, T_r$ extend to bounded maps as follows:
\begin{equation*}
\begin{aligned}
& Q_r \colon H^{\sigma-d,q}_\delta (M;E) \to H^{\sigma,q}_\delta(M;E),
\\
& S_r \colon H^{\sigma,q}_\delta (M;E) \to H^{\sigma,q}_\delta(M;E),
\\
& T_r \colon H^{\sigma,q}_{\delta-1} (M;E) \to H^{\sigma+1,q}_{\delta}(M;E).
\end{aligned}
\end{equation*}
If $u$ is supported in $A_{r/4}$ then
\begin{equation}
\label{eq:Q_r-identity}
Q_r \mathcal Pu = u + S_ru + T_ru.
\end{equation}
Moreover
\begin{equation}
\label{eq:Sr-estimate}
\| S_r u\|_{H^{\sigma,q}_\delta(M)}
\lesssim r^\alpha  \| u \|_{H^{\sigma,q}_\delta(M)},
\end{equation}
where $\alpha\in(0,1]$ is the constant from \eqref{intro:alpha-exists}.

\item The same statement is true replacing
weighted Sobolev
spaces $H$ with Gicquaud-Sakovich spaces
$X$ having the same parameters, and replacing
the weight range $\mathcal D_q(R)$ with
$\mathcal D_\infty(R)$.

\end{enumerate}

\end{proposition}
\begin{proof}
We demonstrate the proof for weighted Sobolev spaces; the proof for Gicquad-Sakovich spaces is similar but easier.

The mapping property of $Q_r$  follows from the fact that
$\delta\in\deltrange_q(R)$ and therefore by Theorem \ref{thm:ball-H-iso} 
 $\breve{\mathcal P}\colon H^{\sigma,q}_\delta(\mathbb H) \to H^{\sigma-d,q}_\delta(\mathbb H)$ is an isomorphism.
This fact also implies the mapping property for $S_r$.
To verify the mapping property for $T_r$, note that Lemma \ref{lem:commutator} implies $[\varphi_j, \mathcal P]\colon H^{\sigma,q}_{\delta-1}(M) \to H^{\sigma-d+1,q}_{\delta}(M)$.

To verify the identity \eqref{eq:Q_r-identity} first note that, analogous to the proof of Proposition \ref{prop:semi-fred-estimate}, the cutoff functions $\eta_{\hat p_j, r}$ appearing in \eqref{alt-breve-P-loc} satisfy $\eta_{\hat p_j, r}\varphi_j= \varphi_j$.
The identity then follows by direct computation, noting that $\sum_{j\ge 1} \varphi^2_j=1$ on $A_{r/4}$.

It remains to establish the estimate \eqref{eq:Sr-estimate}.  
For each M\"obius parametrization
$\Phi_i$,  let $\mathcal N(i)$ be the set of indices $j$ such that the support
of $\varphi_j$ intersects the image of $\Phi_i$.
Using the uniform bound to the size of $\mathcal N(i)$
from  Lemma \ref{lem:boundary-pou}
we find
\[
\|\Phi_i^* S_r u\|_{H^{\sigma,q}(B_2^\Hyp)}^q \lesssim \sum_{j\in \mathcal N(i)}
\|\Phi_i^* S_{j,r} (\varphi_j u)\|_{H^{\sigma,q}(B_2^\Hyp)}^q, 
\]
where $S_{j,r}=
\varphi_j \Xi_j^{-1}\breve{\mathcal  P} ^{-1} \Xi_j (\mathcal P_j - \breve{\mathcal  P}_j)$.
After multiplying by $\rho_i^{-\delta q}$ and summing over $i$ we obtain
\[
\|S_r u\|_{H^{\sigma,q}_\delta(M)}^q \lesssim
\sum_{i} \sum_{j\in \mathcal N(i)}
\rho_i^{-\delta q} \|\Phi_i^* S_{j,r} (\varphi_j u)\|_{H^{\sigma,q}(B_2^\Hyp)}^q
\lesssim
\sum_{j} \|S_{j,r}(\varphi_j u)\|_{H^{\sigma,q}_{\delta}(M)}^q.
\]
Lemmas \ref{lem:xi-norm-equiv} and \ref{lem:local-operator-compare} imply
\begin{equation}\label{eq:Sr-1}
\|S_{j,r}(\varphi_j u)\|_{H^{\sigma,q}_\delta(M)}^q \lesssim r^{\alpha q}
\sum_{j} \|\varphi_j u\|_{H^{\sigma,q}_{\delta}(M)}^q.
\end{equation}
Let $\mathcal N^*(j)$ be the set of indices
$i$ such that the support of $\varphi_j$ intersects the image of $\Phi_i$.
Then
\begin{align*}
\sum_j \|\varphi_j u\|_{H^{\sigma,q}_{\delta}(M)}^q
&= \sum_{j}\sum_{i\in \mathcal N^*(j)}
\rho_i^{-\delta q} \|\Phi_i^* (\varphi_j u)\|_{H^{\sigma,q}(B_2^\Hyp)}^q 
\\
&= \sum_i \sum_{j\in\mathcal N(i)} \rho_i^{-\delta q}
\|\Phi_i^* (\varphi_j u)\|_{H^{\sigma,q}(B_2^\Hyp)}^q.
\end{align*}
From the uniform bound on the size of each $\mathcal N(i)$ and uniform control on
the smooth functions $\varphi_j$ we conclude
\begin{equation} \label{eq:Sr-2}
\sum_j \|\varphi_j u\|_{H^{\sigma,q}_{\delta}(M)}^q
\lesssim \|u\|_{H^{\sigma,q}_\delta(M)}^q,
\end{equation}
where, critically, the implicit constant is independent of the fineness $r$ and the associated
number of terms on the left-hand side of this inequality.
The desired estimate \eqref{eq:Sr-estimate} now follows
from inequalities \eqref{eq:Sr-1} and \eqref{eq:Sr-2}.
\end{proof}

Estimate \eqref{eq:Sr-estimate} implies that for a fixed $(\sigma,q)\in S^{s,p}_d$
the map
$\Id + S_r$ is invertible when $r$ is sufficiently small
and we define\begin{equation}
\widetilde Q_r = (\Id + S_r)^{-1}\circ Q_r
\quad\text{ and }\quad
\widetilde T_r = (\Id + S_r)^{-1}\circ T_r.
\end{equation}
Proposition \ref{prop:basic-Q-S-R} immediately implies the following.
\begin{corollary}\label{cor:parametrix}
\strut
\begin{enumerate}
\item Suppose $(\sigma,q)\in\mathcal S^{s,p}_{d}$
and $\delta,\delta'\in\deltrange_q(R)$ with $\delta'\le \delta+1$
and that $r$ is sufficiently small that $(\Id+S_r)^{-1}$
is well defined on $H^{\sigma,q}_\delta(M;E)$ and on $H^{\sigma,q}_{\delta'}(M;E)$.
The operators
\begin{equation}
\label{eq:parametrix-mapping}
\begin{aligned}
& \widetilde Q_r \colon H^{\sigma-d,q}_\delta(M;E) \to H^{\sigma,q}_\delta(M;E),
\\
& \widetilde T_r \colon H^{\sigma,q}_{\delta}(M;E) \to H^{\sigma+1,q}_{\delta'}(M;E)
\end{aligned}
\end{equation}
are bounded.  Moreover,
for all $u\in H^{\sigma,q}_\delta(M;E)$ having support in $A_{r/4}$
\begin{equation}
\label{eq:tildeQ-identity}
\widetilde Q_r  \mathcal P u = u + \widetilde T_r u.
\end{equation}

\item The same statement is true replacing
weighted Sobolev
spaces $H$ with Gicquaud-Sakovich spaces
$X$ having the same parameters, and replacing
the weight range $\mathcal D_q(R)$ with
$\mathcal D_\infty(R)$.

\end{enumerate}

\end{corollary}

\subsection{Fredholm theorem for weighted Sobolev spaces}

The argument that $\mathcal P$ is Fredholm on general spaces
relies on two key facts. First, $\mathcal P$ is Fredholm when acting on the $L^2$-based space $H^{d/2,2}_0(M;E)$, and second,
that its kernel is independent of the space that it is acting on.
This second fact relies on elliptic regularity up to the boundary
at infinity accounting for the weight $\delta$.  
As a first step toward the Fredholm theorem, we establish smoothing properties of $\widetilde T_r$.

\begin{lemma}\label{lem:T-smoothing}
Suppose
$u\in H^{\sigma_0,q_0}_{\delta_0}(M;E)$
	for some
	$(\sigma_0,q_0)\in\mathcal S^{s,p}_{d}$
and	$\delta_0\in \deltrange_{q_0}(R)$.
Given $(\sigma,q)$ in $\mathcal S^{s,p}_{d}$
and $\delta\in \deltrange_{q}(R)$,
	there exists $\ell\in\Nats$ such that
	$\widetilde T_r^\ell(u)\in H^{\sigma,q}_{\delta}(M;E)$
so long as $r$ is sufficiently small.  

The same statement is true,
for the same values of $r$, replacing
$\widetilde T_r$ with $\eta\widetilde T_r$ for any $\eta\in C^\infty(\bar M;\mathbb R)$.
\end{lemma}
\begin{proof}
We prove the result assuming
$\eta\equiv 1$; because multiplying by $\eta$ does not alter the mapping properties of
$\widetilde T_r$, there is no change to the proof in the general case.
The proof is accomplished by applying $\widetilde T_r$ to a finite sequence of spaces
and we need to ensure that $r$ is small enough that $\widetilde T_r$ is well defined on all of them.
After an initial reading, one verifies that the sequence of spaces is independent
of $r$, and we are therefore free to assume $r$ is small enough that
$\widetilde T_r$ can be applied to all of them.

The idea of the proof is that the operator $\tilde T_r$ improves both the Sobolev indices and the weight, but that these improvements are constrained both by the condition that the Sobolev indices be in $\mathcal S^{s,p}_d$ and by the condition that the weight be in $\mathcal D_q(R)$, which changes depending on the Lebesgue parameter $q$.
We navigate these restrictions by combining the use of $\tilde T_r$ with the basic inclusion of Lemma \ref{lem:basic-inclusions}, which allows one to `trade' weight for Lebesgue parameter, and with using the Sobolev embedding of Proposition \ref{prop:SobolevEmbedding}, which allows one to `trade' differentiability for Lebesgue parameter.
To facilitate our argument we introduce the following notation:
for any $t\in(1,\infty)$, let $\sigma_t = \sup\{\sigma \colon (\sigma,t)\in \mathcal S^{s,p}_{d}\}$ and let
\[
\delta_t=\sup \deltrange_t(R) = R + \frac{n-1}{2}-\frac{n-1}{t}.
\]

We first address the case where $q\ge q_0$.  Noting that $\delta_{q_0}-R\in\deltrange_{q_0}(R)$,
the mapping property
\eqref{eq:parametrix-mapping} implies we can
apply $\widetilde T_r$
a finite number of times to $u$ to arrive at
some
$\hat u\in H^{\sigma_{q_0}+1,q_0}_{\delta_{q_0}-R}(M;E)$.
Now fix $t$ to satisfy
\[
\frac{1}{t} = \max\left(\frac{1}{q},\frac{1}{q_0}-\frac{1}{n},\frac{1}{q_0}-\frac{R/2}{n-1}\right).
\]
Sobolev embedding implies $\hat u\in H^{\sigma_q,t}_{\delta_{q_0}-R/2}(M)$.
Moreover, since $t\ge q$, one readily verifies that $\sigma_q \ge \sigma_t$, see Figure \ref{fig:Sspd},
and therefore $\hat u\in H^{\sigma_t,t}_{\delta_{q_0}-R/2}(M)$
as well.

We claim that $\delta_{q_0}-R\in \deltrange_t(R)$.
Indeed, by construction
\[
-\frac{R}{2}\le \frac{n-1}{t} -\frac{n-1}{q_0} \le 0
\]
and hence
\[
\delta_{q_0}-R +\frac{n-1}{t} -\frac{n-1}{2} = \frac{n-1}{t}-\frac{n-1}{q}
\in [-R/2,0].
\]
Thus $\delta_{q_0} - R\in \mathcal D_t(R)$ as claimed.

We have shown that we can apply $\widetilde T_r$ to
$u$ finitely many times to obtain an element in
$H^{\sigma_t,t}_{\delta^*}(M)$ for
some $\delta^*\in \deltrange_t(R)$ where either $t=q$ or
$1/t$ is a decrease of $1/q_0$ by  $\min(\frac{1}n,\frac{R/2}{n-1})$.  
Repeating this argument finitely many times
we obtain an element of $H^{s_q,q}_{\delta^*}(M)$,
and therefore also  of $H^{\sigma,q}_{\delta^*}(M)$,
for some $\delta^*\in\deltrange_q(R)$. 
Finally, if needed, we can apply
$\widetilde T_r$ finitely many more times to improve
the weight and obtain an element in $H^{\sigma,q}_\delta(M)$ as required.

Now consider the case $q<q_0$ and let $\delta^\prime=\frac{n-1}2-\frac{n-1}{q_0}\in \deltrange_{q_0}(R)$.
We can apply $\widetilde T_r$ finitely many times
to $u$ to obtain some $\hat u\in H^{s_{q_0}+1,q_0}_{\delta'}(M)$.
Now fix $t$ to satisfy
\[
\frac{1}{t} = \min\left( \frac{1}{q},\frac{1}{q_0} + \frac{1}{n}\right).
\]
Since $q_0>t$ and since $\delta'\in\deltrange_{q_0}(R)$
there exists $\delta^*\in\deltrange_{t}(R)$ such
that $\delta^* < \delta'$.  Moreover,
since $1/t \le 1/q_0 +1/n$, a computation shows $s_{t}\le s_{q_0}+1$
and Lemma \ref{lem:basic-inclusions}
implies $\hat u\in H^{\sigma_{t},t}_{\delta^*}(M)$.  That is, we have shown
we can apply $\widetilde T_r$ to $u$ finitely many times
to arrive in $H^{\sigma_{t},t}_{\delta^*}(M)$ where $\delta^*\in\deltrange_t(R)$
and where either $t=q$ or $1/t$ is an increase of $1/q_0$ by $1/n$.
Repeating this argument finitely many times we can
arrive in $H^{\sigma_{q},q}_{\delta^*}(M)$ for some $\delta^*\in \deltrange_q(R)$
and the argument is now finished as in the case when $q_0\leq q$.
\end{proof}

Lemma \ref{lem:T-smoothing} implies the desired regularity
up to the boundary at infinity, at least for functions supported
sufficiently near it.

\begin{corollary}\label{cor:H-boundary-bootstrap}
Suppose
$u\in H^{\sigma_0,q_0}_{\delta_0}(M;E)$ for some
$(\sigma_0,q_0)\in\mathcal S^{s,p}_{d}$
and
$\delta_0\in\deltrange_{q_0}(R)$.  If
$\mathcal P u \in H^{\sigma-d,q}_{\delta}(M;E)$
for some
$(\sigma,q)\in\mathcal S^{s,p}_{d}$ and
$\delta \in \deltrange_q(R)$, and if
$u$ is supported in some $A_r$ with $r$ sufficiently small,
then $u\in H^{\sigma,q}_{\delta}(M;E)$.
\end{corollary}
\begin{proof}
Take $r$ small enough so that the conclusions of Lemma \ref{lem:T-smoothing}
and Corollary \ref{cor:parametrix} hold for the given choice
of parameters $\sigma_0$, $q_0$, etc.
Let $\eta$ be a smooth function on $\bar M$ that equals 1 on $A_{r/8}$
and vanishes outside $A_{r/4}$.

We assume that $u$ is supported in $A_{r/8}$ and that $\mathcal P u\in H^{\sigma,q}_\delta(M)$.
Since $\eta$ is supported in $A_{r/4}$ equation \eqref{eq:tildeQ-identity} implies
\begin{equation}\label{eq:QP-expand}
(-\eta\widetilde T_r)^k u =(\eta (\Id-\widetilde Q_r \mathcal P))^k u
\end{equation}
for any $k\in\Nats$, and in particular for the value $k=\ell$ from Lemma \ref{lem:T-smoothing}.
Note that both $\tilde Q_r\mathcal P$ and multiplication by $\eta$ map $H^{\sigma, q}_\delta(M) \to H^{\sigma,q}_\delta(M)$.
Thus expanding \eqref{eq:QP-expand}, and using $\eta u = u$, it follows from  $(-\eta\widetilde T_r)^\ell u\in H^{\sigma,q}_{\delta}(M)$ and 
$\widetilde Q_r \mathcal P u \in H^{\sigma,q}_\delta(M)$ that
$u \in H^{\sigma,q}_{\delta}(M)$.
\end{proof}

Corollary \ref{cor:H-boundary-bootstrap}, and interior elliptic regularity estimates, imply the following regularity result for $\mathcal P$.

\begin{proposition}\label{prop:P-regularity}
Suppose
$u\in H^{\sigma_0,q_0}_{\delta_0}(M;E)$
for some
$(\sigma_0,q_0)\in\mathcal S^{s,p}_{d}$ and
$\delta_0\in \deltrange_{q_0}(R)$.
If $\mathcal P u\in H^{\sigma-d,q}_{\delta}(M;E)$
for some $\delta\in\deltrange_q(R)$ then
$u\in H^{\sigma,q}_{\delta}(M;E)$.
\end{proposition}

\begin{proof}
Using Lemma \ref{lem:basic-inclusions} if $q_0 \geq p\dual$, or Proposition \ref{prop:SobolevEmbedding} if $q_0< p\dual$, we see that $u\in H^{d-s, p\dual}_\text{loc}(M)$.
Thus Proposition \ref{prop:interior-reg} implies $u\in H^{\sigma,q}_\text{loc}(M)$.

Select $r$ sufficiently small so that the conclusion of Corollary \ref{cor:H-boundary-bootstrap}
holds and let $\eta$ be a smooth function on $\bar M$ that equals 1 in a neighborhood
of $\partial M$ and that vanishes outside $A_r$.  
Since $[\mathcal P, \eta]u$ is supported in the complement of the region where $\eta =1$, Lemma \ref{lem:commutator} implies that $[\mathcal P, \eta]u \in H^{\sigma-d+1,q}_\delta(M)$.
By hypothesis, $\eta\mathcal P u \in H^{\sigma-d, q}_\delta(M)$ and thus $\mathcal P(\eta u)\in H^{\sigma-d, q}_\delta(M)$.
Corollary \ref{cor:H-boundary-bootstrap} thus implies $\eta u \in H^{\sigma, q}_\delta(M)$.

The interior regularity $u\in H^{\sigma,q}_\text{loc}(M)$ implies $(1-\eta)u \in H^{\sigma, q}_\delta(M)$, from which the desired result follows.
\end{proof}

Having established the needed elliptic estimates of Proposition \ref{prop:P-regularity},
we proceed to show that $\mathcal P$ is Fredholm using $L^2$ argument.
Recall that Assumption \ref{Assume-P} implies that $(\sigma,q)=(d/2,2)$ lies in
$\mathcal S^{s,p}_{d}$ and implies that $0\in \deltrange_2(R)$,
since $R>0$.  Let 
$Z$ be the
kernel of $\mathcal P:H^{d/2,2}_0(M;E) \to H^{-d/2, 2}_0(M;E)$.
Proposition \ref{prop:P-regularity}
implies $Z$ is the common kernel of $\mathcal P:H^{\sigma,q}_{\delta}(M;E)\to H^{\sigma-d,q}_{\delta}(M;E)$
for all $(\sigma,q)\in\mathcal S^{s,p}_d$ and $\delta\in\deltrange_q(R)$.
Furthermore, Proposition \ref{prop:semi-fred-estimate}
implies that $Z$ is finite dimensional, and that the image of $\mathcal P$ in each $H^{\sigma-d,q}_{\delta}(M;E)$ is closed.

For a section $v\in Z$ and for a smooth compactly supported section $w$ of $E$
define
\begin{equation}
\label{define-Lv}
L_v(w) = \ip<v,w>_{(M,g)} = \int_M \ip<v, w>_{g}\; dV_g.
\end{equation}
\begin{lemma}\label{lem:Z-perp-continuous}
Suppose $(\sigma,q)\in\mathcal S^{s,p}_{d}$.
For each $v\in Z$, $L_v$ extends to a
 continuous linear functional on $H^{\sigma-d,q}_{\delta}(M;E)$
for each $\delta\in\deltrange_{q}(R)$ and on
$X^{\sigma-d,q}_{\delta}(M;E)$ for each $\delta\in \deltrange_\infty(R)$.
\end{lemma}
\begin{proof}
Let $v\in Z$.
A computation shows that $\mathcal S^{s,p}_{d}$
is invariant under the transformation $(\sigma,q)\mapsto (d-\sigma,q\dual)$,
and that $\delta\in\deltrange_q(R)$ if and only if $-\delta\in \deltrange_{q\dual}(R)$.  Thus $v\in Z$ is an element of  $H^{d-\sigma,q\dual}_{-\delta}(M)$.
Continuity of $L_v$ now follows from Proposition \ref{prop:dualH-early}.

Note that if
$\delta\in \deltrange_\infty(R)$ then $\delta-\frac{n-1}{q}-\epsilon\in \deltrange_q(R)$
for $\epsilon$ sufficiently small.
Thus the continuity of $L_v$ on $X^{\sigma-d,q}_{\delta}(M)$, where $\delta\in \deltrange_\infty(R)$,
follows from the embedding $X^{\sigma-d,q}_{\delta}(M)\hookrightarrow
H^{\sigma-d,q}_{\delta-(n-1)/q-\epsilon}(M)$ provided by Lemma \ref{lem:basic-inclusions}. 
\end{proof}

\begin{theorem}\label{thm:H-fredholm}
Suppose $g$ is an asymptotically hyperbolic metric of either class \eqref{intro:H-class} or class \eqref{intro:X-class}.
Let $\mathcal P = \mathcal P[g]$ be a $d^\text{th}$-order elliptic operator satisfying Assumptions \ref{Assume-P} and \ref{Assume-I}.
Set $Z = \ker\left( \mathcal P\colon H^{d/2,2}_0(M;E) \to H^{-d/2,2}_0(M;E)\right)$.

Let $(\sigma,q)\in \mathcal S^{s,p}_{d}$
and $\delta\in \deltrange_{q}(R)$.
Given $f\in H^{\sigma-d,q}_{\delta}(M;E)$, there
exists $u\in H^{\sigma,q}_{\delta}(M;E)$ solving
$\mathcal P u =f$
if and only if $L_v(f)=0$ for all $v\in Z$.
In particular,
\[
\mathcal P: H^{\sigma,q}_{\delta}(M;E) \to H^{\sigma-d,q}_{\delta}(M;E)
\]
is Fredholm with index zero.
\end{theorem}

\begin{proof}
Let $Y\subset H^{\sigma-d,q}_{\delta}(M;E)$ be the image of $\mathcal P$ acting on $H^{\sigma,q}_{\delta}(M;E)$.  
Since $Y$ is closed by Proposition \ref{prop:semi-fred-estimate}, it is the intersection of the kernels of all continuous linear functionals on $H^{\sigma-d, q}_\delta(M)$ vanishing on $Y$. 
We establish the theorem by showing that such linear functionals aisely those the form $L_v$ for some $v\in Z$.

First, a density argument and the self-adjointedness
of $\mathcal P$ implies that if $v\in Z$, and if $f = \mathcal P u$ for some $u\in H^{\sigma, q}_\delta(M)$, then
\begin{equation*}
L_v(f) = \ip<v,\mathcal Pu>_{(M,g)} =
\ip<\mathcal P v,u>_{(M,g)} = 0.
\end{equation*}
Thus the functionals $L_v$ indeed all vanish on $Y$.

Conversely, suppose $F\in (H^{\sigma-d,q}_{\delta}(M;E))\dual$
vanishes on $Y$.  Proposition \ref{prop:dualH-early}
implies there exists $v\in H^{d-\sigma,q\dual}_{-\delta}(M;E)$
such that
\[
F(\cdot) = \ip<v,\cdot>_{(M, g)}.
\]
Since $F$ vanishes on $Y$ it follows that
$\ip<v,\mathcal P u>_{g} = 0$ for all sufficiently smooth,
compactly supported $u$ and we conclude, again from self-adjointedness,
that $\mathcal Pv=0$ in the sense of distributions.  As
$\mathcal S^{s,p}_d$ is invariant under $\sigma\to d-\sigma$
and $q\to q\dual$, and since $\delta\in\deltrange_q(R)$
if and only if $-\delta\in \deltrange_{q\dual}(R)$,
we conclude $(d-\sigma,q\dual)\in \mathcal S^{s,p}_d$
and $-\delta\in \deltrange_{q\dual}(R)$.
Thus $v$ is in the kernel of $\mathcal P$ acting on $H^{d-\sigma,q\dual}_{-\delta}(M)$.
Hence $v\in Z$ and $F=L_v$.
\end{proof}

\subsection{Fredholm theorem for Gicquaud-Sakovich spaces}

Our approach to establishing Fredholm theorems for operators acting in Gicquaud-Sakovich spaces is to initially solve in some weighted
Sobolev space $H^{\sigma,q}_{\delta'}(M;E)$ and then show
via a bootstrap that the solution in fact
lies in the desired Gicquaud-Sakovich space $X^{\sigma,q}_{\delta}(M;E)$.
The bootstrap relies on the
fact some of the weighted Sobolev spaces are, in fact, also
Gicquaud-Sakovich spaces on which $\mathcal P$ has good mapping properties.
For a low regularity metric, however,
the parameters $(\sigma,q)\in \mathcal S^{s,p}_d$ and $\delta\in \mathcal D_{q}(R)$
may be so limited that this overlap does not occur. Presumably this difficulty
could be overcome by introducing a scale of spaces that transition continuously between
$H^{s,p}_\delta(M;E)$ and $X^{s,p}_\delta(M;E)$.
We overcome this obstacle instead by first
assuming additional regularity from the metric, namely that $s\ge d$, and then using an
index-theoretic argument.

\begin{lemma}\label{lem:T-smoothing-X}
Assume that $s\ge d$.

If $u\in H^{\sigma_0,q_0}_{\delta_0}(M;E)$
for some $(\sigma_0,q_0)\in \mathcal S^{s,p}_{d}$
and $\delta_0\in\deltrange_{q_0}(R)$, then for each
$(\sigma,q)\in \mathcal S^{s,p}_{d}$ and
$\delta\in\deltrange_\infty(R)$ there exists
$\ell\in\Nats$ such that
$\widetilde T^\ell_r(u)\in X^{\sigma,q}_{\delta}(M;E)$
so long as $r$ is sufficiently small.
The same statement
for the same values of $r$ is true replacing
$\widetilde T_r$ with $\eta\widetilde T_r$ for any $\eta\in C^\infty(\bar M;\mathbb R)$.
\end{lemma}
\begin{proof}
As in the proof of Lemma \ref{lem:T-smoothing} we can assume $\eta\equiv 1$
and that $r$ is small enough that $\widetilde T_r$ is well-defined on the
finitely many spaces visited in the bootstrap.

Since $\delta\in \deltrange_{\infty}(R)$, we can find
$\tilde q\in(1,\infty)$ sufficiently large
such that $\delta\in \deltrange_{\tilde q}(R)$.  Increasing
$\tilde q$ if needed we can assume that $\tilde q> p$.
Now pick $\tilde \sigma\in\Reals$
according to
\begin{equation}\label{eq:tilde-sigma}
\frac{1}{\tilde q}-\frac{\tilde \sigma}{n} = \frac{1}{p} -\frac{s}{n}.
\end{equation}
We claim that $(\tilde \sigma,\tilde q)\in \mathcal S^{s,p}_{d}$,
which involves verifying conditions \eqref{eq:S-conds}.
First, since $\tilde q>p$ and since $1/p - s/n<0$,
we conclude from equation \eqref{eq:tilde-sigma} that
$s> \tilde \sigma > 0$. Since $s\ge d$ we conclude $\tilde\sigma \in [d-s,s]$
which is the first of the conditions \eqref{eq:S-conds}.
The second condition follows from equation \eqref{eq:tilde-sigma}
and the fact that the weak $L^2$ condition ensures $1/p-s/n<1/p\dual-(d-s)/n$.

Since $(\tilde \sigma, \tilde q)\in\mathcal S^{s,p}_{d}$
and since $\delta\in \deltrange_{\tilde q}(R)$, Lemma
\ref{lem:T-smoothing} implies there exists $\ell\in\Nats$
such that
\[
\widetilde T^\ell_r u \in H^{\tilde\sigma,\tilde q}_{\delta}(M;E)
\subset X^{\tilde \sigma,\tilde q}_{\delta}(M;E).
\]
The bootstrap from here to $X^{\sigma,q}_{\delta}(M;E)$
now follows the procedure
of Lemma \ref{lem:T-smoothing}, using the
Giquaud-Sakovich variation of Corollary \ref{cor:parametrix}
for the smoothing properties of $\widetilde T_r$.
The argument is indeed somewhat simpler because
there is no need to further
adjust $\delta$ now that it is set to its target value.
\end{proof}

The following variation
of Proposition \ref{prop:P-regularity} now follows from Lemma \ref{lem:T-smoothing-X}
using identical arguments now based on the Gicquaud-Sakovich variation of Corollary
\ref{cor:parametrix}.  Note that we still assume additional regularity of the metric.

\begin{proposition}\label{prop:P-regularity-X}
Assume that $s\ge d$.

Suppose $u\in H^{\hat\sigma,\hat q}_{\hat \delta}(M;E)$
for some $(\hat \sigma,\hat q)\in \mathcal S^{s,p}_{d}$
and some $\hat \delta\in\deltrange_{\hat q}(R)$.
If $\mathcal P u\in X^{\sigma-d,q}_{\delta}(M;E)$ for some $(\sigma,q)\in \mathcal S^{s,p}_{d}$ and
$\delta\in \deltrange_\infty(R)$, then $u\in X^{\sigma,q}_{\delta}(M;E)$.
\end{proposition}

We now establish the main Fredholm theorem
for operators acting on Gicquaud-Sakovich spaces.  
For regular metrics, meaning those for which $s\geq d$, we
use Theorem \ref{thm:H-fredholm},
along with Proposition \ref{prop:P-regularity-X}, to
refine the mapping properties of $\mathcal P$.
For metrics with regularity involving $s<d$, the proof below relies on the we use an index theory argument as follows.
Given an asymptotically hyperbolic metric $g$ of class $\mathscr H^{s,p;m}$, where $s<d$, we use Corollary \ref{cor:metric-approx} to approximate $g$ by smooth, asymptotically hyperbolic metrics that are close to $g$ in the $X^{s,p}_0$ topology.
Theorem \ref{thm:H-fredholm} implies the corresponding operators are Fredholm of index zero.
Since, due to Proposition \ref{prop:L-mapping-S} and Lemma \ref{lem:geometric-op-mapping}, geometric operators are continuous with respect to this norm, approximation in $X^{s,p}_0$ suffices to use index-theoretic arguments in order to conclude that the operator corresponding to $g$ is also Fredholm of index zero.

Recall the definition of $L_v$ from \eqref{define-Lv}.

\begin{theorem} \label{thm:X-fredholm}
Suppose $g$ is an asymptotically hyperbolic metric, of either class \eqref{intro:H-class} or class \eqref{intro:X-class}.
Let $\mathcal P = \mathcal P[g]$ be a $d^\text{th}$-order elliptic operator satisfying Assumptions \ref{Assume-P} and \ref{Assume-I}.
Set $Z = \ker\left(\mathcal P\colon H^{d/2,2}_0(M;E)\to H^{-d/2,2}_0(M;E)\right)$.

Let $(\sigma,q)\in \mathcal S^{s,p}_{d}$ and $\delta\in \deltrange_{\infty}(R)$. 
Given $f\in X^{\sigma-d,q}_\delta(M;E)$ there exists $u\in X^{\sigma,q}_\delta(M;E)$ solving $\mathcal Pu = f$ if and only if $L_v(f) =0$ for all $v\in Z$.
In particular,
\[
\mathcal P: X^{\sigma,q}_{\delta}(M;E) \to X^{\sigma-d,q}_{\delta}(M;E)
\]
is Fredholm with index 0 and kernel $Z$.
\end{theorem}
\begin{proof}
First suppose $s\ge d$.  Let $f\in X^{\sigma-d,q}_\delta(M)$.
Since $\delta\in \deltrange_\infty(R)$ we can choose $\tilde\delta<\delta-(n-1)/q$
sufficiently large so that $\tilde\delta\in \deltrange_q(R)$.
Lemma \ref{lem:basic-inclusions} implies $X^{t,q}_{\delta}(M)\subset H^{t,q}_{\tilde\delta}(M)$ for all $t\in\Reals$. 

Suppose $\mathcal Pu = f$ for some $u\in X^{\sigma, q}_\delta(M)$.
Then $u\in H^{\sigma, q}_{\tilde\delta}(M)$ and Theorem \ref{thm:H-fredholm} implies that $L_v(f) =0$ for all $v\in Z$.
Conversely, suppose $L_v(f) = 0$ for all $v\in Z$. Since $f\in X^{\sigma-d,q}_\delta(M)\subset H^{\sigma-d,q}_{\tilde\delta}(M)$, Theorem \ref{thm:H-fredholm} implies the existence of $u\in H^{\sigma,q}_{\tilde\delta}(M)$ with $\mathcal Pu = f$.
Proposition \ref{prop:P-regularity-X} implies that $u\in X^{\sigma,q}_\delta(M)$.

Proposition \ref{prop:P-regularity-X} further implies $Z\subset X^{\sigma,q}_\delta(M)$ and thus $\dim(\ker \mathcal P) = \dim(Z)$.
Moreover, if $L_v$ vanishes on $X^{\sigma-d,q}_{\delta}(M)$ then $L_v(v)=0$ and therefore
$v=0$.  Thus $\coker\mathcal P=\dim(Z)$ as well and $\mathcal P:X^{\sigma,q}_{\delta}(M)\to
X^{\sigma-d,q}_{\delta}(M)$ has index 0.

Now relax the assumption that $s\ge d$ and let $g_n$
be a sequence of asymptotically hyperbolic
metrics such that $\bar g_n\in C^\infty(\bar M)$
and such that $g_n\to g$ in $X^{s,p}_0(M)$ as provided by
Corollary \ref{cor:metric-approx}.  Let $F$ be the tensor bundle
ambient to $E$ and let $\Pi_E$ denote $g$-orthogonal projection
of sections of $F$ onto sections of $E$.  For each $n$, let
$E_n$ be the geometric tensor bundle associated
with $g_n$ corresponding to $E$ and let $\Pi_{E_n}$ be the 
$g_n$ orthogonal projection onto $E_n$.  We claim 
that $\Pi_{E_n}\to \Pi_E$ as a map on $X^{\sigma,q}_{\delta}(M;F)$
and that
\begin{align*}
\Pi_{E_n}&:X^{\sigma,q}_{\delta}(M;E)\to X^{\sigma,q}_{\delta}(M;E_n),\\
\Pi_{E}&:X^{\sigma-d,q}_{\delta}(M;E_n)\to X^{\sigma-d,q}_{\delta}(M;E)
\end{align*}
are isomorphisms for $n$ sufficiently large.
Assuming this claim for the moment, 
observe from
Lemma \ref{lem:geometric-op-mapping} and Proposition \ref{prop:L-mapping-S} 
that the operators $\mathcal P_n$ converge to $\mathcal P$
as maps $X^{\sigma,q}_{\delta}(M;F)\to X^{\sigma-d,q}_{\delta}(M;F)$.
Hence $\Pi_E\circ \mathcal P_n\circ \Pi_{E_n} \to \mathcal P$
as maps $X^{\sigma,q}_{\delta}(M;E)\to X^{\sigma-d,q}_{\delta}(M;E)$.
Because the maps $\Pi_E$ and $\Pi_{E_n}$ are isomorphisms
between sections of their respective geometric tensor bundles,
each $\Pi_E\circ \mathcal P_n\circ \Pi_{E_n}$ is Fredholm with index zero.
Recall from Proposition \ref{prop:semi-fred-estimate}
that $\mathcal P$ is semi-Fredholm. Since the index of a semi-Fredholm
map is locally constant \cite{schechter2001principles} it follows
that $\mathcal P$ is also Fredholm with index zero.

To establish the claim, using local orthonormal
frames constructed via Lemma \ref{lem:GS} one readily verifies that
the operator norm of $\Pi_{E_n}-\Pi_E$ 
as a map on $X^{\sigma,q}_{\delta}(M;F)$ is controlled
by $\|g_n-g\|_{X^{s,p}_0(M)}$ and hence converges to zero.
On $X^{\sigma,q}_\delta(M;E)$ we have
$\Id - \Pi_E\circ\Pi_{E_n} = \Pi_E\circ(\Pi_E-\Pi_{E_n})$
and hence for $n$ sufficiently large $\Pi_E\circ\Pi_{E_n}$
is invertible on $X^{\sigma,q}(M;E)$
and $\Pi_{E_n}$ is therefore injective.  A parallel argument shows that for $n$
large enough $\Pi_{E_n}$ is also surjective, and similar considerations
show that $\Pi_{E}:X^{\sigma-d,q}_\delta(M;E_n)\to X^{\sigma-d,q}_\delta(M;E)$
is also an isomorphism for $n$ sufficiently large.  

It remains in the low regularity case
to show $\ker \mathcal P=Z$, and to identify
the image of $\mathcal P$.
First, picking $\tilde\delta$ as above, we have the straightforward inclusion
\begin{equation}
\label{eq:ker-P-in-Z}
\begin{aligned}
\ker \left(\mathcal P : X^{\sigma,q}_{\delta}(M) \to X^{\sigma-d,q}_{\delta}(M)\right)
&\subset \ker \left(\mathcal P : H^{\sigma,q}_{\tilde\delta}(M) \to H^{\sigma-d,q}_{\tilde\delta}(M)\right) 
\\
&= Z.
\end{aligned}
\end{equation}
Regarding the image, let
\[
Y=\{w\in X^{\sigma-d,q}_{\delta}(M;E): L_v(w)=0\text{ for all $v\in Z$}\}.
\]
We claim that $v\mapsto L_v$ is injective from $Z$ to $(X^{\sigma-d,q}_{\delta}(M))\dual$,
in which case $\codim Y=\dim Z$.  
If $v\in Z$ is not zero, the
density of smooth compactly
supported sections in $H^{d/2,2}_{0}(M)$ implies we can find such a section
$\psi$ with $L_v(\psi)\neq 0$.  The claim follows from observing $\psi\in X^{\sigma-d,q}_\delta(M)$
as well.

Theorem \ref{thm:H-fredholm} implies
\begin{equation}\label{eq:im-P-in-Y}
\im \left(\mathcal P : X^{\sigma,q}_{\delta}(M) \to X^{\sigma-d,q}_{\delta}(M)\right)
\subset Y.
\end{equation}
Recalling that $\mathcal P$ has index zero we compute
\[
\dim Z \ge \dim(\ker \mathcal P) = \codim (\im\mathcal P)
\ge \codim Y = \dim Z.
\]
So all the intermediate inequalities are equalities, which is only
possible if the containments \eqref{eq:ker-P-in-Z} and \eqref{eq:im-P-in-Y} are equalities.
\end{proof}

As a consequence of Theorem \ref{thm:X-fredholm}, we can
now show that Proposition \ref{prop:P-regularity-X} holds
without making the smooth metric hypothesis.

\begin{corollary}
Suppose $u\in H^{\sigma_0,q_0}_{\delta_0}(M;E)$
for some $(\sigma_0,q_0)\in \mathcal S^{s,p}_{d}$
and some $\delta_0\in\deltrange_{q_0}(R)$.  If $\mathcal P u\in X^{\sigma-d,q}_{\delta}(M;E)$ for some $(\sigma,q)\in \mathcal S^{s,p}_{d}$ and
$\delta\in \deltrange_\infty(R)$, then $u\in X^{\sigma,q}_{\delta}(M;E)$.
\end{corollary}
\begin{proof}

Observe that $\mathcal P u\in X^{\sigma-d,q}_{\delta}(M)\subset H^{\sigma-d,q}_{\delta}(M)$
for some $\tilde\delta\in \deltrange_{q}(R)$.  Proposition \ref{prop:P-regularity} implies
$u\in H^{\sigma,q}_{\tilde\delta}(M)$ and we wish to improve to $X^{\sigma,q}_{\delta}(M)$.

Since $\mathcal P u$ satisfies $\ip<v,\mathcal P u>_{g}=0$
for all $v\in Z$, Theorem \ref{thm:X-fredholm} implies
there exists $w\in X^{\sigma,q}_{\delta}$ such that
$\mathcal P w = \mathcal P u$.  Moreover,
\[
u-w\in \ker \mathcal P: H^{\sigma,q}_{\delta^*}(M)\to H^{\sigma-d,q}_{\delta^*}(M) = Z.
\]
Theorem \ref{thm:X-fredholm} implies $Z\subset X^{s,p}_\delta(M)$
and therefore $u = w + (u-w)\in X^{s,p}_\delta(M)$ as well.
\end{proof}

\subsection{Structure theorems}
\label{secsec:structure-theorems}
We now present structure theorems for elliptic operators satisfying Assumptions \ref{Assume-P} and \ref{Assume-I}.
For such an operator $\mathcal P[g]$, and for asymptotically hyperbolic metric $g$ of class $\mathscr H^{s,p;m}$ satisfying \eqref{intro:H-class} or class $\mathscr X^{s,p;m}$ satisfying \eqref{intro:X-class}, let 
\begin{equation*} 
Z = \ker\left( \mathcal P[g]\colon H^{d/2,2}_0(M;E) \to H^{-d/2,2}_0(M;E)\right).
\end{equation*}

For each $(\sigma,q)\in \mathcal S^{s,p}_d$ and $\delta\in \mathcal D_q(R)$, we invoke Lemma \ref{lem:Z-perp-continuous} to define
\begin{equation*}
Y(H^{\sigma,q}_\delta) = \left\{u\in H^{\sigma,q}_\delta(M;E) \colon L_v(u)=0 \text{ for all }v\in Z \right\}.
\end{equation*}
For $(\sigma,q)\in \mathcal S^{s,p}_d$ and $\delta \in \mathcal D_\infty(R)$, the space $Y(X^{\sigma,q}_\delta)$ is defined analogously.

\begin{theorem}
Suppose $g$ is an asymptotically hyperbolic metric, of either class \eqref{intro:H-class} or class \eqref{intro:X-class}.
Let $\mathcal P = \mathcal P[g]$ be a $d^\text{th}$-order elliptic operator satisfying Assumptions \ref{Assume-P} and \ref{Assume-I}.

Let $(\sigma,q)\in \mathcal S^{s,p}_{d}$ and
$\delta\in \deltrange_{q}(R)$.  
There exist bounded operators
$G,H: H^{\sigma-d,q}_{\delta}(M;E) \to H^{\sigma,d}_\delta(M;E)$
that satisfy
\begin{align*}
&\ker H = Y(H^{\sigma-d,q}_{\delta}),
\\
&\im H = Z,\\
&u = G\mathcal P u + Hu\quad \text{ for all }u \in H^{\sigma,q}_\delta(M;E),\\
&w = \mathcal P G w+ Hw\quad \text{ for all }w\in H^{\sigma-d,q}_\delta(M;E).
\end{align*}
\end{theorem}
\begin{proof}
Recall that $Z\subset L^2(M)$ and let $\{v_j\}_{j=1}^m$ be an $L^2$ orthonormal basis for
$Z$.
For $u\in H^{\sigma,q}_\delta(M)$, we use Lemma \ref{lem:Z-perp-continuous} to define the continuous map $H\colon H^{\sigma-d, q}_\delta(M) \to H^{\sigma, q}_\delta(M)$ by 
\[
H(u) = \sum_{j=1}^m L_{v_j}(u) v_j 
= \sum_{j=1}^m\langle v_j, u\rangle_{(M,g)} v_j.
\]
It follows immediately from the definition that $\im H = Z$ and that $\ker H = Y(H^{\sigma -d, q}_\delta)$.

We now claim that $\mathcal P$ is a bijection $Y(H^{\sigma, q}_\delta) \to Y(H^{\sigma -d, q}_\delta)$.
That the image of $Y(H^{\sigma, q}_\delta)$ under $\mathcal P$ lies in $Y(H^{\sigma -d, q}_\delta)$ follows from the facts that $\mathcal P$ is self-adjoint and that, as noted in the proof of Lemma \ref{lem:Z-perp-continuous}, $Z\subset H^{d-\sigma, q\dual}_{-\delta}(M)$.
The injectivity is a consequence of the definition of $Y(H^{\sigma -d, q}_\delta)$, and thus it remains to show surjectivity.
For $f\in Y(H^{\sigma -d, q}_\delta)$, Theorem \ref{thm:H-fredholm} implies the existence of $\tilde u \in H^{\sigma, q}_\delta(M)$ with $\mathcal P \tilde u = f$.
Let $u = \tilde u - H \tilde u$. Since $H\tilde u\in Z$, and since $Z$ is the kernel of $\mathcal P$ acting on $H^{\sigma, q}_\delta(M)$, we have
\begin{equation*}
\mathcal P u = \mathcal P\tilde u = f.
\end{equation*}

Define $G\colon Y(H^{\sigma -d, q}_\delta) \to Y(H^{\sigma, q}_\delta)$ to be the inverse of $\mathcal P\big|_{Y(H^{\sigma, q}_\delta)}$, so that $G\mathcal P u = u$ for all $u\in Y(H^{\sigma, q}_\delta)$.
The Inverse Mapping Theorem implies $G$ is continuous, and since $Z$ is finite-dimensional, we may extend $G$ continuously to all of $H^{\sigma -d, q}_\delta(M)$ by declaring $G\big|_Z =0$.

Direct computation now shows that for $u\in H^{\sigma, q}_\delta(M)$ we have
\begin{equation*}
u - Hu = G(\mathcal P(u - Hu)) = G\mathcal P u.
\end{equation*}
Furthermore, for $w\in H^{\sigma-d, q}_\delta(M)$ we have $u- Hu\in Y(H^{\sigma -d, q}_\delta)$ and thus
\begin{equation*}
u - Hu = \mathcal P G(u - Hu) = \mathcal P Gu,
\end{equation*}
where we have used that $Hu\in Z$ and $G\big|_Z =0$.
\end{proof}

The analogous result for Gicquaud-Sakovich spaces is proved
identically using Theorem \ref{thm:X-fredholm} in place
of Theorem \ref{thm:H-fredholm}.
\begin{theorem}
Suppose $g$ is an asymptotically hyperbolic metric, of either class \eqref{intro:H-class} or class \eqref{intro:X-class}.
Let $\mathcal P = \mathcal P[g]$ be a $d^\text{th}$-order elliptic operator satisfying Assumptions \ref{Assume-P} and \ref{Assume-I}.

Suppose $(\sigma,q)\in \mathcal S^{s,p}_{d}$ and
$\delta\in \deltrange_{\infty}(R)$.  There exist bounded operators
$G,H: X^{\sigma-d,q}_{\delta}(M;E) \to X^{\sigma,d}_\delta(M;E)$
that satisfy
\begin{align*}
&\ker H = Y(X^{\sigma-d,q}_{\delta}),\\
&\im H = Z,\\
&u = G\mathcal P u + Hu\quad \text{ for all }u\in X^{\sigma,q},_\delta(M;E)\\
&w = \mathcal P G w+ Hw\quad \text{ for all }w\in X^{\sigma-d,q}_\delta(M;E).
\end{align*}
\end{theorem}

\subsection{Regularity bootstrapping and polyhomogeneity}
While stated for metrics with Sobolev-scale regularity, the theorems above apply to more regular metrics as well.
Thus one application of our work is the following regularity bootstrapping principle. 
Given a low-regularity solution to an elliptic PDE arising from a more regular metric (perhaps, for example, from variational methods), then one expects that this solution has as much regularity as the metric allows.
The following results illustrate this principle in three common regularity settings: H\"older-regular metrics, metrics smooth on $M$, and polyhomogeneous metrics.
For simplicity of exposition, we restrict attention to second-order geometric operators acting on scalar-valued functions; more general results can be obtained from analogous reasoning.

We first consider the case of metrics with H\"older regularity, and
suppose that $g$ is an asymptotically hyperbolic metric of class $\mathscr C^{k,\alpha;m}$, meaning that
\begin{equation}
\label{C-class}
\bar g = \rho^2 g\in \mathscr C^{k,\alpha;m}(M)
\quad \text{ with }1\leq m \leq k,\quad \alpha\in (0,1), 
\end{equation}
which implies $\bar g\in C^0(\bar M)$, and that
\begin{equation*}
|d\rho|_{\bar g} =1\quad \text{ along }\partial M.
\end{equation*}

\begin{theorem}
\label{thm:holder-bootstrap-app}
Suppose that $\mathcal P[g]$ is a second-order elliptic operator satisfying Assumptions \ref{Assume-P} and \ref{Assume-I}, and having indicial radius $R$.
Let $g$ be an asymptotically hyperbolic metric of class $\mathscr C^{k,\alpha;m}$, and assume that \eqref{C-class} holds with $k\geq 2$.

Suppose that $u\in H^{\sigma, q}_{\delta^\prime}(M;\mathbb R)$, where $\delta^\prime \in \mathcal D_q(R)$ and where $(\sigma, q)\in S^{k,p}_2$ for some $p>1$.
Suppose also that $\mathcal P[g]u = f$, where $f\in C^{k-2,\alpha}_\delta(M;\mathbb R)$ with $\delta \in \mathcal D_\infty(R)$.
Then $u\in C^{k,\alpha}_\delta(M;\mathbb R)$.
\end{theorem}

\begin{proof}
First, note that we may increase $p$ as needed and still retain $(\sigma, q)\in S^{k,p}_2$; see \eqref{eq:S-conds}.
From \eqref{eq:cC-to-cX-cc}, we have $\bar g\in \mathscr X^{k,p;m}(M)$ for all $p>1$, and thus we may choose $p$ sufficiently large that $g$ is an asymptotically hyperbolic metric of class \eqref{intro:X-class}.

The basic inclusion of Lemma \ref{lem:basic-inclusions} implies that $f\in X^{k-2,p}_\delta(M)$.
Proposition \ref{prop:P-regularity-X} implies $u\in X^{k,p}_\delta(M)$.
With $p$ sufficiently large, the Sobolev embedding of Proposition \ref{prop:SobolevEmbedding} implies that $u\in C^{k-1, \beta}_\delta(M)$ for some $0<\beta<1$.
Finally, \cite[Lemma 5.6]{WAH} implies that $u\in C^{k,\alpha}_\delta(M)$.
\end{proof}

An immediate consequence is the following result for metrics which are smooth on $M$, but have limited regularity on $\bar M$.

\begin{corollary}
Suppose that $\mathcal P[g]$ is a second-order elliptic operator satisfying Assumptions \ref{Assume-P} and \ref{Assume-I}, and having indicial radius $R$.
Let $g$ be an asymptotically hyperbolic metric of class $\mathscr C^{\infty;m}$, and assume that \eqref{C-class} holds for all $k\in\mathbb N$.

Suppose that $u\in H^{\sigma, q}_{\delta^\prime}(M;\mathbb R)$, where $\delta^\prime \in \mathcal D_q(R)$ and where $(\sigma, q)\in S^{k,p}_2$ for some $p>1$.
Suppose also that $\mathcal P[g]u = f$, where $f\in C^{\infty}_\delta(M;\mathbb R)$ with $\delta \in \mathcal D_\infty(R)$.
Then $u\in C^{\infty}_\delta(M;\mathbb R)$.
\end{corollary}

A second consequence of Theorem \ref{thm:holder-bootstrap-app} concerns metrics in the polyhomogeneous category.
Such metrics are smooth in the interior $M$ and admit a formal expansion near $\partial M$ of the form
\begin{equation*}
g \sim \sum_{i=0}^\infty \sum_{j=0}^{N_i}\rho^{s_i}(\log \rho)^j \bar g_{ij},
\end{equation*}
where $s_i$ is a sequence diverging to $+\infty$, $N_i\in \mathbb N_{\geq 0}$, and $\bar g_{ij}$ are smooth tensor fields on $\bar M$; see Appendix A of \cite{WAH} for a detailed definition and a discussion of elliptic operators in the polyhomogeneous category.

The following consequence of Theorem \ref{thm:holder-bootstrap-app} and \cite[Theorem A.14]{WAH} is stated for $m=2$.
Adapting the proof of \cite[Proposition A.14]{WAH}, as done in \cite[Theorem A.18]{WAH}, can relax this restriction to $m\geq 1$.
\begin{corollary}
Suppose that $\mathcal P[g]$ is a second-order elliptic operator satisfying Assumptions \ref{Assume-P} and \ref{Assume-I}, and having indicial radius $R$.
Let $g$ be an asymptotically hyperbolic metric of class $\mathscr C^{2,\alpha;2}$, and further assume that $g$ is polyhomogeneous.

Suppose that $u\in H^{\sigma, q}_{\delta^\prime}(M;\mathbb R)$, where $\delta^\prime \in \mathcal D_q(R)$ and where $(\sigma, q)\in S^{k,p}_2$ for some $p>1$.
Suppose also that $\mathcal P[g]u = f$, where $f\in C^{2,\alpha}_\delta(M;\mathbb R)$ with $\delta \in \mathcal D_\infty(R)$, and that $f$ is polyhomogeneous.
Then $u$ is polyhomogeneous as well.
\end{corollary}

\appendix
\section{Weighted function spaces}
\label{app:bpspaces}

This section concerns weighted function spaces on
a manifold $M$ admitting a continuous conformal compactification $\bar M$.
These spaces generalize the classical weighted Sobolev spaces
$W^{k,p}_\delta(M;E)$ from, e.g., \cite{Lee-FredholmOperators} and as well
as the Gicquaud-Sakovich spaces $X^{k,p}_\delta(M;E)$
from \cite{GicquaudSakovich} to real scales of differentiability.

A \Defn{tensor bundle} $E$ over $\bar M$ is a bundle
$T^{k_1,k_2}\bar M$
with contravariant rank $k_1$ and covariant rank $k_2$, and we
use the same notation $E$ for the induced bundle over $M$.  More
generally, when $\bar M$ is equipped with a metric
a \Defn{geometric tensor bundle} over $\bar M$ is a
subbundle of a tensor bundle $T^{k_1,k_2}\bar M$ that is
associated with an invariant subspace of the standard representation
of $O(n)$ (or $SO(n)$ if $\bar M$ is orientable) on $T^{k_1,k_2}\mathbb R^n$.
Because geometric tensor bundles potentially depend on the regularity of the
underlying metric on $\bar M$ we focus for the moment on the foundation
case where $E$ is a tensor bundle over $\bar M$.
Section \ref{secsec:geometric-E} contains the details needed to treat
the full case of geometric tensor bundles.

Recall that for
 $s\in\Reals$ and $1<p<\infty$ the Bessel potential space $H^{s,p}(\Reals^n)$ is defined by
\[
H^{s,p}(\Reals^n) = \{u\in \mathcal S'(\Reals^n): \| \mathcal F^{-1} \Lambda^s \mathcal F u\|_{L^p(\mathbb R^n)}\}<\infty,
\]
where $S'(\Reals^n)$ is the dual space of the rapidly
decaying smooth functions, $\Lambda(\xi)=(1+|\xi|^2)^{s/2}$ and where $\mathcal F$
is the Fourier transform.  When $s\in\Ints$ then $H^{s,p}(\Reals^n)=W^{k,p}(\Reals^n)$.
If $\Omega$ is an open subset of $\Reals^n$, $H^{s,p}(\Omega)$ is the quotient
space of $H^{s,p}(\Reals^n)$ where two distributions are equivalent if they agree on $\Omega$,
and $H^{s,p}(\Omega)$ is given the quotient norm $\|u\|_{H^{s,p}(\Omega)} = \inf_{w|_\Omega=u} \|w\|_{H^{s,p}(\Reals^n)}$.

To define corresponding asymptotically hyperbolic weighted spaces,
we fix a cutoff function $\chi$ that is equal to 1 on $B_{1/2}^\Hyp$ and vanishes outside $B_{r}^\Hyp$ for some $1/2< r <2$.  Let $\{\Phi_i\}$ be the collection of preferred M\"obius parametrizations of $M$
defined at the end of \S\ref{sec:coords}.
\begin{definition}\label{def:bpspaces}
Let $E$ be a tensor bundle over $M$.
For $s\in\Reals$, $\delta\in\Reals$ and $1<p<\infty$ the
weighted Bessel potential
space $H^{s,p}_\delta(M;E)$ is
the set of all $E$-valued distributions $u$ such that
\begin{equation}\label{eq:Hspdelta-def}
\|u\|_{H^{s,p}_{\delta}(M;E)} = \left[ \sum_{i} \rho_i^{-\delta p} \|\chi \Phi_i^* u\|_{H^{s,p}(\Reals^n)}\right]^\frac{1}{p} <\infty
\end{equation}
where the norms over $\Reals^n$ involve the vector of coefficients
of $\Phi_i^*u$ with respect to the standard basis.

The Gicquaud-Sakovich space $X^{s,p}_\delta(M;E)$ is the set of all
$E$-valued distributions $u$ with
\begin{equation}
\|u\|_{X^{s,p}_{\delta}(M;E)}= \sup_{i} \rho_i^{-\delta} \|\chi \Phi_i^* u\|_{H^{s,p}(\Reals^n)} <\infty.
\end{equation}

To simplify notation, when $E=T^{0,0}\bar M$ we denote the associated spaces
by $H^{s,p}_\delta(M;\Reals)$ and $X^{s,p}_\delta(M;\Reals)$.
\end{definition}

The definition of $H^{s,p}_{\delta}(M;E)$
is analogous to the definition of weighted Bessel potential spaces in Euclidean spaces
found in \cite{Triebel:1976hk} and several of the proofs below rely on generalizations
of techniques found there.

The following lemma is fundamental and shows we can build functions in weighted spaces in a
controlled fashion by adding functions supported in the images of M\"obius charts.
\begin{lemma}\label{lem:Hbuilder}
Let $E$ be a tensor bundle over $M$. Suppose $u=\sum_i u_i$ where
each $u_i\in D'(M;E)$
is a distribution supported in $\Phi_i(B^\Hyp_{r})$ for some $r<2$.
Then
\begin{enumerate}
\item\label{part:builderH} If $s,\delta\in \Reals$ and $1<p<\infty$
\begin{equation}\label{eq:Hspdelta-pieces}
\|u\|_{H^{s,p}_{\delta}(M;E)}^p \lesssim
\sum_{i=1} ^\infty \rho_i^{-p\delta} \|\Phi_i^* u_i\|_{H^{s,p}(\Reals^n)}^p
\end{equation}
and $\sum_i u_i$ converges in norm.
\item\label{part:builderX} If $s,\delta\in \Reals$ and $1<p<\infty$
\begin{equation}\label{eq:Xspdelta-pieces}
\|u\|_{X^{s,p}_{\delta}(M;E)} \lesssim
\sup_{i=1} \rho_i^{-p\delta} \|\Phi_i^* u_i\|_{H^{s,p}(\Reals^n)}.
\end{equation}
\item\label{part:builderC} If $k\in\Nats_{\ge 0}$ and $0<\alpha<1$,
\begin{equation}\label{eq:Ckalphadelta-pieces}
\|u\|_{C^{k,\alpha}_{\delta}(M;E)} \lesssim
\sup_{i=1} \rho_i^{-p\delta} \|\Phi_i^* u_i\|_{C^{k,\alpha}(\Reals^n)}.
\end{equation}
\end{enumerate}
\end{lemma}
\begin{proof}
The transition Jacobians
$\mathop\mathrm{D} \left(\Phi_j^{-1}\circ \Phi_i\right)$ admit uniform $C^k$ bounds for
any choice of $k$.
Indeed if $\Theta_{a_i}$ and $\Theta_{a_j}$ are background coordinates associated with
two M\"obius parametrizations
\begin{equation*}
\mathop\mathrm{D} (\Phi_j^{-1}\circ \Phi_i)_{(x,y)} =
\frac{\rho_i}{\rho_j} (\mathop\mathrm{D}(\Theta_{a_j}^{-1}\circ\Theta_{a_i}))_{(\theta_i + \rho_i y, \rho_i y)}
\end{equation*}
and the uniform bounds follow from uniform estimates for background coordinate transition
functions
the fact that
$\rho_i/\rho_j\le 8$ if the images of $\Phi_i$ and $\Phi_j$ have nontrivial intersection.
With this obsrvation in hand we establish inequality \eqref{eq:Hspdelta-pieces};
inequality \eqref{eq:Xspdelta-pieces}
is demonstrated using similar techniques
and inequality \eqref{eq:Ckalphadelta-pieces} is a straightforward consequence of
the $C^k$ estimates for transition Jacobians along with uniform local finiteness.

Consider indices $i$ and $j$ with $\Phi_i(B^\Hyp_2) \cap \Phi_j(B^\Hyp_2) \neq\emptyset$
and let $U_{j,i}=\Phi_j^*\Phi_{i,*}(B^\Hyp_2)$.
If $w$ is a smooth, compactly supported function on $\Reals^n$ we define
\[
Tw = \chi (\Phi_j^{-1} \circ \Phi_i)^* w|_{U_{j,i}}
\]
and extend this function by zero off of $U_{i,j}=\Phi_i^*\Phi_{j,*}(B^\Hyp_2)$.
Uniform $C^k$ bounds on
the transition Jacobian $D(\Phi_j^{-1} \circ \Phi_i)$
imply that for each $k\in\Nats_{\ge 0}$ that
$T$ determines a continuous map $H^{k,p}(\Reals^n)\to H^{k,p}(\Reals^n)$
with norm controlled by a constant independent of $i$ and $j$.
By interpolation and duality this same statement holds replacing $k$ with any $s\in\Reals$.
As a consequence we conclude
\begin{equation}\label{eq:Hs-transition}
\|\chi \Phi_i^* u_j\|_{H^{s,p}(\Reals^n)} \lesssim
\| \Phi_j^* u_j\|_{H^{s,p}(\Reals^n)}
\end{equation}
with implicit constant independent of $i$ and $j$.

Now fix $i$ and recall that the size of the neighbor index set
$N(i)=\{j:\Phi_i(B^\Hyp_2) \cap \Phi_j(B^\Hyp_2) \neq\emptyset\}$
admits a uniform upper bound. Then
\begin{equation*}
 \begin{aligned}
\|\chi \Phi_i^* u\|_{H^{s,p}(\Reals^n)}^p
&=
\left|\left|\chi \Phi_i^* \sum_{j\in N(i)} u_j\right|\right|_{H^{s,p}(\Reals^n)}^p\\
&\lesssim \sum_{j\in N(i)} \left|\left|\chi \Phi_i^* u_j\right|\right|_{H^{s,p}(\Reals^n)}^p\\
&\lesssim \sum_{j\in N(i)} \left|\left|\Phi_j^* u_j\right|\right|_{H^{s,p}(\Reals^n)}^p
 \end{aligned}
\end{equation*}
with implicit constant independent of $i$.
Moreover, since $\rho_i/\rho_j\le 8$ if $j\in N(i)$,
\begin{equation*}
\begin{aligned}
\|u\|_{H^{s,p}_{\delta}(M;E)}^p &= \sum_{i=1}^\infty \rho_i^{-\delta p} \|\chi \Phi_i^* u\|_{H^{s,p}(\Reals^n)}^p\\
&\lesssim \sum_{i=1}^\infty  \rho_i^{-\delta p} \sum_{j\in N(i)} \| \Phi_j^*u_j\|_{H^{s,p}(\Reals^n)}^p \\
&\lesssim \sum_{i=1}^\infty  8^{-\delta p}\sum_{j\in N(i)} \rho_j^{-\delta p}\| \Phi_j^* u_j\|_{H^{s,p}(\Reals^n)}^p \\
& \lesssim \sum_{j=1}^\infty \rho_j^{-\delta p}\| \Phi_j^* u_j\|_{H^{s,p}(\Reals^n)}^p
\end{aligned}
\end{equation*}
as required.
\end{proof}

As a first application, Lemma \ref{lem:Hbuilder} allows us to construct
a convenient partition of unity based on the cutoff function $\chi$.
Specifically, let
$\chi_i = (\Phi_i)_*\chi$, which we extend by zero to all of $M$.
Lemma \ref{lem:Hbuilder}\eqref{part:builderC} implies
\begin{equation}\label{eq:Xdef}
\bar \chi = \sum_{i=1}^\infty \chi_i \in C^{k,\alpha}_0(M;\Reals)
\end{equation}
for every $k\in\Nats$ and $0<\alpha<1$.
Moreover, since the sets $\Phi_i(B^\Hyp_{1/2})$ cover $M$ it follows that
$\bar \chi\ge 1$
and we conclude $1/\bar\chi\in C^{k,\alpha}_0(M;\Reals)$ as well. The
functions $\chi_i/\bar \chi$ then evidently form a
smooth partition of unity admitting uniform $C^{k,\alpha}_0(M;\Reals)$ bounds.

Definition \ref{def:bpspaces} relies on a number of choices, and \S\ref{sec:equivalent-norms}
shows that the spaces are independent of these choices.  For the moment, we
use Lemma \ref{lem:Hbuilder} to show that the choice of cutoff function is unimportant and
to establish an alternative formulation of the norms that does not employ
a cutoff function.  In light of Lemma \ref{lem:many-norms}, the following
result also shows that when $s\in \Nats_{\ge 0}$ the weighted Bessel potential spaces agree with the weighted
Sobolev spaces $W^{k,p}_{\delta}(M;E)$ .

\begin{lemma}\label{lemma:cutoffequiv}
Let $E$ be a tensor bundle over $\bar M$. For all
$s,\delta\in\Reals$ and $1<p<\infty$ the space $H^{s,p}_{\delta}(M;E)$
is independent of the choice of cutoff function and the norms for any two choices are equivalent.
Moreover, for each $1/2<r\le 2$
\begin{align}
\|u\|_{H^{s,p}_{\delta}(M;E)}&\sim
\left[\sum_{i} \rho_i^{-\delta p}\|\Phi_i^* u\|_{H^{s,p}(B_r^\Hyp)}^p
\right]^{\frac{1}{p}}\label{eq:cutoff-free-H},
\\
\label{eq:cutoff-free-X}
\|u\|_{X^{s,p}_{\delta}(M;E)}&\sim \sup_i \rho_i^{-\delta p}\|\Phi_i^* u\|_{H^{s,p}(B_r^\Hyp)}.
\end{align}
\end{lemma}

\begin{proof}
Let $\chi$ and $\eta$ be two choices of cutoff functions, and let $H^{s,p}_{\delta}(M;\chi)$ and $H^{s,p}_{\delta}(M;\eta)$ be the corresponding spaces.  
Suppose $u\in H^{s,p}_{\delta}(M;\eta)$
and let $u_i=((\Phi_i)_*\eta) u$ which we extend by zero outside of $\Phi_i(B_2)$.
Then each $u_i\in H^{s,p}_{\rm loc}(M)$
and from the definition of the norm on $H^{s,p}_{\delta}(M;\eta)$ it follows that
$
\sum_i \rho_i^{-p\delta} \|\eta \Phi_i^* u\|_{H^{s,p}(\Reals^n)}^p
$
is finite.  But then Lemma \ref{lem:Hbuilder}\eqref{part:builderH}
implies $\hat u := \sum_i u_i\in H^{s,p}_{\delta}(M;\chi)$
and
\begin{equation}\label{eq:equiv1}
\|\hat u\|_{H^{s,p}_{\delta}(M;\chi)}^p
\lesssim \sum_i \rho_i^{-p\delta} \|\eta \Phi_i^* u\|_{H^{s,p}(\Reals^n)}^p
= \|u\|_{H^{s,p}_{\delta}(M;\eta)}^p.
\end{equation}
Let $\overline\eta= \sum_i \Phi_{i,*}\eta$ and observe that $\hat u=\overline\eta u$.
Following the construction that leads to \eqref{eq:Xdef}, but using $\eta$ in place of $\chi$, we obtain 
$\bar\eta\in C^{k,\alpha}_0(M)$ with $1/\bar\eta \in C^{k,\alpha}_0(M)$
for every $k\in\Nats$ and $0<\alpha<1$.
In particular, choosing $k>s$, it is an easy consequence of Definition \ref{def:bpspaces}
that multiplication by $\bar\eta$ or $1/\bar\eta$ is a continuous automorphism of $H^{s,p}_{\delta}(M;\chi)$
 and therefore
\begin{equation}\label{eq:equiv2}
\|u\|_{H^{s,p}_{\delta}(M;\chi)}^p \lesssim \|\overline{\eta} u\|_{H^{s,p}_{\delta}(M;\chi)}^p.
\end{equation}
The continuous inclusion $H^{s,p}_{\delta}(M;\eta)\hookrightarrow H^{s,p}_{\delta}(M;\chi)$
follows from inequalities \eqref{eq:equiv1} and \eqref{eq:equiv2} and the reverse inclusion
follows from symmetry.

Suppose $1/2<r<2$. Pick two cutoff functions $\eta_1$ and $\eta_2$
supported on $B_2^\Hyp$ that both equal $1$ on $B_{1/2}^\Hyp$ and such that
 $\eta_1$ is supported in $B_r^\Hyp$ whereas $\eta_2=1$ on
$B_r^\Hyp$.  Suppose $u\in H^{s,p}_\delta(M)$. For each $i$
\begin{equation}\label{eq:cutoff-local-outer}
\|\Phi_i^* u\|_{H^{s,p}(B_r^\Hyp)} \le \|\eta_2\Phi_i^* u\|_{ H^{s,p}(\Reals^n)}.
\end{equation}
On the other hand, we can pick an extension $w$ of $\Phi_i^*u$ such that
$\|w\|_{H^{s,p}(\Reals^n)} \le 2\|\Phi_i^* u\|_{H^{s,p}(B_r^\Hyp)}$
in which case
\begin{equation}\label{eq:cutoff-local-inner}
\|\eta_1\Phi_i^*u\|_{H^{s,p}(\Reals^n)} = \|\eta_1 w\|_{H^{s,p}(\Reals^n)} \lesssim
\|w\|_{H^{s,p}(\Reals^n)} \le 2 \|\Phi_i^* u\|_{H^{s,p}(B_r^\Hyp)}.
\end{equation}
The alternative norms \eqref{eq:cutoff-free-H} and \eqref{eq:cutoff-free-X}
now follow from inequalites \eqref{eq:cutoff-local-outer} and
\eqref{eq:cutoff-local-inner} along with the norm equivalence just established for
different choices of cutoff functions.
\end{proof}

\subsection{Elementary properties}\label{secsec:funtion-space-elementary}

A standard diagonalization argument using Lemma \ref{lem:Hbuilder}
as a component shows that the spaces
$H^{s,p}(M;E)$ and $X^{s,p}(M;E)$ are complete.

In light of the equivalent norms \eqref{eq:cutoff-free-H} and
\eqref{eq:cutoff-free-X} from Lemma \ref{lemma:cutoffequiv},
the elementary inclusions from Lemma \ref{lem:basic-inclusions}
as well as Sobolev embedding from Proposition \ref{prop:SobolevEmbedding}
are established using standard embedding properties of Sobolev
spaces \cite{triebel2010theory} and identical techniques as used
for weighted spaces with integral degrees of differentiability (see, e.g.,
 \cite{Lee-FredholmOperators} Lemma 3.6(b,c) and its associated
 references).  For example, suppose $u\in X^{s,p}_{\delta}(M;E)$.
 Then if $\delta' + (n-1)/p < \delta$ we find
 \[
\sum_{i} \rho_i^{-\delta' p} \|\Phi_i^* u\|_{H^{s,p}(B_1^\Hyp)}^p
\lesssim \|u\|_{X^{s,p}_{\delta}(M;E)}^p \sum_i \rho_i^{(\delta-\delta')p}
\lesssim \|u\|_{X^{s,p}_{\delta}(M;E)}^p \|\rho^{\delta-\delta'}\|_{L^1(M;\mathbb R)}.
\]
As $\rho^a$ is integrable on $M$ precisely when $a>n-1$, and in
particular for $a=\delta-\delta'$, we conclude that $X^{s,p}_{\delta}(M;E)$
embeds into $H^{s,p}_{\delta'}(M;E)$.

Proposition \ref{prop:multiplication} concerning multiplication
is similarly proved
using the alternative norms
\eqref{eq:cutoff-free-H} and
\eqref{eq:cutoff-free-X} along with the analogous multiplication results
for unweighted spaces from, e.g., the appendix of \cite{HMT2020}.

\begin{lemma}[Rellich Lemma]\label{lem:Rellich}
Let $E$ be a tensor bundle over $M$.
Suppose $1<p<\infty$, $s,s'\in\Reals$ and $\delta,\delta'\in\Reals$.
Then
\begin{equation*}
\begin{aligned}
&H^{s,p}_\delta(M;E) \to  H^{s^\prime, p}_{\delta^\prime}(M;E),
&\quad & s> s^\prime,\quad \delta > \delta^\prime;
\\
&X^{s,p}_\delta(M;E) \to  X^{s^\prime, p}_{\delta^\prime}(M;E),
& &s> s^\prime,\quad \delta > \delta^\prime.
\end{aligned}
\end{equation*}
\end{lemma}
\begin{proof}
Suppose $\{u_k\}$ is a bounded sequence in $H^{s,p}_\delta(M)$
and let $B$ be a bound for the norm of the elements of the sequence.
By a diagonalization argument and the usual Rellich Lemma on unweighted spaces
we can reduce to a subsequence and
assume that each $\chi\Phi_i^* u_k$ converges to some $w_i\in H^{s',p}(\Reals^n)$.
Since each $w_i$ is supported in some $B_r^\Hyp$, we can extend $\Phi_i^* w_i$
by zero to all of $M$ and
$w=\sum_i \Phi_i^* w_i$ is a distribution on $M$.  Fatou's Lemma and
Lemma \ref{lem:basic-inclusions} imply
\[
\sum_i \rho_i^{-\delta p}\|w_i\|_{H^{s',p}(\Reals^n)}^p \le \lim_n
\sum_i \rho_i^{-\delta p}\|\chi\Phi_i^* u\|_{H^{s',p}(\Reals^n)}^p
\lesssim B^p
\]
and we conclude from  Lemma \ref{lem:Hbuilder} that $w\in H^{s',p}_\delta(M)$
with $\|w\|_{H^{s',p}_\delta(M;\Reals)}\lesssim B$.

Recalling the function $\bar\chi$ from equation \eqref{eq:Xdef}
we claim that $\bar\chi u_n\to w$ in $H^{s',p}_{\delta'}(M)$ for $\delta'<\delta$.
Indeed, let $\epsilon>0$
and pick $K$ so that if $i\ge K$ then $\rho_i^{\delta-\delta'}<\epsilon$.
Using Lemma \ref{lem:Hbuilder} we find
\[
\sum_{i\ge I} \rho_i^{-\delta' p}\|w_i-\chi \Phi_i^* u_n\|^p_{H^{s',p}(\Reals^n)}
\lesssim \left(\sup_{i\ge I}\rho_i^{(\delta-\delta')p}\right) \|w-\bar\chi u_n\|_{H^{s',p}(M)}^p
\lesssim \epsilon^p B^p.
\]
On the other hand, by taking $n$ sufficiently large we can ensure
\[
\sum_{i<I} \rho_i^{-\delta' p}\|w_i-\chi \Phi_i^* u_n\|^p_{H^{s,p}(\Reals^n)} < \epsilon^p
\]
and hence for $n$ sufficiently large
\[
\|w-\bar\chi u_n\|_{H^{s,p}_{\delta'}(M;E)} \lesssim \epsilon.
\]
That is, $\bar\chi u_n\to w$ in $H^{s',p}_{\delta'}(M)$ and
therefore $u_n\to (\bar\chi)^{-1} w$ in the same topology.

The analogous proof in the Gicquaud-Sakovich case is similar and is indeed
easier because there is no tail to manage.
\end{proof}

\subsection{Density of smooth tensors}
The following result, a variation of \cite{GicquaudSakovich} Lemma 2.1,
establishes the density of certain classes of smooth tensors.
\begin{proposition}\label{prop:density}
Suppose $E$ is a tensor bundle over $\bar M$.
Let $s,\delta\in\Reals$ and $1<p<\infty$.
\begin{enumerate}
\item\label{part:Hsp-smooth-dense}
The compactly supported smooth sections of $E$
are dense in $H^{s,p}_{\delta}(M;E)$.
\item\label{part:Xsp-smooth-dense} The
set $C^\infty_{\rm loc}(M;E)\cap X^{s,p}_{\delta}(M;E)$ is dense in
$X^{s,p}_{\delta}(M;)$.
\end{enumerate}
\end{proposition}
\begin{proof}
Suppose $u\in X^{s,p}_{\delta}(M)$ and let $\epsilon>0$.
For each $i$, let $w_i$ be a smooth compactly supported tensor field on $\Reals^n$ such that
\begin{equation*}
\rho_i^{-\delta p}\|\chi \Phi_i^* u -w_i \|_{H^{s,p}(\Reals^n)} < \epsilon.
\end{equation*}
Since $\chi \Phi_i^* u$ is compactly supported in $B^\Hyp_{r}$ for some $r<2$
we can assume without loss of generality that each $w_i$ is supported in $B^\Hyp_{r'}$
for some $r<r'<2$ as well.
Observe that
$\rho_i^{-\delta}\|w_i\|_{H^{s,p}(\Reals^n)}\le \|u\|_{X^{s,p}_{\delta}}+\epsilon$.
Lemma \ref{lem:Hbuilder}\eqref{part:builderX} then implies that
 $w=\sum_i (\Phi_i)_* w_i\in X^{s,p}(M)$ and it is evidently smooth.
 Moreover, Lemma \ref{lem:Hbuilder}\eqref{part:builderX} also implies
 \[
\|\bar\chi u -w\|_{X^{s,p}_{\delta}} \lesssim \sup_{i} \rho_i^{-\delta }\| \chi\Phi_i^* u - w_i\|_{H^{s,p}(\Reals^n)} < \epsilon.
\]
Since multiplication by $\bar\chi^{-1}$ is a continuous automorphism
of $X^{s,p}_\delta(M)$ we conclude
\begin{equation*}
\|u - \bar\chi^{-1} w\|_{X^{s,p}_{\delta}(M)} \lesssim \epsilon.
\end{equation*}
Noting that $\bar\chi^{-1} w$ is smooth,
the proof of part \eqref{part:Xsp-smooth-dense} is complete.

Part \eqref{part:Hsp-smooth-dense}
is proved similarly, noting that we first ignore a tail satisfying
\[
\sum_{i\ge K}\rho_i^{-\delta p}\|\chi \Phi_i^* u\|_{H^{s,p}(\Reals^n)}^p<\epsilon^p
\]
and smoothly approximate finitely many terms to obtain a compactly supported
smooth approximating function.
\end{proof}

As a consequence of Proposition \ref{prop:density}, $C^{k,\alpha}_{\delta}(M;\Reals)$
is dense in $H^{k,p}_{\delta}(M;\Reals)$ for each $k\in\Nats_{\ge 0}$.  The question of whether $C^{k,\alpha}_{\delta}(M;\Reals)$
is also dense in $X^{k,p}_{\delta}(M;\Reals)$ was raised in \cite{GicquaudSakovich}, and we show here that this is false in general.
\begin{proposition}\label{prop:Cinf-delta-not-dense-in-X}
Let $1<p<\infty$. For each $0<\alpha<1$, $C^{0,\alpha}_0(M;\Reals)$ is not dense in $X^{0,p}_0(M;\Reals)$.
\end{proposition}
\begin{proof}
Consider a sequence of sets $U_k\subset B_{1/2}^\Hyp$ with Euclidean
measures $|U_k|=|B_{1/2}^\Hyp|/k$
and let $a_k=|U_k|^{-1/p}$.  Let $u_k$ be a function on $B_{1/2}^\Hyp$ that equals
$a_k$ on $U_k$ and is otherwise zero and observe that $\|u_k\|_{L^p(B_{1/2}^\Hyp)}=1$.

Now suppose we have a second sequence of functions $v_k\in C^{0,\alpha}(B_{1/2}^\Hyp)$
such that $\|v_k-u_k\|_{L^p(B_{1/2}^{\Hyp})}<\min\left(1/2,(|B_{1/2}^\Hyp|/2)^{1/p}\right)$ for every $k$.
We claim that for $k\ge 2$ that $\min v_k < 1$ and that $\max v_k > a_k/2$.  Indeed,
if $\min v_k\ge 1$ then an easy computation using the condition $k\ge 2$
shows $\|v_k-u_k\|_{L^p(B_{1/2}^{\Hyp})}\ge (|B_{1/2}^\Hyp|/2)^{1/p}$.  On the other
hand if $\max v_k < a_k/2$ a similar computation shows
that $\|v_k-u_k\|_{L^p(B_{1/2}^{\Hyp})}\ge 1/2$.  As a consequence of these constraints,
$\|v_k\|_{C^{0,\alpha}(B_{1/2}^\Hyp)}\to\infty$ as $k\to\infty$.

The result now follows from selecting M\"obius parametrizations $\Phi_{i_k}$ with
disjoint images and a function $u\in X^{0,p}_0(M)$ with $\Phi_{i_k}^* u = u_k$ on $B_{1/2}^\Hyp$.
The argument above shows that if $v\in X^{0,p}_0(M;\Reals)$ is locally H\"older continuous and if
$\|u-v\|_{X^{0,p}_0(M)}$ is sufficiently small then $\|v\|_{C^{0,\alpha}_0(M)}=\infty$.
\end{proof}

\subsection{Interpolation}\label{secsec:interpolation}
We rapidly review some facts from complex interpolation theory from which, unless otherwise
noted, can be found in \cite{triebel2010theory}.
An interpolation couple is a pair $(X_0,X_1)$ of Banach spaces that
are continuously embedded subspaces of a common Hausdorff topological vector space.
For any $0<\theta<1$ there is an interpolation Banach space $[X,Y]_\theta$
that satisfies continuous inclusions
\begin{equation*}
X_0\cap X_1 \subset [X_0,X_1]_\theta \subset X_0+X_1
\end{equation*}
where $X_0\cap X_1$ and $X_0+X_1$ are endowed with the
natural optimal topologies such that each inclusion
$X_0\cap X_1\hookrightarrow X_i\hookrightarrow X_0+X_1$ is continuous.
Moreover
\begin{equation*}
\|x\|_{[X_0,X_1]_\theta} \le \|x\|_{X_0}^{1-\theta} \|x\|_{X_1}^\theta
\end{equation*}
for all $x\in X_0\cap X_1$.

If $(X_0,X_1)$ and $(Y_0,Y_1)$ are interpolation couples and
$T_i:X_i \rightarrow Y_i$, $i=0,1$, are continuous linear maps with
$T_0=T_1$ on $X_0\cap X_1$ then we have an associated continuous linear map
$T:[X_0,X_1]_\theta \rightarrow [Y_0,Y_1]_\theta$ for all $\theta\in (0,1)$.

Given a Banach space $X$, let $cX$ be the same set $X$ with its norm scaled by $c$.
If $X$ and $Y$ are Banach spaces and $X'=a X_1$ and $Y' = b Y$ for
$a,b>0$ then $[X',Y']_{\theta} = a^{1-\theta} b^\theta [X,Y]$; see \cite{Triebel:1976hk} Lemma 5.

Given a sequence $Y=\{Y_j\}_{j=1}^\infty$ of Banach spaces let
$\ell^p(\{Y_j\}_{j=1}^\infty)$ be the set of sequences $y\in\prod Y_j$
such that
\begin{equation*}
\sum_j \|y_j\|_{Y_j}^p < \infty.
\end{equation*}
This is a Banach space with the obvious norm.  If $X=\{X_j\}_{j=1}^\infty$ and
$Y=\{Y_j\}_{j=1}^\infty$ are sequences of Banach spaces, $1<p_1,p_2<\infty$ and $0<\theta<1$ then
\begin{equation*}
[\ell^{p_1}(X),\ell^{p_2}(Y)]_{\theta} = \ell^p([X,Y]_{\theta})
\end{equation*}
where $(1/p)=(1-\theta)(1/p_1) + \theta (1/p_2)$ and where $[X,Y]_{\theta}$ is the
sequence of spaces with elements $([X,Y]_{\theta})_j = [X_j,Y_j]_\theta$.

The following interpolation theorem for weighted Bessel potential spaces
is modeled on \cite{Triebel:1976hk} Theorem 3
for weighted function spaces on Euclidean $\Reals^n$.
\begin{theorem} Let $E$ be a tensor bundle over $\bar M$.
Consider spaces $H^{s_1,p_1}_{\delta_1}(M;E)$, $H^{s_2,p_2}_{\delta_2}(M;E)$
with $s_1,s_2\in\Reals$, $1<p_1,p_2<\infty$, and $\delta_1,\delta_2\in\Reals$.  Let $\theta\in(0,1)$ and let
\begin{equation}\label{eq:Hsp-interp}
\begin{aligned}
s &= (1-\theta) s_1 + \theta s_2\\
p &= (1-\theta) \frac{1}{p_1} + \theta \frac{1}{p_2}\\
\delta &= (1-\theta) \delta_1 + \theta \delta_2.
\end{aligned}
\end{equation}
Then
\begin{equation*}
[H^{s_1,p_1}_{\delta_1}(M;E),H^{s_2,p_2}_{\delta_2}(M;E)]_\theta = H^{s,p}_{\delta}(M;E).
\end{equation*}
\end{theorem}
\begin{proof}
For simplicity we treat the case $E=\Reals$; the general case is proved
verbatim with some additional notation needed to denote spaces of vector-valued functions
over $\Reals^n$.

For the moment fix $t,\sigma\in\Reals$ and $1<p<\infty$. For $j\in\Nats$ define
\begin{equation*}
H^{t,q}_{\sigma,(j)}(\Reals^n) = \rho_j^{-\sigma} H^{t,q}(\Reals^n)
\end{equation*}
and let
\begin{equation*}
Z^{t,q}_{\sigma} = \ell^q(\{H^{t,q}_{\sigma,(j)}(\Reals^n)\}_{j=1}^\infty).
\end{equation*}

Define $S_j: H^{t,q}_{\sigma}(M) \rightarrow H^{t,q}_{\sigma,(j)}(\Reals^n)$ by
\begin{equation*}
S_j(u) = \chi \Phi_j^* u.
\end{equation*}
A straightforward computation shows
\begin{equation*}
u\mapsto (S_1 u, S_2 u, \ldots )
\end{equation*}
defines a continuous linear map $S:H^{t,q}_{\sigma}(M)\rightarrow Z^{t,q}_{\sigma}$.

Let $\eta$ be a cutoff function that equals 1 on the support of $\chi$
and equals $0$ outside $B_2$ and define
\begin{equation*}
T:Z^{t,q}_\sigma \rightarrow  H^{t,q}_{\sigma}(M)
\end{equation*}
by
\begin{equation*}
T(w_1,w_2,\ldots) = \sum_i \Phi_{i,*} (\eta w_i).
\end{equation*}
Lemma \ref{lem:Hbuilder} implies $T$ is continuous. Since $\eta\chi =\chi$,
\begin{equation*}
(T\circ S)(u) = \overline\chi \cdot u,
\end{equation*}
where $\overline\chi$ is defined in equation \eqref{eq:Xdef}.  Finally let $ R u = \overline\chi^{-1}\cdot T u$,
which is again a continuous linear map
\begin{equation*}
Z^{t,q}_{\sigma} \rightarrow  H^{t,q}_{\sigma}(M)
\end{equation*}
and moreover $R \circ S = \mathrm{Id}$.

We now apply interpolation to the maps $R$ and $S$. From \cite{Triebel:1976hk} Lemma 6,
$R$ and $S$ determine continuous linear maps
\begin{equation*}
\begin{aligned}
S&:[H^{s_1,p_1}_{\delta_1}(M),H^{s_2,p_2}_{\delta_2}(M)]_{\theta}
\rightarrow [Z^{s_1,p_1}_{\delta_1},Z^{s_2,p_2}_{\delta_2}]_{\theta}\\
R&:
[Z^{s_1,p_1}_{\delta_1},Z^{s_2,p_2}_{\delta_2}]_{\theta}
\rightarrow
[H^{s_1,p_1}_{\delta_1}(M),H^{s_2,p_2}_{\delta_2}(M)]_{\theta}
\end{aligned}
\end{equation*}
such that $R\circ S=\mathrm{Id}$.

From interpolation of vector-valued $\ell^q$ spaces we find
\begin{equation*}
\begin{aligned}\relax
[Z^{s_1,p_1}_{\delta_1},Z^{s_2,p_2}_{\delta_2}]_{\theta} &=
[\ell^{p_1}(\{H^{s_1,p_1}_{\delta_1,(j)}(\Reals^n)\}_j),
 \ell^{p_2}(\{H^{s_2,p_2}_{\delta_2,(j)}(\Reals^n)\}_j)]_{\theta}\\
 & =
\ell^p(\{[H^{s_1,p_1}_{\delta_1,(j)}(\Reals^n),
          H^{s_2,p_2}_{\delta_2,(j)}(\Reals^n)]_\theta\}_p ).
\end{aligned}
\end{equation*}
Moreover,
\begin{equation*}
\begin{aligned}\relax
[H^{s_1,p_1}_{\delta_1,(j)}(\Reals^n),
          H^{s_2,p_2}_{\delta_2,(j)}(\Reals^n)]_\theta
&= [\rho_j^{-\delta_1}H^{s_1,p_1}(\Reals^n),\rho_j^{-\delta_2}H^{s_2,p_2}(\Reals^n)]_{\theta}\\
&= \rho_j^{-[(1-\theta)\delta_1+\theta\delta_2]}[H^{s_1,p_1}(\Reals^n),H^{s_2,p_2}(\Reals^n)]_{\theta} \\
&= \rho_j^{-\delta} H^{s,p}(\Reals^n).
\end{aligned}
\end{equation*}
That is, $[Z^{s_1,p_1}_{\delta_1},Z^{s_2,p_2}_{\delta_2}]_{\theta} = Z^{s,p}_{\delta}$.

Now $S:[H^{s_1,p_1}_{\delta_1}(M),H^{s_2,p_2}_{\delta_2}(M)]_{\theta}\rightarrow Z^{s,p}_{\delta}$
is an isomorphism onto its image since it admits a continuous inverse, namely
the restriction of $R$. Hence
\begin{equation}\label{eq:norminterp}
\|u\|_{[H^{s_1,p_1}_{\delta_1}(M),H^{s_2,p_2}_{\delta_2}(M)]_{\theta}}^p
\sim \sum_i \rho_i^{-\delta p} \|\chi \Phi_i^* u\|_{H^{s,p}(\Reals^n)}^p = \|u\|_{H^{s,p}_{\delta}(M)}^p.
\end{equation}
Equation \eqref{eq:norminterp} establishes equivalence of norms, but leaves open the possibility that
the interpolation space $[H^{s_1,p_1}_{\delta_1}(M),H^{s_2,p_2}_{\delta_2}(M)]_{\theta}$
is a closed subspace
of $H^{s,p}_{\delta}(M)$.  However, Proposition \ref{prop:density} shows that
the smooth compactly supported functions are contained in
$H^{s_1,p_1}_{\delta_1}(M)\cap H^{s_2,p_2}_{\delta_2}(M)\subset [H^{s_1,p_1}_{\delta_1}(M),H^{s_2,p_2}_{\delta_2}(M)]_{\theta}$ and are dense in $H^{s,p}_{\delta}(M)$.
\end{proof}

Gicquaud-Sakovtich spaces are more technical from the perspective of interpolation
than the weighted Bessel potential spaces are
because they are based on $\ell^\infty$ rather than $\ell^p$ with $1<p<\infty$.
The following partial result suffices for the purposes of our work.
\begin{proposition}
Let $E$ be a tensor bundle over $\bar M$.
Suppose $(A_0, A_1)$ is an interpolation couple and that $T$
is a linear map
\[
T: A_0 + A_1 \to X^{s_0,p_0}_{\delta_0}(M;E) + X^{s_1,p_1}_{\delta_1}(M;E)
\]
for some $s_0,s_1\in\Reals$, $\delta_0,\delta_1\in\Reals$, and $1<p_0,p_1<\infty$. If $T$ restricts to continuous maps $T_k:A_k\to X^{s_k,p_k}_{\delta_k}(M;E)$ for $k=0,1$, then for each $\theta\in (0,1)$, then $T$
also restricts to a continuous map
\[
T_\theta: [A_0, A_1]_\theta \to X^{s,p}_{\delta}(M;E),
\]
where $s$, $p$ and $\delta$ are given by equations \eqref{eq:Hsp-interp}. Moreover,
\[
\|T_\theta\|\le \|T_0\|^{1-\theta} \|T_1\|^\theta
\]
with implicit constant independent of $T_0$, $T_1$.
\end{proposition}
\begin{proof}
Consider a preferred M\"obius chart $\Phi_i$ and let
\[
T_{(i)} =  \chi\cdot(\Phi_i^*\circ T).
\]
Recall that if $Y$ is a normed space and $r>0$, the $rY$
norm is $\|\cdot\|_{rY}= r\|\cdot\|_{Y}$.
Then $T_{(i)}$ determines continuous linear maps
$A_k\to \rho_i^{-\delta_k} H^{s_k,p_k}(\Reals^n)$ for $k=0,1$
with norm bounded by $\|T_k\|$; note that $H^{s_k,p_k}(\Reals)$
denotes a space of vector-valued functions here.
It therefore also determines a continuous linear map
\[
[A_0,A_1]_{\theta} \to [\rho_i^{-\delta_0} H^{s_0,p_0}(\Reals^n),
\rho_i^{-\delta_1} H^{s_1,p_1}(\Reals^n)]_\theta =
\rho_i^{-\delta}H^{s,p}(\Reals^n)
\]
with norm bounded by $\|T_0\|^{1-\theta}\|T_1\|^{\theta}$.
But the norm of $T:[A_0,A_1]_\theta\to X^{s,p}_{\delta}(M)$
is exactly the supremum of all such norms.
\end{proof}

\subsection{Duality}

Let $g_{\mathbb E}$ be the Euclidean metric on $\Reals^n$.  The $L^2$ inner product
\begin{equation*}
\left<v,u\right>_{(\Reals^n, g_{\mathbb E})} = \int_{\Reals^n} vu\; dV_{g_{\mathbb E}}
\end{equation*}
on smooth, compactly supported functions $v$ and $u$ extends to a continuous bilinear
form on $H^{-s,p\dual}(\Reals^n)\times H^{s,p}(\Reals^n)$ where $\frac 1 p +\frac 1 {p\dual} = 1$,
and the map
$v\mapsto \left<v,\cdot\right>_{\Reals^n}$ is an isomorphism between
$H^{-s,p}(\Reals^n)$ and $(H^{s,p}(\Reals^n))\dual$.  In this section
we establish analogous results for weighted Bessel potential spaces. Because
Gicquaud-Sakovich spaces are based on $\ell^\infty$ their duality theory
is more complex and is beyond the scope of our work.

Let $\check g$ denote a metric on $M$ with smooth
compactification $\rho^2\check g$ on $\bar M$
and let $dV_{\check g}$ be its volume element.  We use the notation
\begin{equation*}
\left<v,u\right>_{(M,\check g)} = \int_M vu\; dV_{\check g}
\end{equation*}
for smooth compactly supported functions $v$ and $u$ on $M$.
In this section we show that $\left<\cdot,\cdot\right>_{(M,\check g)}$ induces a corresponding
isomorphism between $H^{-s,p\dual}_{-\delta}(M;\Reals)$ and $(H^{s,p}_{\delta}(M;\Reals))\dual$.
Theorem \ref{thm:dual-H-gtb} in \S\ref{secsec:geometric-E}
builds on this result to obtain a duality statement for general geometric tensor bundles
and less regular metrics.

We begin by showing the map $\ip<\cdot,\cdot>_{(M,\check g)}$ is well defined between
nominally dual weighted spaces.
\begin{lemma}\label{lem:dual-makes-sense}
Let $s,\delta\in\Reals$ and $1<p<\infty$ and let $p\dual$ be the dual exponent of $p$.
Then
\[
\left<v,u\right>_{(M,\check g)} \lesssim \|v\|_{H^{-s,p\dual}_{-\delta}(M;\Reals)} \|u\|_{H^{s,p}_{\delta}(M;\Reals)}\\
\]
for all $u,v\in C_{\mathrm c}^\infty(M;\Reals)$ and hence
$\left<\cdot,\cdot\right>_{(M,\check g)}$ extends to a continuous bilinear form
on $H^{-s,p\dual}_{-\delta}(M;\Reals)\times H^{s,p}_{\delta}(M;\Reals)$.
\end{lemma}
\begin{proof}
For each M\"obius parametrization $\Phi_i$ define $\omega_i$ by $\Phi_i^* dV_{\check g} = \omega_i dV_{g_{\mathbb E}}$.
Since $g$ admits a smooth compactification, an easy computation shows that for any $k\in\Nats$,
\begin{equation}\label{eq:omega-i-est}
\sup_{i} \|\omega_i\|_{C^k(B_2)} <\infty,\qquad \sup_{i} \|(\omega_i)^{-1}\|_{C^k(B_2)} <\infty.
\end{equation}
The usual change of variables formula shows that if $u$ and $v$ are smooth functions
supported in $\Phi_i(B_2^\Hyp)$ then
\[
\left<v,u\right>_{(M,\check g)} = \left< \Phi_i^* v,\, \omega_i \Phi_i^* u\right>_{(\Reals^n,g_{\mathbb E})}.
\]

Suppose $u$ and $v$ are smooth compactly supported functions. Let $\eta$ be a cutoff
function on $B_2^\Hyp$ that equals 1 on the support of $\chi$, and pick an integer $k>|s|$. Then
recalling the function $\bar\chi$ from equation \eqref{eq:Xdef} we find
\[
\begin{aligned}
\left|\left<v,u\right>_{(M,\check g)}\right| &= \left|\left<\overline\chi v,(\overline \chi)^{-1}u\right>_{(M,\check g)}\right|\\
&\le \sum_{i} \left|\left< \chi_i v, (\overline \chi)^{-1} u\right>_{(M,\check g)}\right|
\\
&= \sum_{i} \left|\left< \chi \Phi_i^*v, \omega_i \eta \Phi_i^* ((\overline\chi)^{-1} u) \right>_{(\Reals^n,g_{\mathbb E})}\right|
\\
&\lesssim \sum_{i} \|\chi \Phi_i^* v \|_{H^{-s,p\dual}(\Reals^n)}
\|\omega_i\|_{C^{k}(B_2)} \| \eta \Phi_i^*( (\overline\chi)^{-1} u) \|_{H^{s,p}(\Reals^n)}\\
&\le \|v\|_{H^{-s,p\dual}_{-\delta}}
\left[\sum_{i} \rho_i^{-\delta p} \| \eta \Phi_i^*((\overline \chi)^{-1}) u) \|_{H^{s,p}(\Reals^n)}\right]^{1/p}
\end{aligned}
\]
where we have used Sobolev duality on $\Reals^n$ along with uniform $C^k$ bounds on $\omega_i$.
To finish the proof it suffices to show that
\begin{equation*}
\left[\sum_i \rho_i^{-\delta p} \| \eta\, \Phi_i^* ( (\overline\chi)^{-1} u) \|_{H^{s,p}(\Reals^n)}^p\right]^{1/p} \lesssim
\|u \|_{H^{s,p}_{\delta}(M)}.
\end{equation*}
But this follows from Lemma \ref{lemma:cutoffequiv} and our earlier observation
that $u\mapsto (\overline \chi)^{-1}u$ is continuous on $H^{s,p}_{\delta}(M)$.
\end{proof}

\begin{theorem}\label{thm:Hdual}
Let $s,\delta\in\Reals$ and $1<p<\infty$ and let $p\dual$ be the Lebesgue dual exponent of $p$.
The map
\begin{equation}\label{eq:dualmap}
v\mapsto \left<v,\cdot\right>_{(M,\check g)}
\end{equation}
is an isomorphism between $H^{-s,p\dual}_{-\delta}(M;\Reals)$ and $(H^{s,p}_{\delta}(M;\Reals))\dual$.
\end{theorem}
\begin{proof}
Suppose $F\in (H^{s,p}_{\delta}(M))\dual$. Fix $i\in\Nats$ for the moment
and define $F_i\in (H^{s,p}(\Reals^n))\dual$ by
\begin{equation*}
F_i(u) = F((\Phi_i)_*(\chi u) ) = F(\chi_i (\Phi_i)_*u).
\end{equation*}
Since $F_i$ is supported on $B_{r}^\Hyp$ for some $r<2$
there exists $w_i\in H^{-s,p\dual}(\Reals^n)$ supported on $B_{r'}^{\Hyp}$, where $r'<2$, such that
\begin{equation*}
F_i(u) = \left< w_i, u\right>_{(\Reals^n,g_{\mathbb E})}
\end{equation*}
for all $u\in H^{s,p}(\Reals^n)$.  Using the usual isomorphism between
$H^{-s,p*}(\Reals^n)$ and
$(H^{s,p}(\Reals^n))\dual$ there exists
a constant $c>0$ arising from
the isomorphism and a function
$z_i\in H^{s,p}(\Reals^n)$ such that
$\|z_i\|_{H^{s,p}(\Reals^n)}=1$ and such that
\begin{equation*}
\left<w_i, z_i\right>_{(\Reals^n,g^{\mathbb E})} \ge c \|w_i\|_{H^{-s,p\dual}(\Reals^n)}.
\end{equation*}
Moreover, since $w_i$ is supported in $B_{r'}^{\Hyp}$ we can assume,
after multiplication by a cutoff function depending only on $r'$, that
$z_i$ is supported in $B_{r''}^\Hyp$ for some $r''$, and that the functions
$z_i$ remain uniformly bounded in $H^{s,p}(\Reals^n)$.
As in the proof of Lemma \ref{lem:dual-makes-sense} define
$\omega_i$ by $\Phi_i^* dV_{\check g} = \omega_i dV_{g_{\mathbb E}}$
and let
\begin{equation*}
\begin{aligned}
v_i &= (\Phi_{i})_* ( \omega_i^{-1} w_i)\\
u_i &= (\Phi_{i})_* z_i.
\end{aligned}
\end{equation*}
Recalling that $\omega_i^{-1}$ admits the uniform bounds
\eqref{eq:omega-i-est} we find that
$v_i$ and $u_i$ are well-defined elements of $H^{-s,p\dual}_{0}(M)$
and $H^{s,p}_{0}(M)$ respectively.

Partition $\Nats$ into sets finitely many sets $N_a$ such that
if $i,j\in N_a$ with $i\neq j$ then $\Phi_i(B_2^\Hyp)\cap \Phi_j(B_2^\Hyp)=\emptyset$;
the existence of such a partition
is an easy consequence of local finiteness.  Fix a partition index $a$ and let $\{b_i\}_{i\in N_a}$ be a sequence of real numbers such that
$\sum_{i\in N_a} |b_i|^p \le 1$.  Let $u_{(a)} = \sum_{i\in N_a} b_i\rho_i^{\delta} u_i$.
Then
\begin{equation*}
\sum_{i\in N_a} \rho_i^{-\delta p}
\| b_i \rho_i^{\delta} \Phi_i^* u_i\|_{H^{s,p}(\Reals^n)}^p
\lesssim  \sum_{i\in N_a} |b_i|^p \lesssim 1
\end{equation*}
where we have used uniform $H^{s,p}(\Reals^n)$ bounds for the functions $z_i=\Phi_i^*u_i$.
Lemma \ref{lem:Hbuilder} implies that $u_{(a)}$ belongs to $H^{s,p}_{\delta}(\Reals^n)$
and its norm is bounded independent of the choice of sequence $\{b_j\}$.
From continuity of the map $F$ and this uniform bound we find
\begin{equation*}
\begin{aligned}
\|F\| \gtrsim F(u_{(a)}) &= F\left(\sum_{i\in N_a} \chi_i \rho_i^\delta b_i u_i \right) \\
&= \sum_{i\in N_a} \rho_i^\delta b_i F_i(  \Phi_i^* u_i)\\
&= \sum_{i\in N_a} \rho_i^\delta b_i \left<w_i, z_i\right>_{(\Reals^n,g^{\mathbb E})}\\
&\ge  \sum_{i\in N_a} \rho_i^\delta b_i \alpha \| w_i\|_{H^{-s,p\dual}(\Reals^n)},
\end{aligned}
\end{equation*}
where the second equality relies on the fact that we only sum over $i\in N_a$.
Since the choice of sequence $\{b_i\}\in \ell^p$ with norm less than one is arbitrary
we conclude $\{ \rho_i^\delta \| w_i\|_{H^{-s,p}(\Reals^n)} \}_{i\in N_a}\in \ell^{p\dual}$.
Recall that $v_i=(\Phi_i)_*( \omega_i^{-1} w_i)$.  The uniform $C^k$ bounds on the
functions $(\omega_i)^{-1}$ and Lemma \ref{lem:Hbuilder} imply
$v_{(a)} := \sum_{i\in N_a} v_i \in H^{-s,p\dual}_{-\delta}(M)$
and consequently $v$ defined by
\begin{equation*}
v = \sum_{a} v_{(a)} = \sum_i v_i
\end{equation*}
also belongs to $H^{-s,p\dual}_{-\delta}(M)$.
Moreover, for any $u\in H^{s,p}_{\delta}(M)$
\begin{equation*}
\begin{aligned}
\left<v,u\right>_{(M,\check g)} &= \sum_{i} \left<v_i,u\right>_{(M,\check g)}\\
&= \sum_{i} \left< \omega_i^{-1} w_i, \eta \omega_i \Phi_i^*  u\right>_{(\Reals^n,g_{\mathbb E})}\\
&= \sum_{i} F( \chi_i  u ) = F(\overline\chi u )
\end{aligned}
\end{equation*}
where $\eta$ is a cutoff function on $B_2^\Hyp$ with $\eta\chi=1$ and with
$\eta w_i=w_i$ for all $i$.
But then $\overline\chi ^{-1} v \in H^{-s,p\dual}_{\delta}(M;\Reals)$ and
\begin{equation*}
\left<\overline\chi^{-1} v,  u\right>_{(M,\check g)} = F(u)
\end{equation*}
for all $u\in H^{s,p}_{\delta}(M)$.

We have now established surjectivity of the map in \eqref{eq:dualmap}.
To show the map is injective, suppose $v\neq 0$.
Then $\chi \Phi_i^* v\neq 0$ for some $i$ and we can find $w\in H^{s,p}(\Reals^n)$,
supported in $B_2$, such that $\left<\chi \Phi_i^* v,w\right>_{\Reals^n}\neq 0$.
Let $u=(\Phi_i)_*((\omega_i)^{-1}\chi w)$. Then
\begin{equation*}
\left<v,u\right>_{(M,\check g)} = \left<\chi \Phi_i^* v, \omega_i (\omega_i)^{-1} w\right>_{(\Reals^n,g_{\mathbb E})} =
\left<\Phi_i^* v, \chi u\right>_{\Reals^n} \neq 0.
\end{equation*}
Thus the map is a linear isomorphism \eqref{eq:dualmap} and its continuity
follows from Lemma \ref{lem:dual-makes-sense}.
\end{proof}

\subsection{Geometric tensor bundles}\label{secsec:geometric-E}

Until this point we have only considered tensor bundles $E=T^{k_1,k_2}M$
and now turn our attention to geometric tensor bundles, which
are subbundles of tensor bundles and which potentially depend on a metric
for their definition. In fact, subbundles such as the symmetric tensors in $T^{2,0}\bar M$
which are associated with $GL(n;\Reals)$ representations
could have largely been treated earlier with very minor modifications to the proofs.
In the general case, however, the regularity of the metric plays a role in the results.
Moreover, even for tensor bundles, the
generalization of the duality result Theorem \ref{thm:Hdual} requires a choice
of (perhaps not smooth) reference metric.
Therefore, throughout this section we fix a metric $h\in X^{s_0,p_0}_0(M)$
with $1<p_0<\infty$ and $s_0>n/p_0$ (and consequently is locally H\"older continuous)
that admits a compactification $\rho^2 h$ having
an extension to a continuous metric on $\bar M$. 
We also fix a particular geometric tensor
bundle $E$ over $\bar M$ with ambient tensor bundle F.  For technical reasons
we assume additionally that $s_0\ge 1$. In the event that $E$ happens to be
a subbundle associated with a $GL(n;\Reals)$ representation
(and therefore does not require a metric for its definition)
we can simply select an arbitrary metric $h$ with smooth conformal compactification for use in the constructions below.

The regularity of the metric restricts the permissible  regularity
classes of sections of $E$. Using the notation of \S\ref{sec:differential-ops},
let $\mathcal S^{s_0,p_0}_0$ be the set of pairs $(s,p)\in \Reals\times (1,\infty)$
satisfying
\[
\begin{gathered}
-s_0 \le s \le s_0,\\
\frac{1}{p_0}-\frac{s}{n} \le \frac{1}{p}-\frac{s}{n} \le \frac{1}{p_0\dual}+\frac{s}{n}.
\end{gathered}
\]
The significance of these conditions comes from multiplication in Sobolev spaces:
since $s_0>n/p_0$, multiplication is continuous
$H^{s_0,p_0}(\Reals^n)\times H^{s,p}(\Reals^n)\to H^{s,p}(\Reals^n)$
when $(s,p)\in \mathcal S^{s_0,p_0}_0$.  
Similarly, Proposition \ref{prop:multiplication}.
shows that when $(s,p)\in \mathcal S^{s_0,p_0}_0$, multiplication is continuous
\[
\begin{aligned}
H^{s,p}_\delta(M;\Reals)\times X^{s_0,p_0}_0(M;F)&\to H^{s,p}_\delta(M;F),\\
X^{s,p}_\delta(M;\Reals)\times X^{s_0,p_0}_0(M;F)&\to X^{s,p}_\delta(M;F).
\end{aligned}
\]

\begin{definition}\label{def:gtb} Suppose $(s,p)\in S^{s_0,p_0}_0$ and $\delta\in\Reals$. Let
$E$ be a geometric tensor bundle over $\bar M$
with ambient tensor bundle $F$.  The space
$H^{s,p}_\delta(M;E)$ is the subspace of $H^{s,p}_\delta(M;F)$ consisting
of the elements $u$ for which there exist
finitely many sections
$\sigma_\alpha$ in $X^{s_0,p_0}_{0}(M;F)$ that are also
(continuous) sections of $E$, along with coefficients
$u_\alpha\in H^{s,p}_{\delta}(M;\Reals)$, such that
\[
u = \sum_{\alpha} u_{\alpha} \sigma_\alpha.
\]
The space $X^{s,p}_\delta(M;E)$ is defined similarly, with coefficients
in $X^{s,p}_{\delta}(M;\Reals)$.
The spaces $H^{s,p}_\delta(M;E)$
and $X^{s,p}_\delta(M;E)$ inherit
their norms from the ambient spaces $H^{s,p}_\delta(M;F)$
and $X^{s,p}_\delta(M;F)$, respectively.
\end{definition}

The main strategy for the remainder of this section is to find
a single collection of basis sections in $X^{s,p}_{0}(M;E)$
and then apply earlier results to coefficients with respect to this
basis. The basis sections
themselves are obtained using well-controlled local orthonormal frames
for the tangent bundle, and to build these we start with
a local coordinate version of the frame construction.
Let $\Sym_+^2(\Reals^n)$ denote the the set of symmetric, positive definite
$n\times n$ matrices which we identify with the set of metrics on $\Reals^n$.
Let $\GS:\Sym_+^2(\Reals^n)\times\GL(n,\Reals)\to\GL(n,\Reals)$ be the map taking
a metric $g$ and a matrix $A$
to the matrix with columns consisting of a $g$-orthonormal
frame obtained via the Gram-Schmidt procedure starting at
the columns of $A$.  We extend $\GS$ to sections of metrics
and frames in the obvious way and use the same notation.

The following sequence of results are used to build well controlled
local frames for $E$ using $\GS$.

\begin{lemma}\label{lem:power} Suppose $1<p<\infty$ and $s\ge 1$ with $s>n/p$,
and that $U\subset \Reals^n$ is open and bounded.
Let $I=[a,b]$ with $0<a<b$, and let 
$$Z_I=\{u\in H^{s,p}(U): u(x)\in I \text{ for all } x\in U\}.$$
For each $a\in\Reals$ the map $u\mapsto u^a$ is
globally Lipschitz continuous as a map $Z_{I}\to H^{s,p}(U)$.
\end{lemma}
\begin{proof}
We first recall a fact about nonlinearities and Sobolev spaces.
A straightforward bootstrap using the inequalities $s>n/p$ and $s\ge 1$
shows that if $F:\Reals\to\Reals$ is smooth
and $u\in H^{s,p}(U)$, then $F(u)\in H^{s,p}(U)$ with its
norm is controlled by a bound depending on $F$ and $\|u\|_{H^{s,p}(U)}$.
Alternatively, this fact is also a consequence of the more sophisticated estimates
of \cite[Theorem 1]{WS-Composition-96}
and \cite[Section 3.1]{taylor_pseudodifferential_1991}.

Suppose $u_0,u_1\in Z_{\epsilon,K}$ and let $\delta u =u_1-u_0$.  Then
\[
(u_1)^a - (u_0)^a = (u_0)^{a}\left[\left(1+\frac{\delta u}{u_0}\right)^a-1\right]
= (u_0)^{a-1} G_a\left(\frac{\delta u}{u_0}\right)\,\delta u
\]
where $G_a$ is a smooth function on $\Reals_{>0}$. Our initial remarks about nonlinearities
show that the $H^{s,p}(U)$ norms
of $u_0$, $u_0^{a-1}$, $\delta u/u_0$ and $G_a(u/u_0)$ are all uniformly bounded
by a constant depending on $a$ and $I$, and the result follows
the fact that $H^{s,p}(U)$ is an algebra.
\end{proof}

\begin{lemma}\label{lem:GS}
Let $U\subset \Reals^n$ be open and bounded, let $1<p<\infty$, and let $s\in\Reals$
with $s\ge 1$ and $s>n/p$. Let $\epsilon>0$ and $K>0$ and define
\begin{multline*}
Z_{K}=\big\{(g,A): g\in H^{s,p}(U;\Sym_+^2(\Reals^n)),\, 
A\in H^{s,p}(U;GL(n;\Reals)),
\\
\det g\geq \epsilon, \, \det A\ge \epsilon,\,
\|g\|_{H^{s,p}(U)}\le K,\,
\|A\|_{H^{s,p}(U)}\le K \big\}.
\end{multline*}
Then $GS: Z_{\epsilon,K}\to H^{s,p}(U;\GL(n,\Reals))$ and is globally Lipschitz continuous.
Moreover the map $(g,A)\mapsto (GS(g,A))^{-1}$ similarly determines a globally Lipschitz
map $Z_{K}\to H^{s,p}(U;\GL(n,\Reals))$.
\end{lemma}
\begin{proof}
The coefficients of $\GS(g,A)$ are constructed iteratively from those of $g$
and $A$ using addition, multiplication, and the map $x\mapsto x^{-1/2}$.
Since $\{(g(x),A(x)):x\in U\}$ is contained in a compact subset of
$\Sym_+^2\times\, GL(n;\Reals)$ we can
find a compact interval $I$ of positive real numbers, depending only on
$\epsilon$ and $K$,
such that each application of $x\mapsto x^{-1/2}$ needed to
compute $\GS(g(x),A(x))$ for all $x\in U$ involves
an argument in $I$.  Hence using Lemma \ref{lem:power}, and the fact that
$H^{s,p}(U)$ is an algebra, we find that $\GS:Z_{\epsilon,K}\to H^{s,p}(U;GL(n;\Reals))$
and indeed globally Lipschitz continuous.  Moreover,
$\{\GS(g(x),A(x)):x\in U\}$ is contained in a compact subset of $\GL(n,\Reals)$
determined by $\epsilon$ and $K$,
and therefore there exist bounds $C\ge c>0$, depending only on $\epsilon$ and $K$,
such that $c\le \det \GS(g(x),A(x))\le C$ for all
$x\in U$; recall that the Gram Schmidt procedure is orientation preserving.  Lemma \ref{lem:power}, and the fact that $H^{s,p}(U)$
is an algebra, then imply that $g\mapsto (\GS(g,A))^{-1}$ is also globally Lipschitz continuous.
\end{proof}

The previous lemma allows us to find a local orthonormal frame for $E$ with support in
the image of each M\"obius parametrization.
\begin{lemma}\label{lem:frame-basis}
Let $d$ be the fiber dimension of $E$, and recall the cutoff function $\chi$ used in \eqref{eq:Hspdelta-def}.
For each M\"obius parametrization $\Phi_i$ there
exist sections $\sigma_{\alpha,i} \in X^{s_0,p_0}(M;E)$, $1\le\alpha\le d$,
supported in $\Phi_i(B_2^\Hyp)$ such that for each
$x$ in a neighborhood of the support of $(\Phi_i)_*\chi$,
the collection $\{\sigma_{\alpha,a}(x)\}_{\alpha}$ is an $h$-orthonormal
basis for the fiber of $E$ over $x$.  Moreover, $\|\sigma_{\alpha,i}\|_{X^{s_0,p_0}_{0}(M;E)}$
is bounded independent of $\alpha$ and $i$.
\end{lemma}
\begin{proof}
The collection of metrics $\Phi_i^* h$ are uniformly
bounded in $H^{s,p}(B_r^\Hyp)$.  Moreover,
using the fact that $h$ has a continuous compactification,
the argument of Lemma \ref{lem:geometric-op-mapping} shows that
$\det (\Phi_i^* h)$ is uniformly bounded away from zero. 
So Lemma \ref{lem:GS} implies
that for each $i$ we can find a section of matrices $A_i\in H^{s_0,p_0}(B_r^\Hyp)$,
with columns consisting of an $h$-orthonormal frame, such that the norms
$\|A_i\|_{H^{s_0,p_0}(B_r^\Hyp)}$ and
$\|(A_i)^{-1}\|_{H^{s_0,p_0}(B_r^\Hyp)}$ are uniformly bounded in $i$.

Let $\eta$ be a cutoff function supported in some $B_r^\Hyp$, with $r<2$,
that equals $1$ in a neighborhood
of the support of $\chi$.
Let $\sigma^{\mathbb E}_{\alpha}$, $1\le \alpha\le d$ be a Euclidean orthonormal frame
for the subspace of the tensor bundle over $\Reals^n$ associated with $E$.
From the bounds on $A_i$ and $A_i^*$ we find that the sections
$(A_i)_* \sigma^{\mathbb E}_{\alpha}$ induced from the tensor bundle automorphism
determined by $A_i$ are bounded in $H^{s,p}(B_r^\Hyp)$
independent of $i$ and $\alpha$.
Lemma \ref{lem:Hbuilder} implies that the sections
$\sigma_{\alpha,i} = (\Phi_i)_* ( \eta (A_i)_* \sigma^{\mathbb E}_{\alpha})$
are uniformly bounded in $X^{s,p}_0(M;E)$.
Since $\eta\chi=\chi$,
for each $i$ the sections $\sigma_{\alpha,i}$ form a pointwise $h$-orthonormal
frame for $x$ in the support of $(\Phi_i)_*\chi$ .
\end{proof}

As in the proof of Theorem \ref{thm:Hdual},
partition $\Nats$ into finitely many sets $\{N_a\}$ with finite so that if
if $i,j\in N_a$ and if $i\neq j$ then $\Phi_i(B_2^\Hyp)\cap \Phi_j(B_2^\Hyp)=\emptyset$.
Let $\chi_{(a)} = \overline\chi^{-1}\sum_{i\in N_a}(\Phi_i)_* \chi$ so
that the collection $\chi_{(a)}$ is a partition of unity in $C^k_0(M;\Reals^n)$
for any $k$.  Similarly, for each bundle index $\alpha$ let
$\sigma_{\alpha,(a)}=\sum_{i\in N_a} \sigma_{\alpha,i}$,
so each $\sigma_{\alpha,(a)}\in X^{s_0,p_0}_0(M;E)$.
For each index $(\alpha,(a))$ we define projections
\[
\pi_{\alpha,(a)}(u) = \ip<\chi_{(a)}u, \sigma_{\alpha,(a)}>_{h}.
\]
Proposition \ref{prop:multiplication}
ensures that the projections
are continuous $H^{s,p}_{\delta}(M;F)\to H^{s,p}_{\delta}(M;\Reals)$
and $X^{s,p}_{\delta}(M;F)\to X^{s,p}_{\delta}(M;\Reals)$ so long
as $(s,p)\in \mathcal S^{s_0,p_0}_0$.

\begin{lemma}\label{lem:bundle-split}
Suppose $\delta\in\Reals$ and $(s,p)\in \mathcal S^{s_0,p_0}_0$.
If $u\in H^{s,p}_{\delta}(M;E)$ or $u\in X^{s,p}_{\delta}(M;E)$
then
\begin{equation}\label{eq:g-bundle-decomp}
u = \sum_{\alpha,a} \pi_{\alpha,(a)}(u)\, \sigma_{\alpha,(a)}.
\end{equation}
\end{lemma}
\begin{proof}
Suppose that $w\in X^{s_0,p_0}_{0}(M)$. On each fiber over the support
of $\chi_{(a)}$ the sections $\{\sigma_{\alpha,(a)}\}$ form
an orthonormal frame for $E$ and hence a pointwise computation shows
\[
\chi_{(a)} w = \sum_{\alpha} \ip<\chi_{(a)} w, \sigma_{\alpha,(a)}>_{h} \sigma_{\alpha,(a)}.
\]
But then
\[
w = \sum_{a} \chi_{(a)} w =
\sum_{a,\alpha} \ip<\chi_{(a)} w, \sigma_{\alpha,(a)}>_{h} \sigma_{\alpha,(a)}
\]
as required.  The general case now follows from linearity.
\end{proof}

At this point, the generalizations of all the results discussed in
Section \ref{secsec:funtion-space-elementary}
(completeness, the basic inclusions of Lemma \ref{lem:basic-inclusions},
Sobolev embedding (Proposition \ref{prop:SobolevEmbedding}),
Proposition \ref{prop:multiplication}) concerning multiplication,
and the Rellich Lemma \ref{lem:Rellich}) are easy consequences
of their tensor space counterparts, Lemma \ref{lem:bundle-split} and
Definition \ref{def:gtb}, and we leave these details to the reader.

Regarding density of `smooth' sections, the regularity of the metric
provides a cap to the amount of regularity that can be expected from
a dense subset.
\begin{proposition}\label{prop:density-gtb}
Suppose $(s,p)\in \mathcal S^{s_0,p_0}_0$ and $\delta\in\Reals$.
\begin{enumerate}
  \item The compactly supported elements of $H^{s_0,p_0}_{\delta}(M;E)$
  are dense in $H^{s,p}_{\delta}(M;E)$.
  \item The set $H^{s_0,p_0}_{\rm loc}(M;E)\cap X^{s,p}_\delta(M;E)$
  is dense in $X^{s,p}_{\delta}(M;E)$.
\end{enumerate}
\end{proposition}
\begin{proof}
Let $u\in H^{s_0,p_0}_{\delta}(M;E)$ and use Lemma \ref{lem:bundle-split} to decompose
\[
u = \sum_{\alpha,a} u_{\alpha,(a)}\sigma_{\alpha,(a)}
\]
for functions $u_{\alpha,(a)}\in H^{s,p}_{\delta}(M)$.
Let $\epsilon>0$ and use Proposition \ref{prop:density} to find
a smooth compactly supported function $\tilde u_{\alpha,(a)}$
with $\|u_{\alpha,(a)}-\tilde u_{\alpha,(a)}\|_{H^{s,p}_{\delta}(M;E)}<\epsilon$.
Setting $\tilde u = \sum_{\alpha,a} \tilde u_{\alpha,(a)} \sigma_{\alpha,(a)}$
we find $\tilde u\in H^{s_0,p_0}_{\delta}(M)$ and $\|u-\tilde u\|_{H^{s,p}_{\delta}(M)}\lesssim\epsilon$.

The Gicquaud-Sakovich case is proved similarly.
\end{proof}

Interpolation for sections of geometric tensor bundles
essentially follows from applying interpolation to the coefficients of a section.
\begin{proposition}
Consider spaces $H^{s_1,p_1}_{\delta_1}(M;E)$ and $H^{s_2,p_2}_{\delta_2}(M;E)$, with
$(s_1,p_1)$ and $(s_2,p_2)$ in $S^{s_0,p_0}_{0}$, and with
$\delta_1,\delta_2\in\Reals$.  Let $\theta\in(0,1)$ and let $s$, $p$, $d$ be given by \eqref{eq:Hsp-interp}.
Then $(s,p)\in S^{s_0,p_0}_0$ and
\begin{equation*}
[H^{s_1,p_1}_{\delta_1}(M;E),H^{s_2,p_2}_{\delta_2}(M;E)]_\theta = H^{s,p}_{\delta}(M;E).
\end{equation*}
\end{proposition}
\begin{proof}
For $(s,p)\in S^{s_0,p_0}_0$
define $S:H^{s,p}_\delta(M;E)\to \prod_{\alpha,a} H^{s,p}_{\delta}(M;\Reals)$
by $S(u)= (\pi_{\alpha,(a)}(u))$ and observe that $S$ is continuous.
Define $R:\prod_{\alpha,a} H^{s,p}_{\delta}(M;\Reals) \to H^{s,p}_\delta(M;E)$
to be the map taking $(u_{\alpha,(a)})$
to $\sum_{\alpha,(a)}
u_{\alpha,(a)} \sigma_{\alpha,(a)}$, which is again continuous.  Moreover,
Lemma \ref{lem:bundle-split} implies $R\circ S=\Id$.

Using the fact that
\[
\left[\prod_{\alpha,a} H^{s_1,p_1}_{\delta_1}(M;\Reals), \prod_{\alpha,a} H^{s_2,p_2}_{\delta_2}(M;\Reals)\right]_{\theta} = \prod_{\alpha,a}H^{s,p}_{\delta}(M;\Reals)
\]
we find that
\begin{align*}
&S:\left[H^{s_1,p_1}_{\delta_1}(M;E),H^{s_2,p_2}_{\delta_2}(M;E)\right]_\theta
\to \prod_{\alpha,a} H^{s,p}_{\delta}(M;\Reals),
\\
&R:\prod_{\alpha,a} H^{s,p}_{\delta}(M;\Reals)
\to \left[H^{s_1,p_1}_{\delta_1}(M;E),H^{s_2,p_2}_{\delta_2}(M;E)\right]_\theta,
\end{align*}
and that $R\circ S = \Id$.
The image of $R$ with domain $\prod_{\alpha,a} H^{s,p}_{\delta}(M;\Reals)$
is contained in $H^{s,p}_\delta(M;E)$ and hence
$[H^{s_1,p_1}_{\delta_1}(M;E),H^{s_2,p_2}_{\delta_2}(M;E)]_\theta\subset
H^{s,p}_\delta(M;E)$.  Moreover,
$$S:[H^{s_1,p_1}_{\delta_1}(M;E),H^{s_2,p_2}_{\delta_2}(M;E)]_\theta
\to \prod_{\alpha,a} H^{s,p}_{\delta}(M;\Reals)$$ is an isomorphism
onto its image since it admits a continuous inverse, the restriction
of $R$.  Hence
\[
\|u\|_{[H^{s_1,p_1}_{\delta_1}(M;E),H^{s_2,p_2}_{\delta_2}(M;E)]_\theta}
\sim\sum_{\alpha,a} \|u_{\alpha,(a)}\|_{H^{s,p}_\delta(M;\Reals)}
\sim \|u\|_{H^{s,p}_\delta(M;E)},
\]
and we conclude $[H^{s_1,p_1}_{\delta_1}(M;E),H^{s_2,p_2}_{\delta_2}(M;E)]_\theta$
is a closed subspace of $H^{s,p}_\delta(M;E)$.  
Since it contains
the dense subspace of compactly supported elements of $H^{s_0,p_0}_\delta(M;E)$
we obtain the desired equality.
\end{proof}

The limited interpolation result Lemma \ref{lem:Xinterp-early}
for Gicquaud-Sakovich spaces generalizes to geometric tensor bundles
by using the identity
\[
Tu = \sum_{\alpha_a} (\pi_{\alpha,(a)}\circ T)(u)
\]
and applying Lemma \ref{lem:Xinterp-early} to the maps
$\pi_{\alpha,(a)}\circ T$.

The only remaining task is to generalize  the duality results of Theorem \ref{thm:Hdual} , which only treated the scalar functions and only with a smooth background
metric, to tensors defined by rough metric $h$.  
We first relax the regularity of the metric.
\begin{lemma}\label{lem:relax-dual-scalar}
Suppose $(s,p)\in \mathcal S^{s_0,p_0}_0$.  The map
\[
\ip<v,u>_{(M,h)} = \int_{M} vu\;dV_h,
\]
defined for smooth compactly supported functions $u,v$ on $M$,
extends to a continuous bilinear map
$H^{-s,p\dual}_{-\delta}(M;\Reals)\times H^{s,p}_{\delta}(M;\Reals)\to \Reals$.
Moreover, the map
\[
v\mapsto \ip<v,\cdot>_{(M,h)}
\]
is a continuous isomorphism $H^{-s,p\dual}_{-\delta}(M;\Reals) \to (H^{s,p}_{\delta}(M;\Reals))\dual$.
\end{lemma}
\begin{proof}
Let $\check g$ be a smooth background metric as in Theorem \ref{thm:Hdual}.
Working in the domains of M\"obius parametrizations, using the fact that
$h$ admits a continuous compactification, one readily verifies that
$dV_h = \omega dV_{\check g}$ where $\omega\in X^{s_0,p_0}_0(M)$ and
where $\omega$ is uniformly bounded below away from zero.  Proposition \ref{prop:nonlin}
ensures that $\omega^{-1}\in X^{s_0,p_0}_0(M)$ as well.
Since
\[
\ip<\cdot,\cdot>_{(M,h)} = \ip<\cdot,\cdot\omega>_{(M,\check g)}
\]
for smooth functions, continuity of multiplication
$X^{s_0,p_0}_0(M)\times H^{s,p}_{\delta}(M)\to
H^{s,p}_{\delta}(M;\Reals)$ implies that $\ip<\cdot,\cdot>_{(M,h)}$
extends to a continuous bilinear map
$H^{-s,p\dual}_{-\delta}(M;\Reals)\times H^{s,p}_{\delta}(M;\Reals)\to \Reals$.
If $L\in (H^{s,p}_{\delta}(M))\dual$ we can find a unique
$v\in H^{-s,p\dual}_{\delta}(M)$ with $L=\ip<v,\cdot>_{(M,\check g)}$.
Since $(s,p)\in \mathcal S^{s_0,p_0}_0$ a simple computation shows
$(-s,p\dual)\in \mathcal S^{s_0,p_0}_0$ as well and hence
$v\mapsto \omega^{-1} v$ is a continuous automorphism of $H^{-s,p\dual}_{-\delta}(M)$.
But then
\[
L(\cdot) = \ip<\omega^{-1} v, \cdot \omega>_{(M,\check g)} = \ip<\omega^{-1} v,\cdot>_{(M,h)}.
\]
Since $L\mapsto v\mapsto \omega^{-1}v$ is a chain of isomorphisms, the proof is complete.
\end{proof}

Turning now to sections of geometric tensor bundles, define
\[
\ip<u,v>_{(M,h)} = \int_M \ip<u,v>_{h}\;dV_h
\]
for compactly supported sections $u,v\in H^{s_0,p_0}_0(M;E)$. 
For such sections we have
\[
\ip<u,v>_{(M,h)} = \sum_{\alpha,a,\alpha',a'} \int_M \pi_{\alpha,(a)}(u) \pi_{\alpha',(a')}(v)
\ip<\sigma_{\alpha,(a)},\sigma_{\alpha',(a')}>_{h}\; dV_h.
\]
Noting that each
$\ip<\sigma_{\alpha,(a)},\sigma_{\alpha',(a')}>_{h}\in X^{s_0,p_0}_0(M;\Reals)$
it follows from Lemma \ref{lem:relax-dual-scalar},
Propositions \ref{prop:density-gtb}
and \ref{prop:multiplication}, and the continuity of the projections that
$\ip<\cdot,\cdot>_{(M,h)}$
extends to a
continuous, bilinear functional on $H^{-s,p\dual}_{-\delta}(M;E)\times H^{s,p}_{\delta}(M;E)$.
\begin{theorem}\label{thm:dual-H-gtb}
Suppose $(s,p)\in \mathcal S^{s_0,p_0}_0$.  The map
\[
v\mapsto \ip<v,\cdot>_{(M,h)}
\]
is a continuous isomorphism $H^{-s,p\dual}_{-\delta}(M;E)\to (H^{s,p}_\delta(M;E))\dual$.
\end{theorem}
\begin{proof}
Let $L\in H^{s,p}_\delta(M;E)$. For each index $(\alpha,(a))$ Lemma \ref{lem:relax-dual-scalar}
implies there exists $v_{\alpha,(a)}\in H^{-s,p\dual}_{-\delta}(M;\Reals)$ with
$L(w\sigma_{\alpha,(a)}) = \ip<v,u>_{(M,h)}$ for all $w\in H^{s,p}_{\delta}(M;\Reals)$.
The proof of Lemma \ref{lem:bundle-split} shows that for each $a$,
\[
\chi_{(a)} u  = \sum_{\alpha} \pi_{\alpha,(a)}(u) \sigma_{\alpha,(a)}.
\]
Using the fact that on the support of $\chi_{(a)}$ the sections $\sigma_{\alpha,(a)}$
form an orthonormal frame we  then find
\begin{align*}
L(\chi_{(a)} u ) &= \sum_{\alpha} \ip<v_{\alpha,(a)},\pi_{\alpha,(a)}(u)>\\
&= \sum_{\alpha} \ip< v_{\alpha,(a)},\chi_{(a)}
\ip<u,\sigma_{\alpha,(a)}>_h>_{(M,h)}\\
&= \sum_{\alpha} \ip< v_{\alpha,(a)}\sigma_{\alpha,(a)},
\chi_{(a)}\ip<u,\sigma_{\alpha,(a)}>_h\sigma_{\alpha,(a)}>_{(M,h)}\\
&= \ip<v_{(a)},\chi_{(a)}u>_{(M,h)}
\end{align*}
where $v_{(a)}=\sum_{\alpha} v_{\alpha,(a)}\sigma_{\alpha,(a)}\in H^{-s,p\dual}_{-\delta}(M;E)$.
Summing over $a$ we find
\[
L(u) = \ip<\sum_a \chi_{(a)}v_{(a)},u>_{(M,h)}.
\]
Since $\sum_{a} v_{(a)}\in H^{-s,p\dual}_{-\delta}(M;E)$ we find
$v\mapsto \ip<v,\cdot>_{(M,h)}$ is a surjection
$H^{-s,p\dual}_{-\delta}(M;E)\to (H^{s,p}_\delta(M;E))\dual$.
The continuity of the map follows from the remarks preceding
the statement of the theorem.
To see that it has trivial kernel, suppose
$v\in H^{-s,p\dual}_{-\delta}(M;E)\neq 0$.
Appealing to Lemma \ref{lem:bundle-split} there
is are indices $(\alpha,(a))$ with
$\pi_{\alpha,(a)}(v)\neq 0$.
Choose a smooth compactly supported function $\phi$
such that $\ip<\pi_{\alpha,(a)}(v),\phi>_{(M,h)}\neq 0$.
Since $\chi_{(a)}v=\sum_{\alpha} \pi_{\alpha,(a)}(v)\sigma_{\alpha,(a)}$,
\begin{equation*}
\ip<v,\chi_{(a)}\phi\sigma_{\alpha,(a)}>_{(M,h)} =
\ip<\sum_{\alpha'} \pi_{\alpha',(a)}(v)\sigma_{\alpha',(a)}, \phi\sigma_{\alpha,(a)}>_{(M,h)}
=\ip<\pi_{\alpha,(a)}(v),\phi>_{(M,h)}\neq 0.
\qedhere
\end{equation*}
\end{proof}

\subsection{Equivalent norms}
\label{sec:equivalent-norms}
The weighted space norms depend implicitly on a number of choices:
\begin{itemize}
\item A boundary defining function $\rho$.
\item A projection from a collar neighborhood of $\partial M$ onto the boundary.
\item A choice of finitely many charts for the boundary.
\item A selection of preferred M\"obius parametrizations $\Phi_i$.
\end{itemize}
We call these collective choices \Defn{M\"obius data} and
now show that the function spaces are independent of the choice of M\"obius
data. These considerations also lead to Corollary \ref{cor:X-no-cutoff}, 
which establishes an equivalent norm for the Gicquaud-Sakovich spaces that is
comparable to the authors' original definition.
Since spaces of sections of geometric tensor bundles inherit their norms from their
ambient spaces of sections of tensor bundles, throughout this section we work
in the simpler setting of tensor bundles, and thus the regularity of a defining metric
is not a consideration.  The resulting norm equivalences extend to geometric tensor
bundles automatically.

\begin{proposition}\label{prop:H-Mobius-data} Let $E$ be a tensor bundle over $\bar M$.
For each $s,\delta\in\Reals$ and $1<p<\infty$
the space $H^{s,p}_{\delta}(M;E)$
is independent of the choice of M\"obius data and
the norms determined by two choices are equivalent.
\end{proposition}
\begin{proof}
From the density
of smooth compactly supported sections (Proposition \ref{prop:density}), to show two spaces
are the same it suffices to
show that the norms corresponding to different choices
are equivalent.

First, consider the case $s=k\in\Nats_{\ge 0}$.
Lemma \ref{lemma:cutoffequiv} and Lemma
\ref{lem:many-norms} show that any choice of $H^{k,p}_{\delta}(M)$
norm is equivalent to a $W^{k,p}_\delta(M)$ norm
from \eqref{eq:Wkpdelta}, which
depends only on a (separate, independent) choice of reference metric $\check g$
and the boundary defining function $\rho$.  Moreover, the
$W^{k,p}_\delta(M)$ norm for two choices $\rho$ and $\tilde \rho$
are equivalent because
$(\tilde\rho/\rho)^\delta$ is smooth on $\bar M$ and
hence multiplication by this function is an automorphism of
$W^{k,p}_\delta(M)$.  Thus any two choices of M\"obius data
lead to equivalent $H^{k,p}_{\delta}(M)$ norms.

Turning to the general case, the result for $s\ge 0$
follows from what we have shown for $s\in \Nats_{\ge 0}$
and the the characterization
of $H^{s,p}_\delta(M)$ as an interpolation space.
Suppose instead $s<0$.
Choosing a reference metric $\check g$ with smooth compactification,
the identification of $H^{s,p}_\delta(M)$ with $(H^{-s,p\dual}_{-\delta}(M))\dual$
via Theorem \ref{thm:dual-H-gtb} establishes the result in this case.
\end{proof}

The analogue of Proposition \ref{prop:H-Mobius-data} for Gicquaud-Sakovich
spaces is an easy consequence
of the following lemma, which is independently used in Section \ref{app:hyp-elliptic}
concerning elliptic operators on the ball model of hyperbolic space.
\begin{lemma} \label{lem:X-other-charts}
Let $E$ be a tensor bundle over $\bar M$
and let $(\tilde \Omega,\tilde \Theta)$
be a boundary chart associated with some boundary projection and
some boundary defining function $\tilde \rho$.
Suppose $s,\delta\in\Reals$ and $1<p<\infty$.
For all M\"obius parametrizations $\tilde \Phi$ associated
with $\tilde\Theta$ and all $u\in H^{s,p}_{\loc}(M)$,
\[
\tilde \rho(\tilde\Phi(0))^{-\delta}\|\chi \tilde\Phi^*u\|_{H^{s,p}(\Reals^n)} \lesssim \|u\|_{X^{s,p}_\delta(M)}
\]
with implicit constant independent of $u$ but depending on the remaining data.
\end{lemma}
\begin{proof}
Let $I$ be the set of indices $i$ such that
$\tilde \Phi(B_2) \cap \Phi_i(B_2)\neq \emptyset$.
From Lee \cite{Lee-FredholmOperators} Lemma 2.1
the $\check g$ diameter of $\tilde \Phi(B_2)$
is bounded independent of $\tilde \Phi$ and hence
from property \eqref{part:diam-bound-to-I-bound}
of the parametrizations $\Phi_i$
enumerated at the end of Section \ref{sec:coords}
we obtain a uniform upper bound for the size of $I$.

An argument essentially similar to the one
from Lemma \ref{lem:Hbuilder}, using the fact that
$\tilde \rho_0/\rho_0$ is uniformly bounded above and below,
shows that the transition Jacobians
$
D(\Phi^{-1}_i \circ \tilde \Phi)
$
admit $C^k$ bounds for each $k$ that are uniform in $i$.

Recalling the function $\overline\chi$ from equation \eqref{eq:Xdef} we have
\[
\|\chi \tilde\Phi^* (\overline{\chi} u)\|_{H^{s,p}(\Reals^n)}
=
\|\chi \tilde\Phi^* \sum_{i\in I}( (\Phi_i)_*\chi) u\|_{H^{s,p}(\Reals^n)}
\lesssim
\sum_{i\in I} \|\chi (\Phi^{-1}_i\circ \Phi)^* (\chi \Phi_i^*u)\|_{H^{s,p}(\Reals^n)}.
\]
The uniform bound for the size of $I$,
the uniform $C^k$ control of the transition functions,
and the fact that $\tilde \rho(\tilde\Phi(0))/\rho_i$ is uniformly
bounded above and below for all $i\in I$ then imply
\[
\rho(\Phi(0))^{-\delta} \|\Phi^* (\overline\chi u)\|_{H^{s,p}(B_r)}
\lesssim \|u\|_{X^{s,p}_\delta(M)}.
\]
Since $u\mapsto u/\overline\chi $ is continuous on $H^{s,p}_\delta(M)$
the proof is complete.
\end{proof}

Lemma \ref{lem:X-other-charts} immediately implies the following.
\begin{corollary} \label{cor:X-any-rho}
Let $E$ be a tensor bundle over $\bar M$.
For each
$s,\delta\in\Reals$ and $1<p<\infty$,
the space $X^{s,p}_\delta(M;E)$ is independent of the choice of M\"obius
data and any two choices lead to equivalent norms.
\end{corollary}

We also obtain from Lemma \ref{lem:X-other-charts}
the following equivalent norm for Gicquaud-Sakovich spaces
which can be compared with the original definition of `weighted local Sobolev spaces'
from \cite{GicquaudSakovich} as well as the H\"older analogue, equation (3.2)
of \cite{Lee-FredholmOperators}.

\begin{corollary} \label{cor:X-no-cutoff}
Let $E$ be a tensor bundle over $\bar M$. Suppose $s,\delta\in\Reals$
and $1<p<\infty$.
For any $1/2<r <2$ the $X^{s,p}_\delta$ norm is equivalent to
\[
\sup_{\Phi} \rho(\Phi(0))^{-\delta} \|\Phi^* u\|_{H^{s,p}(B_r)}
\]
where the supremum is taken over all M\"obius parametrizations.
\end{corollary}
\begin{proof}
One direction follows immediately from Lemma \ref{lemma:cutoffequiv}
and the converse is a consequence of
Lemma \eqref{lem:X-other-charts} and the fact that
we can take our cutoff function $\chi$ to equal 1
on $B_r^\Hyp$.
\end{proof}

\section{Elliptic theory for hyperbolic space}\label{app:hyp-elliptic}

Following \cite{Lee-FredholmOperators} Chapter 5,
we work with the Poincar\'e ball model $\HypB$ of hyperbolic
space, so $\HypB$ is the unit ball in $\Reals^{n}$
with metric
\[
g_\HypB = 4(1-|\xi|)^{-2}\sum_i (d\xi^i)^2.
\]
For $\xi, \eta \in\HypB$ we
define $\rho(\xi, \eta)=1/\cosh(d_{g_\HypB}(\xi,\eta))$ and set
\begin{equation*}
\rho(\xi)=\rho(\xi,0) = \frac{1-|\xi|^2}{1+|\xi|^2},
\end{equation*}
which is a defining function on $\mathbb B$.

Throughout this section we consider a fixed operator $\opP = \opP[g_\HypB]$
acting on a geometric tensor bundle $E\to \HypB$
and satisfying Assumption \ref{Assume-P}.  
We assume moreover that, Assumption \ref{Assume-I} holds, so that $\opP:H^{d,2}_0(\HypB;E)\to H^{0,2}_0(\HypB; E)$
is an isomorphism.  From \cite{Lee-FredholmOperators} Lemma 4.10 this
is equivalent to assuming $\opP$ satisfies
the $L^2$ estimate at infinity, namely
there exists a compact set $K\subset \HypB$ and a constant $C$
such that for all $u\in C^\infty_{\rm cpct}(\HypB\setminus K;E)$,
\begin{equation}\label{eq:est_at_inf_H}
\|u\|_{L^2(\HypB;E)} \le C \|\opP u\|_{L^2(\HypB;E)}.
\end{equation}

The symmetries of $\HypB$ allow for a convenient alternative formulation
of weighted space norms. Following the construction of \cite{Lee-FredholmOperators}
Lemma 2.2 we can find a positive integer $N$ and
a countable collection of points $p_i\in\HypB$
such that the balls $B_{1/2}^\HypB(p_i)$ cover $\HypB$
and such that each $B_2^\HypB(p_i)$ intersects at
most $N$ others.  Moreover, for any diameter $\ell>0$ there
is a bound $K_\ell$ such that if $U\subset \HypB$ has diameter at most $\ell$,
then $B_{2}^\HypB(p_i)\cap U = \emptyset$ for all but at most $K_\ell$ points $p_i$.
For each $i$ we select a hyperbolic isometry $\Gamma_i$ with $\Gamma_i(0)=p_i$.

\begin{lemma}\label{lem:alt-B-norms}
Suppose $s,\delta\in\Reals$ and $1<p<\infty$.  For each $1/2<r<2$, we
have the following norm equivalences:
\begin{enumerate}
\item
\begin{equation}\label{eq:ball-Hspdelta}
\|u\|_{H^{s,p}_\delta(\HypB;E)}^p \sim \sum_i
\rho(\Gamma_i(0))^{-\delta p}\|\Gamma_i^* u\|_{H^{s,p}(B_r^\HypB)}^p.
\end{equation}
\item\label{eq:ball-Xspdelta}
\[
\|u\|_{X^{s,p}_\delta(\HypB;E)} \sim \sup_i
\rho(\Gamma_i(0))^{-\delta}\|\Gamma_i^* u\|_{H^{s,p}(B_r^\HypB)}.
\]
\item\label{eq:ball-Xspdelta-all}
\[
\|u\|_{X^{s,p}_\delta(\HypB;E)} \sim \sup_{\Gamma}
\rho(\Gamma(0))^{-\delta}\|\Gamma^* u\|_{H^{s,p}(B_r^\HypB)},
\]
where the supremum is taken over all hyperbolic isometries of $\HypB$.
\end{enumerate}
\end{lemma}
\begin{proof}
The central issue is that the coordinate
function $y$ in the half space model $\Hyp$ of hyperbolic space does not arise from
a smooth boundary defining function on $\HypB$;
it is singular at a point of $\partial\HypB$.
To mitigate this, choose open cover $\{\Omega_{\alpha}\}_{\alpha\in \{1,2\}}$
for $\HypB$ and isometries $\Theta_\alpha:\HypB\to \Hyp$ such that $\Theta_\alpha(\Omega_\alpha)=Y_r$
for some number $r$ and such that for each
$p\in\HypB$ we can select an index $\alpha$ so that $B_2(p)\subset \Omega_\alpha$.

Fix an isometry $\Psi_\Hyp^{\HypB}:B_{2}^\Hyp\to B_2^\HypB$.  For each $\alpha$ we can select
a boundary defining function $\rho_\alpha$ that equals $\Theta_\alpha^*y$ on $\Omega_\alpha$.
Let $\Gamma:\HypB\to \HypB$ be an isometry and pick $\alpha$ so that
$\Gamma(B_2^\Hyp)\subset \Omega_\alpha$.
Choosing a suitable Euclidean rotation $R_\Gamma$ we can arrange so that
\[
(\Theta_\alpha\circ\Gamma\circ R_\Gamma\circ \Psi_\Hyp^{\HypB})(x,y) = (\theta_0+ \rho_\Gamma x,\rho_\Gamma y)
\]
for some $\theta_0$, where $\rho_\Gamma = \rho_\alpha(\Gamma(0))$.
Let $\Phi_\Gamma = \Gamma\circ R\circ \Psi$ so that in
effect, $\Phi_{\Gamma}$ is a M\"obius parametrization with respect
to the chart $\Omega_\alpha$ and the boundary defining function $\rho_\alpha$.
Hence, arguing as in  the
proof of Lemma \ref{lem:X-other-charts}, and using the fact that
$\rho/\rho_\alpha$ is uniformly bounded above and bounded
below away from zero,
we find that the Jacobians
$D(\Phi^{-1}\circ\Phi_\Gamma)$ and $D(\Phi^{-1}_\Gamma\circ \Phi)$ admit
uniform $C^k$ bounds independent of $\Phi$ and $\Gamma$.

Let $\eta$ be a cutoff function with support in $B_{r}^\Hyp$
that equals $1$ on $B_{1/2}^\Hyp$.
For each $i$, let $N(i)$ be the set of indices $j$ such that $\Gamma_j(B_{2}^\HypB)
\cap \Phi_i(B_2^\Hyp)\neq \emptyset$.  Since the sets $\Phi_i(B_2^\Hyp)$
admits an upper bound on their diameters, the size of $N(i)$ is bounded above
independent of $i$.  The function
\[
\sum_{j\in N(i)} \Phi_i^* (\Phi_{\Gamma_j})_*\eta
\]
is bounded below on $B_2^\Hyp$ and admits uniform $C^k$ bounds
independent of $i$. Hence, using the uniform $C^k$ bounds for
the transition Jacobians,
\[
\|\chi \Phi_i^* u\|_{H^{s,p}(\Reals^n)}^p
\lesssim \|\chi \Phi_i^*
(\sum_{j\in N(i)} (\Phi_{\Gamma_j})_*\eta) u\|_{H^{s,p}(\Reals^n)}^p
\lesssim \sum_{j\in N(i)} \|\eta \Phi_{\Gamma_j}^* u\|_{H^{s,p}(\Reals^n)}^p.
\]
Since $\rho(\Gamma_j(0))/\rho_i$ is uniformly bounded above and below
independent of $i$ and $j$, and using the fact that
that there is a uniform upper bound for the number of indices $i$
such that $\Gamma_j(B_2)\cap\Phi_i(B_2)\neq \emptyset$
we find
\[
\|u\|_{H^{s,p}_\delta(\mathbb B)}^p \lesssim
\sum_i \sum_{j\in N(i)}
\rho(\Gamma_j(0))^{-\delta p} \|\eta \Phi_{\Gamma_j}^* u\|_{H^{s,p}(\Reals^n)}^p
\lesssim \sum_j \rho(\Gamma_j(0))^{-\delta p}\|\Phi_{\Gamma_j}^* u\|_{H^{s,p}(B_r^\Hyp)}^p.
\]
From uniform $C^k$ estimates for $R_{\Gamma_k}\circ \Psi_\Hyp^\HypB$ we then find
\[
\|u\|_{H^{s,p}_\delta(\mathbb B)}^p \lesssim
\sum_j \rho(\Gamma_j(0))^{-\delta p}\| \Gamma_j^* u\|_{H^{s,p}(B_r^\HypB)}^p.
\]
The reverse inequality is proved by analogous arguments, which establishes
the equivalence \eqref{eq:ball-Hspdelta}.  The
equivalences \eqref{eq:ball-Xspdelta} and
\eqref{eq:ball-Xspdelta-all}
are proved with similar, easier techniques; see also
Lemma \ref{lem:X-other-charts}.
\end{proof}

The following elliptic regularity result is a straightforward
consequence of the alternative norms of Lemma \ref{lem:alt-B-norms}
and local elliptic regularity (Proposition \ref{prop:interior-reg}),
noting that the operator $\Gamma_i^*\opP$ is independent of $i$ and hence
the implicit constant from equation \eqref{eq:int-reg} is also independent of $i$.

\begin{lemma}\label{lem:hyp-elliptic-reg}
Suppose $s,s',\delta\in\Reals$ and $1<p<\infty$.
\begin{enumerate}
  \item Suppose $u\in H^{s,p}_\delta(\HypB;E)$ and $\opP u\in H^{s'-d,p}_\delta(\HypB;E)$.
Then $u\in H^{s',p\dual}_\delta(\HypB;E)$.
\item Suppose $u\in X^{s,p}_\delta(\HypB;E)$ and $\opP u\in X^{s'-d,p}_\delta(\HypB;E)$.
Then $u\in X^{s',p}_\delta(\HypB;E)$.
\end{enumerate}
\end{lemma}

The mapping properties of $\opP$ on the weighted Sobolev
spaces is a straightforward consequence interpolation and
duality arguments; see, for example \cite[Lemma 3.5]{Maxwell-RoughAE}.
We now establish the main isomorphism theorem; recall the definition of $\mathcal D_q(R)$ from \eqref{intro:define-DqR}.
\begin{theorem}\label{thm:ball-H-iso}
Suppose $s\in\Reals$, $1<p<\infty$ and that
$\delta\in\mathcal D_q(R)$, where $R$ is the indicial radius of $\mathcal P$.
Then $\opP:H^{s,p}_\delta(\HypB)\rightarrow H^{s-d,p}_{\delta}(\HypB)$
is an isomorphism.
\end{theorem}
\begin{proof}
The proof proceeds in three cases, $s\ge d$, $s\le 0$ and $s\in [0,d]$.

For $s\in\Reals$, let $\opP_s$ denote $\opP$ as
a map from $H^{s,p}_{\delta}(\HypB; E)$ to $H^{s-d,p}_{\delta}(\HypB;E)$.
Suppose $s\ge d$ and let
$k$ be the integer such that $d\le k\le s< k+1$.  From
\cite{Lee-FredholmOperators} Theorem 5.6, $\opP_{k}$ and $\opP_{k+1}$
are isomorphisms. Hence so is $\opP_s$ by interpolation.

Now suppose $s\le 0$.  Given $v\in H^{s-d,p}_{\delta}(\HypB)$
we wish to solve $\opP u =v$.  To this end we apply Theorem \ref{thm:dual-H-gtb}
with background metric $g_{\mathbb B}$ to obtain a unique
$u\in H^{s,p}_{-\delta}(\HypB)$ satisfying
\[
\ip< u, w>_{(\HypB, g_{\HypB})} = \ip< v,(\opP_{d-s})^{-1} w>_{(\HypB,g_\HypB)}
\]
for all $w\in H^{-s,p\dual}_{-\delta}(\HypB,g_{\HypB})$.
Note that in this context $\opP_{d-s}^{-1}:H^{-s,p\dual}_{-\delta}(\HypB) \to H^{d-s,p\dual}_{-\delta}(\HypB)$, 
which is well-defined by our work in the first case since $d-s\ge d$ and
since 
$-\delta \in \mathcal D_{p\dual}(R)$
if and only if 
$\delta \in \mathcal D_q(R)$.

Since $\opP_s$ is formally self-adjoint with respect to
$g^{\HypB}$, if $w$ is
any compactly supported smooth section,
\[
\ip< \opP_s u, w>_{(\HypB,g_\HypB)} =
\ip< u, \opP_{d-s} w>_{(\HypB,g_\HypB)} = \ip< v, \opP_{d-s}^{-1}\opP_{d-s} w>_{(\HypB,g_\HypB)} =
\ip<v,w>_{(\HypB,g_\HypB)}.
\]
Thus $\opP_s u = v$ in the sense of distributions, and
$\opP_s$ is surjective. That it is injective follows from the injectivity
of $\opP_{d}$ and Lemma \ref{lem:hyp-elliptic-reg}.

Finally, suppose $s\in [0,d]$. Then $\opP_0$ and $\opP_d$
are isomorphisms, and thus by interpolation $\opP_s$ is an isomorphism as well.
\end{proof}

We now turn to the analogue of Theorem \ref{thm:ball-H-iso}
for Gicquaud-Sakovich spaces, which requires additional labor
due to the absence of good interpolation/duality tools.
The following lemma from \cite{GicquaudSakovich} is the
crucial ingredient for obtaining mapping properties
on spaces $X^{s,p}_\delta(\HypB;E)$ when $s\ge m$.  In it,
$\opP^{-1}$ refers to convolution against the Greens function
of $\opP$ discussed in \cite[Chapter 5]{Lee-FredholmOperators}.

\begin{lemma}[\cite{GicquaudSakovich} Lemma A.3]\label{lem:L2-X}
Suppose $1<p<\infty$ and $\delta\in \mathcal D_\infty(R)$.
Then
\[
\|\opP^{-1} f\|_{X^{0,p}_{\delta}(\HypB;E)}  \lesssim
\|f\|_{X^{0,p}_{\delta}(\HypB;E)}
\]
for all $f\in X^{0,p}_{\delta}(M;\HypB)$.
\end{lemma}

The following discrete analogue of \cite{Lee-FredholmOperators}
Lemma 5.4 is the main tool needed for establishing mapping
properties for spaces with $s<0$. In it, the points $p_j$
are those associated with the isometries $\Gamma_j$ discussed
in Proposition \ref{lem:alt-B-norms}.
\begin{lemma}\label{lem:discrete-convolve}
Suppose $\delta$, $\delta^*$ are real numbers
such that $\delta< \delta*$ and
$\delta+\delta^*>n-1$.
For all $q\in\HypB$,
\[
\sum_j \rho(0,p_j)^{\delta}\rho(p_j,q)^{\delta^*}
\lesssim \rho(0,q)^\delta
\]
with implicit constant independent of $q$.

\end{lemma}
\begin{proof}
From local finiteness we can partition the centers $\{p_j\}$
into finitely many sets $J_1,\ldots,J_{N+1}$ such that if
$p_i,p_j\in J_k$, then $B_2(p_i)\cap B_2(p_j)=\emptyset$. It is
then enough to show that each
\[
\sum_{j\in J_k} \rho(0,p_j)^{\delta}\rho(p_j,q)^{\delta^*}
\lesssim \rho(0,q),
\]
where the implicit constant $C$ independent of $q$.

Fix one of the partitions $J_k$ and consider the function $\chi$
that equals $\rho(0,p_j)^\delta$ on $B_1^{\HypB}(p_j)$ for each $j\in J_k$
and is otherwise zero.
Define $\hat \chi$ similarly with $\hat \chi = \rho(p_j,q)$
on each $B_{1}^\HypB(p_j)$. A straightforward computation
shows there is a constant $C>0$ such that
$C\rho(x,0)\ge \chi(x)$
and $C \rho(x,q) \ge \hat \chi(x)$
for all $x\in \HypB$; indeed,
$C=\exp(1)=\sup_{x\ge 0} \cosh(x+1)/\cosh(x)$ suffices.

Now
\[
\begin{aligned}
\sum_{j\in P_k} \rho(0,p_j)^{\delta}\rho(p_j,q)^{\delta^*}
&= \frac{1}{\mathop{\rm Vol} B_1^\HypB} \int \chi\hat\chi\;dV_{g_\HypB}\\
&\le \frac{C^2}{\mathop{\rm Vol} B_1^\HypB}
\int \rho(0,x)^\delta\rho(x,q)^{\delta^*}\;dV_{g_\HypB}(x).
\end{aligned}
\]
Lemma 5.4 of \cite{Lee-FredholmOperators}
implies there is a constant $c$, independent of $q$, such
that
\[
\int \rho(0,x)^\delta\rho(x,q)^{\delta^*}\;dV_{g_\HypB}(x)
\le   c \rho(0,q)^\delta,
\]
which proves the result.
\end{proof}

\begin{theorem}\label{thm:ball-X-iso}
Suppose $s\in\Reals$, $1<p<\infty$, and $\delta\in \mathcal D_\infty(R)$.
Then
$\opP:X^{s,p}_\delta(\HypB;E)\rightarrow X^{s-d,p}_{\delta}(\HypB;E)$
is an isomorphism.
\end{theorem}
\begin{proof}
Let $\opP_s$ denote $\opP$ as a map
from $X^{s,p}_{\delta}(\HypB)$ to
$X^{s-d,p}_{\delta}(\HypB)$.
Pick $\delta'$
such that
$-R < \delta'+\frac{n-1}p -\frac{n-1}2< \delta-\frac{n-1}2< R$.
Noting that $\delta'\in\mathcal D_p(R)$, the injectivity
of $\opP_s$ follows from the embedding
of $X^{s,p}_{\delta}(\HypB)$ into $H^{s,p}_{\delta'}(\HypB)$
given by Lemma \ref{lem:basic-inclusions}
along with the injectivity of $\opP$ on this latter space.
Hence, by the Closed Graph Theorem,
to show that $\opP_s$ is an isomorphism
it suffices to show that $\opP_s$ is surjective.

Let $f\in X^{s-m,p}_{\delta}(\HypB)\subset H^{s-m,p}_{\delta'}(\HypB)$.
Theorem \ref{thm:ball-H-iso} implies there exists
$u\in H^{s,p}_{\delta'}(\HypB)$ such that $\opP u = f$.  Hence
the main issue is to prove that in fact $u\in X^{s,p}_{\delta}(\HypB)$.

Suppose first that $s\ge m$.  Then \cite{Lee-FredholmOperators}
Chapter 5
shows that  $u\in H^{m,p}_{\delta'}(\HypB)$ is obtained
by convolving $f\in H^{0,p}_{\delta'}(\Hyp)$
against the Greens function of $\opP$ and
Lemma \ref{lem:L2-X} therefore implies $u\in X^{0,p}_{\delta}(\HypB)$.
We conclude from Lemma \ref{lem:hyp-elliptic-reg}
that $u\in X^{s,p}_{\delta}(\HypB)$ and hence $\opP_s$ is surjective.

Now suppose $s\le 0$.
Given a smooth section $\phi$ of $E$ that is compactly supported
in $B_1^\HypB(0)$ we compute for any hyperbolic isometry $\Gamma$
\[
\ip< \Gamma^* u,  \phi>_{(\HypB,g_{\HypB})}
=  \ip< \opP^{-1}_s f, \Gamma_* \phi>_{(\HypB,g_{\HypB})}
= \ip< f, \opP^{-1}_{-s+m} \Gamma_*  \phi>_{(\HypB,g_{\HypB})}
= \ip< f, \Gamma_* \opP^{-1}_{-s+m} \phi>_{(\HypB,g_{\HypB})}
\]
where $\ip<\cdot,\cdot>_{(\HypB,g_{\HypB})}$ denotes duality
pairing on weighted Bessel potential spaces and where
each step is obvious when $f$ is smooth and compactly supported
and is hence justified by a density argument.

Let $\{\Gamma_i\}$ be the collection of isometries
considered in Lemma \ref{lem:alt-B-norms} and set $p_i=\Gamma_i(0)$.
Making straightforward adjustments, 
Lemma \ref{lem:Hbuilder} can be generalized to the current setting.
and using it we can then construct a partition of unity
$\{\chi_i\}$ subordinate to $\{B_1^\HypB(p_i)\}$ that is uniformly
bounded in $C^{k}_0(\HypB)$ for any choice of $k$.
Pick $\delta^*> \delta$ such that $\delta^*\in \mathcal D_\infty(R)$
and such that $\delta+\delta^*>n-1$.  Then
\[
\begin{aligned}
\ip< f, \Gamma_* \opP^{-1}_{m-s} \phi>_{(\HypB,g_{\HypB})}
&= \sum_j \ip< f, \chi_j \Gamma_* \opP^{-1}_{m-s}  \phi>_{(\HypB,g_\HypB)}\\
&= \sum_j \ip< \Gamma_j^* f,(\Gamma_j^* \chi_j) (\Gamma_j^* \Gamma_*\opP^{-1}_{m-s} \phi)>_{(\HypB,g_\HypB)}\\
&\lesssim \sum_j \|\Gamma_j^*f\|_{H^{s-m,p}(B_1^\HypB)}
\|(\Gamma_j^* \Gamma_*\opP^{-1}_{m-s} \phi)\|_{H^{m-s,p\dual}(B_1^\HypB)}\\
&\le \|f\|_{X^{s-m,p}_\delta(\HypB)} \sum_j \rho(p_j)^{\delta}
\|(\Gamma_j^* \Gamma_* \opP^{-1}_{m-s} \phi)\|_{H^{m-s,p\dual}(B_1^\HypB)}.
\end{aligned}
\]
From the alternate norm of Lemma \ref{lem:alt-B-norms},
the fact that $\rho(\cdot,\cdot)$ is invariant under
hyperbolic isometries, and the fact that
$m-s\ge m$ (and therefore $\opP_{m-s}$
is an isomorphism) we compute
\[
\begin{aligned}
\|\Gamma_j^* \Gamma_* \opP^{-1}_{m-s} \phi\|
_{H^{m-s,p\dual}(B_{1}(0;\HypB))} &\le
\|\opP^{-1}_{m-s} \phi\|_{X^{m-s,p\dual}_{\delta^*}(\HypB;E)}
\rho( (\Gamma^{-1}\circ\Gamma_j)(0),0)^{\delta^*}\\
&\le
\|\opP^{-1}_{m-s} \phi\|_{X^{m-s,p\dual}_{\delta^*}(\HypB;E)}
\rho( \Gamma_j(0),\Gamma(0))^{\delta^*}\\
&\lesssim
\|\phi\|_{X^{-s,p\dual}_{\delta^*}(\HypB;E)}
\rho(p_j,\Gamma(0))^{\delta^*}\\
&\lesssim \|\phi\|_{H^{-s,p\dual}(B_1^\HypB)}
\rho(p_j,\Gamma(0))^{\delta^*}.
\end{aligned}
\]
Combining these inequalities and applying
Lemma \ref{lem:discrete-convolve} we conclude
\[
\begin{aligned}
\ip< \Gamma^* u,  \phi>_{(\HypB,g_{\HypB})} &\lesssim
\|\phi\|_{H^{-s,p\dual}(B_1^\HypB)}
\|f\|_{X^{s-m,p}_\delta(\HypB)} \sum_j \rho(p_j)^{\delta}
\rho(p_j,\Gamma(0))^{\delta^*}\\
&\lesssim \|\phi\|_{H^{-s,p\dual}(B_1^\HypB)}
\|f\|_{X^{s-m,p}_\delta(\HypB)} \rho(0,\Gamma(0))^\delta
\end{aligned}
\]
for all smooth sections $\phi$ that are compactly supported in $B_1^\HypB$.
Hence
\[
\|\Gamma^* u\|_{H^{s,p}(B_1^\HypB)}
\lesssim \|f\|_{X^{s-m,p}_\delta(\HypB)}
\rho(0,\Gamma(0))^\delta
\]
and we conclude $u\in X^{s,p}_{\delta}(\HypB)$.  Thus
$\opP_s$ is surjective.

Finally, suppose $s\in[0,m]$.  Given $f\in X^{s-m,p}_{\delta}(\HypB)
\subset X^{-m,p}_{\delta}(\HypB)$
we can find $u\in X^{0,p}_{\delta}(\HypB)$ with $\opP_{0} u = f$.
Lemma \ref{lem:hyp-elliptic-reg}
then implies $u\in X^{0,p}_{\delta}(\HypB)$.
\end{proof}

\section{Smooth approximation of metrics in fortified spaces}
The proof of Theorem \ref{thm:X-fredholm} concerning Fredholm mapping
properties of elliptic operators
acting on Gicquaud-Sakovich spaces makes use, at least for
asympotically hyperbolic metrics having very low regularity, 
of the fact that
these metrics can be approximated in a limited sense
by smooth asymptotically hyperbolic metrics. 
This section contains the construction.

First, we show that if $u$ is a tensor of weight $w$ 
in some space $\mathscr H^{s,p;m}(M;E)$ with
parameters good enough to imply H\"older continuity on $\bar M$,
we can find a nearby smooth tensor on $\bar M$, but
only with respect to the $X^{s,p}_{w}(M;E)$ norm.
In the case of a metric $g$,
the compactified metric $\bar g$ is approximated by some
smooth $\tilde g$ in $X^{s,p}_{w}(M)$ and hence
$\rho^{-2}\tilde g$ is smooth on $M$ and is close to $g$ in $X^{s,p}_{0}(M)$.
Although the approximation is not in the topology of the fortified space, 
the result is sufficient to show that
differential operators derived from $\rho^{-2} \tilde g$ are close to those derived from $g$.

\begin{proposition}\label{prop:smooth-approx}
Suppose $1<p<\infty$, $s\in\Reals$, $m\in\Nats$,
$s>n/p$ and $s\ge m\ge 1$.
Let $E$ be a tensor bundle of weight $w$ over $M$
and suppose either
$u\in \mathscr H^{s,p;m}(M;E)$ with $m>n/p$
or $u\in \mathscr X^{s,p;m}(M;E)$.
For every $\epsilon>0$ there exists
$u_\epsilon\in  C^\infty(\bar M;E)$ such that
\[
\| u- \bar u_\epsilon\|_{X^{s,p}_{w}(M;E)} < \epsilon.
\]
\end{proposition}
\begin{proof}
For definiteness, suppose $u\in \mathscr H^{s,p;m}(M;E)$;
the proof in the Gicquaud-Sakovich case is similar.
Let $\{\hat p_j\}_{j=1}^J$ and
$\{\psi_j\}_{j=0}^J$ be the boundary points and
partition of unity of $M$ constructed by Lemma
\ref{lem:boundary-pou} for some $r\le r_0/2$ chosen
small enough so that the Lemma applies.

For $j\ge 1$ let $u^{A}_B[\hat p_j]$ be the components of $u$ at $\hat p_j$
in the local coordinates $\Theta$ associated with $\hat p_j$ and let
$u_j$ be a smooth tensor on $\bar M$ that equals
\[
u^{A}_B[\hat p_j] \partial_{\Theta^A}d\Theta^B
\]
on $Z_{r_0/2}(\hat p_j)$ and that vanishes outside $Z_{r_0}(p_j)$.
From the boundary continuity of $u$ we can arrange to do this so
\begin{equation}\label{eq:tilde-u-control}
\|u_j\|_{\mathscr H^{s,p;m}(M)} \lesssim \|u\|_{\mathscr H^{s,p;m}(M)}.
\end{equation}

Pick a smooth compactly supported tensor $v_0$ on $M$
such that $\|\psi_0 u-v_0\|_{X^{s,p}_0}\le r^\alpha$
and define $v = \sum_{j=0}^J \psi_j u_j$.
We claim that
\[
\|\bar u - v\|_{X^{s,p}_w(M)}  \lesssim r^\alpha
\]
with implicit constant independent of $r$.
Since $v$ has the desired smoothness,
we are done once the claim is demonstrated.

Consider a M\"obius parametrization $\Phi_i$
and let $\mathcal N(i)$ be the set of indices $j$ such that $v_j$ is not identically 0
on the image of $\Phi_i$.  The uniform local finiteness
provided by Lemma \ref{lem:boundary-pou} ensures that the size of $\mathcal N(i)$
is uniformly bounded with respect to $r$ and hence
\begin{equation}\label{eq:limited-smooth}
\begin{aligned}
\rho_i^{-w}\|\Phi^*_i( u-v)\|_{H^{s,p}(B_2^\Hyp)}
&\lesssim \sum_{j\in\mathcal N(i)}\rho_i^{-w}\|\Phi_i^*(\psi_j u - v_j)\|_{H^{s,p}(B_2^\Hyp)}\\
&\lesssim \max_j \|\psi_j u-v_j\|_{X^{s,p}_{w}(M)}.
\end{aligned}
\end{equation}

For $i\ge 1$, Lemma \ref{lem:r_rescale} and the
observation \eqref{eq:tilde-u-control} imply
\[
\left|\left| \psi_i  u - v_i\right|\right|_{X^{s,p}_w(M)}
=\left|\left| \psi_i ( u - u_i)\right|\right|_{X^{s,p}_w(M)}
\lesssim r^\alpha \| u\|_{\mathscr H^{s,p;m}(M)} .
\]
A similar inequality holds for $i=0$ by construction, and the result follows
from taking a supremum in inequality \eqref{eq:limited-smooth}.
\end{proof}

\begin{corollary}\label{cor:metric-approx}
Suppose $1<p<\infty$, $s\in\Reals$, $m\in\Nats$, $s>n/p$ and $s\ge m$.
Suppose $g$ is
an asymptotically hyperbolic metric
on $M$ of class $\mathscr H^{s,p;m}$ with $m>n/p$ or of class
$\mathscr X^{s,p;m}$.  Given $\epsilon>0$
there is an asymptotically hyperbolic metric
$g_\epsilon$ with $\rho^2 g_\epsilon\in C^\infty(\bar M)$
and
\[
\|g- g_\epsilon\|_{X^{s,p}_0(M;T^{0,2}M)} < \epsilon.
\]
\end{corollary}
\begin{proof}
We assume that $g$ is of class $\mathscr H^{s,p;m}$;
the $\mathscr X^{s,p;m}$ case is proved similarly.

Let $\bar g^{-1}$ be the contravariant tensor dual to $\bar g=\rho^2 g$
and recall, as discussed in \S\ref{sec:ah-metrics}, that
 $\bar g^{-1}\in \mathscr H^{s,p;m}(M)$.
Let $\tilde g^{-1}$ be an approximating
tensor obtained by the construction of the proof of
Proposition \ref{prop:smooth-approx}.  So in a neighborhood of
$\partial M$, $\tilde g^{-1} = \sum_{j=1}^J \psi_j \bar g^{-1}[\hat p_j]$
for some collection of boundary points $\hat p_j$
and partition of unity $\{\psi_j\}$
provided by Lemma \ref{lem:boundary-pou}.
At each $\hat p_j$, $\bar g^{-1}(d\rho,d\rho)=1$
and from the partition of unity construction it follows that
$\tilde g^{-1}(d\rho,d\rho)=1$ along all of $\partial M$.
Thus $\rho^2\tilde g^{-1}$ is dual to an asymptotically hyperbolic metric
with smooth compactification
and can be made as close to $g^{-1}$ as we want in $X^{s,p}_0(M)$.
 The result now follows, noting from the technique
of Lemma \ref{lem:geometric-op-mapping}
that the map between metrics and their duals is continuous
on $X^{s,p}_0(M)$.
\end{proof}

\section{Improved regularity for Laplacians}\label{app:Lap-extra}

The general Fredholm mapping properties from Theorems \eqref{thm:H-fredholm}
and \eqref{thm:X-fredholm} can be improved in some settings with
respect to the Sobolev parameters.  Consider the case of an asymptotically
hyperbolic metric $g$ of class $\mathscr H^{1,p;1}$ with $p>n$ and its
conformal Laplacian $L = -a_n \Delta_{g}+\R[g]$.
Proposition \ref{prop:L-mapping-S} indicates that the operator
is continuous $H^{\sigma,q}_{\delta}(M;\mathbb R)\to H^{\sigma-2,q}_{\delta}(M;\mathbb R)$
so long as $\sigma=1$ and $1/p\le 1/q \le 1/p\dual$.
The source of the very strong restriction on $\sigma$ is the scalar curvature term.
First, since $\R[g]\in H^{-1,p}_{\loc}(M;\mathbb R)$, in order for $\R[g]u$ to be well defined,
$u$ needs at least $1$ derivative and therefore $\sigma \ge 1$.  On the other hand,
we cannot expect $L u$ to have more derivatives than the low order term,
which leads to the condition $\sigma-2\ge -1$ and hence $\sigma\ge 1$ as well.

For the standard Laplacian without the scalar curvature term the mapping range can be improved.
Indeed for a metric of class $\mathscr H^{s,p;m}$ with $s>n/p$ and $s\ge m\ge 1$
a computation using Proposition \ref{prop:multiplication} shows that
$\Delta_{g} : H^{\sigma,q}_{\delta}(M;\mathbb R) \to H^{\sigma-2,q}_{\delta}(M;\mathbb R)$
is continuous so long as
\begin{equation}\label{eq:S-lap}
\begin{aligned}
\sigma&\in [2-s,s+1]\\
\frac 1q -\frac\sigma{n} &\in \left[ \frac{1}{p} - \frac{s+1}{n},
\frac{1}{p\dual}-\frac{2-s}{n}\right]
\end{aligned}
\end{equation}
which should be compared with the standard restrictions \eqref{eq:S-conds}
defining $\mathcal S^{s,p}_2$. Whereas for a general
second order operator the most regular allowable spaces are $H^{s,p}_{\delta}(M;\mathbb R)$,
for the Laplacian they are $H^{s+1,p}_{\delta}(M;\mathbb R)$.  This distinction can matter
in applications, for example, where conformal factors obtained using the Laplacian
are required to have more regularity than the metric itself.

A Fredholm theory is possible for the Laplacian
where the hypothesis
$(\sigma,q)\in S^{s,p}_2$ in Theorems \ref{thm:H-fredholm}
and \ref{thm:X-fredholm} is replaced with $(\sigma,q)$ satisfying
\eqref{eq:S-lap}.  To see this , we first observe that the only role
that $S^{s,p}_2$ plays in Section \ref{sec:differential-ops}
is to ensure continuity of the differential operator as a map
$H^{\sigma,q}_{\delta}(M)\to H^{\sigma-2,q}_{\delta}(M)$.
With a missing low-order term this range can be increased to
$(\sigma,q)$ satisfying \eqref{eq:S-lap} instead; the proofs of the results of
Section \ref{sec:differential-ops} with the weaker hypothesis on $(\sigma,q)$
then go through unchanged.  In particular, we obtain
a variation of Proposition \ref{prop:semi-fred-estimate}
specific to the Laplacian by replacing the condition
$(\sigma,q)\in S^{s,p}_2$ with $(\sigma,q)$
satisfying \eqref{eq:S-lap}, nothing that in this case the
indicial radius $R$ appearing in the Proposition is that of the Laplacian
on hyperbolic space, $(n-1)/2$.
With this semi-Fredholm result in hand, we
obtain the following Fredholm theorem with the larger range of parameters
\eqref{eq:S-lap} by using a variation of the
approximation technique from Theorem \ref{thm:X-fredholm}.

\begin{proposition}\label{prop:lap-extra}
Suppose $1<p<\infty$, $s\in\Reals$, $m\in\Nats$, with $s>n/p$ and $s\ge m \ge 1$.
Let $g$ be an asymptotically hyperbolic metric on $M$ that is
either
\begin{itemize}
  \item of class $\mathscr H^{s,p;m}$ with $m>n/p$ or
  \item of class $\mathscr X^{s,p;m}$.
\end{itemize}

Then $-\Delta_g$ is an isomorphism
\[
H^{\sigma,q}_{\delta}(M;\mathbb R)\to H^{\sigma-2,q}_{\delta}(M;\mathbb R)
\]
so long as $(\sigma,q)$ satisfy \eqref{eq:S-lap} and
\begin{equation}\label{eq:Delta-delta-H}
\left|\delta+\frac{n-1}{q}-\frac{n-1}{2}\right| < \frac{n-1}{2}.
\end{equation}
Similarly, $-\Delta_g$ is an isomorphism
\[
X^{s+1,p}_{\delta}(M;\mathbb R)\to X^{s-1,p}_{\delta}(M;\mathbb R)
\]
so long as $(\sigma,q)$ satisfy \eqref{eq:S-lap} and
\[
\left|\delta-\frac{n-1}{2}\right| < \frac{n-1}{2}.
\]
\end{proposition}
\begin{proof}
Fix some $(\sigma,q)$ satisfying \eqref{eq:S-lap} and some
$\delta$ satisfying \eqref{eq:Delta-delta-H}.
As in the proof of Theorem \ref{thm:X-fredholm}, let $g_n$
be a sequence of metrics obtained from Corollary \ref{cor:metric-approx}
with compactifications $\bar g_n\in C^\infty(\bar M)$
that converge to $g$ in $X^{s,p}_0(M)$. As discussed in the proof
of Proposition \ref{prop:Lap-Fredholm-range}, the indicial radius
of the Laplacian on hyperbolic space is $(n-1)/2$ and hence for the
smooth metrics $g_n$,
$\Delta_{g_n}:H^{\sigma,q}_{\delta}(M)\to H^{\sigma-2,q}_{\delta}(M)$
is an isomorphism.  The operators $\Delta_{g_n}$ converge to $\Delta_g$,
which is semi-Fredholm by the variation of Proposition \ref{prop:semi-fred-estimate}
discussed above. Thus $\Delta_g:H^{\sigma,q}_{\delta}(M)\to H^{\sigma-2,q}_{\delta}(M)$
is Fredholm with index zero.

It remains to show that the kernel is trivial.  But if $u\in H^{\sigma,q}_{\delta}(M)$
is in the kernel and if $(\sigma,q)$ does not already satisfy $(\sigma,q)\in S^{p,s}_2$
then either by lowering $\sigma$ or by applying Lemma \ref{lem:basic-inclusions}
and lowering $q$ while simultaneously lowering $\delta$ suitably
we have $u\in H^{\sigma',q'}_{\delta'}(M)$ with $(\sigma',q')\in \mathcal S^{s,p}_2$ and
$|\delta'+(n-1)/q'-(n-1)/2|<(n-1)/2$.  At this point the parameters are within the
range covered by Proposition \ref{prop:Lap-Fredholm-range} and we conclude $u=0$.

The Gicquad-Sakovich case is proved similarly.
\end{proof}

Results analogous to Proposition \ref{prop:lap-extra} can be proved by similar techniques
for other second-order operators with either a missing low order term
such as the vector Laplacian or with a manifestly smooth low-order term
such as $-\Delta_g +\Lambda$ with $\Lambda$ a constant.

\bibliographystyle{amsalpha}
\bibliography{RoughAHManual,Rough-AH-Maxwell,RoughSF}

\end{document}